\documentclass[12pt]{amsart}

\usepackage{etoolbox}

\makeatletter
\let\old@tocline\@tocline
\let\section@tocline\@tocline
\newcommand{\subsection@dotsep}{4.5}
\newcommand{\subsubsection@dotsep}{4.5}
\patchcmd{\@tocline}
  {\hfil}
  {\nobreak
     \leaders\hbox{$\m@th
        \mkern \subsection@dotsep mu\hbox{.}\mkern \subsection@dotsep mu$}\hfill
     \nobreak}{}{}
\let\subsection@tocline\@tocline
\let\@tocline\old@tocline

\patchcmd{\@tocline}
  {\hfil}
  {\nobreak
     \leaders\hbox{$\m@th
        \mkern \subsubsection@dotsep mu\hbox{.}\mkern \subsubsection@dotsep mu$}\hfill
     \nobreak}{}{}
\let\subsubsection@tocline\@tocline
\let\@tocline\old@tocline

\let\old@l@subsection\l@subsection
\let\old@l@subsubsection\l@subsubsection

\def\@tocwriteb#1#2#3{%
  \begingroup
    \@xp\def\csname #2@tocline\endcsname##1##2##3##4##5##6{%
      \ifnum##1>\c@tocdepth
      \else \sbox\z@{##5\let\indentlabel\@tochangmeasure##6}\fi}%
    \csname l@#2\endcsname{#1{\csname#2name\endcsname}{\@secnumber}{}}%
  \endgroup
  \addcontentsline{toc}{#2}%
    {\protect#1{\csname#2name\endcsname}{\@secnumber}{#3}}}%

\newlength{\@tocsectionindent}
\newlength{\@tocsubsectionindent}
\newlength{\@tocsubsubsectionindent}
\newlength{\@tocsectionnumwidth}
\newlength{\@tocsubsectionnumwidth}
\newlength{\@tocsubsubsectionnumwidth}
\newcommand{\settocsectionnumwidth}[1]{\setlength{\@tocsectionnumwidth}{#1}}
\newcommand{\settocsubsectionnumwidth}[1]{\setlength{\@tocsubsectionnumwidth}{#1}}
\newcommand{\settocsubsubsectionnumwidth}[1]{\setlength{\@tocsubsubsectionnumwidth}{#1}}
\newcommand{\settocsectionindent}[1]{\setlength{\@tocsectionindent}{#1}}
\newcommand{\settocsubsectionindent}[1]{\setlength{\@tocsubsectionindent}{#1}}
\newcommand{\settocsubsubsectionindent}[1]{\setlength{\@tocsubsubsectionindent}{#1}}

\renewcommand{\l@section}{\section@tocline{1}{\@tocsectionvskip}{\@tocsectionindent}{}{\@tocsectionformat}}%
\renewcommand{\l@subsection}{\subsection@tocline{1}{\@tocsubsectionvskip}{\@tocsubsectionindent}{}{\@tocsubsectionformat}}%
\renewcommand{\l@subsubsection}{\subsubsection@tocline{1}{\@tocsubsubsectionvskip}{\@tocsubsubsectionindent}{}{\@tocsubsubsectionformat}}%
\newcommand{\@tocsectionformat}{}
\newcommand{\@tocsubsectionformat}{}
\newcommand{\@tocsubsubsectionformat}{}
\expandafter\def\csname toc@1format\endcsname{\@tocsectionformat}
\expandafter\def\csname toc@2format\endcsname{\@tocsubsectionformat}
\expandafter\def\csname toc@3format\endcsname{\@tocsubsubsectionformat}
\newcommand{\settocsectionformat}[1]{\renewcommand{\@tocsectionformat}{#1}}
\newcommand{\settocsubsectionformat}[1]{\renewcommand{\@tocsubsectionformat}{#1}}
\newcommand{\settocsubsubsectionformat}[1]{\renewcommand{\@tocsubsubsectionformat}{#1}}
\newlength{\@tocsectionvskip}
\newcommand{\settocsectionvskip}[1]{\setlength{\@tocsectionvskip}{#1}}
\newlength{\@tocsubsectionvskip}
\newcommand{\settocsubsectionvskip}[1]{\setlength{\@tocsubsectionvskip}{#1}}
\newlength{\@tocsubsubsectionvskip}
\newcommand{\settocsubsubsectionvskip}[1]{\setlength{\@tocsubsubsectionvskip}{#1}}

\patchcmd{\tocsection}{\indentlabel}{\makebox[\@tocsectionnumwidth][l]}{}{}
\patchcmd{\tocsubsection}{\indentlabel}{\makebox[\@tocsubsectionnumwidth][l]}{}{}
\patchcmd{\tocsubsubsection}{\indentlabel}{\makebox[\@tocsubsubsectionnumwidth][l]}{}{}

\newcommand{\@sectypepnumformat}{}
\renewcommand{\contentsline}[1]{%
  \expandafter\let\expandafter\@sectypepnumformat\csname @toc#1pnumformat\endcsname%
  \csname l@#1\endcsname}
\newcommand{\@tocsectionpnumformat}{}
\newcommand{\@tocsubsectionpnumformat}{}
\newcommand{\@tocsubsubsectionpnumformat}{}
\newcommand{\setsectionpnumformat}[1]{\renewcommand{\@tocsectionpnumformat}{#1}}
\newcommand{\setsubsectionpnumformat}[1]{\renewcommand{\@tocsubsectionpnumformat}{#1}}
\newcommand{\setsubsubsectionpnumformat}[1]{\renewcommand{\@tocsubsubsectionpnumformat}{#1}}
\renewcommand{\@tocpagenum}[1]{%
  \hfill {\mdseries\@sectypepnumformat #1}}

\let\oldappendix\appendix
\renewcommand{\appendix}{%
  \leavevmode\oldappendix%
  \addtocontents{toc}{%
    \protect\settowidth{\protect\@tocsectionnumwidth}{\protect\@tocsectionformat\sectionname\space}%
    \protect\addtolength{\protect\@tocsectionnumwidth}{2em}}%
}
\makeatother

\makeatletter
\settocsectionnumwidth{2em}
\settocsubsectionnumwidth{2.5em}
\settocsubsubsectionnumwidth{3em}
\settocsectionindent{1pc}%
\settocsubsectionindent{\dimexpr\@tocsectionindent+\@tocsectionnumwidth}%
\settocsubsubsectionindent{\dimexpr\@tocsubsectionindent+\@tocsubsectionnumwidth}%
\makeatother

\settocsectionvskip{10pt}
\settocsubsectionvskip{0pt}
\settocsubsubsectionvskip{0pt}

\settocsectionformat{\scshape}
\settocsubsectionformat{\mdseries}
\settocsubsubsectionformat{\mdseries}
\setsectionpnumformat{\scshape}
\setsubsectionpnumformat{\mdseries}
\setsubsubsectionpnumformat{\mdseries}

\let\oldtableofcontents\tableofcontents
\renewcommand{\tableofcontents}{%
  \vspace*{-\linespacing}
  \oldtableofcontents}

\usepackage{graphicx}
\usepackage{amssymb}
\usepackage{amsmath, amscd}
\usepackage[boxsize=5pt]{ytableau}

\usepackage[bookmarks=true, bookmarksopen=true,%
bookmarksdepth=3,bookmarksopenlevel=2,%
colorlinks=true,%
linkcolor=blue,%
citecolor=blue,%
filecolor=blue,%
menucolor=blue,%
urlcolor=blue]{hyperref}

\usepackage{enumitem, verbatim}
\usepackage{float}
\usepackage{physics}
\usepackage{tikz-cd, tikz-3dplot}
\usetikzlibrary{patterns, knots, calc, math, arrows.meta, decorations, decorations.markings, shapes.misc}
\tikzset{anchorbase/.style={baseline={([yshift=-0.5ex]current bounding box.center)}}}
\tikzstyle directed=[postaction={decorate,decoration={markings,
    mark=at position #1 with {\arrow{>}}}}]
\tikzset{cross/.style={cross out, draw=black, minimum size=2*(#1-\pgflinewidth), inner sep=0pt, outer sep=0pt},
cross/.default={1pt}}
\tikzset{
    partial ellipse/.style args={#1:#2:#3}{
        insert path={+ (#1:#3) arc (#1:#2:#3)}
    }
}

\usepackage[draft]{say}

\newtheorem{thm}{Theorem}[section]

\newtheorem{cor}[thm]{Corollary}

\newtheorem{theorem}[thm]{Theorem}
\newtheorem{proposition}[thm]{Proposition}
\newtheorem{corollary}[thm]{Corollary}
\newtheorem{lemma}[thm]{Lemma}

\theoremstyle{definition}
\newtheorem{definition}[thm]{Definition}

\newtheorem{defn}[thm]{Definition}

\theoremstyle{remark}
\newtheorem{remark}[thm]{Remark}
\newtheorem{rmk}[thm]{Remark}
\newtheorem{example}[thm]{Example}
\newtheorem{eg}[thm]{Example}
\newtheorem{question}[thm]{Question}

\numberwithin{equation}{section}

\newcommand{\R}{{\mathbb{R}}}
\newcommand{\C}{{\mathbb{C}}}
\newcommand{\Q}{{\mathbb{Q}}}
\newcommand{\Z}{{\mathbb{Z}}}

\newcommand{\area}{\operatorname{area}}

\newcommand{\Sk}{\operatorname{Sk}}
\newcommand{\SkAlg}{\operatorname{Sk}}

\newcommand{\Loc}{\operatorname{Loc}}

\usepackage[margin=1in,marginparwidth=0.8in, marginparsep=0.1in]{geometry}

\begin{document}

\title{Skein traces from curve counting}

\author{Tobias Ekholm}
\address{Department of mathematics, Uppsala University, Box 480, 751 06 Uppsala, Sweden \and
Institut Mittag-Leffler, Aurav 17, 182 60 Djursholm, Sweden}
\email{tobias.ekholm@math.uu.se}

\author{Pietro Longhi}
\address{Department of Physics and Astronomy \and Department of Mathematics, Uppsala University, Box 516, 751 20 Uppsala, Sweden}
\email{pietro.longhi@physics.uu.se}

\author{Sunghyuk Park}
\address{Department of Mathematics \and Center of Mathematical Sciences and Applications, Harvard University, Cambridge, MA 02138, USA}
\email{sunghyukpark@math.harvard.edu}

\author{Vivek Shende}
\address{Center for Quantum Mathematics, Syddansk Univ., Campusvej 55
5230 Odense Denmark \and 
Department of mathematics, UC Berkeley, 970 Evans Hall,
Berkeley CA 94720 USA}
\email{vivek.vijay.shende@gmail.com}

\thanks{TE is supported by the Knut and Alice Wallenberg Foundation, KAW2020.0307 Wallenberg Scholar and by the Swedish Research Council, VR 2022-06593, Centre of Excellence in Geometry and Physics at Uppsala University and VR 2024-04417, project grant. \\ \indent
The work of P.L. is supported by the Knut and Alice Wallenberg Foundation, KAW2020.0307 Wallenberg Scholar and by the Swedish Research Council, VR 2022-06593, Centre of Excellence in Geometry and Physics at Uppsala University, and he records the preprint number UUITP-30/25. \\ \indent
S.P. gratefully acknowledges support from Simons Foundation through Simons Collaboration on Global Categorical Symmetries. \\ \indent
V.S. is supported by  Villum Fonden Villum Investigator grant 37814, Novo Nordisk Foundation grant NNF20OC0066298, and Danish National Research Foundation grant DNRF157. 
}

\begin{abstract}
Given a 3-manifold $M$, and a branched cover arising from the projection of a Lagrangian 3-manifold $L$ in the cotangent bundle of $M$ to the zero-section,  
we define a map from the skein of $M$ to the skein of $L$, via the skein-valued counting of holomorphic curves. 
When $M$ and $L$ are products of surfaces and 
intervals, we show that wall crossings in the space of the branched covers obey a skein-valued lift of the Kontsevich-Soibelman wall-crossing formula.   

Holomorphic curves in cotangent bundles correspond to Morse flow graphs; in the case of branched double covers, this allows us to give an explicit formula for the the skein trace.  
After specializing to the case where $M$ is a surface times an interval, and additionally specializing the HOMFLYPT skein to the
$\mathfrak{gl}(2)$ skein on $M$ and the $\mathfrak{gl}(1)$ skein on $L$, we recover an existing prescription of Neitzke and Yan. 
\end{abstract}

\maketitle
\thispagestyle{empty}

\vspace{3mm}
\renewcommand{\contentsname}{}
\tableofcontents

\section{Introduction}\label{Intro}
The use of what would later be called `cluster coordinates' on the moduli space of $\mathrm{SL}(2, \C)$ local systems on a Riemann surface $S$ dates back at least to the work of Penner \cite{Penner}; 
it is by now a part of the `higher Teichm\"uller theory' of Fock and Goncharov \cite{Fock-Goncharov-moduli}, which treats local systems of arbitrary rank $n$. 
It was subsequently discovered by Gaiotto, Moore, and Neitzke \cite{Gaiotto-Moore-Neitzke-WKB, Gaiotto-Moore-Neitzke-spectral} that the resulting coordinate charts can be identified with the moduli of abelian local systems on a spectral curve $\Sigma \subset T^* S$, whose projection to the zero section is a degree $n$ branched cover $\pi\colon \Sigma\to S$. 
The coordinate systems have at least two interpretations, indicated already in \cite{Gaiotto-Moore-Neitzke-WKB}, either as governing Stokes phenomena in the exact WKB analysis of the Schr\"odinger equation on a Riemann surface \cite{Iwaki-Nakanishi, Allegretti-voros, Kuwagaki-WKB}, or in terms of certain moduli of objects in the Fukaya category \cite{Shende-Treumann-Williams-Zaslow}, especially as calculated in the flow-tree limit \cite{Ekholm-morse} by \cite{Nho-spectral, Casals-Nho-spectral}.

Fock and Goncharov have shown that cluster varieties in general can be $q$-deformed by $q$-deforming their charts \cite{Fock-Goncharov-quantization-dilogarithm}.  On the other hand,  Turaev had previously established that moduli of rank $n$ local systems on $S$ admit $q$-deformations to the $\mathfrak{gl}(n)$ skein modules \cite{Turaev-quantization}.
Here, we recall that given a 3-manifold $M$, its HOMFLYPT skein $\Sk_{a, z}(M)$ is the $\Z[a^\pm, z^\pm]$-module generated by framed links in $M$, modulo the following relations: 
\begin{align}
\vcenter{\hbox{
\begin{tikzpicture}[scale=0.7]
\draw[dotted] (0,0) circle (1);
\draw[ultra thick, ->] ({sqrt(2)/2},{-sqrt(2)/2}) -- ({-sqrt(2)/2},{sqrt(2)/2});
\draw[white, line width=2.5mm] ({-sqrt(2)/2},{-sqrt(2)/2}) -- ({sqrt(2)/2},{sqrt(2)/2});
\draw[ultra thick, ->] ({-sqrt(2)/2},{-sqrt(2)/2}) -- ({sqrt(2)/2},{sqrt(2)/2});
\end{tikzpicture}
}}
\;\;-\;\;
\vcenter{\hbox{
\begin{tikzpicture}[scale=0.7]
\draw[dotted] (0,0) circle (1);
\draw[ultra thick, ->] ({-sqrt(2)/2},{-sqrt(2)/2}) -- ({sqrt(2)/2},{sqrt(2)/2});
\draw[white, line width=2.5mm] ({sqrt(2)/2},{-sqrt(2)/2}) -- ({-sqrt(2)/2},{sqrt(2)/2});
\draw[ultra thick, ->] ({sqrt(2)/2},{-sqrt(2)/2}) -- ({-sqrt(2)/2},{sqrt(2)/2});
\end{tikzpicture}
}}
\;\;&=\;\;
z\;
\vcenter{\hbox{
\begin{tikzpicture}[scale=0.7]
\draw[dotted] (0,0) circle (1);
\draw[ultra thick, <-] ({sqrt(2)/2},{sqrt(2)/2}) arc (135:225:1);
\draw[ultra thick, ->] ({-sqrt(2)/2},{-sqrt(2)/2}) arc (-45:45:1);
\end{tikzpicture}
}}
\;, \label{eq:skeinrel1}
\\
a
\vcenter{\hbox{
\begin{tikzpicture}[scale=0.7]
\draw[dotted] (0,0) circle (1);
\end{tikzpicture}
}}
\;-\;
a^{-1}
\vcenter{\hbox{
\begin{tikzpicture}[scale=0.7]
\draw[dotted] (0,0) circle (1);
\end{tikzpicture}
}}
\;\;&=\;\;
z\;\vcenter{\hbox{
\begin{tikzpicture}[scale=0.7]
\draw[dotted] (0,0) circle (1);
\draw[ultra thick, ->] (0.5,0) arc (0:370:0.5);
\end{tikzpicture}
}}
\;, \label{eq:skeinrel2}
\\
\vcenter{\hbox{
\begin{tikzpicture}[scale=0.25]
\draw[dotted] (0, 0) circle (3);
\draw [ultra thick] (1,-1) to [out=180,in=-90] (0,0);
\draw [ultra thick, ->] (0,0) -- (0,3);
\draw [ultra thick] (1,1) to [out=0,in=90] (2,0) to [out=-90,in=0] (1,-1);
\draw [white, line width=2.5mm] (0,-3) to [out=90,in=-90] (0,0) to [out=90,in=180] (1,1);
\draw [ultra thick] (0,-3) to [out=90,in=-90] (0,0) to [out=90,in=180] (1,1);
\end{tikzpicture}
}}
\;\;&=\;\;
a\;
\vcenter{\hbox{
\begin{tikzpicture}[scale=0.25]
\draw[dotted] (0, 0) circle (3);
\draw[ultra thick, <-] (0, 3) -- (0, -3);
\end{tikzpicture}
}}
\;. \label{eq:skeinrel3}
\end{align}
We will always set $z = q^{1/2} - q^{-1/2}$ and extend scalars to $\Z[q^{\pm 1/2}, (1-q)^{-1}, (1-q^2)^{-1}, \cdots]$. 
The $\mathfrak{gl}(n)$ skein modules are the corresponding $\Z[q^{\pm 1/2}]$-modules arising from the specialization of the above relations at $a = q^{n/2}$.\footnote{To be precise, the $\mathfrak{gl}(n)$ skein modules are quotients of such specializations that kill, e.g., all the $\lambda$-colored strands, for any Young diagram $\lambda$ with more than $n$ rows. 
}

For the Fock-Goncharov deformation to agree with that of Turaev, there must be corresponding maps from the $\mathfrak{gl}(n)$ skein of $S$ to the $\mathfrak{gl}(1)$ skein of $\Sigma$, as indeed Fock conjectured explicitly \cite{Fock-dual}.  
Such a map was constructed by Bonahon and Wong \cite{Bonahon-Wong-trace} for $n=2$ and termed a {\em quantum trace}; their methods have been pursued further by various authors \cite{TTQLe-trace, Panitch-Park, Garoufalidis-Yu-trace}.   Neitzke and Yan \cite{Neitzke-Yan-nonabelianization, Neitzke-Yan-gl3}, building on earlier work of Gabella \cite{Gabella}, showed that a $q$-deformation of the Gaiotto-Moore-Neitzke formulas would also provide such maps; this construction was  later shown to recover the Bonahon-Wong map  \cite{Panitch-Park-compatibility}. 
We will give in Theorem \ref{thm:UV-IRmap smooth} a generalization of the Neitzke-Yan formulas, both to more general 3-manifolds (rather than surface times interval) and to the HOMFLYPT skein on both base and cover (rather than $\mathfrak{gl}(2)$ and $\mathfrak{gl}(1)$ skeins, respectively).  

\vspace{2mm}
Our main purpose  in the present article is to give a geometric construction of these and more general skein traces, by counting holomorphic curves.  

Let us recall from \cite{SOB, ghost, bare} that counting holomorphic curves with (Maslov zero) Lagrangian boundary conditions in Calabi-Yau 3-folds by the class of their boundary in the HOMFLYPT skein modules of the Lagrangians lead to invariant quantities. 
This is because it is possible to match the wall-crossings in parameter space corresponding to boundaries of moduli space with the HOMFLYPT skein relations.

Consider a 3-manifold $M$, possibly noncompact with cylindrical ends $\partial_\infty M \times \R_{>0}$.  
We fix $L \subset T^*M$; if $M$ is noncompact, we ask that $L$ is cylindrical in the cylindrical end, i.e., there is a Lagrangian $\partial_\infty L \subset T^* \partial_\infty M$ and, in the cylindrical end, $L$ takes the form $\partial_\infty L \times \R_{> 0} \subset T^*{\partial_\infty M} \times T^*\R_{> 0}$.

Consider a knot or link $K \subset M$ and its conormal $N^*K\subset T^\ast M$. We identify $M$ with the $0$-section in $T^\ast M$, $M\subset T^\ast$; then $N^*K \cap M = K$. 
We may shift $N^*K$ along the $1$-form dual to the tangent of $K$. This is a (non-Hamiltonian) symplectic isotopy which takes $N^\ast K$ to a non-exact Lagrangian $N^\circ K$ off the $0$-section, $N^\circ K\cap M=\emptyset$. 
In the geometry $(T^*M; M \sqcup  N^\circ K)$, there is a single bounded holomorphic curve, namely the trace of $K$ along the isotopy. 
(When $M=S^3$, this geometry motivated the Ooguri-Vafa conjecture expressing the HOMFLYPT invariant of a knot in as a count of curves in the holomorphic curves in the resolved conifold \cite{OV}, and its proof \cite{SOB, ekholm-shende-colored}.) 
We may shift further so that $N^\circ K$ sits {\em outside} the unit cotangent bundle; in particular, so that it is disjoint from $L$.  

After such further shift, we count curves in the geometry $(T^*M; L \sqcup N^\circ K)$. 
In particular, if we restrict attention to those which go once positively 
around the longitude $K\subset N^\circ K$ (or count all curves and take the appropriate coefficient with respect to a basis of $\Sk(N^\circ K)$, see Section \ref{basic classes}) we obtain an element of $\Sk(L)$. 

\begin{definition} \label{main definition}
We write $[K]_L \in \Sk(L)$ for the  count of curves in $(T^*M; L \sqcup N^\circ K)$
which have boundary going once positively around the longitude $k$ of the solid torus $N^\circ K$. 
\end{definition} 

\begin{remark}
This formulation is similar to the proposal of \cite[Sec. 1.6.2-3]{Neitzke-Yan-nonabelianization}. 
One difference is that we push the conormal off the zero section. 
Another is that we work in the HOMFLYPT skein on both $L$ and $M$. 
Yet another is that the skein-valued curve counting is most immediately expressed in the variable $z = e^{g_s/2} - e^{- g_s/2}$, see \cite[Theorem 1.4]{bare}, 
rather than $g_s$, and in the variable $a$ rather than $e^{N g_s/2}$, see \cite{ekholm-shende-colored}. 
All of these may be related to the assertion in loc.\ cit.\ that those authors work ``\thinspace`before' the conifold transition''.
\end{remark}

We  understand the Lagrangians appearing  in Definition \ref{main definition} as coming equipped with certain additional ``brane data''.  This data includes -- as always for studying holomorphic curves with boundary -- a choice of spin structure.  Related issues lead us to introduce certain {\em sign lines} into our skeins; these being some fixed 1-cycles in the 3-manifold which our links should avoid, but may cross at the cost of changing sign.  
\begin{gather}
\vcenter{\hbox{
\begin{tikzpicture}[scale=0.7]
\draw[dotted] (0,0) circle (1);
\draw[ultra thick, red] (-1, 0) -- (1, 0);
\draw[white, line width=2.5mm] (0, -1) -- (0, 1);
\draw[ultra thick, ->] (0, -1) -- (0, 1);
\end{tikzpicture}
}}
\;\;=\;\;
(-1)\cdot
\vcenter{\hbox{
\begin{tikzpicture}[scale=0.7]
\draw[dotted] (0,0) circle (1);
\draw[ultra thick, ->] (0, -1) -- (0, 1);
\draw[white, line width=2.5mm] (-1, 0) -- (1, 0);
\draw[ultra thick, red] (-1, 0) -- (1, 0);
\end{tikzpicture}
}}
\;, \label{eq:skeinrel4}
\end{gather}
The effect on the isomorphism class of the skein module only depends on the class of the sign lines in $H^1(M, \Z/2\Z)$.  
These lines correspond to the appearance of twisted local systems in \cite{Gaiotto-Moore-Neitzke-WKB}, and appeared previously in a skein-valued curve-counting setting in \cite{Scharitzer-Shende-mirror}. 
Brane data also includes, uniquely to the skein-valued curve counting, a choice of a 4-chain $W \subset T^*M$ where $\partial W$ is twice the Lagrangian of interest \cite{SOB}.  We discuss appropriate choices for these data in Section \ref{sec: brane data} below. 

We show that $[ \, \cdot \,]_L$ defines a skein trace:  

\begin{theorem} \label{skein trace}
The map $K \mapsto [K]_L$ factors through $\Sk(M)|_{a_M = a_L^n}$, hence defining a map 
\begin{eqnarray*}
\Sk(M)|_{a_M = a_L^n} & \to & \Sk(L) \\
K & \mapsto & [K]_L
\end{eqnarray*} 
\end{theorem}

This result does not follow immediately from the well-definedness of the skein-valued curve counting, and 
requires us to compare the counts of curves ending on conormals to knots which are related by a skein
relation.  We do so by studying holomorphic curves in the Morse flow graph limit of \cite{Ekholm-morse}; the proof of Theorem \ref{skein trace} is given in Section \ref{existence of skein trace}.  

In general, finding higher genus Morse flow graphs is a
difficult problem, but when the cover has degree two, it simplifies enough that we can solve it explicitly, showing: 
\begin{theorem}[Theorem \ref{explicit description}] \label{degree two}
When $L \to M$ has degree two, the skein trace is computed by an explicit finite sum of graphs, recovering the prescription of Neitzke and Yan \cite{Neitzke-Yan-nonabelianization}, and, more generally, our Theorem \ref{thm:UV-IRmap smooth}.
\end{theorem}

For a Riemann surface $S$, we write $\Sk(S) := \Sk(S \times \R)$.
Recall that concatenation of the $\R$ factor makes $\Sk(S)$ an algebra, and, similarly, if $M$ is a 3-manifold and $S$ is a component of its boundary, then $\Sk(S)$ acts on $\Sk(M)$. 

Fix $\Sigma \subset T^*S$ a Lagrangian which bounds no (non-constant) holomorphic curves,
which can be arranged by choosing $\Sigma$
either to be an exact Lagrangian with Legendrian ends as in \cite{Shende-Treumann-Williams-Zaslow}\footnote{In \cite{Shende-Treumann-Williams-Zaslow}, the Legendrian ends were drawn in the cosphere bundle of a surface $S$. However, those ends projected to a neighborhood of certain marked points on $S$; if we puncture $S$ at these points to get $S^\circ$ and view $T^* S^\circ$ as having convex contact boundary, then we may isotope the ends to lie in the `vertical' part of the boundary, compatibly with our current conventions.}, or, as in the original \cite{Gaiotto-Moore-Neitzke-WKB}, 
considering holomorphic $S$ and $\Sigma$, and, if $\Omega$ is the holomorphic symplectic form on $T^*S$, taking the real symplectic form to be  $\mathrm{Re}(e^{i \theta} \Omega)$ where $\theta$ is generic. We show the following result in Section \ref{wall crossing lifts}.
\begin{theorem}\label{algebra structure}
    Suppose $\Sigma \subset T^*S$ bounds no nonconstant holomorphic curves and $\Sigma \to S$ is a degree $n$ cover.  Then 
    $$[\;\cdot\;]_{\Sigma}: \Sk_{a^n}(S) \to \Sk_{a}(\Sigma)$$ is an algebra homomorphism.  

    More generally, if $M$ is a 3-manifold, $S$ is a component of $\partial M$, and $L \subset T^*M$ is a Lagrangian eventually cylindrical on $\Sigma \subset T^* S$ then the map $\Sk_{a^n}(M) \to \Sk_a(L)$ is a map of $\Sk_{a^n}(S)$ modules,  where $\Sk_{a^n}(S)$ acts on $\Sk_a(L)$ through the algebra map $[\,\cdot\,]_{\Sigma}$. 
\end{theorem}
\begin{proof}
    The algebra structure concerns knots in $S \times \R$ whose projections to the $\R$ direction are disjoint.  The holomorphic curves involved in determining $[\, \cdot \,]_\Sigma$ then necessarily have correspondingly disjoint projections to $T^*\R$. 
\end{proof}

In fact, letting (totally real) traces of sufficiently slow Lagrangian isotopies play the role of the Lagrangian $L$ in Theorem \ref{algebra structure}, we also establish the following `wall crossing' formula in Section \ref{wall crossing lifts}: 

\begin{theorem} \label{general wall crossing}
Let $\phi_t\colon \Sigma \to T^*S$ be a path of Lagrangians interpolating between some $\Sigma_-$ and $\Sigma_+$.  
Then there is an invertible element $\Omega(\phi) \in \Sk(\Sigma)$ such that, for every $[K] \in \Sk(S)$, 
\begin{equation} \label{general skein mutation} 
[K]_{\Sigma_-}  \cdot \Omega(\phi) = \Omega(\phi) \cdot [K]_{\Sigma_+}
\end{equation} 
Moreover, $\Omega(\phi)$ is invariant under homotopy of the path $\phi$ with fixed endpoints. 
\end{theorem}

Furthermore, for the simplest such paths, we determine $\Omega(\phi)$ explicitly in Section \ref{(-1)disks}.  The result is formulated in terms of the exponentiated skein dilogarithm, which 
is a certain lift of the $q$-dilogarithm of Fock and Goncharov from the $\mathfrak{gl}(1)$ skein to the HOMFLYPT skein,  see Definition \ref{def : skein dilog}. 
The skein dilogarithm first appeared as the skein-valued count of curves ending on an Aganagic-Vafa brane in $\C^3$ \cite{Ekholm-Shende-unknot}.

\begin{theorem} \label{skein valued cluster}
    If $\phi_t(\Sigma)$ bounds no holomorphic disks away from $t=0$, and $\phi_0(\Sigma)$ bounds a single embedded index $-1$ holomorphic disk $D$, which is a transversely cut out holomorphic disk instance in the $1$-parameter family, see Section \ref{(-1)disks} for details, then $\Omega(\phi)$ is the exponentiated skein dilogarithm on $\partial D$. 
\end{theorem}

Theorem \ref{skein valued cluster} is the skein-quantized version of the fact that the simplest wall crossings of the charts $\Loc_1(\Sigma) \to \Loc_n(S)$ in \cite{Gaiotto-Moore-Neitzke-WKB} or \cite{Shende-Treumann-Williams-Zaslow} are cluster transformations.  A related result, showing that a certain algebra of boundary operators $A(\Sigma)$ transforms similarly,  has previously appeared in \cite{Scharitzer-Shende-cluster, Hu-Schrader-Zaslow}.

\section{The skein dilogarithm}

\subsection{Skein of the solid torus}
Consider a solid torus, $D^2 \times S^1$.  
Choose an orientation of the longitude, hence distinguishing an oriented meridian $P_{1,0}$ of the boundary which links the longitude positively.  
According to \cite{Hadji-Morton}, $P_{1,0}$ acts diagonally with distinct eigenvalues, one for each pair of partitions $(\lambda,\mu)$ equal to 
$$
\frac{a-a^{-1}}{q^{1/2}-q^{-1/2}}+
(q^{1/2}-q^{-1/2})\left(a C_\lambda(q)-a^{-1}C_\mu(q^{-1})\right),
$$ 
where $C_\lambda$ denotes the content polynomial of $\lambda$, on the $\Sk(D^2\times S^1)$ and correspondingly determines, up to scalar multiples, a basis $W_{\lambda, \mu}$ indexed by pairs of partitions $\lambda,  \mu$. 
A choice of framing of the longitude determines the scalars by requiring that, when completing the torus to the three-sphere, $W_{\lambda, \mu}$ goes to the appropriate colored HOMFLYPT invariant of the longitude-turned-unknot with the given framing.  Note that the framing of the longitude also determines a choice of longitude in the boundary of the solid torus.
Thus for $\Psi \in \Sk(D^2 \times S^1)$, we may expand:
\begin{equation} \label{torus skein expansion}
    \Psi = \sum \psi_{\lambda, \mu} W_{\lambda, \mu}
\end{equation}
The $W_\lambda:=W_{\lambda,\emptyset}$ are known to span the part of the skein generated by links whose tangent in the $S^1$ direction is everywhere positive; we refer to this as the \emph{positive skein} and denote it $\Sk_+(D^2\times S^1)$.  We write $\widehat{\Sk}(D^2 \times S^1)$ for the completion in which we allow infinite sums of links, so long as for any $N$, there are only finitely many whose winding around $S^1$ is bounded above by $N$. 

We will later want to consider the $a=1$ specialization of $\Sk_+(D^2\times S^1)$. 
As the eigenvalues of the meridian operator among the $W_{\lambda}$ remain distinct, these remain linearly independent.

\subsection{Skein dilogarithm}\label{ssec: skein dilog} 
Let $P_{1,0}$ and $P_{0,1}$ be the meridian and longitude of the boundary of the solid torus.   

\begin{defn}\label{def : skein dilog}
The \emph{exponentiated skein dilogarithm} $\Psi[\xi] \in \Q[\xi]\otimes \widehat{\Sk}_{+}(D^2 \times S^1)$ is the unique solution to the 3-term recurrence relation 
\begin{equation}\label{eq:dilog-recurrence}
(\bigcirc - P_{1,0} - a\xi P_{0,1}) \Psi[\xi] = 0
\end{equation}
of the form $\Psi[\xi] = 1 + \cdots$, i.e.\ the coefficient of $W_\emptyset$ is $1$. 
\end{defn}
Explicitly, 
\begin{equation}\label{eq:explicit-skein-dilog}
\Psi[\xi]:= \sum_{\lambda} \prod_{\square \in \lambda}\frac{-q^{-c(\square)/2}\xi}{q^{h(\square)/2} - q^{-h(\square)/2}} W_{\lambda} =  \exp\left( -\sum_d \frac{1}{d} \frac{\xi^d\, P_d}{q^{d/2}-q^{-d/2}} \right).
\end{equation}
For the existence and uniqueness of the solution, and the first expression for it, see \cite{Ekholm-Shende-unknot}.  
In the second expression, the $P_d$ are the images of the power sum symmetric functions under the identification of symmetric functions and $\Sk_+(D^2 \times S^1)$ determined by identifying $W_\lambda$ with the corresponding Schur function \cite{Aiston-powersum, Morton_powersum}. 
The equality of the two formulas can be derived by manipulation of symmetric functions, see e.g. \cite[Theorem 1.1 (d)]{Nakamura}. 
As is customary in the literature, we will often omit the word `exponentiated'; we always mean $\Psi[\xi]$, rather than its logarithm. 
Also, we will simply write $\Psi$ for $\Psi[1]$. 

The skein dilogarithm satisfies a relative version of the 3-term recurrence relation \eqref{eq:dilog-recurrence} in a thickened annulus with 2 marked points on the boundary, see \cite{Ekholm-Shende-unknot} or \cite[Equation (104)]{Nakamura}: 
\begin{equation}\label{eq:dilog-relative-recurrence}
\vcenter{\hbox{
\begin{tikzpicture}
\draw[very thick, blue, <-] (0, 0) [partial ellipse = 0 : 360 : 0.6];
\draw[white, line width=5] (0, 0.3) -- (0, 1);
\draw[very thick, ->] (0, 0.3) -- (0, 1);
\draw[very thick] (0, 0) circle (1);
\draw[very thick] (0, 0) circle (0.3);
\node[blue, right] at (1.0, 0){$\Psi$};
\end{tikzpicture}
}}
\;\;=\;\;
\vcenter{\hbox{
\begin{tikzpicture}
\draw[very thick, ->] (0, 0.3) -- (0, 1);
\draw[white, line width=5] (0, 0) [partial ellipse = 0 : 360 : 0.6];
\draw[very thick, blue, <-] (0, 0) [partial ellipse = 0 : 360 : 0.6];
\draw[very thick] (0, 0) circle (1);
\draw[very thick] (0, 0) circle (0.3);
\node[blue, right] at (1.0, 0){$\Psi$};
\end{tikzpicture}
}}
\;\;+\;\;
\vcenter{\hbox{
\begin{tikzpicture}
\draw[very thick, blue, <-] (0, 0) [partial ellipse = 0 : 360 : 0.5];
\draw[white, line width=5] (0, 0.3) to[out=90, in=150] ({0.7*1/2}, {0.7*sqrt(3)/2});
\draw[very thick] (0, 0.3) to[out=90, in=150] ({0.7*1/2}, {0.7*sqrt(3)/2});
\draw[very thick] (0, 0) [partial ellipse = 60 : -240 : 0.7];
\draw[very thick, ->] ({-0.7*1/2}, {0.7*sqrt(3)/2}) to[out=30, in=-90] (0, 1); 
\draw[very thick] (0, 0) circle (1);
\draw[very thick] (0, 0) circle (0.3);
\node[blue, right] at (1.0, 0){$\Psi$};
\end{tikzpicture}
}}
\;.
\end{equation}

The skein dilogarithm has an inverse:
\begin{equation}\label{eq:skein-dilog-inverse}
\Psi^{-1} = \sum_{\lambda} \prod_{\square \in \lambda}\frac{q^{c(\square)/2}}{q^{h(\square)/2} - q^{-h(\square)/2}} W_{\lambda}
\end{equation}
which satisfies an analogous 3-term recurrence relation
\begin{equation}
(\bigcirc - P_{-1,0} - a^{-1} P_{0,1}) \Psi^{-1} = 0 
\end{equation}
where $P_{-1,0}$ is the meridian with opposite orientation, and the corresponding relative version.

\section{Branched double covers from triangulations}

In this section we recall how to associate smooth double covers to certain decorated ideal triangulations of 3-manifolds, and how these can be understood as smoothings of certain singular double covers depending on the triangulation alone.  Then we define a skein module associated to the singular double cover, and identify conditions under which it maps to the skein module of a smoothing.

\subsection{Smooth covers from marked triangulations}\label{ssec:smoothcoversfromtriang}

Given any manifold $M$ containing a cooriented codimension two submanifold $\sigma \subset M$, there is a canonical double cover $\widetilde{M} \to M$ branched over $\sigma$.  In this context we will write $\widetilde{\sigma}$ for the preimage of $\sigma$.  

For triangulated $3$-manifolds $M$, we will define branch loci $\sigma\subset M$ by a patching local pieces in each tetrahedron of the triangulation. Recall that an ideal tetrahedron is a tetrahedron minus its vertices, and an ideal triangulation $\Delta$ of a noncompact 3-manifold $M$ is a stratification of $M$ by ideal tetrahedra. 
We will use the notation $\Delta^k$ to denote the set of $k$-cells in the ideal triangulation. 

\begin{definition}
    A {\em marking} of an ideal triangulation $\Delta$ of a 3-manifold $M$ is a choice of a pair of opposed edges on each tetrahedron in $\Delta$.
\end{definition}

A marking of a tetrahedron determines a tangle by the local rules in Figure \ref{fig:marked_tetrahedron_resolutions}, 
and hence an associated smooth branched double cover with topology $S^1\times D^2\stackrel{2 : 1}{\longrightarrow} D^3$.
We may alternatively specify a marking by choosing quadrilaterals which slice the tetrahedron in a way separating the tangle associated to the marking, as in Figure \ref{fig:tetrahedron-boundary-cycles}.

\begin{figure}
\centering
$\vcenter{\hbox{
\tdplotsetmaincoords{60}{80}
\begin{tikzpicture}[tdplot_main_coords]
\begin{scope}[scale = 0.8, tdplot_main_coords]
    \coordinate (o) at (0, 0, 0);
    \coordinate (a) at (0, 0, 3);
    \coordinate (a1) at (0, 0, -1);
    
    \coordinate (b) at ({2*sqrt(2)}, 0, -1);
    \coordinate (b1) at ({-2*sqrt(2)/3}, 0, 1/3);
    
    \coordinate (c) at ({-sqrt(2)}, {sqrt(6)}, -1);
    \coordinate (c1) at ({sqrt(2)/3}, {-sqrt(6)/3}, 1/3);
    
    \coordinate (d) at ({-sqrt(2)}, {-sqrt(6)}, -1);
    \coordinate (d1) at ({sqrt(2)/3}, {sqrt(6)/3}, 1/3);
    
    \coordinate (ab) at ({sqrt(2)}, 0, 1);
    \coordinate (ac) at ({-sqrt(2)/2}, {sqrt(6)/2}, 1);
    \coordinate (ad) at ({-sqrt(2)/2}, {-sqrt(6)/2}, 1);
    \coordinate (bc) at ({sqrt(2)/2}, {sqrt(6)/2}, -1);
    \coordinate (bd) at ({sqrt(2)/2}, {-sqrt(6)/2}, -1);
    \coordinate (cd) at ({-sqrt(2)}, 0, -1);

    \draw[very thick, dashed] (c) -- (d);
    
    \draw[red, very thick] (a1) .. controls (o) .. (c1);
    \draw[red, very thick] (b1) .. controls (o) .. (d1);

    \filldraw[red] (a1) circle (0.15em);
    \filldraw[red] (b1) circle (0.15em);
    \filldraw[red] (c1) circle (0.2em);
    \filldraw[red] (d1) circle (0.2em);

    \draw[very thick] (a) -- (b);
    \draw[very thick, blue] (a) -- (c);
    \draw[very thick] (a) -- (d);
    \draw[very thick] (b) -- (c);
    \draw[very thick, blue] (b) -- (d);

    \filldraw (a) circle (0.05em);
    \filldraw (b) circle (0.05em);
    \filldraw (c) circle (0.05em);
    \filldraw (d) circle (0.05em);

\end{scope}
\end{tikzpicture}
}}
\quad\quad
\vcenter{\hbox{
\tdplotsetmaincoords{60}{80}
\begin{tikzpicture}[tdplot_main_coords]
\begin{scope}[scale = 0.8, tdplot_main_coords]
    \coordinate (o) at (0, 0, 0);
    \coordinate (a) at (0, 0, 3);
    \coordinate (a1) at (0, 0, -1);
    
    \coordinate (b) at ({2*sqrt(2)}, 0, -1);
    \coordinate (b1) at ({-2*sqrt(2)/3}, 0, 1/3);
    
    \coordinate (c) at ({-sqrt(2)}, {sqrt(6)}, -1);
    \coordinate (c1) at ({sqrt(2)/3}, {-sqrt(6)/3}, 1/3);
    
    \coordinate (d) at ({-sqrt(2)}, {-sqrt(6)}, -1);
    \coordinate (d1) at ({sqrt(2)/3}, {sqrt(6)/3}, 1/3);
    
    \coordinate (ab) at ({sqrt(2)}, 0, 1);
    \coordinate (ac) at ({-sqrt(2)/2}, {sqrt(6)/2}, 1);
    \coordinate (ad) at ({-sqrt(2)/2}, {-sqrt(6)/2}, 1);
    \coordinate (bc) at ({sqrt(2)/2}, {sqrt(6)/2}, -1);
    \coordinate (bd) at ({sqrt(2)/2}, {-sqrt(6)/2}, -1);
    \coordinate (cd) at ({-sqrt(2)}, 0, -1);

    \draw[very thick, dashed] (c) -- (d);
    
    \draw[red, very thick] (a1) .. controls (o) .. (d1);
    \draw[red, very thick] (b1) .. controls (o) .. (c1);

    \filldraw[red] (a1) circle (0.15em);
    \filldraw[red] (b1) circle (0.15em);
    \filldraw[red] (c1) circle (0.2em);
    \filldraw[red] (d1) circle (0.2em);

    \draw[very thick] (a) -- (b);
    \draw[very thick] (a) -- (c);
    \draw[very thick, blue] (a) -- (d);
    \draw[very thick, blue] (b) -- (c);
    \draw[very thick] (b) -- (d);

    \filldraw (a) circle (0.05em);
    \filldraw (b) circle (0.05em);
    \filldraw (c) circle (0.05em);
    \filldraw (d) circle (0.05em);

\end{scope}
\end{tikzpicture}
}}
\quad\quad
\vcenter{\hbox{
\tdplotsetmaincoords{60}{80}
\begin{tikzpicture}[tdplot_main_coords]
\begin{scope}[scale = 0.8, tdplot_main_coords]
    \coordinate (o) at (0, 0, 0);
    \coordinate (a) at (0, 0, 3);
    \coordinate (a1) at (0, 0, -1);
    
    \coordinate (b) at ({2*sqrt(2)}, 0, -1);
    \coordinate (b1) at ({-2*sqrt(2)/3}, 0, 1/3);
    
    \coordinate (c) at ({-sqrt(2)}, {sqrt(6)}, -1);
    \coordinate (c1) at ({sqrt(2)/3}, {-sqrt(6)/3}, 1/3);
    
    \coordinate (d) at ({-sqrt(2)}, {-sqrt(6)}, -1);
    \coordinate (d1) at ({sqrt(2)/3}, {sqrt(6)/3}, 1/3);
    
    \coordinate (ab) at ({sqrt(2)}, 0, 1);
    \coordinate (ac) at ({-sqrt(2)/2}, {sqrt(6)/2}, 1);
    \coordinate (ad) at ({-sqrt(2)/2}, {-sqrt(6)/2}, 1);
    \coordinate (bc) at ({sqrt(2)/2}, {sqrt(6)/2}, -1);
    \coordinate (bd) at ({sqrt(2)/2}, {-sqrt(6)/2}, -1);
    \coordinate (cd) at ({-sqrt(2)}, 0, -1);

    \draw[very thick, dashed, blue] (c) -- (d);
    
    \draw[red, very thick] (a1) .. controls (o) .. (b1);
    \draw[white, line width=5] (c1) .. controls (o) .. (d1);
    \draw[red, very thick] (c1) .. controls (o) .. (d1);

    \filldraw[red] (a1) circle (0.15em);
    \filldraw[red] (b1) circle (0.15em);
    \filldraw[red] (c1) circle (0.2em);
    \filldraw[red] (d1) circle (0.2em);

    \draw[very thick, blue] (a) -- (b);
    \draw[very thick] (a) -- (c);
    \draw[very thick] (a) -- (d);
    \draw[very thick] (b) -- (c);
    \draw[very thick] (b) -- (d);

    \filldraw (a) circle (0.05em);
    \filldraw (b) circle (0.05em);
    \filldraw (c) circle (0.05em);
    \filldraw (d) circle (0.05em);

\end{scope}
\end{tikzpicture}
}}
$
\caption{Red tangle associated to the marked blue edges.}
\label{fig:marked_tetrahedron_resolutions}
\end{figure}
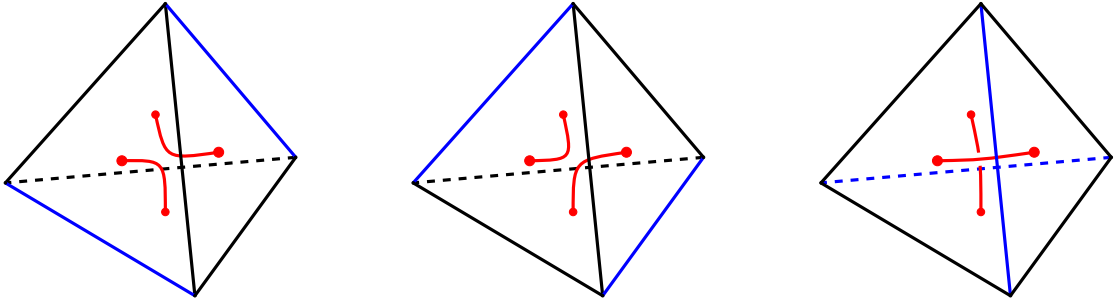

\begin{figure}
\centering
$
\vcenter{\hbox{
\tdplotsetmaincoords{60}{80}
\begin{tikzpicture}[scale=0.7, tdplot_main_coords]
\begin{scope}[scale = 0.8, tdplot_main_coords]
    \coordinate (o) at (0, 0, 0);
    \coordinate (a) at (0, 0, 3);
    \coordinate (a1) at (0, 0, -1);
    
    \coordinate (b) at ({2*sqrt(2)}, 0, -1);
    \coordinate (b1) at ({-2*sqrt(2)/3}, 0, 1/3);
    
    \coordinate (c) at ({-sqrt(2)}, {sqrt(6)}, -1);
    \coordinate (c1) at ({sqrt(2)/3}, {-sqrt(6)/3}, 1/3);
    
    \coordinate (d) at ({-sqrt(2)}, {-sqrt(6)}, -1);
    \coordinate (d1) at ({sqrt(2)/3}, {sqrt(6)/3}, 1/3);
    
    \coordinate (ab) at ({sqrt(2)}, 0, 1);
    \coordinate (ac) at ({-sqrt(2)/2}, {sqrt(6)/2}, 1);
    \coordinate (ad) at ({-sqrt(2)/2}, {-sqrt(6)/2}, 1);
    \coordinate (bc) at ({sqrt(2)/2}, {sqrt(6)/2}, -1);
    \coordinate (bd) at ({sqrt(2)/2}, {-sqrt(6)/2}, -1);
    \coordinate (cd) at ({-sqrt(2)}, 0, -1);
    
    \draw[very thick] (c) -- (d);

    \draw[very thick, blue] (ab) -- (bc) -- (cd) -- (ad) -- cycle;

    \draw[white, ultra thick] (a) -- (b);
    \draw[very thick] (a) -- (b);
    \draw[very thick] (a) -- (c);
    \draw[very thick] (a) -- (d);
    \draw[very thick] (b) -- (c);
    \draw[very thick] (b) -- (d);
    
    \filldraw (a) circle (0.05em);
    \filldraw (b) circle (0.05em);
    \filldraw (c) circle (0.05em);
    \filldraw (d) circle (0.05em);
\end{scope}
\end{tikzpicture}
}}
\quad
\quad
\vcenter{\hbox{
\tdplotsetmaincoords{60}{80}
\begin{tikzpicture}[scale=0.7, tdplot_main_coords]
\begin{scope}[scale = 0.8, tdplot_main_coords]
    \coordinate (o) at (0, 0, 0);
    \coordinate (a) at (0, 0, 3);
    \coordinate (a1) at (0, 0, -1);
    
    \coordinate (b) at ({2*sqrt(2)}, 0, -1);
    \coordinate (b1) at ({-2*sqrt(2)/3}, 0, 1/3);
    
    \coordinate (c) at ({-sqrt(2)}, {sqrt(6)}, -1);
    \coordinate (c1) at ({sqrt(2)/3}, {-sqrt(6)/3}, 1/3);
    
    \coordinate (d) at ({-sqrt(2)}, {-sqrt(6)}, -1);
    \coordinate (d1) at ({sqrt(2)/3}, {sqrt(6)/3}, 1/3);
    
    \coordinate (ab) at ({sqrt(2)}, 0, 1);
    \coordinate (ac) at ({-sqrt(2)/2}, {sqrt(6)/2}, 1);
    \coordinate (ad) at ({-sqrt(2)/2}, {-sqrt(6)/2}, 1);
    \coordinate (bc) at ({sqrt(2)/2}, {sqrt(6)/2}, -1);
    \coordinate (bd) at ({sqrt(2)/2}, {-sqrt(6)/2}, -1);
    \coordinate (cd) at ({-sqrt(2)}, 0, -1);
    
    \draw[very thick] (c) -- (d);

    \draw[very thick, blue] (ab) -- (ac) -- (cd) -- (bd) -- cycle;

    \draw[white, ultra thick] (a) -- (b);
    \draw[very thick] (a) -- (b);
    \draw[very thick] (a) -- (c);
    \draw[very thick] (a) -- (d);
    \draw[very thick] (b) -- (c);
    \draw[very thick] (b) -- (d);
    
    \filldraw (a) circle (0.05em);
    \filldraw (b) circle (0.05em);
    \filldraw (c) circle (0.05em);
    \filldraw (d) circle (0.05em);
\end{scope}
\end{tikzpicture}
}}
\quad
\vcenter{\hbox{
\tdplotsetmaincoords{60}{80}
\begin{tikzpicture}[scale=0.7, tdplot_main_coords]
\begin{scope}[scale = 0.8, tdplot_main_coords]
    \coordinate (o) at (0, 0, 0);
    \coordinate (a) at (0, 0, 3);
    \coordinate (a1) at (0, 0, -1);
    
    \coordinate (b) at ({2*sqrt(2)}, 0, -1);
    \coordinate (b1) at ({-2*sqrt(2)/3}, 0, 1/3);
    
    \coordinate (c) at ({-sqrt(2)}, {sqrt(6)}, -1);
    \coordinate (c1) at ({sqrt(2)/3}, {-sqrt(6)/3}, 1/3);
    
    \coordinate (d) at ({-sqrt(2)}, {-sqrt(6)}, -1);
    \coordinate (d1) at ({sqrt(2)/3}, {sqrt(6)/3}, 1/3);
    
    \coordinate (ab) at ({sqrt(2)}, 0, 1);
    \coordinate (ac) at ({-sqrt(2)/2}, {sqrt(6)/2}, 1);
    \coordinate (ad) at ({-sqrt(2)/2}, {-sqrt(6)/2}, 1);
    \coordinate (bc) at ({sqrt(2)/2}, {sqrt(6)/2}, -1);
    \coordinate (bd) at ({sqrt(2)/2}, {-sqrt(6)/2}, -1);
    \coordinate (cd) at ({-sqrt(2)}, 0, -1);
    
    \draw[very thick] (c) -- (d);

    \draw[very thick, blue] (ac) -- (ad) -- (bd) -- (bc) -- cycle;

    \draw[white, ultra thick] (a) -- (b);
    \draw[very thick] (a) -- (b);
    \draw[very thick] (a) -- (c);
    \draw[very thick] (a) -- (d);
    \draw[very thick] (b) -- (c);
    \draw[very thick] (b) -- (d);
    
    \filldraw (a) circle (0.05em);
    \filldraw (b) circle (0.05em);
    \filldraw (c) circle (0.05em);
    \filldraw (d) circle (0.05em);
\end{scope}
\end{tikzpicture}
}}
$
\caption{Boundaries of slicing quadrilaterals.}
\label{fig:tetrahedron-boundary-cycles}
\end{figure}
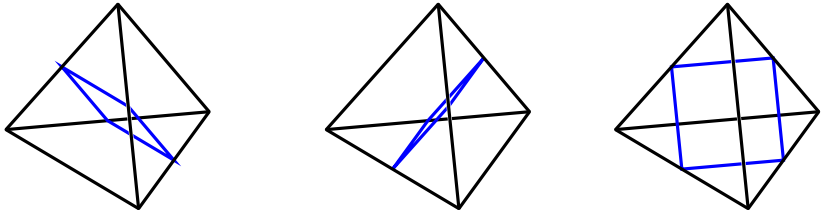

\subsection{Singular branched double covers}\label{ssec:brancheddoublecoversfoliations}
In fact, given an ideal triangulation of a 3-manifold $M$ without any additional structure,  there is a singular branched double cover $\widetilde{M}^{\mathrm{sing}}$ of $M$ canonically associated to it; see e.g.\ \cite{Cecotti-Cordova-Vafa, Freed-Neitzke}. 
In this section, we briefly review this construction. 

Let $M$ be a 3-manifold with an ideal triangulation $\Delta$.  We write $M^\circ\subset M$ for the complement of the barycenters of the tetrahedra in $\Delta$. 
We write $\tau \subset M$ for the singular tangle that consists of the edges ($1$-cells) of the polygonal decomposition dual to the triangulation, i.e.\ the edges which connect barycenters in adjacent tetrahedra of $\Delta$, passing through the barycenters of faces, see Figure \ref{fig:tetrahedron_branch_locus}. 

We write $\widetilde{M}^\circ$ for branched double cover of $M^\circ$ corresponding to the smooth tangle $\tau \cap M^\circ$. 
The topology of $\widetilde{M}^\circ$ near the barycenters is $T^2\times \R_{\ge 0}\subset \widetilde{M}^\circ$, where the torus $T^2$ double covers the sphere $S^2$ in the punctured neighborhood $\R_{\ge 0}\times S^2\subset M^\circ$ of the barycenter. 
We compactify $\widetilde{M}^\circ$ by gluing in a cone on the torus $T^2$ near each barycenter. We call the resulting singular space $\widetilde{M}^{\mathrm{sing}}$. Then $\widetilde{M}^{\mathrm{sing}}$ is smooth save for the isolated points (corresponding to the barycenters) where it is the cone on a torus and it comes with a branched double cover $\widetilde{M}^{\mathrm{sing}} \to M$. 
We write $\widetilde{\tau}\subset \widetilde{M}^{\mathrm{sing}}$ for the preimage of $\tau$ under this branched double cover. 

Consider a tetrahedron in $\Delta$ and consider inside this tetrahedron the triangles with the following
vertices: one vertex of the tetrahedron, the barycenter of a face adjacent to this tetrahedron vertex, and the barycenter of the tetrahedron. For each vertex of the tetrahedron there are three such triangles and therefore there are $4 \cdot 3 = 12$ such triangles in each tetrahedron, see Figure \ref{fig:tetrahedron_spectralnetwork}. In \cite{Freed-Neitzke}, the complex formed by all such triangles in $M$ is called a 3d spectral network in $M$, as an analogue in the 3-dimensional context of the spectral networks for surfaces \cite{Gaiotto-Moore-Neitzke-spectral}.

Consider next the complement of the 3d spectral network triangles in a tetrahedron of $\Delta$. It has six connected components, each one contains exactly one edge of the tetrahedron and exactly two of its ideal vertices in its boundary. We next consider the faces ($2$-cells) of the polygonal decomposition $\Delta^\vee$ dual to $\Delta$. In each tetrahedron these are the quadrilaterals depicted in Figure \ref{fig:tetrahedron_branchcut}. (We will use the $2$-skeleton $\Delta^\vee_{(2)}$ as a branch cut for the double cover $\widetilde{M}^{\mathrm{sing}}\to M$, see Definition \ref{defn: sheet labels}.) 
Each component of the complement of the 3d spectral network triangles retracts to one of the quadrilaterals in $\Delta^\vee_{(2)}$. In fact, it retracts along the lines of a foliation in which each line connects the two ideal vertices in the boundary of the complement component and intersects the quadrilateral exactly once. We extend the foliation across the 3d spectral network by allowing critical leaves: the critical leaves connect points on the branch locus to ideal vertices as follows: each edge in the branch locus is a boundary component of three of the spectral network triangles and the critical leaves through a point in the branch locus connects this point inside the network triangles to the ideal vertex which is also a vertex of the triangle, see Figure \ref{fig:tetrahedron_foliation} for an illustration of this 1-dimensional foliation.

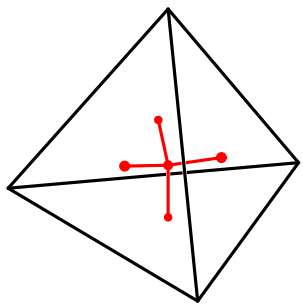
\begin{figure}
\centering
$\vcenter{\hbox{
\tdplotsetmaincoords{60}{80}
\begin{tikzpicture}[tdplot_main_coords]
\begin{scope}[scale = 0.8, tdplot_main_coords]
    \coordinate (o) at (0, 0, 0);
    \coordinate (a) at (0, 0, 3);
    \coordinate (a1) at (0, 0, -1);
    
    \coordinate (b) at ({2*sqrt(2)}, 0, -1);
    \coordinate (b1) at ({-2*sqrt(2)/3}, 0, 1/3);
    
    \coordinate (c) at ({-sqrt(2)}, {sqrt(6)}, -1);
    \coordinate (c1) at ({sqrt(2)/3}, {-sqrt(6)/3}, 1/3);
    
    \coordinate (d) at ({-sqrt(2)}, {-sqrt(6)}, -1);
    \coordinate (d1) at ({sqrt(2)/3}, {sqrt(6)/3}, 1/3);
    
    \coordinate (ab) at ({sqrt(2)}, 0, 1);
    \coordinate (ac) at ({-sqrt(2)/2}, {sqrt(6)/2}, 1);
    \coordinate (ad) at ({-sqrt(2)/2}, {-sqrt(6)/2}, 1);
    \coordinate (bc) at ({sqrt(2)/2}, {sqrt(6)/2}, -1);
    \coordinate (bd) at ({sqrt(2)/2}, {-sqrt(6)/2}, -1);
    \coordinate (cd) at ({-sqrt(2)}, 0, -1);
    
    \draw[very thick] (c) -- (d);

    \draw[white, ultra thick] (o) -- (a1);
    
    \draw[red, very thick] (o) -- (a1);
    \draw[red, very thick] (o) -- (b1);
    \draw[red, very thick] (o) -- (c1);
    \draw[red, very thick] (o) -- (d1);

    \filldraw[red] (o) circle (0.18em);
    \filldraw[red] (a1) circle (0.15em);
    \filldraw[red] (b1) circle (0.15em);
    \filldraw[red] (c1) circle (0.2em);
    \filldraw[red] (d1) circle (0.2em);

    \draw[white, ultra thick] (a) -- (b);
    \draw[very thick] (a) -- (b);
    \draw[very thick] (a) -- (c);
    \draw[very thick] (a) -- (d);
    \draw[very thick] (b) -- (c);
    \draw[very thick] (b) -- (d);
    
    \filldraw (a) circle (0.05em);
    \filldraw (b) circle (0.05em);
    \filldraw (c) circle (0.05em);
    \filldraw (d) circle (0.05em);
\end{scope}
\end{tikzpicture}
}}$
\caption{The branch locus for the singular branched double cover of a tetrahedron}
\label{fig:tetrahedron_branch_locus}
\end{figure}

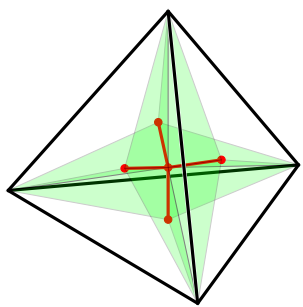
\begin{figure}
\centering
$\vcenter{\hbox{
\tdplotsetmaincoords{60}{80}
\begin{tikzpicture}[tdplot_main_coords]
\begin{scope}[scale = 0.8, tdplot_main_coords]
    \coordinate (o) at (0, 0, 0);
    \coordinate (a) at (0, 0, 3);
    \coordinate (a1) at (0, 0, -1);
    
    \coordinate (b) at ({2*sqrt(2)}, 0, -1);
    \coordinate (b1) at ({-2*sqrt(2)/3}, 0, 1/3);
    
    \coordinate (c) at ({-sqrt(2)}, {sqrt(6)}, -1);
    \coordinate (c1) at ({sqrt(2)/3}, {-sqrt(6)/3}, 1/3);
    
    \coordinate (d) at ({-sqrt(2)}, {-sqrt(6)}, -1);
    \coordinate (d1) at ({sqrt(2)/3}, {sqrt(6)/3}, 1/3);
    
    \coordinate (ab) at ({sqrt(2)}, 0, 1);
    \coordinate (ac) at ({-sqrt(2)/2}, {sqrt(6)/2}, 1);
    \coordinate (ad) at ({-sqrt(2)/2}, {-sqrt(6)/2}, 1);
    \coordinate (bc) at ({sqrt(2)/2}, {sqrt(6)/2}, -1);
    \coordinate (bd) at ({sqrt(2)/2}, {-sqrt(6)/2}, -1);
    \coordinate (cd) at ({-sqrt(2)}, 0, -1);
    
    \draw[very thick] (c) -- (d);

    \draw[white, ultra thick] (o) -- (a1);
    
    \filldraw [very thin, fill= green, opacity = 0.2] (o) -- (b1) -- (d) -- (o);
    \filldraw [very thin, fill= green, opacity = 0.2] (o) -- (b1) -- (c) -- (o);
    \filldraw [very thin, fill= green, opacity = 0.2] (o) -- (a1) -- (d) -- (o);
    \filldraw [very thin, fill= green, opacity = 0.2] (o) -- (a1) -- (c) -- (o);
    
    \filldraw [very thin, fill= green, opacity = 0.2] (o) -- (a1) -- (b) -- (o);
    \filldraw [very thin, fill= green, opacity = 0.2] (o) -- (b1) -- (a) -- (o);

    \filldraw [very thin, fill= green, opacity = 0.2] (o) -- (c1) -- (d) -- (o);
    \filldraw [very thin, fill= green, opacity = 0.2] (o) -- (d1) -- (c) -- (o);
    
    \draw[red, very thick] (o) -- (a1);
    \draw[red, very thick] (o) -- (b1);
    \draw[red, very thick] (o) -- (c1);
    \draw[red, very thick] (o) -- (d1);

    \filldraw[red] (o) circle (0.15em);
    \filldraw[red] (a1) circle (0.15em);
    \filldraw[red] (b1) circle (0.15em);
    \filldraw[red] (c1) circle (0.15em);
    \filldraw[red] (d1) circle (0.15em);

    \filldraw [very thin, fill= green, opacity = 0.2] (o) -- (c1) -- (b) -- (o);
    \filldraw [very thin, fill= green, opacity = 0.2] (o) -- (c1) -- (a) -- (o);
    \filldraw [very thin, fill= green, opacity = 0.2] (o) -- (d1) -- (b) -- (o);
    \filldraw [very thin, fill= green, opacity = 0.2] (o) -- (d1) -- (a) -- (o);

    \draw[white, ultra thick] (a) -- (b);
    \draw[very thick] (a) -- (b);
    \draw[very thick] (a) -- (c);
    \draw[very thick] (a) -- (d);
    \draw[very thick] (b) -- (c);
    \draw[very thick] (b) -- (d);
    
    \filldraw (a) circle (0.05em);
    \filldraw (b) circle (0.05em);
    \filldraw (c) circle (0.05em);
    \filldraw (d) circle (0.05em);
\end{scope}
\end{tikzpicture}
}}$
\caption{3d spectral network on a tetrahedron (twelve green triangles).}
\label{fig:tetrahedron_spectralnetwork}
\end{figure}

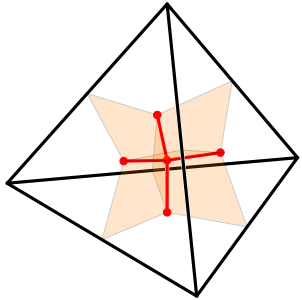
\begin{figure}
\centering
$\vcenter{\hbox{
\tdplotsetmaincoords{60}{80}
\begin{tikzpicture}[tdplot_main_coords]
\begin{scope}[scale = 0.8, tdplot_main_coords]
    \coordinate (o) at (0, 0, 0);
    \coordinate (a) at (0, 0, 3);
    \coordinate (a1) at (0, 0, -1);
    
    \coordinate (b) at ({2*sqrt(2)}, 0, -1);
    \coordinate (b1) at ({-2*sqrt(2)/3}, 0, 1/3);
    
    \coordinate (c) at ({-sqrt(2)}, {sqrt(6)}, -1);
    \coordinate (c1) at ({sqrt(2)/3}, {-sqrt(6)/3}, 1/3);
    
    \coordinate (d) at ({-sqrt(2)}, {-sqrt(6)}, -1);
    \coordinate (d1) at ({sqrt(2)/3}, {sqrt(6)/3}, 1/3);
    
    \coordinate (ab) at ({sqrt(2)}, 0, 1);
    \coordinate (ac) at ({-sqrt(2)/2}, {sqrt(6)/2}, 1);
    \coordinate (ad) at ({-sqrt(2)/2}, {-sqrt(6)/2}, 1);
    \coordinate (bc) at ({sqrt(2)/2}, {sqrt(6)/2}, -1);
    \coordinate (bd) at ({sqrt(2)/2}, {-sqrt(6)/2}, -1);
    \coordinate (cd) at ({-sqrt(2)}, 0, -1);
    
    \draw[very thick] (c) -- (d);

    \draw[white, ultra thick] (o) -- (a1);
    
    \filldraw [very thin, fill= orange, opacity = 0.2] (o) -- (a1) -- (cd) -- (b1) -- (o);
    \filldraw [very thin, fill= orange, opacity = 0.2] (o) -- (b1) -- (ac) -- (d1) -- (o);
    \filldraw [very thin, fill= orange, opacity = 0.2] (o) -- (b1) -- (ad) -- (c1) -- (o);
    \filldraw [very thin, fill= orange, opacity = 0.2] (o) -- (a1) -- (bc) -- (d1) -- (o);
    \filldraw [very thin, fill= orange, opacity = 0.2] (o) -- (a1) -- (bd) -- (c1) -- (o);
    
    \draw[red, very thick] (o) -- (a1);
    \draw[red, very thick] (o) -- (b1);
    \draw[red, very thick] (o) -- (c1);
    \draw[red, very thick] (o) -- (d1);

    \filldraw[red] (o) circle (0.15em);
    \filldraw[red] (a1) circle (0.15em);
    \filldraw[red] (b1) circle (0.15em);
    \filldraw[red] (c1) circle (0.15em);
    \filldraw[red] (d1) circle (0.15em);

    \filldraw [very thin, fill= orange, opacity = 0.2] (o) -- (c1) -- (ab) -- (d1) -- (o);

    \draw[white, ultra thick] (a) -- (b);
    \draw[very thick] (a) -- (b);
    \draw[very thick] (a) -- (c);
    \draw[very thick] (a) -- (d);
    \draw[very thick] (b) -- (c);
    \draw[very thick] (b) -- (d);
    
    \filldraw (a) circle (0.05em);
    \filldraw (b) circle (0.05em);
    \filldraw (c) circle (0.05em);
    \filldraw (d) circle (0.05em);
\end{scope}
\end{tikzpicture}
}}$
\caption{Leaf space / branch cuts for the branched cover (six orange quadrilaterals).}
\label{fig:tetrahedron_branchcut}
\end{figure}

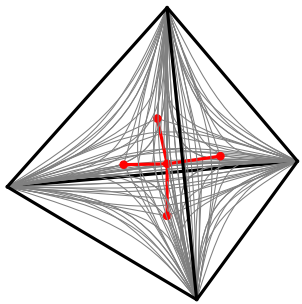
\begin{figure}
\centering
$\vcenter{\hbox{
\tdplotsetmaincoords{60}{80}
\begin{tikzpicture}[tdplot_main_coords]
\begin{scope}[scale = 0.8, tdplot_main_coords]
    \coordinate (o) at (0, 0, 0);
    
    \coordinate (a) at (0, 0, 3);
    \coordinate (a1) at (0, 0, -1);
    
    \coordinate (b) at ({2*sqrt(2)}, 0, -1);
    \coordinate (b1) at ({-2*sqrt(2)/3}, 0, 1/3);
    
    \coordinate (c) at ({-sqrt(2)}, {sqrt(6)}, -1);
    \coordinate (c1) at ({sqrt(2)/3}, {-sqrt(6)/3}, 1/3);
    
    \coordinate (d) at ({-sqrt(2)}, {-sqrt(6)}, -1);
    \coordinate (d1) at ({sqrt(2)/3}, {sqrt(6)/3}, 1/3);
    
    \coordinate (ab) at ({sqrt(2)}, 0, 1);
    \coordinate (ac) at ({-sqrt(2)/2}, {sqrt(6)/2}, 1);
    \coordinate (ad) at ({-sqrt(2)/2}, {-sqrt(6)/2}, 1);
    \coordinate (bc) at ({sqrt(2)/2}, {sqrt(6)/2}, -1);
    \coordinate (bd) at ({sqrt(2)/2}, {-sqrt(6)/2}, -1);
    \coordinate (cd) at ({-sqrt(2)}, 0, -1);
    
    \draw[very thick] (c) -- (d);

    \foreach \t in {0, 1/2, 2/2} { 
        \foreach \n in {1/2, 2/2} {
            \draw[gray] (c) .. controls ({\n*\t*(-2*sqrt(2)/3) + (1-\n)*(-sqrt(2))}, {\n*\t*(0) + (1-\n)*(sqrt(6)/3)}, {\n*\t*(1/3) + (1-\n)*(-1)}) and ({\n*\t*(-2*sqrt(2)/3) + (1-\n)*(-sqrt(2))}, {\n*\t*(0) + (1-\n)*(-sqrt(6)/3)}, {\n*\t*(1/3) + (1-\n)*(-1)}) .. (d);
        }
    }
    \foreach \t in {0, 1/2, 2/2} { 
        \foreach \n in {1/2, 2/2} {
            \draw[gray] (c) .. controls ({\n*\t*(0) + (1-\n)*(-sqrt(2))}, {\n*\t*(0) + (1-\n)*(sqrt(6)/3)}, {\n*\t*(-1) + (1-\n)*(-1)}) and ({\n*\t*(0) + (1-\n)*(-sqrt(2))}, {\n*\t*(0) + (1-\n)*(-sqrt(6)/3)}, {\n*\t*(-1) + (1-\n)*(-1)}) .. (d);
        }
    }

    \foreach \t in {0, 1/3, 2/3, 3/3} { 
        \draw[gray] (d) -- ({\t*(0)}, {\t*(0)}, {\t*(-1)});
        \draw[gray] (c) -- ({\t*(0)}, {\t*(0)}, {\t*(-1)});
    }
    \foreach \t in {0, 1/3, 2/3, 3/3} { 
        \draw[gray] (d) -- ({\t*(-2*sqrt(2)/3)}, {\t*(0)}, {\t*(1/3)});
        \draw[gray] (c) -- ({\t*(-2*sqrt(2)/3)}, {\t*(0)}, {\t*(1/3)});
    }

    \draw[red, very thick] (o) -- (a1);
    \draw[red, very thick] (o) -- (b1);

    \foreach \t in {0, 1/3, 2/3, 3/3} { 
        \draw[gray] (a) -- ({\t*(-2*sqrt(2)/3)}, {\t*(0)}, {\t*(1/3)});
    }
    \foreach \t in {0, 1/3, 2/3, 3/3} { 
        \draw[gray] (b) -- ({\t*(0)}, {\t*(0)}, {\t*(-1)});
    }

    \foreach \t in {0, 1/3, 2/3, 3/3} { 
        \draw[gray] (d) -- ({\t*(sqrt(2)/3)}, {\t*(-sqrt(6)/3)}, {\t*(1/3)});
    }
    \foreach \t in {0, 1/3, 2/3, 3/3} { 
        \draw[gray] (c) -- ({\t*(sqrt(2)/3)}, {\t*(sqrt(6)/3)}, {\t*(1/3)});
    }

    \foreach \t in {0, 1/2, 2/2} { 
        \foreach \n in {1/2, 2/2} {
            \draw[gray] (a) .. controls ({\n*\t*(-2*sqrt(2)/3) + (1-\n)*(-sqrt(2)/3)}, {\n*\t*(0) + (1-\n)*(-sqrt(6)/3)}, {\n*\t*(1/3) + (1-\n)*(5/3)}) and ({\n*\t*(-2*sqrt(2)/3) + (1-\n)*(-2*sqrt(2)/3)}, {\n*\t*(0) + (1-\n)*(-2*sqrt(6)/3)}, {\n*\t*(1/3) + (1-\n)*(1/3)}) .. (d);
        }
    }
    \foreach \t in {0, 1/2, 2/2} { 
        \foreach \n in {1/2, 2/2} {
            \draw[gray] (a) .. controls ({\n*\t*(sqrt(2)/3) + (1-\n)*(-sqrt(2)/3)}, {\n*\t*(-sqrt(6)/3) + (1-\n)*(-sqrt(6)/3)}, {\n*\t*(1/3) + (1-\n)*(5/3)}) and ({\n*\t*(sqrt(2)/3) + (1-\n)*(-2*sqrt(2)/3)}, {\n*\t*(-sqrt(6)/3) + (1-\n)*(-2*sqrt(6)/3)}, {\n*\t*(1/3) + (1-\n)*(1/3)}) .. (d);
        }
    }

    \foreach \t in {0, 1/2, 2/2} { 
        \foreach \n in {1/2, 2/2} {
            \draw[gray] (a) .. controls ({\n*\t*(-2*sqrt(2)/3) + (1-\n)*(-sqrt(2)/3)}, {\n*\t*(0) + (1-\n)*(sqrt(6)/3)}, {\n*\t*(1/3) + (1-\n)*(5/3)}) and ({\n*\t*(-2*sqrt(2)/3) + (1-\n)*(-2*sqrt(2)/3)}, {\n*\t*(0) + (1-\n)*(2*sqrt(6)/3)}, {\n*\t*(1/3) + (1-\n)*(1/3)}) .. (c);
        }
    }
    \foreach \t in {0, 1/2, 2/2} { 
        \foreach \n in {1/2, 2/2} {
            \draw[gray] (a) .. controls ({\n*\t*(sqrt(2)/3) + (1-\n)*(-sqrt(2)/3)}, {\n*\t*(sqrt(6)/3) + (1-\n)*(sqrt(6)/3)}, {\n*\t*(1/3) + (1-\n)*(5/3)}) and ({\n*\t*(sqrt(2)/3) + (1-\n)*(-2*sqrt(2)/3)}, {\n*\t*(sqrt(6)/3) + (1-\n)*(2*sqrt(6)/3)}, {\n*\t*(1/3) + (1-\n)*(1/3)}) .. (c);
        }
    }

    \foreach \t in {0, 1/2, 2/2} { 
        \foreach \n in {1/2, 2/2} {
            \draw[gray] (b) .. controls ({\n*\t*(0) + (1-\n)*(sqrt(2))}, {\n*\t*(0) + (1-\n)*(-sqrt(6)/3)}, {\n*\t*(-1) + (1-\n)*(-1)}) and ({\n*\t*(0) + (1-\n)*(0)}, {\n*\t*(0) + (1-\n)*(-2*sqrt(6)/3)}, {\n*\t*(-1) + (1-\n)*(-1)}) .. (d);
        }
    }
    \foreach \t in {0, 1/2, 2/2} { 
        \foreach \n in {1/2, 2/2} {
            \draw[gray] (b) .. controls ({\n*\t*(sqrt(2)/3) + (1-\n)*(sqrt(2))}, {\n*\t*(-sqrt(6)/3) + (1-\n)*(-sqrt(6)/3)}, {\n*\t*(1/3) + (1-\n)*(-1)}) and ({\n*\t*(sqrt(2)/3) + (1-\n)*(0)}, {\n*\t*(-sqrt(6)/3) + (1-\n)*(-2*sqrt(6)/3)}, {\n*\t*(1/3) + (1-\n)*(-1)}) .. (d);
        }
    }

    \foreach \t in {0, 1/2, 2/2} { 
        \foreach \n in {1/2, 2/2} {
            \draw[gray] (b) .. controls ({\n*\t*(0) + (1-\n)*(sqrt(2))}, {\n*\t*(0) + (1-\n)*(sqrt(6)/3)}, {\n*\t*(-1) + (1-\n)*(-1)}) and ({\n*\t*(0) + (1-\n)*(0)}, {\n*\t*(0) + (1-\n)*(2*sqrt(6)/3)}, {\n*\t*(-1) + (1-\n)*(-1)}) .. (c);
        }
    }
    \foreach \t in {0, 1/2, 2/2} { 
        \foreach \n in {1/2, 2/2} {
            \draw[gray] (b) .. controls ({\n*\t*(sqrt(2)/3) + (1-\n)*(sqrt(2))}, {\n*\t*(sqrt(6)/3) + (1-\n)*(sqrt(6)/3)}, {\n*\t*(1/3) + (1-\n)*(-1)}) and ({\n*\t*(sqrt(2)/3) + (1-\n)*(0)}, {\n*\t*(sqrt(6)/3) + (1-\n)*(2*sqrt(6)/3)}, {\n*\t*(1/3) + (1-\n)*(-1)}) .. (c);
        }
    }

    \foreach \t in {0, 1/3, 2/3, 3/3} { 
        \draw[gray] (a) -- ({\t*(sqrt(2)/3)}, {\t*(-sqrt(6)/3)}, {\t*(1/3)});
        \draw[gray] (b) -- ({\t*(sqrt(2)/3)}, {\t*(-sqrt(6)/3)}, {\t*(1/3)});
    }
    \foreach \t in {0, 1/3, 2/3, 3/3} { 
        \draw[gray] (a) -- ({\t*(sqrt(2)/3)}, {\t*(sqrt(6)/3)}, {\t*(1/3)});
        \draw[gray] (b) -- ({\t*(sqrt(2)/3)}, {\t*(sqrt(6)/3)}, {\t*(1/3)});
    }

    \draw[red, very thick] (o) -- (c1);
    \draw[red, very thick] (o) -- (d1);

    \filldraw[red] (o) circle (0.15em);
    \filldraw[red] (a1) circle (0.15em);
    \filldraw[red] (b1) circle (0.15em);

    \foreach \t in {0, 1/2, 2/2} { 
        \foreach \n in {1/2, 2/2} {
            \draw[gray] (a) .. controls ({\n*\t*(sqrt(2)/3) + (1-\n)*(2*sqrt(2)/3)}, {\n*\t*(-sqrt(6)/3) + (1-\n)*0},{\n*\t*(1/3) + (1-\n)*(5/3)}) and ({\n*\t*(sqrt(2)/3) + (1-\n)*(4*sqrt(2)/3)}, {\n*\t*(-sqrt(6)/3) + (1-\n)*0},{\n*\t*(1/3) + (1-\n)*(1/3)}) .. (b);
        }
    }
    \foreach \t in {0, 1/2, 2/2} { 
        \foreach \n in {1/2, 2/2} {
            \draw[gray] (a) .. controls ({\n*\t*(sqrt(2)/3) + (1-\n)*(2*sqrt(2)/3)}, {\n*\t*(sqrt(6)/3) + (1-\n)*0},{\n*\t*(1/3) + (1-\n)*(5/3)}) and ({\n*\t*(sqrt(2)/3) + (1-\n)*(4*sqrt(2)/3)}, {\n*\t*(sqrt(6)/3) + (1-\n)*0},{\n*\t*(1/3) + (1-\n)*(1/3)}) .. (b);
        }
    }

    \filldraw[red] (c1) circle (0.15em);
    \filldraw[red] (d1) circle (0.15em);

    \draw[very thick] (a) -- (b);
    \draw[very thick] (a) -- (c);
    \draw[very thick] (a) -- (d);
    \draw[very thick] (b) -- (c);
    \draw[very thick] (b) -- (d);
    
    \filldraw (a) circle (0.05em);
    \filldraw (b) circle (0.05em);
    \filldraw (c) circle (0.05em);
    \filldraw (d) circle (0.05em);
\end{scope}
\end{tikzpicture}
}}$
\caption{1-dimensional foliation on a tetrahedron}
\label{fig:tetrahedron_foliation}
\end{figure}

Consider the foliation inside a tetrahedron $\delta\in\Delta$ in a connected component $U_\delta\subset \delta\setminus\Delta_{(2)}^{\vee}$ of the complement of the branch cut. It admits two orientations, either all leaves are oriented away from the branch cut toward the ideal vertex or vice versa. We will lift these oriented foliations to the branched double cover. Consider two copies of $U_\delta$ that are the sheets of the double cover over $U_\delta$.
\begin{definition}[Sheet labeling convention]\label{defn: sheet labels}
In ``sheet 1'' over $U_\delta$, the leaves are oriented away from the branch cut. 
In ``sheet 2'' over $U_\delta$, they are oriented towards the branch cut.
\end{definition}
Note that as we go across the branch cut, the two sheets (hence the labels for these orientations) flip. 
See Figure \ref{fig:foliation_orientation_convention} for an illustration of this orientation convention. 
Consider the lift of the foliation in $M$ to the branched double cover $\widetilde{M}^{\mathrm{sing}}$. Note that the two orientations of the foliation in $M$ together give a genuine orientation of the lifted foliation in $\widetilde{M}^{\mathrm{sing}}$. 
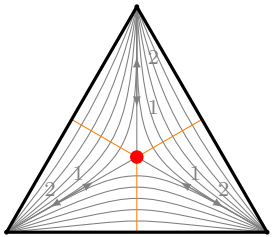
\begin{figure}
\centering
$\vcenter{\hbox{
\begin{tikzpicture}
\begin{scope}[scale = 2.0]
    \coordinate (o) at (0, 0);
    \coordinate (a) at (0, 1);
    \coordinate (a1) at (0, -1/2);
    \coordinate (b) at ({sqrt(3)/2}, -1/2);
    \coordinate (b1) at ({-sqrt(3)/4}, 1/4);
    \coordinate (c) at (-{sqrt(3)/2}, -1/2);
    \coordinate (c1) at ({sqrt(3)/4}, 1/4);

    \draw[gray] (a) -- (o);
    \draw[gray] (b) -- (o);
    \draw[gray] (c) -- (o);
    \foreach \n in {0, 1/5, 2/5, 3/5, 4/5} {
        \draw[gray] (a) .. controls ({1/(2*sqrt(3)) * \n}, {1/2 * \n}) and ({1/(sqrt(3)) * \n}, 0) .. (b);
        \draw[gray] (a) .. controls ({-1/(2*sqrt(3)) * \n}, {1/2 * \n}) and ({-1/(sqrt(3)) * \n}, 0) .. (c);
        \draw[gray] (b) .. controls ({-1/(2*sqrt(3)) * \n}, {-1/2 * \n}) and ({1/(2*sqrt(3)) * \n}, {-1/2 * \n}) .. (c);
    }

    \draw[gray, latex-latex] ({sqrt(3)/6}, -1/6) -- ({sqrt(3)/3}, -1/3);
    \node[gray, anchor = south] at ({sqrt(3)/6 + 0.1}, {-1/6 - 0.06}) {\tiny $1$};
    \node[gray, anchor = south] at ({sqrt(3)/3}, -1/3) {\tiny $2$};

    \draw[gray, latex-latex] ({-sqrt(3)/6}, -1/6) -- ({-sqrt(3)/3}, -1/3);
    \node[gray, anchor = south] at ({-sqrt(3)/6 - 0.1}, {-1/6 - 0.06}) {\tiny $1$};
    \node[gray, anchor = south] at ({-sqrt(3)/3}, -1/3) {\tiny $2$};

    \draw[gray, latex-latex] (0, 1/3) -- (0, 2/3);
    \node[gray, anchor = west] at (0, 1/3) {\tiny $1$};
    \node[gray, anchor = west] at (0, 2/3) {\tiny $2$};

    \draw[orange] (o) -- (a1);
    \draw[orange] (o) -- (b1);
    \draw[orange] (o) -- (c1);
    
    \draw[very thick] (a) -- (b);
    \draw[very thick] (a) -- (c);
    \draw[very thick] (b) -- (c);
    \filldraw[red] (o) circle (0.1em);
    \filldraw (a) circle (0.03em);
    \filldraw (b) circle (0.03em);
    \filldraw (c) circle (0.03em);
\end{scope}
\end{tikzpicture}
}}$
\caption{Orientation convention for the foliation, as seen on a boundary triangle of the tetrahedron}
\label{fig:foliation_orientation_convention}
\end{figure}

\subsection{Angle structures}
In this section we introduce the notions of angle groups and structures used to encode natural geometric structures on the 
the leaf spaces of the foliations discussed in Section \ref{ssec:brancheddoublecoversfoliations}. 

\begin{defn} \label{angle structure}
For a $3$-manifold $M$ with ideal triangulation $\Delta$, the \emph{angle group} $\mathbf{a}(\Delta)$ is the abelian group with generators a set of \emph{angle symbols}, one for each opposing pair of dihedral angles in each tetrahedron of $\Delta$ (see Figure \ref{fig:dihedral_angles}) and one auxiliary generator that we denote $\pi$,  subject to the following relations: 
\begin{itemize}
\begin{figure}\centering
$\vcenter{\hbox{
\tdplotsetmaincoords{60}{80}
\begin{tikzpicture}[tdplot_main_coords]
\begin{scope}[scale = 0.8, tdplot_main_coords]
    \coordinate (o) at (0, 0, 0);
    \coordinate (a) at (0, 0, 3);
    \coordinate (a1) at (0, 0, -1);
    \coordinate (b) at ({2*sqrt(2)}, 0, -1);
    \coordinate (b1) at ({-2*sqrt(2)/3}, 0, 1/3);
    \coordinate (c) at ({-sqrt(2)}, {sqrt(6)}, -1);
    \coordinate (c1) at ({sqrt(2)/3}, {-sqrt(6)/3}, 1/3);
    \coordinate (d) at ({-sqrt(2)}, {-sqrt(6)}, -1);    \coordinate (d1) at ({sqrt(2)/3}, {sqrt(6)/3}, 1/3);
    \coordinate (ab) at ({sqrt(2)}, 0, 1);
    \coordinate (ac) at ({-sqrt(2)/2}, {sqrt(6)/2}, 1);
    \coordinate (ad) at ({-sqrt(2)/2}, {-sqrt(6)/2}, 1);
    \coordinate (bc) at ({sqrt(2)/2}, {sqrt(6)/2}, -1);
    \coordinate (bd) at ({sqrt(2)/2}, {-sqrt(6)/2}, -1);
    \coordinate (cd) at ({-sqrt(2)}, 0, -1);
    \draw[very thick, dashed] (c) -- (d);
    \draw[white, ultra thick] (a) -- (b);
    \draw[very thick] (a) -- (b);
    \draw[very thick] (a) -- (c);
    \draw[very thick] (a) -- (d);
    \draw[very thick] (b) -- (c);
    \draw[very thick] (b) -- (d);
    \filldraw (a) circle (0.05em);
    \filldraw (b) circle (0.05em);
    \filldraw (c) circle (0.05em);
    \filldraw (d) circle (0.05em);
    \node[anchor = south east] at (ad) {$\theta'$};
    \node[anchor = south west] at (ac) {$\theta$};
    \node[anchor = north east] at (bd) {$\theta$};
    \node[anchor = north west] at (bc) {$\theta'$};
    \node[anchor = south west] at (ab) {$\theta''$};
    \node[anchor = north] at (cd) {$\theta''$};
\end{scope}
\end{tikzpicture}
}}$
\caption{Dihedral angles $\theta, \theta', \theta''$}
\label{fig:dihedral_angles}
\end{figure}
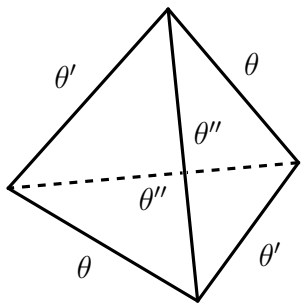
\item[$(1)$] The three angle symbols  $\theta$, $\theta'$, and $\theta''$ of any tetrahedron in $\Delta$ satisfies 
\[
\theta+\theta'+\theta''=\pi.
\]
\item[$(2)$] The sum of angle symbols around any internal edge is $2\pi$.
\item[$(2')$] In case $M$ has boundary, the sum of angle symbols around a boundary edge is $\pi$.
\end{itemize}
\end{defn}

In case $\Delta$ was an ideal hyperbolic triangulation, then specializing the angle symbols to the actual values of the angles gives a map $\mathbf{a}(\Delta) \to \R$ that takes the generator $\pi$ to the number $\pi$ and each angle symbol to some number in $(0, \pi)$. More generally, any map of the angle group with these properties is called an \emph{angle structure}.  Even more generally, any map $\mathbf{a}(\Delta) \to \R$ sending $\pi$ to $\pi$ is called a \emph{generalized angle structure}. 
If a generalized angle structure maps each angle symbol to $0$ or $\pi$, it is called a \emph{taut angle structure}. 

While the angle group may become trivial if some of the imposed relations are inconsistent, it is well-known, see e.g.\ \cite[Theorem 1]{Luo-Tillmann}, that any ideal triangulation of $M$ admits a generalized angle structure -- and hence has a non-trivial angle group -- if all cusps of $M$ are tori. 

We will write $\Z[\mathbf{a}(\Delta)]$ for the group ring of the angle group and denote the generators of this group ring as $a:= [\pi]$ and correspondingly $a^{\theta/\pi} := [\theta]$; $a$-powers then have the expected behavior of exponentials.

\subsection{An effectivity condition}\label{sec: combinatorial criterion for 1-form}
\begin{figure}
\centering
$\vcenter{\hbox{
\tdplotsetmaincoords{60}{65}
\begin{tikzpicture}[tdplot_main_coords]
\begin{scope}[scale = 0.8, tdplot_main_coords]
    \coordinate (o) at (0, 0, 0);
    \coordinate (a) at (0, 0, 3);
    \coordinate (a1) at (0, 0, -1);
    
    \coordinate (b) at ({2*sqrt(2)}, 0, -1);
    \coordinate (b1) at ({-2*sqrt(2)/3}, 0, 1/3);
    
    \coordinate (c) at ({-sqrt(2)}, {sqrt(6)}, -1);
    \coordinate (c1) at ({sqrt(2)/3}, {-sqrt(6)/3}, 1/3);
    
    \coordinate (d) at ({-sqrt(2)}, {-sqrt(6)}, -1);
    \coordinate (d1) at ({sqrt(2)/3}, {sqrt(6)/3}, 1/3);
    
    \coordinate (ab) at ({sqrt(2)}, 0, 1);
    \coordinate (ac) at ({-sqrt(2)/2}, {sqrt(6)/2}, 1);
    \coordinate (ad) at ({-sqrt(2)/2}, {-sqrt(6)/2}, 1);
    \coordinate (bc) at ({sqrt(2)/2}, {sqrt(6)/2}, -1);
    \coordinate (bd) at ({sqrt(2)/2}, {-sqrt(6)/2}, -1);
    \coordinate (cd) at ({-sqrt(2)}, 0, -1);

    \draw[red, very thick] (a1) .. controls (o) .. (c1);
    \draw[red, very thick] (b1) .. controls (o) .. (d1);

    \filldraw[red] (a1) circle (0.15em);
    \filldraw[red] (b1) circle (0.15em);
    \filldraw[red] (c1) circle (0.2em);
    \filldraw[red] (d1) circle (0.2em);

    \draw[blue, very thick] (a) -- (c);
    \draw[very thick] (a) -- (d);
    \draw[very thick] (b) -- (c);
    \draw[blue, very thick] (b) -- (d);
    \draw[very thick, dashed] (c) -- (d);

    \begin{scope}[thick, decoration={markings, mark=at position 0.5 with {\arrow{>}}}]
    \draw[very thick, darkgreen, postaction={decorate}] (ac) -- (ad);
    \draw[very thick, darkgreen, postaction={decorate}] (ad) -- (bd);
    \draw[very thick, darkgreen, postaction={decorate}] (bd) -- (bc);
    \draw[very thick, darkgreen, postaction={decorate}] (bc) -- (ac);
    \end{scope}
    \node[darkgreen, right] at ($1/2*(bc) + 1/2*(ac)$){$2$};
    \node[darkgreen, above] at ($1/2*(ac) + 1/2*(ad)$){$1$};
    \node[darkgreen, left] at ($1/2*(ad) + 1/2*(bd)$){$2$};
    \node[darkgreen, below] at ($1/2*(bd) + 1/2*(bc)$){$1$};

    \draw[very thick] (a) -- (b);

    \filldraw (a) circle (0.05em);
    \filldraw (b) circle (0.05em);
    \filldraw (c) circle (0.05em);
    \filldraw (d) circle (0.05em);

\end{scope}
\end{tikzpicture}
}}
\;\;\;=
\vcenter{\hbox{
\tdplotsetmaincoords{60}{65}
\begin{tikzpicture}[tdplot_main_coords]
\begin{scope}[scale = 0.8, tdplot_main_coords]
    \coordinate (o) at (0, 0, 0);
    \coordinate (a) at (0, 0, 3);
    \coordinate (a1) at (0, 0, -1);
    
    \coordinate (b) at ({2*sqrt(2)}, 0, -1);
    \coordinate (b1) at ({-2*sqrt(2)/3}, 0, 1/3);
    
    \coordinate (c) at ({-sqrt(2)}, {sqrt(6)}, -1);
    \coordinate (c1) at ({sqrt(2)/3}, {-sqrt(6)/3}, 1/3);
    
    \coordinate (d) at ({-sqrt(2)}, {-sqrt(6)}, -1);
    \coordinate (d1) at ({sqrt(2)/3}, {sqrt(6)/3}, 1/3);
    
    \coordinate (ab) at ({sqrt(2)}, 0, 1);
    \coordinate (ac) at ({-sqrt(2)/2}, {sqrt(6)/2}, 1);
    \coordinate (ad) at ({-sqrt(2)/2}, {-sqrt(6)/2}, 1);
    \coordinate (bc) at ({sqrt(2)/2}, {sqrt(6)/2}, -1);
    \coordinate (bd) at ({sqrt(2)/2}, {-sqrt(6)/2}, -1);
    \coordinate (cd) at ({-sqrt(2)}, 0, -1);

    \begin{scope}[thick, decoration={markings, mark=at position 0.5 with {\arrow{>}}}]
    \draw[very thick, darkgreen, postaction={decorate}] (bd) -- (cd);
    \draw[very thick, darkgreen, postaction={decorate}] (cd) -- (ac);
    \end{scope}

    \draw[white, line width=5] (a1) .. controls (o) .. (c1);
    \draw[white, line width=5] (b1) .. controls (o) .. (d1);
    \draw[red, very thick] (a1) .. controls (o) .. (c1);
    \draw[red, very thick] (b1) .. controls (o) .. (d1);

    \filldraw[red] (a1) circle (0.15em);
    \filldraw[red] (b1) circle (0.15em);
    \filldraw[red] (c1) circle (0.2em);
    \filldraw[red] (d1) circle (0.2em);

    \draw[white, line width=5] (ab) -- (bd);
    \draw[white, line width=5] (ac) -- (ab);
    \draw[very thick, darkgreen] (ac) -- ($1/2*(ac)+1/2*(cd)$);
    \draw[very thick, darkgreen] (bd) -- ($1/2*(bd)+1/2*(cd)$);
    
    \draw[blue, very thick] (a) -- (c);
    \draw[very thick] (a) -- (d);
    \draw[very thick] (b) -- (c);
    \draw[blue, very thick] (b) -- (d);
    \draw[very thick, dashed] (c) -- (d);

    \begin{scope}[thick, decoration={markings, mark=at position 0.5 with {\arrow{>}}}]
    \draw[very thick, darkgreen, postaction={decorate}] (ab) -- (bd);
    \draw[very thick, darkgreen, postaction={decorate}] (ac) -- (ab);
    \end{scope}
    \node[darkgreen, right] at ($1/2*(ac) + 1/2*(ab)$){$1$};
    \node[darkgreen, below] at ($1/2*(ab) + 1/2*(bd)$){$2$};
    \node[darkgreen, left] at ($1/2*(bd) + 1/2*(cd)$){$1$};
    \node[darkgreen, above] at ($1/2*(cd) + 1/2*(ac)$){$2$};

    \draw[very thick] (a) -- (b);

    \filldraw (a) circle (0.05em);
    \filldraw (b) circle (0.05em);
    \filldraw (c) circle (0.05em);
    \filldraw (d) circle (0.05em);

\end{scope}
\end{tikzpicture}
}}
$
\caption{Distinguished cycle (green) in the branched double cover associated to a marked ideal tetrahedron; the numbers indicate the sheet labels away from the standard branch cut (Figure \ref{fig:tetrahedron_branchcut}). These cycles remain the same if we reverse the orientation and flip the sheet labels at the same time. }
\label{fig:cycles-signed-taut-tetrahedra}
\end{figure}
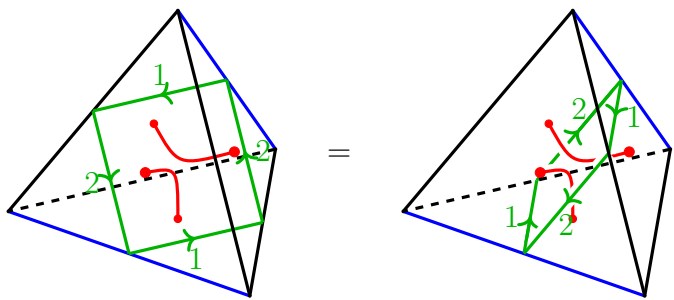

For each marked tetrahedron, there is a \emph{distinguished cycle} in the first homology of the corresponding branched double cover, see Figure \ref{fig:cycles-signed-taut-tetrahedra}.  Throughout this article, we will consider only covers satisfying the following: 

\begin{definition}[Effectivity condition]\label{defn: positivity condition}
We say that a marking on an ideal triangulation is {\em effective} if the distinguished cycles span a strictly convex cone in $H^1(\widetilde{M}, \R)$, or equivalently, if there exists a closed 1-form $\zeta$ on $\widetilde{M}$ that is positive on all distinguished cycles. 
\end{definition}

We will need this condition to define the following completion of the skein: 

\begin{definition}\label{defn: completed skein}
For a cover $\widetilde{M}$ associated to an effective marking, and some (any) choice of corresponding closed 1-form $\zeta$,
we write $\widehat{\Sk}^{\zeta}(\widetilde{M}; \widetilde{\sigma})$ for the completion of $\Sk(\widetilde{M}; \widetilde{\sigma})$ that consists of all infinite sums $\sum_{\alpha} \kappa_\alpha$,
$\kappa_\alpha\in \Sk(\widetilde{M},\widetilde{\sigma})$, such that if $[\kappa_\alpha]\in H_1(\widetilde{M})$ denotes the homology class represented by $\kappa_\alpha$ then for any $N>0$, $\int_{[\kappa_\alpha]}\zeta<N$ for only finitely many $\kappa_\alpha$.  
\end{definition}

We split the homology and cohomology of $\widetilde M$ by its characters for the involution $p$: 
\begin{eqnarray*}
H^1(\widetilde M, \R) & = & H^1(\widetilde M, \R)^{\mathrm{even}} \oplus H^1(\widetilde M, \R)^{\mathrm{odd}} \\ 
H_1(\widetilde M, \R) & = & H_1(\widetilde M, \R)^{\mathrm{even}} \oplus H_1(\widetilde M, \R)^{\mathrm{odd}}
\end{eqnarray*}
and we use similar notation for the corresponding splitting of 1-forms or 1-chains; i.e.
\[
\zeta^{\mathrm{odd}} := \frac{1}{2}(\zeta - p^*\zeta)\quad\text{and}\quad \zeta^{\mathrm{even}}:= \frac{1}{2}(\zeta + p^*\zeta).
\]
As always for Galois covers, the invariant co/homology is  identified with the co/homology of $M$ by pushforward or pullback.  A distinguished cycle has image contained in a tetrahedron of $M$, so has vanishing even class.  Thus if $\zeta$ is positive on all distinguished cycles, so is $\zeta^{\mathrm{odd}}$. 

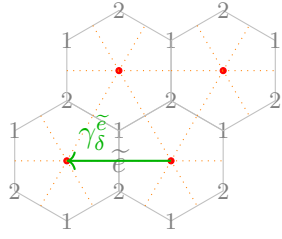
\begin{figure}
\centering
\[
\vcenter{\hbox{
\begin{tikzpicture}[scale=0.8]
\begin{scope}
    \filldraw[red] (0, 0) circle (0.05);
    \draw[lightgray] (0, 1) -- ({sqrt(3)/2}, 1/2) -- ({sqrt(3)/2}, -1/2) -- (0, -1) -- ({-sqrt(3)/2}, -1/2) -- ({-sqrt(3)/2}, 1/2) -- cycle;
    \draw[orange, dotted] ({-sqrt(3)/4}, -3/4) -- ({sqrt(3)/4}, 3/4);
    \draw[orange, dotted] ({-sqrt(3)/2}, 0) -- ({sqrt(3)/2}, 0);
    \draw[orange, dotted] ({-sqrt(3)/4}, 3/4) -- ({sqrt(3)/4}, -3/4);
    \node[gray] at ({sqrt(3)/2}, 1/2){\scriptsize $1$};
    \node[gray] at (0, 1){\scriptsize $2$};
    \node[gray] at ({-sqrt(3)/2}, 1/2){\scriptsize $1$};
    \node[gray] at ({-sqrt(3)/2}, -1/2){\scriptsize $2$};
    \node[gray] at (0, -1){\scriptsize $1$};
    \node[gray] at ({sqrt(3)/2}, -1/2){\scriptsize $2$};
\end{scope}
\begin{scope}[shift={({sqrt(3)}, 0)}]
    \filldraw[red] (0, 0) circle (0.05);
    \draw[lightgray] (0, 1) -- ({sqrt(3)/2}, 1/2) -- ({sqrt(3)/2}, -1/2) -- (0, -1) -- ({-sqrt(3)/2}, -1/2) -- ({-sqrt(3)/2}, 1/2) -- cycle;
    \draw[orange, dotted] ({-sqrt(3)/4}, -3/4) -- ({sqrt(3)/4}, 3/4);
    \draw[orange, dotted] ({-sqrt(3)/2}, 0) -- ({sqrt(3)/2}, 0);
    \draw[orange, dotted] ({-sqrt(3)/4}, 3/4) -- ({sqrt(3)/4}, -3/4);
    \node[gray] at ({sqrt(3)/2}, 1/2){\scriptsize $1$};
    \node[gray] at (0, 1){\scriptsize $2$};
    \node[gray] at (0, -1){\scriptsize $1$};
    \node[gray] at ({sqrt(3)/2}, -1/2){\scriptsize $2$};
\end{scope}
\begin{scope}[shift={({sqrt(3)/2}, 3/2)}]
    \filldraw[red] (0, 0) circle (0.05);
    \draw[lightgray] (0, 1) -- ({sqrt(3)/2}, 1/2) -- ({sqrt(3)/2}, -1/2) -- (0, -1) -- ({-sqrt(3)/2}, -1/2) -- ({-sqrt(3)/2}, 1/2) -- cycle;
    \draw[orange, dotted] ({-sqrt(3)/4}, -3/4) -- ({sqrt(3)/4}, 3/4);
    \draw[orange, dotted] ({-sqrt(3)/2}, 0) -- ({sqrt(3)/2}, 0);
    \draw[orange, dotted] ({-sqrt(3)/4}, 3/4) -- ({sqrt(3)/4}, -3/4);
    \node[gray] at ({sqrt(3)/2}, 1/2){\scriptsize $1$};
    \node[gray] at (0, 1){\scriptsize $2$};
    \node[gray] at ({-sqrt(3)/2}, 1/2){\scriptsize $1$};
\end{scope}
\begin{scope}[shift={({3*sqrt(3)/2}, 3/2)}]
    \filldraw[red] (0, 0) circle (0.05);
    \draw[lightgray] (0, 1) -- ({sqrt(3)/2}, 1/2) -- ({sqrt(3)/2}, -1/2) -- (0, -1) -- ({-sqrt(3)/2}, -1/2) -- ({-sqrt(3)/2}, 1/2) -- cycle;
    \draw[orange, dotted] ({-sqrt(3)/4}, -3/4) -- ({sqrt(3)/4}, 3/4);
    \draw[orange, dotted] ({-sqrt(3)/2}, 0) -- ({sqrt(3)/2}, 0);
    \draw[orange, dotted] ({-sqrt(3)/4}, 3/4) -- ({sqrt(3)/4}, -3/4);
    \node[gray] at ({sqrt(3)/2}, 1/2){\scriptsize $1$};
    \node[gray] at (0, 1){\scriptsize $2$};
    \node[gray] at ({sqrt(3)/2}, -1/2){\scriptsize $2$};
\end{scope}
\node[gray] at ({sqrt(3)/2}, 0){$\widetilde{e}$};
\draw[darkgreen, thick, <-] (0, 0) -- ({sqrt(3)}, 0);
\node[darkgreen, above left] at ({sqrt(3)/2}, 0){$\gamma_\delta^{\widetilde{e}}$};
\end{tikzpicture}
}}
\]
\caption{The 1-chain $\gamma_\delta^{\widetilde{e}} \in H_1(\widetilde{\delta}, \partial \widetilde{\sigma})$, drawn on the boundary, viewed from outside of $\widetilde{\delta}$.}
\label{fig:angular-1-chain}
\end{figure}

To an edge $\widetilde{e}$ of the double cover $\widetilde{\delta}$ of a marked tetrahedron $\delta \in \Delta^3$,
we associate a 1-chain $\gamma_\delta^{\widetilde{e}}$ in $\widetilde{\delta}$ connecting the two barycenters of the adjacent faces, oriented so that the puncture on sheet 1 (resp., 2) is to the right (resp., left); see Figure \ref{fig:angular-1-chain}. 
We will be interested in the periods 
\[
x_\delta^{\widetilde{e}} := \int_{\gamma_{\delta}^{\widetilde{e}}} \zeta^{\mathrm{odd}}. 
\]
If $\widetilde{e}_1$ and $\widetilde{e}_2$ are two edges of $\widetilde{\delta}$ that project down to the same edge $e$ of $\delta$, then 
$\gamma_\delta^{\widetilde{e}_1} + \gamma_\delta^{\widetilde{e}_2}$ is a cycle, see Figure \ref{fig:1-cycle-gamma}. 
This cycle is odd: $p( \gamma_\delta^{\widetilde{e}_1})= -\gamma_\delta^{\widetilde{e}_2}$. 
It follows that 
\[
\int_{\gamma_\delta^{\widetilde{e}_1} + \gamma_\delta^{\widetilde{e}_2}} \zeta = 
\int_{\gamma_\delta^{\widetilde{e}_1}} (\zeta -p^*\zeta) = 2 x_\delta^{\widetilde{e}_1} = 2x_\delta^{\widetilde{e}_2}. 
\]
Since $x_\delta^{\widetilde{e}_1} = x_\delta^{\widetilde{e}_2}$, we simply write $x_\delta^{e}$ for this common value. 

\begin{figure}
\centering
\[
\vcenter{\hbox{
\begin{tikzpicture}[scale=0.8]
\begin{scope}
    \filldraw[red] (0, 0) circle (0.05);
    \draw[lightgray] (0, 1) -- ({sqrt(3)/2}, 1/2) -- ({sqrt(3)/2}, -1/2) -- (0, -1) -- ({-sqrt(3)/2}, -1/2) -- ({-sqrt(3)/2}, 1/2) -- cycle;
    \draw[orange, dotted] ({-sqrt(3)/4}, -3/4) -- ({sqrt(3)/4}, 3/4);
    \draw[orange, dotted] ({-sqrt(3)/2}, 0) -- ({sqrt(3)/2}, 0);
    \draw[orange, dotted] ({-sqrt(3)/4}, 3/4) -- ({sqrt(3)/4}, -3/4);
    \node[gray] at ({sqrt(3)/2}, 1/2){\scriptsize $1$};
    \node[gray] at (0, 1){\scriptsize $2$};
    \node[gray] at ({-sqrt(3)/2}, 1/2){\scriptsize $1$};
    \node[gray] at ({-sqrt(3)/2}, -1/2){\scriptsize $2$};
    \node[gray] at (0, -1){\scriptsize $1$};
    \node[gray] at ({sqrt(3)/2}, -1/2){\scriptsize $2$};
    \node[gray] at ({-sqrt(3)/2}, 0){$\widetilde{e}_2$};
    \node[gray] at ({sqrt(3)/2}, 0){$\widetilde{e}_1$};
\end{scope}
\begin{scope}[shift={({sqrt(3)}, 0)}]
    \filldraw[red] (0, 0) circle (0.05);
    \draw[lightgray] (0, 1) -- ({sqrt(3)/2}, 1/2) -- ({sqrt(3)/2}, -1/2) -- (0, -1) -- ({-sqrt(3)/2}, -1/2) -- ({-sqrt(3)/2}, 1/2) -- cycle;
    \draw[orange, dotted] ({-sqrt(3)/4}, -3/4) -- ({sqrt(3)/4}, 3/4);
    \draw[orange, dotted] ({-sqrt(3)/2}, 0) -- ({sqrt(3)/2}, 0);
    \draw[orange, dotted] ({-sqrt(3)/4}, 3/4) -- ({sqrt(3)/4}, -3/4);
    \node[gray] at ({sqrt(3)/2}, 1/2){\scriptsize $1$};
    \node[gray] at (0, 1){\scriptsize $2$};
    \node[gray] at (0, -1){\scriptsize $1$};
    \node[gray] at ({sqrt(3)/2}, -1/2){\scriptsize $2$};
    \node[gray] at ({sqrt(3)/2}, 0){$\widetilde{e}_2$};
\end{scope}
\begin{scope}[shift={({sqrt(3)/2}, 3/2)}]
    \filldraw[red] (0, 0) circle (0.05);
    \draw[lightgray] (0, 1) -- ({sqrt(3)/2}, 1/2) -- ({sqrt(3)/2}, -1/2) -- (0, -1) -- ({-sqrt(3)/2}, -1/2) -- ({-sqrt(3)/2}, 1/2) -- cycle;
    \draw[orange, dotted] ({-sqrt(3)/4}, -3/4) -- ({sqrt(3)/4}, 3/4);
    \draw[orange, dotted] ({-sqrt(3)/2}, 0) -- ({sqrt(3)/2}, 0);
    \draw[orange, dotted] ({-sqrt(3)/4}, 3/4) -- ({sqrt(3)/4}, -3/4);
    \node[gray] at ({sqrt(3)/2}, 1/2){\scriptsize $1$};
    \node[gray] at (0, 1){\scriptsize $2$};
    \node[gray] at ({-sqrt(3)/2}, 1/2){\scriptsize $1$};
\end{scope}
\begin{scope}[shift={({3*sqrt(3)/2}, 3/2)}]
    \filldraw[red] (0, 0) circle (0.05);
    \draw[lightgray] (0, 1) -- ({sqrt(3)/2}, 1/2) -- ({sqrt(3)/2}, -1/2) -- (0, -1) -- ({-sqrt(3)/2}, -1/2) -- ({-sqrt(3)/2}, 1/2) -- cycle;
    \draw[orange, dotted] ({-sqrt(3)/4}, -3/4) -- ({sqrt(3)/4}, 3/4);
    \draw[orange, dotted] ({-sqrt(3)/2}, 0) -- ({sqrt(3)/2}, 0);
    \draw[orange, dotted] ({-sqrt(3)/4}, 3/4) -- ({sqrt(3)/4}, -3/4);
    \node[gray] at ({sqrt(3)/2}, 1/2){\scriptsize $1$};
    \node[gray] at (0, 1){\scriptsize $2$};
    \node[gray] at ({sqrt(3)/2}, -1/2){\scriptsize $2$};
\end{scope}
\draw[darkgreen, thick, <-] ({-sqrt(3)/4}, 3/4) -- ({sqrt(3)*7/4}, 3/4);
\end{tikzpicture}
}}
\;\;=\;\;
\vcenter{\hbox{
\begin{tikzpicture}[scale=0.8]
\begin{scope}
    \filldraw[red] (0, 0) circle (0.05);
    \draw[lightgray] (0, 1) -- ({sqrt(3)/2}, 1/2) -- ({sqrt(3)/2}, -1/2) -- (0, -1) -- ({-sqrt(3)/2}, -1/2) -- ({-sqrt(3)/2}, 1/2) -- cycle;
    \draw[orange, dotted] ({-sqrt(3)/4}, -3/4) -- ({sqrt(3)/4}, 3/4);
    \draw[orange, dotted] ({-sqrt(3)/2}, 0) -- ({sqrt(3)/2}, 0);
    \draw[orange, dotted] ({-sqrt(3)/4}, 3/4) -- ({sqrt(3)/4}, -3/4);
    \node[gray] at ({sqrt(3)/2}, 1/2){\scriptsize $1$};
    \node[gray] at (0, 1){\scriptsize $2$};
    \node[gray] at ({-sqrt(3)/2}, 1/2){\scriptsize $1$};
    \node[gray] at ({-sqrt(3)/2}, -1/2){\scriptsize $2$};
    \node[gray] at (0, -1){\scriptsize $1$};
    \node[gray] at ({sqrt(3)/2}, -1/2){\scriptsize $2$};
    \node[gray] at ({-sqrt(3)/2}, 0){$\widetilde{e}_2$};
    \node[gray] at ({sqrt(3)/2}, 0){$\widetilde{e}_1$};
\end{scope}
\begin{scope}[shift={({sqrt(3)}, 0)}]
    \filldraw[red] (0, 0) circle (0.05);
    \draw[lightgray] (0, 1) -- ({sqrt(3)/2}, 1/2) -- ({sqrt(3)/2}, -1/2) -- (0, -1) -- ({-sqrt(3)/2}, -1/2) -- ({-sqrt(3)/2}, 1/2) -- cycle;
    \draw[orange, dotted] ({-sqrt(3)/4}, -3/4) -- ({sqrt(3)/4}, 3/4);
    \draw[orange, dotted] ({-sqrt(3)/2}, 0) -- ({sqrt(3)/2}, 0);
    \draw[orange, dotted] ({-sqrt(3)/4}, 3/4) -- ({sqrt(3)/4}, -3/4);
    \node[gray] at ({sqrt(3)/2}, 1/2){\scriptsize $1$};
    \node[gray] at (0, 1){\scriptsize $2$};
    \node[gray] at (0, -1){\scriptsize $1$};
    \node[gray] at ({sqrt(3)/2}, -1/2){\scriptsize $2$};
    \node[gray] at ({sqrt(3)/2}, 0){$\widetilde{e}_2$};
\end{scope}
\begin{scope}[shift={({sqrt(3)/2}, 3/2)}]
    \filldraw[red] (0, 0) circle (0.05);
    \draw[lightgray] (0, 1) -- ({sqrt(3)/2}, 1/2) -- ({sqrt(3)/2}, -1/2) -- (0, -1) -- ({-sqrt(3)/2}, -1/2) -- ({-sqrt(3)/2}, 1/2) -- cycle;
    \draw[orange, dotted] ({-sqrt(3)/4}, -3/4) -- ({sqrt(3)/4}, 3/4);
    \draw[orange, dotted] ({-sqrt(3)/2}, 0) -- ({sqrt(3)/2}, 0);
    \draw[orange, dotted] ({-sqrt(3)/4}, 3/4) -- ({sqrt(3)/4}, -3/4);
    \node[gray] at ({sqrt(3)/2}, 1/2){\scriptsize $1$};
    \node[gray] at (0, 1){\scriptsize $2$};
    \node[gray] at ({-sqrt(3)/2}, 1/2){\scriptsize $1$};
\end{scope}
\begin{scope}[shift={({3*sqrt(3)/2}, 3/2)}]
    \filldraw[red] (0, 0) circle (0.05);
    \draw[lightgray] (0, 1) -- ({sqrt(3)/2}, 1/2) -- ({sqrt(3)/2}, -1/2) -- (0, -1) -- ({-sqrt(3)/2}, -1/2) -- ({-sqrt(3)/2}, 1/2) -- cycle;
    \draw[orange, dotted] ({-sqrt(3)/4}, -3/4) -- ({sqrt(3)/4}, 3/4);
    \draw[orange, dotted] ({-sqrt(3)/2}, 0) -- ({sqrt(3)/2}, 0);
    \draw[orange, dotted] ({-sqrt(3)/4}, 3/4) -- ({sqrt(3)/4}, -3/4);
    \node[gray] at ({sqrt(3)/2}, 1/2){\scriptsize $1$};
    \node[gray] at (0, 1){\scriptsize $2$};
    \node[gray] at ({sqrt(3)/2}, -1/2){\scriptsize $2$};
\end{scope}
\draw[darkgreen, thick, <-] (0, 0) -- ({sqrt(3)}, 0);
\draw[darkgreen, thick] (0, 0) -- ({-sqrt(3)/2}, 0);
\draw[darkgreen, thick, <-] ({sqrt(3)}, 0) -- ({sqrt(3)*3/2}, 0);
\end{tikzpicture}
}}
\]
\caption{
$\gamma_\delta^{\widetilde{e}_1} + \gamma_\delta^{\widetilde{e}_2}$ is the distinguished cycle if $e$ is marked.
}
\label{fig:1-cycle-gamma}
\end{figure}
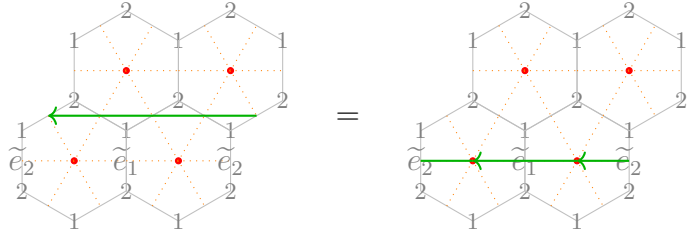

\begin{lemma}\label{lem: gamma relations}
    If three edges $\widetilde{e}, \widetilde{e}', \widetilde{e}''$ of $\widetilde{\delta}$ are adjacent to the same vertex, then $\gamma_{\delta}^{\widetilde{e}}+\gamma_{\delta}^{\widetilde{e}'}+\gamma_{\delta}^{\widetilde{e}''}$ is homologous to zero. 
    Moreover, fixing a lift $\widetilde{e}$ of an edge $e \in \Delta^1$, the sum (with multiplicity)
    \begin{equation}\label{eq:gluing-equation 1-chains}
\sum_{\delta\text{ {\rm abutting} }e}\gamma_{\delta}^{\widetilde{e}}
\end{equation}
    is also homologous to zero. 
\end{lemma}
\begin{proof}
For the first claim, it is enough to observe that $\gamma_{\delta}^{\widetilde{e}}+\gamma_{\delta}^{\widetilde{e}'}+\gamma_{\delta}^{\widetilde{e}''}$ is homologous to a small contractible loop close to the vertex shared by the three edges. 
For the second claim, the sum is homologous to a link of $\widetilde{e}$, which is contractible. 
\end{proof}

\begin{corollary}
$\quad$
\begin{enumerate}
\item If $e, e', e''$ are three edges of $\delta$ sharing a vertex,  $x_{\delta}^{e}+x_{\delta}^{e'}+x_{\delta}^{e''} = 0$. 
\item If $\{e,f\}$ is a pair of opposite edges of $\delta$, then $x_{\delta}^{e} = x_{\delta}^{f}$. 
\item If $e \in \Delta^1$ is an edge of the triangulation, then
\begin{equation}\label{eqn:gluing-rel}
\sum_{\delta\text{ {\rm abutting} }e}x_{\delta}^e = 0.
\end{equation}
(If $\delta$ has multiple edges abutting $e$, then $x_\delta^e$ should be counted with multiplicity accordingly in \eqref{eqn:gluing-rel}.)
\end{enumerate}
\end{corollary}
\begin{proof}
(1) and (3) are immediate corollaries of Lemma \ref{lem: gamma relations}, and (2) follows from (1). 
\end{proof}

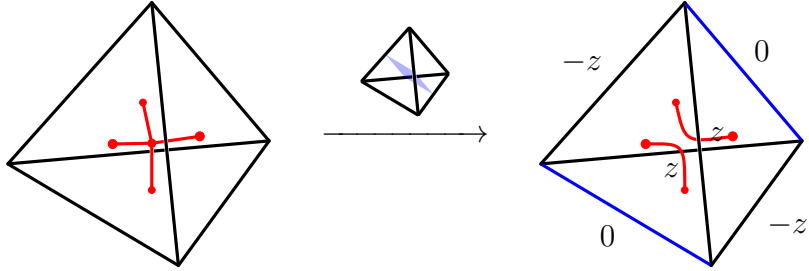
\begin{figure}
\centering
$\vcenter{\hbox{
\tdplotsetmaincoords{60}{80}
\begin{tikzpicture}[scale=0.9, tdplot_main_coords]
\begin{scope}[scale = 0.8, tdplot_main_coords]
    \coordinate (o) at (0, 0, 0);
    \coordinate (a) at (0, 0, 3);
    \coordinate (a1) at (0, 0, -1);
    
    \coordinate (b) at ({2*sqrt(2)}, 0, -1);
    \coordinate (b1) at ({-2*sqrt(2)/3}, 0, 1/3);
    
    \coordinate (c) at ({-sqrt(2)}, {sqrt(6)}, -1);
    \coordinate (c1) at ({sqrt(2)/3}, {-sqrt(6)/3}, 1/3);
    
    \coordinate (d) at ({-sqrt(2)}, {-sqrt(6)}, -1);
    \coordinate (d1) at ({sqrt(2)/3}, {sqrt(6)/3}, 1/3);
    
    \coordinate (ab) at ({sqrt(2)}, 0, 1);
    \coordinate (ac) at ({-sqrt(2)/2}, {sqrt(6)/2}, 1);
    \coordinate (ad) at ({-sqrt(2)/2}, {-sqrt(6)/2}, 1);
    \coordinate (bc) at ({sqrt(2)/2}, {sqrt(6)/2}, -1);
    \coordinate (bd) at ({sqrt(2)/2}, {-sqrt(6)/2}, -1);
    \coordinate (cd) at ({-sqrt(2)}, 0, -1);
    
    \draw[very thick] (c) -- (d);

    \draw[white, ultra thick] (o) -- (a1);
    
    \draw[red, very thick] (o) -- (a1);
    \draw[red, very thick] (o) -- (b1);
    \draw[red, very thick] (o) -- (c1);
    \draw[red, very thick] (o) -- (d1);

    \filldraw[red] (o) circle (0.18em);
    \filldraw[red] (a1) circle (0.15em);
    \filldraw[red] (b1) circle (0.15em);
    \filldraw[red] (c1) circle (0.2em);
    \filldraw[red] (d1) circle (0.2em);

    \draw[white, ultra thick] (a) -- (b);
    \draw[very thick] (a) -- (b);
    \draw[very thick] (a) -- (c);
    \draw[very thick] (a) -- (d);
    \draw[very thick] (b) -- (c);
    \draw[very thick] (b) -- (d);
    
    \filldraw (a) circle (0.05em);
    \filldraw (b) circle (0.05em);
    \filldraw (c) circle (0.05em);
    \filldraw (d) circle (0.05em);
\end{scope}
\end{tikzpicture}
}}
\overset{
\vcenter{\hbox{
\tdplotsetmaincoords{60}{80}
\begin{tikzpicture}[scale=0.3, tdplot_main_coords]
\begin{scope}[scale = 0.8, tdplot_main_coords]
    \coordinate (o) at (0, 0, 0);
    \coordinate (a) at (0, 0, 3);
    \coordinate (a1) at (0, 0, -1);
    
    \coordinate (b) at ({2*sqrt(2)}, 0, -1);
    \coordinate (b1) at ({-2*sqrt(2)/3}, 0, 1/3);
    
    \coordinate (c) at ({-sqrt(2)}, {sqrt(6)}, -1);
    \coordinate (c1) at ({sqrt(2)/3}, {-sqrt(6)/3}, 1/3);
    
    \coordinate (d) at ({-sqrt(2)}, {-sqrt(6)}, -1);
    \coordinate (d1) at ({sqrt(2)/3}, {sqrt(6)/3}, 1/3);
    
    \coordinate (ab) at ({sqrt(2)}, 0, 1);
    \coordinate (ac) at ({-sqrt(2)/2}, {sqrt(6)/2}, 1);
    \coordinate (ad) at ({-sqrt(2)/2}, {-sqrt(6)/2}, 1);
    \coordinate (bc) at ({sqrt(2)/2}, {sqrt(6)/2}, -1);
    \coordinate (bd) at ({sqrt(2)/2}, {-sqrt(6)/2}, -1);
    \coordinate (cd) at ({-sqrt(2)}, 0, -1);
    
    \draw[very thick] (c) -- (d);

    \fill[blue, opacity=0.3] (ab) -- (bc) -- (cd) -- (ad) -- cycle;

    \draw[white, ultra thick] (a) -- (b);
    \draw[very thick] (a) -- (b);
    \draw[very thick] (a) -- (c);
    \draw[very thick] (a) -- (d);
    \draw[very thick] (b) -- (c);
    \draw[very thick] (b) -- (d);
    
    \filldraw (a) circle (0.05em);
    \filldraw (b) circle (0.05em);
    \filldraw (c) circle (0.05em);
    \filldraw (d) circle (0.05em);
\end{scope}
\end{tikzpicture}
}}
}{\quad\xrightarrow{\hspace*{2cm}}}
\vcenter{\hbox{
\tdplotsetmaincoords{60}{80}
\begin{tikzpicture}[scale=0.9, tdplot_main_coords]
\begin{scope}[scale = 0.8, tdplot_main_coords]
    \coordinate (o) at (0, 0, 0);
    \coordinate (a) at (0, 0, 3);
    \coordinate (a1) at (0, 0, -1);
    
    \coordinate (b) at ({2*sqrt(2)}, 0, -1);
    \coordinate (b1) at ({-2*sqrt(2)/3}, 0, 1/3);
    
    \coordinate (c) at ({-sqrt(2)}, {sqrt(6)}, -1);
    \coordinate (c1) at ({sqrt(2)/3}, {-sqrt(6)/3}, 1/3);
    
    \coordinate (d) at ({-sqrt(2)}, {-sqrt(6)}, -1);
    \coordinate (d1) at ({sqrt(2)/3}, {sqrt(6)/3}, 1/3);
    
    \coordinate (ab) at ({sqrt(2)}, 0, 1);
    \coordinate (ac) at ({-sqrt(2)/2}, {sqrt(6)/2}, 1);
    \coordinate (ad) at ({-sqrt(2)/2}, {-sqrt(6)/2}, 1);
    \coordinate (bc) at ({sqrt(2)/2}, {sqrt(6)/2}, -1);
    \coordinate (bd) at ({sqrt(2)/2}, {-sqrt(6)/2}, -1);
    \coordinate (cd) at ({-sqrt(2)}, 0, -1);
    
    \draw[very thick] (c) -- (d);

    \draw[red, very thick] (a1) .. controls (o) .. (c1);
    \draw[red, very thick] (b1) .. controls (o) .. (d1);

    \filldraw[red] (a1) circle (0.15em);
    \filldraw[red] (b1) circle (0.15em);
    \filldraw[red] (c1) circle (0.2em);
    \filldraw[red] (d1) circle (0.2em);

    \draw[white, ultra thick] (a) -- (b);
    \draw[very thick] (a) -- (b);
    \draw[blue, very thick] (a) -- (c);
    \draw[very thick] (a) -- (d);
    \draw[very thick] (b) -- (c);
    \draw[blue, very thick] (b) -- (d);
    
    \filldraw (a) circle (0.05em);
    \filldraw (b) circle (0.05em);
    \filldraw (c) circle (0.05em);
    \filldraw (d) circle (0.05em);

    \node[anchor=west] at (ab){$z$};
    \node[anchor=north] at (cd){$z$};
    \node[anchor=south west] at (ac){$0$};
    \node[anchor=south east] at (ad){$-z$};
    \node[anchor=north west] at (bc){$-z$};
    \node[anchor=north east] at (bd){$0$};
\end{scope}
\end{tikzpicture}
}}
$
\caption{Branch locus for the smoothing of a marked tetrahedron.  The quantity $z>0$ records the half of the value of a 1-form on the distinguished 1-cycle (Figure \ref{fig:cycles-signed-taut-tetrahedra}) on the boundary.}
\label{fig:tetrahedron_resolved_Lagrangian}
\end{figure}

\begin{corollary}\label{cor:z_delta}
For a marked tetrahedron $\delta$ and a closed 1-form $\zeta$ on the branched double cover $\widetilde{M}$, there is a unique $z_\delta \in \R$ such that
\begin{equation}\label{eqn:x-to-z}
x_\delta^e = 
\begin{cases}
0 &\text{if }e\text{ is a marked edge}, \\
\pm z_\delta &\text{otherwise}, 
\end{cases}
\end{equation}
where the sign is $+$ (resp., $-$) if the marked edges are counterclockwise (resp., clockwise) from $e$ when viewed from vertices; 
see the right-hand side of Figure \ref{fig:tetrahedron_resolved_Lagrangian}.  
\end{corollary}
\begin{proof}
The period of $\zeta$ along one of the three cycles $(1,0)$, $(0,1)$, and $(-1,-1)$ -- corresponding to quadrilateral slices in Figure \ref{fig:tetrahedron-boundary-cycles} -- is equal to $2 x_{\delta}^e$, where $e$ is an edge avoiding the quadrilateral slice; see Figure \ref{fig:1-cycle-gamma}.
Since the cycle corresponding to the marked edges is collapsed, 
the result follows. 
\end{proof}

We call $z_\delta$ the \emph{angular period} of $\zeta$ over $\delta$. 
The period of $\zeta$ along the distinguished cycle of $\widetilde{\delta}$ is $2 z_\delta$, so $\zeta$ is positive on all distinguished cycles if and only if the $z_\delta$ are all positive. 
By rewriting the gluing equations \eqref{eqn:gluing-rel} in terms of $\{z_\delta\}_{\delta \in \Delta}$ using \eqref{eqn:x-to-z}, we obtain a system of homogeneous equations 
\begin{equation}\label{eq:z-gluing-equations}
\left\{\sum_{\delta \text{ abutting }e}\epsilon_\delta^e z_\delta^e = 0 \right\}_{e\in \Delta^1},
\end{equation}
where $\Delta^{1}$ denotes the edges in the triangulation $\Delta$, and 
$\epsilon_\delta^e \in \{0, \pm 1\}$ is determined as in \eqref{eqn:x-to-z}. Since the sum of angles in the generalized angle structure around any edge $e\in\Delta^1$ equals zero, each boundary component of $M$ is a torus.  It follows that the Euler characteristic of $M$ is zero.  
Then, if there are $t$ tetrahedra in $\Delta^3$, there must be $2t$ faces in $\Delta^2$, and hence $t$ edges in $\Delta^1$. The $t$ equations \eqref{eq:z-gluing-equations} are well-studied in the context of normal surface theory \cite{Luo-Tillmann}, and it is well-known that if $M$ has $b$ boundary components then at most $t-b$ of the $t$ equations are linearly independent. In particular, it admits an at least $b$-dimensional space of solutions. 
\begin{lemma}
Any solution $(z_1, \cdots, z_t)$ to \eqref{eq:z-gluing-equations} induces a closed and bounded 1-form $\zeta$ on $\widetilde{M}$, odd with respect to $p$, with given periods. 
\end{lemma}
\begin{proof}
Each $z_\delta$ determines a closed 1-form on $\widetilde{\delta}$, odd with respect to $p$, with the given period $z_\delta$. 
Since we are specifying the periods along the 1-chains ending on the barycenters, these 1-forms glue together to give a closed 1-form on the 3-manifold obtained by gluing $\widetilde{\delta}$'s along small neighborhoods of barycenters. 
The resulting 3-manifold is $\widetilde{M}$ minus the preimages of the edges in $\Delta^1$. We obtain $\widetilde{M}$ by attaching 2-handles along the links of those preimages. 
The gluing equations \eqref{eq:z-gluing-equations} are exactly the conditions that the closed 1-form extends over these 2-handles as a closed 1-form. After extension we get a closed $1$-form near the boundary of $\widetilde{M}$, $\sqcup_{j=1}^{2b} T^2\times [0,1)$, and we extend it to a bounded closed $1$-form over $\sqcup_{j=1}^{2b} T^2\times \R_+$.
\end{proof}
To this end, the effectivity condition simply requires all $z_\delta$'s to be positive. 
We summarize our discussion so far with the following:
\begin{proposition}\label{prop: existence of smoothing}
Let $\Delta$ be a marked ideal triangulation of $M$ with $t$ tetrahedra and $b$ boundary components. 
Then $\widetilde{M}$ admits a closed 1-form $\zeta$ positive on all distinguished cycles if and only if
there is a tuple of angular periods $(z_1,\dots,z_t)$ that satisfies \eqref{eq:z-gluing-equations} and has all $z_j > 0$. 
In particular, such solutions exist provided the $d$-dimensional solution space to \eqref{eq:z-gluing-equations}, $d\geq b \geq 1$, intersects the positive orthant $\bigcap_{1\leq j\leq t}\{z_j > 0\}$. \qed
\end{proposition}
\begin{remark}
By Stiemke's alternative \cite{Stiemke}, the criterion in Proposition \ref{prop: existence of smoothing} is equivalent to the the following condition: 
there is no $\{p_e\}_{e\in \Delta^1} \in \mathbb{R}^t$ such that $\{\sum_{e} p_e \epsilon_\delta^e\}_{\delta \in \Delta^3}$ is nonnegative but not identically zero. 
\end{remark}

\begin{example}
Consider the figure-eight knot complement ideally triangulated as in Figure \ref{fig:figure-eight-complement}. 
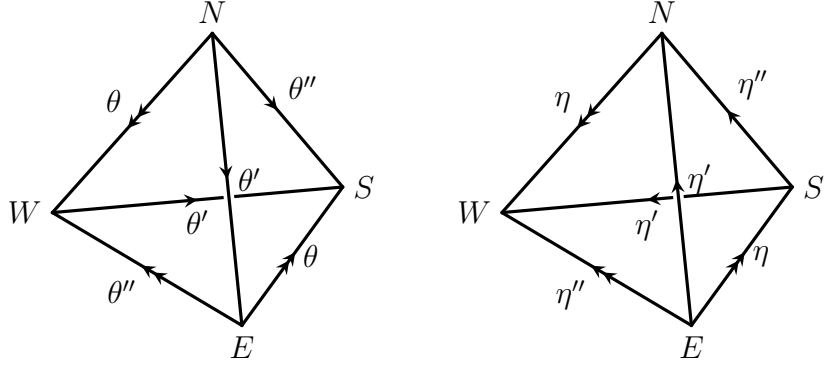
\begin{figure}
\centering
\[
\tdplotsetmaincoords{60}{80}
\begin{tikzpicture}[tdplot_main_coords]
\begin{scope}[scale = 0.8, tdplot_main_coords]
    \coordinate (o) at (0, 0, 0);
    \coordinate (a) at (0, 0, 3);
    \coordinate (a1) at (0, 0, -1);
    
    \coordinate (b) at ({2*sqrt(2)}, 0, -1);
    \coordinate (b1) at ({-2*sqrt(2)/3}, 0, 1/3);
    
    \coordinate (c) at ({-sqrt(2)}, {sqrt(6)}, -1);
    \coordinate (c1) at ({sqrt(2)/3}, {-sqrt(6)/3}, 1/3);
    
    \coordinate (d) at ({-sqrt(2)}, {-sqrt(6)}, -1);
    \coordinate (d1) at ({sqrt(2)/3}, {sqrt(6)/3}, 1/3);
    
    \coordinate (ab) at ({sqrt(2)}, 0, 1);
    \coordinate (ac) at ({-sqrt(2)/2}, {sqrt(6)/2}, 1);
    \coordinate (ad) at ({-sqrt(2)/2}, {-sqrt(6)/2}, 1);
    \coordinate (bc) at ({sqrt(2)/2}, {sqrt(6)/2}, -1);
    \coordinate (bd) at ({sqrt(2)/2}, {-sqrt(6)/2}, -1);
    \coordinate (cd) at ({-sqrt(2)}, 0, -1);
    
    \begin{scope}[thick,decoration={
    markings,
    mark=at position .5 with \arrow{stealth}}
    ] 
    \draw[postaction={decorate},very thick] (d) -- (c);
    \end{scope}
    
    \draw[white, line width = 4] (a) -- (b);
   
    \begin{scope}[thick,decoration={
    markings,
    mark=between positions 0.47 and .53 step 0.06 with \arrow{stealth}}
    ] 
    \draw[postaction={decorate},very thick] (b) -- (c);
    \draw[postaction={decorate},very thick] (b) -- (d);
    \draw[postaction={decorate},very thick] (a) -- (d);
    \end{scope}

    \begin{scope}[thick,decoration={
    markings,
    mark=at position .5 with \arrow{stealth}}
    ] 
    \draw[postaction={decorate},very thick] (a) -- (b);
    \draw[postaction={decorate},very thick] (a) -- (c);
    \end{scope}
    
    \filldraw (a) circle (0.05em);
    \filldraw (b) circle (0.05em);
    \filldraw (c) circle (0.05em);
    \filldraw (d) circle (0.05em);

    \node[above] at (a) {$N$};
    \node[below] at (b) {$E$};
    \node[right] at (c) {$S$};
    \node[left] at (d) {$W$};
    \node[above right] at (ac) {$\theta''$};
    \node[right] at (ab) {$\theta'$};
    \node[above left] at (ad) {$\theta$};
    \node[right] at (bc) {$\theta$};
    \node[below left] at (bd) {$\theta''$};
    \node[below] at (cd) {$\theta'$};
\end{scope}
\end{tikzpicture}
\hspace{2em}
\tdplotsetmaincoords{60}{80}
\begin{tikzpicture}[tdplot_main_coords]
\begin{scope}[scale = 0.8, tdplot_main_coords]
    \coordinate (o) at (0, 0, 0);
    \coordinate (a) at (0, 0, 3);
    \coordinate (a1) at (0, 0, -1);
    
    \coordinate (b) at ({2*sqrt(2)}, 0, -1);
    \coordinate (b1) at ({-2*sqrt(2)/3}, 0, 1/3);
    
    \coordinate (c) at ({-sqrt(2)}, {sqrt(6)}, -1);
    \coordinate (c1) at ({sqrt(2)/3}, {-sqrt(6)/3}, 1/3);
    
    \coordinate (d) at ({-sqrt(2)}, {-sqrt(6)}, -1);
    \coordinate (d1) at ({sqrt(2)/3}, {sqrt(6)/3}, 1/3);
    
    \coordinate (ab) at ({sqrt(2)}, 0, 1);
    \coordinate (ac) at ({-sqrt(2)/2}, {sqrt(6)/2}, 1);
    \coordinate (ad) at ({-sqrt(2)/2}, {-sqrt(6)/2}, 1);
    \coordinate (bc) at ({sqrt(2)/2}, {sqrt(6)/2}, -1);
    \coordinate (bd) at ({sqrt(2)/2}, {-sqrt(6)/2}, -1);
    \coordinate (cd) at ({-sqrt(2)}, 0, -1);
    
    \begin{scope}[thick,decoration={
    markings,
    mark=at position .5 with \arrow{stealth}}
    ] 
    \draw[postaction={decorate},very thick] (c) -- (d);
    \end{scope}
    
    \draw[white, line width = 4] (a) -- (b);
   
    \begin{scope}[thick,decoration={
    markings,
    mark=between positions 0.47 and .53 step 0.06 with \arrow{stealth}}
    ] 
    \draw[postaction={decorate},very thick] (b) -- (c);
    \draw[postaction={decorate},very thick] (b) -- (d);
    \draw[postaction={decorate},very thick] (a) -- (d);
    \end{scope}

    \begin{scope}[thick,decoration={
    markings,
    mark=at position .5 with \arrow{stealth}}
    ] 
    \draw[postaction={decorate},very thick] (b) -- (a);
    \draw[postaction={decorate},very thick] (c) -- (a);
    \end{scope}
    
    \filldraw (a) circle (0.05em);
    \filldraw (b) circle (0.05em);
    \filldraw (c) circle (0.05em);
    \filldraw (d) circle (0.05em);

    \node[above] at (a) {$N$};
    \node[below] at (b) {$E$};
    \node[right] at (c) {$S$};
    \node[left] at (d) {$W$};
    \node[above right] at (ac) {$\eta''$};
    \node[right] at (ab) {$\eta'$};
    \node[above left] at (ad) {$\eta$};
    \node[right] at (bc) {$\eta$};
    \node[below left] at (bd) {$\eta''$};
    \node[below] at (cd) {$\eta'$};
\end{scope}
\end{tikzpicture}
\]
\caption{Figure-eight knot complement triangulated into two ideal tetrahedra}
\label{fig:figure-eight-complement}
\end{figure}
Since there are only two tetrahedra (hence 2 edges), besides $\theta + \theta' + \theta'' = \pi$ and $\eta + \eta' + \eta'' = \pi$, there is only one equation that defines a generalized angle structure, which is
\[
2\theta + \theta'' + 2\eta + \eta'' = 2\pi.
\]
As a result, this triangulation admits 3 taut angle structures:
\begin{align*}
(\theta, \theta', \theta'', \eta, \eta', \eta'') &= (\pi, 0, 0, 0, \pi, 0), \\
\text{or } (\theta, \theta', \theta'', \eta, \eta', \eta'') &= (0, 0, \pi, 0, 0, \pi), \\
\text{or } (\theta, \theta', \theta'', \eta, \eta', \eta'') &= (0, \pi, 0, \pi, 0, 0). 
\end{align*}
Note, for any pair $\theta^{\square} \in \{\theta, \theta', \theta''\}$ and $\eta^{\square} \in \{\eta, \eta', \eta''\}$, there is a taut angle structure for which both $\theta^{\square}$ and $\eta^{\square}$ are zeros, meaning, for any choice of tangle (corresponding to a choice of marked quad type) in each tetrahedron, there is a taut angle structure compatible with it.

In this example, the equation \eqref{eq:z-gluing-equations} is given by
\begin{equation}\label{eq:figure-eight-example-gluing}
2 x_1^{\theta} + x_1^{\theta''} + 2 x_2^{\eta} + x_2^{\eta''} = 0,
\end{equation}
where 
\[
x_1^{\theta^{\bullet}} = 
\begin{cases}
0 &\text{if }\theta^{\bullet} \text{ is the marked quad type }\theta^{\square}, \\
z_1 &\text{if }\theta^{\square}\text{ is counterclockwise from }\theta^{\bullet},\\
-z_1 &\text{if }\theta^{\square}\text{ is clockwise from }\theta^{\bullet},
\end{cases}
\]
and likewise for $x_2^{\eta^{\bullet}}$. 
There are solutions $(z_1, z_2)$ such that $z_1, z_2 > 0$ if and only if the coefficient of $z_1$ and that of $z_2$ in \eqref{eq:figure-eight-example-gluing} have opposite signs, which is the case when the marking $(\theta^{\square}, \eta^{\square})$ is either $(\theta, \eta''), (\theta', \eta''), (\theta'', \eta)$, or $(\theta'', \eta')$. 
\end{example}

\subsection{HOMFLYPT skein modules of singular branched double covers}

\begin{defn}\label{defn:SkeinModuleOfBranchedDoubleCover}
Let $M$ be a 3-manifold with ideal triangulation $\Delta$, associated (singular) branched double cover $\widetilde{M}^{\mathrm{sing}}$ and branch locus $\widetilde{\tau}\subset \widetilde{M}^{\mathrm{sing}}$. 
We write $\mathrm{Sk}_{a,z}(\widetilde{M}^{\mathrm{sing}};\widetilde{\tau})$ for the $\Z[\mathbf{a}(\Delta)][z^{\pm}]$
module generated by the isotopy classes of framed oriented links in $\widetilde{M}^{\mathrm{sing}}\setminus \widetilde{\tau}$, subject to the following relations 
\begin{itemize}
\item
the usual skein relations \eqref{eq:skeinrel1}, \eqref{eq:skeinrel2}, \eqref{eq:skeinrel3},
\item 
relation \eqref{eq:skeinrel4} holds as links cross $\widetilde{\tau}$; i.e., the branch locus $\widetilde{\tau}$ is a sign line, 
\item
the following relation near the singular cone point in any tetrahedron of $\Delta$:
\begin{gather}
\vcenter{\hbox{
\begin{tikzpicture}[scale=0.7]
\draw[ultra thick, ->] ({-sqrt(3)/2}, 1/2) -- ({sqrt(3)/2}, 1/2);
\draw[white, line width=5] (0, 0) -- (0, 1);
\draw[ultra thick, red] (0, 0) -- (0, 1);
\draw[ultra thick, red] (0, 0) -- ({-sqrt(3)/2}, -1/2);
\draw[ultra thick, red] (0, 0) -- ({sqrt(3)/2}, -1/2);
\draw[ultra thick, red, dotted] (0, 0) -- (0, -1);
\filldraw[orange, opacity=0.2] ({sqrt(3)}, 1) -- ({-sqrt(3)}, 1) -- (0, -2) -- cycle;
\node[anchor=north west] at ({-sqrt(3)}, 1){$\theta'$};
\node[anchor=north east] at ({sqrt(3)}, 1){$\theta$};
\node[anchor=south] at (0, -2){$\theta''$};
\node[anchor=south] at ({-sqrt(3)/2}, 1/2){$1$};
\node[anchor=south] at ({sqrt(3)/2}, 1/2){$2$};
\end{tikzpicture}
}}
\;\;=\;\;
a^{\frac{\theta}{\pi}}\;
\vcenter{\hbox{
\begin{tikzpicture}[scale=0.7]
\node[anchor=north west] at ({-sqrt(3)}, 1){$\theta'$};
\node[anchor=north east] at ({sqrt(3)}, 1){$\theta$};
\node[anchor=south] at (0, -2){$\theta''$};
\draw[ultra thick, ->] (0.1, -1/2) to[out=0, in=180] ($({sqrt(3)/2}, 1/2)$);
\draw[white, line width=5] (0, 0) -- ({sqrt(3)/2}, -1/2);
\draw[ultra thick, red] (0, 0) -- ({sqrt(3)/2}, -1/2);
\draw[ultra thick, red] (0, 0) -- (0, 1);
\draw[ultra thick, red, dotted] (0, 0) -- (0, -1);
\draw[ultra thick, red] (0, 0) -- ({-sqrt(3)/2}, -1/2);
\draw[white, line width=5] ({-sqrt(3)/2}, 1/2) to[out=0, in=180] (0.1, -1/2);
\draw[ultra thick] ({-sqrt(3)/2}, 1/2) to[out=0, in=180] (0.1, -1/2);
\node[anchor=south] at ({-sqrt(3)/2}, 1/2){$1$};
\node[anchor=south] at ({sqrt(3)/2}, 1/2){$2$};
\filldraw[orange, opacity=0.2] ({sqrt(3)}, 1) -- ({-sqrt(3)}, 1) -- (0, -2) -- cycle;
\end{tikzpicture}
}}
\;\;+\;\;
a^{-\frac{\theta'}{\pi}}\;
\vcenter{\hbox{
\begin{tikzpicture}[scale=0.7]
\node[anchor=north west] at ({-sqrt(3)}, 1){$\theta'$};
\node[anchor=north east] at ({sqrt(3)}, 1){$\theta$};
\node[anchor=south] at (0, -2){$\theta''$};
\draw[ultra thick] ({-sqrt(3)/2}, 1/2) to[out=0, in=180] (-0.1, -1/2);
\draw[white, line width=5] (0, 0) -- ({-sqrt(3)/2}, -1/2);
\draw[ultra thick, red] (0, 0) -- ({-sqrt(3)/2}, -1/2);
\draw[ultra thick, red] (0, 0) -- (0, 1);
\draw[ultra thick, red] (0, 0) -- ({sqrt(3)/2}, -1/2);
\draw[ultra thick, red, dotted] (0, 0) -- (0, -1);
\draw[white, line width=5] (-0.1, -1/2) to[out=0, in=180] ($({sqrt(3)/2}, 1/2)$);
\draw[ultra thick, ->] (-0.1, -1/2) to[out=0, in=180] ($({sqrt(3)/2}, 1/2)$);
\node[anchor=south] at ({-sqrt(3)/2}, 1/2){$1$};
\node[anchor=south] at ({sqrt(3)/2}, 1/2){$2$};
\filldraw[orange, opacity=0.2] ({sqrt(3)}, 1) -- ({-sqrt(3)}, 1) -- (0, -2) -- cycle;
\end{tikzpicture}
}}
\;. \label{eq:skeinrel5}
\end{gather}
\end{itemize}

Unlike the first four relations in $\Sk(\widetilde{M}^{\mathrm{sing}};\widetilde{\tau})$, which are drawn in $\widetilde{M}^{\mathrm{sing}}$, the relation \eqref{eq:skeinrel5} is drawn in the projection to $M$ -- the numbers $1$ and $2$ denote the sheet labels in accordance with the convention in Definition \ref{defn: sheet labels}. 
It is a local relation near the singular cone point at the barycenter of some tetrahedron in $\Delta$, here projected from a vertex of tetrahedron to the face opposing it. 
For the corresponding 3-dimensional picture in $M$, see Figure \ref{fig:3termrelation}, and for the relation in $\widetilde{M}^{\mathrm{sing}}$, projected to the torus $T^2$ which is the link of the cone point in $\widetilde{M}^{\mathrm{sing}}$, see Figure \ref{fig:3termrelation_torus}. 
\begin{figure}
\centering
$\vcenter{\hbox{
\tdplotsetmaincoords{60}{80}
\begin{tikzpicture}[tdplot_main_coords]
\begin{scope}[scale = 1.0, tdplot_main_coords]
    \coordinate (o) at (0, 0, 0);
    \coordinate (a) at (0, 0, 3);
    \coordinate (a1) at (0, 0, -1);
    
    \coordinate (b) at ({2*sqrt(2)}, 0, -1);
    \coordinate (b1) at ({-2*sqrt(2)/3}, 0, 1/3);
    
    \coordinate (c) at ({-sqrt(2)}, {sqrt(6)}, -1);
    \coordinate (c1) at ({sqrt(2)/3}, {-sqrt(6)/3}, 1/3);
    
    \coordinate (d) at ({-sqrt(2)}, {-sqrt(6)}, -1);
    \coordinate (d1) at ({sqrt(2)/3}, {sqrt(6)/3}, 1/3);
    
    \coordinate (ab) at ({sqrt(2)}, 0, 1);
    \coordinate (ac) at ({-sqrt(2)/2}, {sqrt(6)/2}, 1);
    \coordinate (ad) at ({-sqrt(2)/2}, {-sqrt(6)/2}, 1);
    \coordinate (bc) at ({sqrt(2)/2}, {sqrt(6)/2}, -1);
    \coordinate (bd) at ({sqrt(2)/2}, {-sqrt(6)/2}, -1);
    \coordinate (cd) at ({-sqrt(2)}, 0, -1);

    \draw[white, ultra thick] (o) -- (a1);

    \draw[very thick] ({-sqrt(2)/3}, {-0.1}, 1/6) .. controls ({-sqrt(2)/3}, 0, {1/6-0.1}) .. ({-sqrt(2)/3}, {0.1}, 1/6);
    \filldraw ({-sqrt(2)/3}, {-0.1}, 1/6) circle (0.05em);
    \filldraw ({-sqrt(2)/3}, {0.1}, 1/6) circle (0.05em);
    
    \filldraw [very thin, fill= orange, opacity = 0.2] (o) -- (a1) -- (cd) -- (b1) -- (o);
    \filldraw [very thin, fill= orange, opacity = 0.2] (o) -- (b1) -- (ac) -- (d1) -- (o);
    \filldraw [very thin, fill= orange, opacity = 0.2] (o) -- (b1) -- (ad) -- (c1) -- (o);
    \filldraw [very thin, fill= orange, opacity = 0.2] (o) -- (a1) -- (bc) -- (d1) -- (o);
    \filldraw [very thin, fill= orange, opacity = 0.2] (o) -- (a1) -- (bd) -- (c1) -- (o);
    
    \draw[red, very thick] (o) -- (a1);
    \draw[red, very thick] (o) -- (b1);
    \draw[red, very thick] (o) -- (c1);
    \draw[red, very thick] (o) -- (d1);

    \filldraw[red] (o) circle (0.15em);

    \filldraw [very thin, fill= orange, opacity = 0.2] (o) -- (c1) -- (ab) -- (d1) -- (o);

    \draw[very thick] ({-sqrt(2)/4}, {-sqrt(6)/4}, {1/2 + 0.5}) -- ({-sqrt(2)/4}, {-0.1}, {1/2 + 0.5}) -- ({-sqrt(2)/3}, {-0.1}, 1/6); 
    \draw[very thick, ->] ({-sqrt(2)/3}, {0.1}, 1/6) -- ({-sqrt(2)/4}, {0.1}, {1/2 + 0.5}) -- ({-sqrt(2)/4}, {sqrt(6)/4}, {1/2 + 0.5});

    \node[anchor=south] at ({-sqrt(2)/4}, {-sqrt(6)/4}, {1/2 + 0.5}) {$1$};
    \node[anchor=south] at ({-sqrt(2)/4}, {sqrt(6)/4}, {1/2 + 0.5}) {$2$};

    \node[anchor = south east] at (ad) {$\theta'$};
    \node[anchor = south west] at (ac) {$\theta$};
    \node[anchor = south west] at (ab) {$\theta''$};
\end{scope}
\end{tikzpicture}
}}
\; = \;
a^{\frac{\theta}{\pi}}
\vcenter{\hbox{
\tdplotsetmaincoords{60}{80}
\begin{tikzpicture}[tdplot_main_coords]
\begin{scope}[scale = 1.0, tdplot_main_coords]
    \coordinate (o) at (0, 0, 0);
    \coordinate (a) at (0, 0, 3);
    \coordinate (a1) at (0, 0, -1);
    
    \coordinate (b) at ({2*sqrt(2)}, 0, -1);
    \coordinate (b1) at ({-2*sqrt(2)/3}, 0, 1/3);
    
    \coordinate (c) at ({-sqrt(2)}, {sqrt(6)}, -1);
    \coordinate (c1) at ({sqrt(2)/3}, {-sqrt(6)/3}, 1/3);
    
    \coordinate (d) at ({-sqrt(2)}, {-sqrt(6)}, -1);
    \coordinate (d1) at ({sqrt(2)/3}, {sqrt(6)/3}, 1/3);
    
    \coordinate (ab) at ({sqrt(2)}, 0, 1);
    \coordinate (ac) at ({-sqrt(2)/2}, {sqrt(6)/2}, 1);
    \coordinate (ad) at ({-sqrt(2)/2}, {-sqrt(6)/2}, 1);
    \coordinate (bc) at ({sqrt(2)/2}, {sqrt(6)/2}, -1);
    \coordinate (bd) at ({sqrt(2)/2}, {-sqrt(6)/2}, -1);
    \coordinate (cd) at ({-sqrt(2)}, 0, -1);
    
    \draw[white, ultra thick] (o) -- (a1);

    \draw[very thick] ({sqrt(2)/6 -0.1*(sqrt(3)/2)}, {sqrt(6)/6 -0.1*(1/2)}, 1/6) .. controls ({sqrt(2)/6}, {sqrt(6)/6}, {1/6-0.1}) .. ({sqrt(2)/6 +0.1*(sqrt(3)/2)}, {sqrt(6)/6 +0.1*(1/2)}, 1/6);
    \filldraw ({sqrt(2)/6 -0.1*(sqrt(3)/2)}, {sqrt(6)/6 -0.1*(1/2)}, 1/6) circle (0.05em);
    \filldraw ({sqrt(2)/6 +0.1*(sqrt(3)/2)}, {sqrt(6)/6 +0.1*(1/2)}, 1/6) circle (0.05em);
    
    \filldraw [very thin, fill= orange, opacity = 0.2] (o) -- (a1) -- (cd) -- (b1) -- (o);
    \filldraw [very thin, fill= orange, opacity = 0.2] (o) -- (b1) -- (ac) -- (d1) -- (o);
    \filldraw [very thin, fill= orange, opacity = 0.2] (o) -- (b1) -- (ad) -- (c1) -- (o);
    \filldraw [very thin, fill= orange, opacity = 0.2] (o) -- (a1) -- (bc) -- (d1) -- (o);
    \filldraw [very thin, fill= orange, opacity = 0.2] (o) -- (a1) -- (bd) -- (c1) -- (o);
    
    \draw[red, very thick] (o) -- (a1);
    \draw[red, very thick] (o) -- (b1);
    \draw[red, very thick] (o) -- (c1);
    \draw[red, very thick] (o) -- (d1);

    \filldraw[red] (o) circle (0.15em);

    \filldraw [very thin, fill= orange, opacity = 0.2] (o) -- (c1) -- (ab) -- (d1) -- (o);

    \draw[very thick, ->] ({sqrt(2)/6 -0.1*(sqrt(3)/2)}, {sqrt(6)/6 -0.1*(1/2)}, 1/6) -- ({sqrt(2)/6 -0.1*(sqrt(3)/2)}, {sqrt(6)/6 -0.1*(1/2)}, {1/2 + 0.5}) -- ({-sqrt(2)/4}, {sqrt(6)/4}, {1/2 + 0.5});
     \draw[orange!20, line width=0.1cm] ({0.3*(-sqrt(2)/4) + 0.7*(sqrt(2)/6 +0.1*(sqrt(3)/2))}, {0.3*(-sqrt(6)/4)+0.7*(sqrt(6)/6 +0.1*(1/2))}, {1/2 + 0.5}) -- ({sqrt(2)/6 +0.1*(sqrt(3)/2)}, {sqrt(6)/6 +0.1*(1/2)}, {1/2 + 0.5});
    \draw[very thick] ({-sqrt(2)/4}, {-sqrt(6)/4}, {1/2 + 0.5}) -- ({sqrt(2)/6 +0.1*(sqrt(3)/2)}, {sqrt(6)/6 +0.1*(1/2)}, {1/2 + 0.5}) -- ({sqrt(2)/6 +0.1*(sqrt(3)/2)}, {sqrt(6)/6 +0.1*(1/2)}, 1/6); 
    
    \node[anchor=south] at ({-sqrt(2)/4}, {-sqrt(6)/4}, {1/2 + 0.5}) {$1$};
    \node[anchor=south] at ({-sqrt(2)/4}, {sqrt(6)/4}, {1/2 + 0.5}) {$2$};

    \node[anchor = south east] at (ad) {$\theta'$};
    \node[anchor = south west] at (ac) {$\theta$};
    \node[anchor = south] at ({sqrt(2)}, -0.1, 1) {$\theta''$};
\end{scope}
\end{tikzpicture}
}}
\;+\;
a^{-\frac{\theta'}{\pi}}
\vcenter{\hbox{
\tdplotsetmaincoords{60}{80}
\begin{tikzpicture}[tdplot_main_coords]
\begin{scope}[scale = 1.0, tdplot_main_coords]
    \coordinate (o) at (0, 0, 0);
    \coordinate (a) at (0, 0, 3);
    \coordinate (a1) at (0, 0, -1);
    
    \coordinate (b) at ({2*sqrt(2)}, 0, -1);
    \coordinate (b1) at ({-2*sqrt(2)/3}, 0, 1/3);
    
    \coordinate (c) at ({-sqrt(2)}, {sqrt(6)}, -1);
    \coordinate (c1) at ({sqrt(2)/3}, {-sqrt(6)/3}, 1/3);
    
    \coordinate (d) at ({-sqrt(2)}, {-sqrt(6)}, -1);
    \coordinate (d1) at ({sqrt(2)/3}, {sqrt(6)/3}, 1/3);
    
    \coordinate (ab) at ({sqrt(2)}, 0, 1);
    \coordinate (ac) at ({-sqrt(2)/2}, {sqrt(6)/2}, 1);
    \coordinate (ad) at ({-sqrt(2)/2}, {-sqrt(6)/2}, 1);
    \coordinate (bc) at ({sqrt(2)/2}, {sqrt(6)/2}, -1);
    \coordinate (bd) at ({sqrt(2)/2}, {-sqrt(6)/2}, -1);
    \coordinate (cd) at ({-sqrt(2)}, 0, -1);

    \draw[white, ultra thick] (o) -- (a1);

    \draw[very thick] ({sqrt(2)/6 -0.1*(sqrt(3)/2)}, {-sqrt(6)/6 -0.1*(1/2)}, 1/6) .. controls ({sqrt(2)/6}, {-sqrt(6)/6}, {1/6-0.1}) .. ({sqrt(2)/6 +0.1*(sqrt(3)/2)}, {-sqrt(6)/6 +0.1*(1/2)}, 1/6);
    \filldraw ({sqrt(2)/6 -0.1*(sqrt(3)/2)}, {-sqrt(6)/6 -0.1*(1/2)}, 1/6) circle (0.05em);
    \filldraw ({sqrt(2)/6 +0.1*(sqrt(3)/2)}, {-sqrt(6)/6 +0.1*(1/2)}, 1/6) circle (0.05em);
    
    \filldraw [very thin, fill= orange, opacity = 0.2] (o) -- (a1) -- (cd) -- (b1) -- (o);
    \filldraw [very thin, fill= orange, opacity = 0.2] (o) -- (b1) -- (ac) -- (d1) -- (o);
    \filldraw [very thin, fill= orange, opacity = 0.2] (o) -- (b1) -- (ad) -- (c1) -- (o);
    \filldraw [very thin, fill= orange, opacity = 0.2] (o) -- (a1) -- (bc) -- (d1) -- (o);
    \filldraw [very thin, fill= orange, opacity = 0.2] (o) -- (a1) -- (bd) -- (c1) -- (o);
    
    \draw[red, very thick] (o) -- (a1);
    \draw[red, very thick] (o) -- (b1);
    \draw[red, very thick] (o) -- (c1);
    \draw[red, very thick] (o) -- (d1);

    \filldraw[red] (o) circle (0.15em);

    \filldraw [very thin, fill= orange, opacity = 0.2] (o) -- (c1) -- (ab) -- (d1) -- (o);

    \draw[very thick] ({-sqrt(2)/4}, {-sqrt(6)/4}, {1/2 + 0.5}) -- ({sqrt(2)/6 -0.1*(sqrt(3)/2)}, {-sqrt(6)/6 -0.1*(1/2)}, {1/2 + 0.5}) -- ({sqrt(2)/6 -0.1*(sqrt(3)/2)}, {-sqrt(6)/6 -0.1*(1/2)}, 1/6); 
    \draw[very thick, ->] ({sqrt(2)/6 +0.1*(sqrt(3)/2)}, {-sqrt(6)/6 +0.1*(1/2)}, 1/6) -- ({sqrt(2)/6 +0.1*(sqrt(3)/2)}, {-sqrt(6)/6 +0.1*(1/2)}, {1/2 + 0.5}) -- ({-sqrt(2)/4}, {sqrt(6)/4}, {1/2 + 0.5});
    
    \node[anchor=south] at ({-sqrt(2)/4}, {-sqrt(6)/4}, {1/2 + 0.5}) {$1$};
    \node[anchor=south] at ({-sqrt(2)/4}, {sqrt(6)/4}, {1/2 + 0.5}) {$2$};

    \node[anchor = south east] at (ad) {$\theta'$};
    \node[anchor = south west] at (ac) {$\theta$};
    \node[anchor = south west] at (ab) {$\theta''$};
\end{scope}
\end{tikzpicture}
}}
$
\caption{The 3-term relation \eqref{eq:skeinrel5} near the cone point}
\label{fig:3termrelation}
\end{figure}
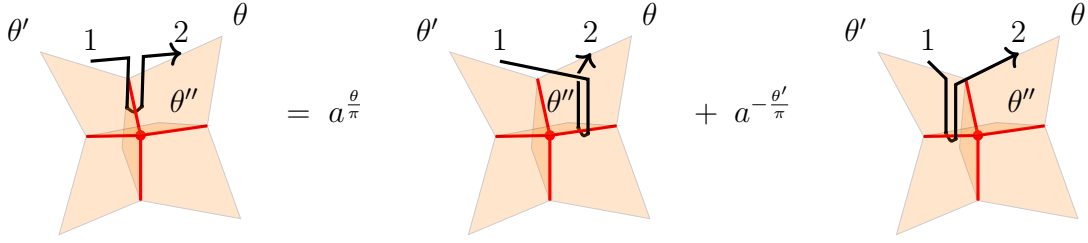
\begin{figure}
\centering
$
\vcenter{\hbox{
\begin{tikzpicture}[scale=0.8]
\begin{scope}
    \filldraw[red] (0, 0) circle (0.05);
    \draw[lightgray] (0, 1) -- ({sqrt(3)/2}, 1/2) -- ({sqrt(3)/2}, -1/2) -- (0, -1) -- ({-sqrt(3)/2}, -1/2) -- ({-sqrt(3)/2}, 1/2) -- cycle;
    \draw[orange, dotted] ({-sqrt(3)/4}, -3/4) -- ({sqrt(3)/4}, 3/4);
    \draw[orange, dotted] ({-sqrt(3)/2}, 0) -- ({sqrt(3)/2}, 0);
    \draw[orange, dotted] ({-sqrt(3)/4}, 3/4) -- ({sqrt(3)/4}, -3/4);
    \node[gray] at ({sqrt(3)/2}, 1/2){\scriptsize $1$};
    \node[gray] at (0, 1){\scriptsize $2$};
    \node[gray] at ({-sqrt(3)/2}, 1/2){\scriptsize $1$};
    \node[gray] at ({-sqrt(3)/2}, -1/2){\scriptsize $2$};
    \node[gray] at (0, -1){\scriptsize $1$};
    \node[gray] at ({sqrt(3)/2}, -1/2){\scriptsize $2$};
    \node[gray] at ({sqrt(3)/2}, 0){$\theta''$};
    \node[gray] at ({sqrt(3)/4}, 3/4){$\theta'$};
\end{scope}
\begin{scope}[shift={({sqrt(3)}, 0)}]
    \filldraw[red] (0, 0) circle (0.05);
    \draw[lightgray] (0, 1) -- ({sqrt(3)/2}, 1/2) -- ({sqrt(3)/2}, -1/2) -- (0, -1) -- ({-sqrt(3)/2}, -1/2) -- ({-sqrt(3)/2}, 1/2) -- cycle;
    \draw[orange, dotted] ({-sqrt(3)/4}, -3/4) -- ({sqrt(3)/4}, 3/4);
    \draw[orange, dotted] ({-sqrt(3)/2}, 0) -- ({sqrt(3)/2}, 0);
    \draw[orange, dotted] ({-sqrt(3)/4}, 3/4) -- ({sqrt(3)/4}, -3/4);
    \node[gray] at ({sqrt(3)/2}, 1/2){\scriptsize $1$};
    \node[gray] at (0, 1){\scriptsize $2$};
    \node[gray] at (0, -1){\scriptsize $1$};
    \node[gray] at ({sqrt(3)/2}, -1/2){\scriptsize $2$};
    \node[gray] at ({-sqrt(3)/4}, 3/4){$\theta$};
\end{scope}
\begin{scope}[shift={({sqrt(3)/2}, 3/2)}]
    \filldraw[red] (0, 0) circle (0.05);
    \draw[lightgray] (0, 1) -- ({sqrt(3)/2}, 1/2) -- ({sqrt(3)/2}, -1/2) -- (0, -1) -- ({-sqrt(3)/2}, -1/2) -- ({-sqrt(3)/2}, 1/2) -- cycle;
    \draw[orange, dotted] ({-sqrt(3)/4}, -3/4) -- ({sqrt(3)/4}, 3/4);
    \draw[orange, dotted] ({-sqrt(3)/2}, 0) -- ({sqrt(3)/2}, 0);
    \draw[orange, dotted] ({-sqrt(3)/4}, 3/4) -- ({sqrt(3)/4}, -3/4);
    \node[gray] at ({sqrt(3)/2}, 1/2){\scriptsize $1$};
    \node[gray] at (0, 1){\scriptsize $2$};
    \node[gray] at ({-sqrt(3)/2}, 1/2){\scriptsize $1$};
\end{scope}
\begin{scope}[shift={({3*sqrt(3)/2}, 3/2)}]
    \filldraw[red] (0, 0) circle (0.05);
    \draw[lightgray] (0, 1) -- ({sqrt(3)/2}, 1/2) -- ({sqrt(3)/2}, -1/2) -- (0, -1) -- ({-sqrt(3)/2}, -1/2) -- ({-sqrt(3)/2}, 1/2) -- cycle;
    \draw[orange, dotted] ({-sqrt(3)/4}, -3/4) -- ({sqrt(3)/4}, 3/4);
    \draw[orange, dotted] ({-sqrt(3)/2}, 0) -- ({sqrt(3)/2}, 0);
    \draw[orange, dotted] ({-sqrt(3)/4}, 3/4) -- ({sqrt(3)/4}, -3/4);
    \node[gray] at ({sqrt(3)/2}, 1/2){\scriptsize $1$};
    \node[gray] at (0, 1){\scriptsize $2$};
    \node[gray] at ({sqrt(3)/2}, -1/2){\scriptsize $2$};
\end{scope}
\draw[->, thick] ($({sqrt(3)/2}, 1)$) to[out=90, in=-90] ($({sqrt(3)/2 - 0.2}, 3/2)$);
\draw[thick] ($({sqrt(3)/2 - 0.2}, 3/2)$) to[out=90, in=-90] ($({sqrt(3)/2}, 2)$);
\node at ({sqrt(3)/2}, 1){\scriptsize $\otimes$};
\node at ({sqrt(3)/2}, 2){\scriptsize $\odot$}; 
\end{tikzpicture}
}}
\;=\;
a^{\frac{\theta}{\pi}}
\vcenter{\hbox{
\begin{tikzpicture}[scale=0.8]
\begin{scope}
    \filldraw[red] (0, 0) circle (0.05);
    \draw[lightgray] (0, 1) -- ({sqrt(3)/2}, 1/2) -- ({sqrt(3)/2}, -1/2) -- (0, -1) -- ({-sqrt(3)/2}, -1/2) -- ({-sqrt(3)/2}, 1/2) -- cycle;
    \draw[orange, dotted] ({-sqrt(3)/4}, -3/4) -- ({sqrt(3)/4}, 3/4);
    \draw[orange, dotted] ({-sqrt(3)/2}, 0) -- ({sqrt(3)/2}, 0);
    \draw[orange, dotted] ({-sqrt(3)/4}, 3/4) -- ({sqrt(3)/4}, -3/4);
    \node[gray] at ({sqrt(3)/2}, 1/2){\scriptsize $1$};
    \node[gray] at (0, 1){\scriptsize $2$};
    \node[gray] at ({-sqrt(3)/2}, 1/2){\scriptsize $1$};
    \node[gray] at ({-sqrt(3)/2}, -1/2){\scriptsize $2$};
    \node[gray] at (0, -1){\scriptsize $1$};
    \node[gray] at ({sqrt(3)/2}, -1/2){\scriptsize $2$};
    \node[gray] at ({sqrt(3)/2}, 0){$\theta''$};
    \node[gray] at ({sqrt(3)/4}, 3/4){$\theta'$};
\end{scope}
\begin{scope}[shift={({sqrt(3)}, 0)}]
    \filldraw[red] (0, 0) circle (0.05);
    \draw[lightgray] (0, 1) -- ({sqrt(3)/2}, 1/2) -- ({sqrt(3)/2}, -1/2) -- (0, -1) -- ({-sqrt(3)/2}, -1/2) -- ({-sqrt(3)/2}, 1/2) -- cycle;
    \draw[orange, dotted] ({-sqrt(3)/4}, -3/4) -- ({sqrt(3)/4}, 3/4);
    \draw[orange, dotted] ({-sqrt(3)/2}, 0) -- ({sqrt(3)/2}, 0);
    \draw[orange, dotted] ({-sqrt(3)/4}, 3/4) -- ({sqrt(3)/4}, -3/4);
    \node[gray] at ({sqrt(3)/2}, 1/2){\scriptsize $1$};
    \node[gray] at (0, 1){\scriptsize $2$};
    \node[gray] at (0, -1){\scriptsize $1$};
    \node[gray] at ({sqrt(3)/2}, -1/2){\scriptsize $2$};
    \node[gray] at ({-sqrt(3)/4}, 3/4){$\theta$};
\end{scope}
\begin{scope}[shift={({sqrt(3)/2}, 3/2)}]
    \filldraw[red] (0, 0) circle (0.05);
    \draw[lightgray] (0, 1) -- ({sqrt(3)/2}, 1/2) -- ({sqrt(3)/2}, -1/2) -- (0, -1) -- ({-sqrt(3)/2}, -1/2) -- ({-sqrt(3)/2}, 1/2) -- cycle;
    \draw[orange, dotted] ({-sqrt(3)/4}, -3/4) -- ({sqrt(3)/4}, 3/4);
    \draw[orange, dotted] ({-sqrt(3)/2}, 0) -- ({sqrt(3)/2}, 0);
    \draw[orange, dotted] ({-sqrt(3)/4}, 3/4) -- ({sqrt(3)/4}, -3/4);
    \node[gray] at ({sqrt(3)/2}, 1/2){\scriptsize $1$};
    \node[gray] at (0, 1){\scriptsize $2$};
    \node[gray] at ({-sqrt(3)/2}, 1/2){\scriptsize $1$};
\end{scope}
\begin{scope}[shift={({3*sqrt(3)/2}, 3/2)}]
    \filldraw[red] (0, 0) circle (0.05);
    \draw[lightgray] (0, 1) -- ({sqrt(3)/2}, 1/2) -- ({sqrt(3)/2}, -1/2) -- (0, -1) -- ({-sqrt(3)/2}, -1/2) -- ({-sqrt(3)/2}, 1/2) -- cycle;
    \draw[orange, dotted] ({-sqrt(3)/4}, -3/4) -- ({sqrt(3)/4}, 3/4);
    \draw[orange, dotted] ({-sqrt(3)/2}, 0) -- ({sqrt(3)/2}, 0);
    \draw[orange, dotted] ({-sqrt(3)/4}, 3/4) -- ({sqrt(3)/4}, -3/4);
    \node[gray] at ({sqrt(3)/2}, 1/2){\scriptsize $1$};
    \node[gray] at (0, 1){\scriptsize $2$};
    \node[gray] at ({sqrt(3)/2}, -1/2){\scriptsize $2$};
\end{scope}
\node at ({sqrt(3)/2}, 1){\scriptsize $\otimes$};
\node at ({sqrt(3)/2}, 2){\scriptsize $\odot$}; 
\draw[->, thick] ($({sqrt(3)/2}, 1)$) to[out=-90, in=120] ($({7/8*sqrt(3)}, -1/8)$);
\draw[thick] ($({7/8*sqrt(3)}, -1/8)$) to[out=-60, in=120] ($({11/8*sqrt(3)}, -5/8)$); 
\draw[thick] ($({3/8*sqrt(3)}, 19/8)$) to[out=-60, in=90] ($({sqrt(3)/2}, 2)$); 
\end{tikzpicture}
}}
\;+\;
a^{-\frac{\theta'}{\pi}}
\vcenter{\hbox{
\begin{tikzpicture}[scale=0.8]
\begin{scope}
    \filldraw[red] (0, 0) circle (0.05);
    \draw[lightgray] (0, 1) -- ({sqrt(3)/2}, 1/2) -- ({sqrt(3)/2}, -1/2) -- (0, -1) -- ({-sqrt(3)/2}, -1/2) -- ({-sqrt(3)/2}, 1/2) -- cycle;
    \draw[orange, dotted] ({-sqrt(3)/4}, -3/4) -- ({sqrt(3)/4}, 3/4);
    \draw[orange, dotted] ({-sqrt(3)/2}, 0) -- ({sqrt(3)/2}, 0);
    \draw[orange, dotted] ({-sqrt(3)/4}, 3/4) -- ({sqrt(3)/4}, -3/4);
    \node[gray] at ({sqrt(3)/2}, 1/2){\scriptsize $1$};
    \node[gray] at (0, 1){\scriptsize $2$};
    \node[gray] at ({-sqrt(3)/2}, 1/2){\scriptsize $1$};
    \node[gray] at ({-sqrt(3)/2}, -1/2){\scriptsize $2$};
    \node[gray] at (0, -1){\scriptsize $1$};
    \node[gray] at ({sqrt(3)/2}, -1/2){\scriptsize $2$};
    \node[gray] at ({sqrt(3)/2}, 0){$\theta''$};
    \node[gray] at ({sqrt(3)/4}, 3/4){$\theta'$};
\end{scope}
\begin{scope}[shift={({sqrt(3)}, 0)}]
    \filldraw[red] (0, 0) circle (0.05);
    \draw[lightgray] (0, 1) -- ({sqrt(3)/2}, 1/2) -- ({sqrt(3)/2}, -1/2) -- (0, -1) -- ({-sqrt(3)/2}, -1/2) -- ({-sqrt(3)/2}, 1/2) -- cycle;
    \draw[orange, dotted] ({-sqrt(3)/4}, -3/4) -- ({sqrt(3)/4}, 3/4);
    \draw[orange, dotted] ({-sqrt(3)/2}, 0) -- ({sqrt(3)/2}, 0);
    \draw[orange, dotted] ({-sqrt(3)/4}, 3/4) -- ({sqrt(3)/4}, -3/4);
    \node[gray] at ({sqrt(3)/2}, 1/2){\scriptsize $1$};
    \node[gray] at (0, 1){\scriptsize $2$};
    \node[gray] at (0, -1){\scriptsize $1$};
    \node[gray] at ({sqrt(3)/2}, -1/2){\scriptsize $2$};
    \node[gray] at ({-sqrt(3)/4}, 3/4){$\theta$};
\end{scope}
\begin{scope}[shift={({sqrt(3)/2}, 3/2)}]
    \filldraw[red] (0, 0) circle (0.05);
    \draw[lightgray] (0, 1) -- ({sqrt(3)/2}, 1/2) -- ({sqrt(3)/2}, -1/2) -- (0, -1) -- ({-sqrt(3)/2}, -1/2) -- ({-sqrt(3)/2}, 1/2) -- cycle;
    \draw[orange, dotted] ({-sqrt(3)/4}, -3/4) -- ({sqrt(3)/4}, 3/4);
    \draw[orange, dotted] ({-sqrt(3)/2}, 0) -- ({sqrt(3)/2}, 0);
    \draw[orange, dotted] ({-sqrt(3)/4}, 3/4) -- ({sqrt(3)/4}, -3/4);
    \node[gray] at ({sqrt(3)/2}, 1/2){\scriptsize $1$};
    \node[gray] at (0, 1){\scriptsize $2$};
    \node[gray] at ({-sqrt(3)/2}, 1/2){\scriptsize $1$};
\end{scope}
\begin{scope}[shift={({3*sqrt(3)/2}, 3/2)}]
    \filldraw[red] (0, 0) circle (0.05);
    \draw[lightgray] (0, 1) -- ({sqrt(3)/2}, 1/2) -- ({sqrt(3)/2}, -1/2) -- (0, -1) -- ({-sqrt(3)/2}, -1/2) -- ({-sqrt(3)/2}, 1/2) -- cycle;
    \draw[orange, dotted] ({-sqrt(3)/4}, -3/4) -- ({sqrt(3)/4}, 3/4);
    \draw[orange, dotted] ({-sqrt(3)/2}, 0) -- ({sqrt(3)/2}, 0);
    \draw[orange, dotted] ({-sqrt(3)/4}, 3/4) -- ({sqrt(3)/4}, -3/4);
    \node[gray] at ({sqrt(3)/2}, 1/2){\scriptsize $1$};
    \node[gray] at (0, 1){\scriptsize $2$};
    \node[gray] at ({sqrt(3)/2}, -1/2){\scriptsize $2$};
\end{scope}
\node at ({sqrt(3)/2}, 1){\scriptsize $\otimes$};
\node at ({sqrt(3)/2}, 2){\scriptsize $\odot$}; 
\draw[thick, ->] ($({sqrt(3)/2}, 1)$) to[out=-90, in=60] ($({-sqrt(3)/8}, 1/8)$);
\draw[thick] ({-sqrt(3)/8}, 1/8) -- ({-3/8*sqrt(3)}, -5/8);
\draw[thick] ($({5/8*sqrt(3)}, 19/8)$) to[out=-120, in=90] ($({sqrt(3)/2}, 2)$);
\end{tikzpicture}
}}
$
\caption{The 3-term relation \eqref{eq:skeinrel5} drawn on the projection to the torus}
\label{fig:3termrelation_torus}
\end{figure}
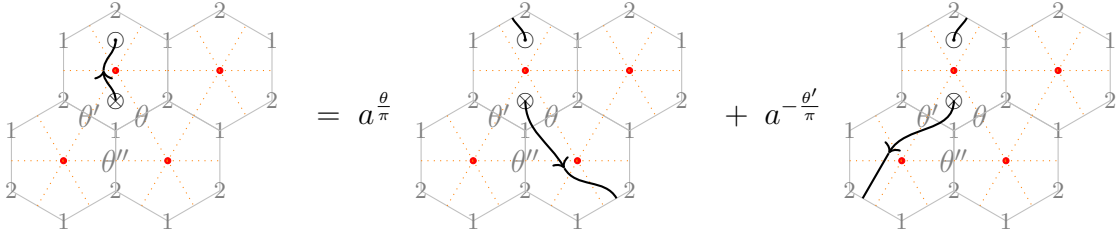
\end{defn}

\subsection{Specializing the singular skein}

Following \cite{Lackenby-taut}, we say that a generalized angle structure $\Theta: \mathbf{a}(\Delta) \to \mathbb{R}$ is {\em taut} if it maps each angle symbol to $0$ or $\pi$.\footnote{To be precise, our notion of ``taut'' is weaker than that of Lackenby, in that Lackenby's notion also requires a choice of transverse orientation. 
If we are also given a transverse orientation, then half-integer powers of $a$ would not appear in our skein trace map. See also Remark \ref{rmk:transverse-orientation}.} 

\begin{figure}
\centering
$\vcenter{\hbox{
\tdplotsetmaincoords{60}{80}
\begin{tikzpicture}[tdplot_main_coords]
\begin{scope}[scale = 0.8, tdplot_main_coords]
    \coordinate (o) at (0, 0, 0);
    \coordinate (a) at (0, 0, 3);
    \coordinate (a1) at (0, 0, -1);
    
    \coordinate (b) at ({2*sqrt(2)}, 0, -1);
    \coordinate (b1) at ({-2*sqrt(2)/3}, 0, 1/3);
    
    \coordinate (c) at ({-sqrt(2)}, {sqrt(6)}, -1);
    \coordinate (c1) at ({sqrt(2)/3}, {-sqrt(6)/3}, 1/3);
    
    \coordinate (d) at ({-sqrt(2)}, {-sqrt(6)}, -1);
    \coordinate (d1) at ({sqrt(2)/3}, {sqrt(6)/3}, 1/3);
    
    \coordinate (ab) at ({sqrt(2)}, 0, 1);
    \coordinate (ac) at ({-sqrt(2)/2}, {sqrt(6)/2}, 1);
    \coordinate (ad) at ({-sqrt(2)/2}, {-sqrt(6)/2}, 1);
    \coordinate (bc) at ({sqrt(2)/2}, {sqrt(6)/2}, -1);
    \coordinate (bd) at ({sqrt(2)/2}, {-sqrt(6)/2}, -1);
    \coordinate (cd) at ({-sqrt(2)}, 0, -1);
    
    \draw[very thick] (a) -- (b);
    \draw[very thick] (a) -- (c);
    \draw[very thick] (a) -- (d);
    \draw[very thick] (b) -- (c);
    \draw[very thick] (b) -- (d);
    \draw[very thick, dashed] (c) -- (d);

    \filldraw (a) circle (0.05em);
    \filldraw (b) circle (0.05em);
    \filldraw (c) circle (0.05em);
    \filldraw (d) circle (0.05em);

    \node[anchor=east,  yshift=1.5em, xshift = -0.3em] at (ab){$\pi$};
    \node[anchor=north, xshift = -1em, yshift = -0.3em] at (cd){$\pi$};
    \node[anchor=south west] at (ac){$0$};
    \node[anchor=south east] at (ad){$0$};
    \node[anchor=north west] at (bc){$0$};
    \node[anchor=north east] at (bd){$0$};
\end{scope}
\end{tikzpicture}
}}$
\caption{A taut ideal tetrahedron}
\label{fig:taut_tetrahedron}
\end{figure}
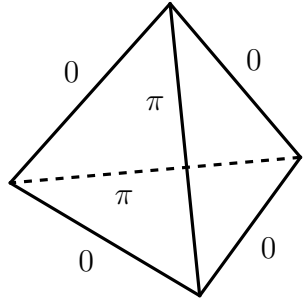

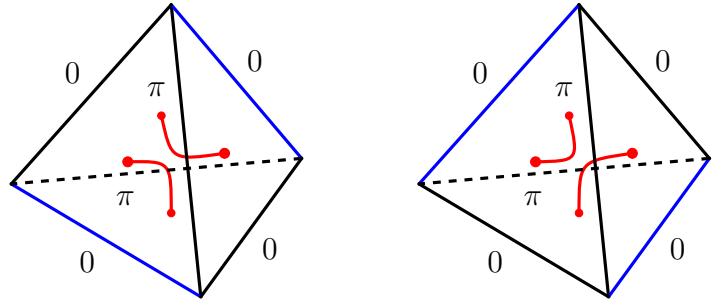
\begin{figure}
\centering
$\vcenter{\hbox{
\tdplotsetmaincoords{60}{80}
\begin{tikzpicture}[tdplot_main_coords]
\begin{scope}[scale = 0.8, tdplot_main_coords]
    \coordinate (o) at (0, 0, 0);
    \coordinate (a) at (0, 0, 3);
    \coordinate (a1) at (0, 0, -1);
    
    \coordinate (b) at ({2*sqrt(2)}, 0, -1);
    \coordinate (b1) at ({-2*sqrt(2)/3}, 0, 1/3);
    
    \coordinate (c) at ({-sqrt(2)}, {sqrt(6)}, -1);
    \coordinate (c1) at ({sqrt(2)/3}, {-sqrt(6)/3}, 1/3);
    
    \coordinate (d) at ({-sqrt(2)}, {-sqrt(6)}, -1);
    \coordinate (d1) at ({sqrt(2)/3}, {sqrt(6)/3}, 1/3);
    
    \coordinate (ab) at ({sqrt(2)}, 0, 1);
    \coordinate (ac) at ({-sqrt(2)/2}, {sqrt(6)/2}, 1);
    \coordinate (ad) at ({-sqrt(2)/2}, {-sqrt(6)/2}, 1);
    \coordinate (bc) at ({sqrt(2)/2}, {sqrt(6)/2}, -1);
    \coordinate (bd) at ({sqrt(2)/2}, {-sqrt(6)/2}, -1);
    \coordinate (cd) at ({-sqrt(2)}, 0, -1);
    
    \draw[red, very thick] (a1) .. controls (o) .. (c1);
    \draw[red, very thick] (b1) .. controls (o) .. (d1);

    \filldraw[red] (a1) circle (0.15em);
    \filldraw[red] (b1) circle (0.15em);
    \filldraw[red] (c1) circle (0.2em);
    \filldraw[red] (d1) circle (0.2em);

    \draw[very thick] (a) -- (b);
    \draw[very thick, blue] (a) -- (c);
    \draw[very thick] (a) -- (d);
    \draw[very thick] (b) -- (c);
    \draw[very thick, blue] (b) -- (d);
    \draw[very thick, dashed] (c) -- (d);

    \filldraw (a) circle (0.05em);
    \filldraw (b) circle (0.05em);
    \filldraw (c) circle (0.05em);
    \filldraw (d) circle (0.05em);

    \node[anchor=east,  yshift=2em, xshift = -0.3em] at (ab){$\pi$};
    \node[anchor=north, xshift = -1em, yshift = -0.3em] at (cd){$\pi$};
    \node[anchor=south west] at (ac){$0$};
    \node[anchor=south east] at (ad){$0$};
    \node[anchor=north west] at (bc){$0$};
    \node[anchor=north east] at (bd){$0$};
\end{scope}
\end{tikzpicture}
}}
\quad\quad
\vcenter{\hbox{
\tdplotsetmaincoords{60}{80}
\begin{tikzpicture}[tdplot_main_coords]
\begin{scope}[scale = 0.8, tdplot_main_coords]
    \coordinate (o) at (0, 0, 0);
    \coordinate (a) at (0, 0, 3);
    \coordinate (a1) at (0, 0, -1);
    
    \coordinate (b) at ({2*sqrt(2)}, 0, -1);
    \coordinate (b1) at ({-2*sqrt(2)/3}, 0, 1/3);
    
    \coordinate (c) at ({-sqrt(2)}, {sqrt(6)}, -1);
    \coordinate (c1) at ({sqrt(2)/3}, {-sqrt(6)/3}, 1/3);
    
    \coordinate (d) at ({-sqrt(2)}, {-sqrt(6)}, -1);
    \coordinate (d1) at ({sqrt(2)/3}, {sqrt(6)/3}, 1/3);
    
    \coordinate (ab) at ({sqrt(2)}, 0, 1);
    \coordinate (ac) at ({-sqrt(2)/2}, {sqrt(6)/2}, 1);
    \coordinate (ad) at ({-sqrt(2)/2}, {-sqrt(6)/2}, 1);
    \coordinate (bc) at ({sqrt(2)/2}, {sqrt(6)/2}, -1);
    \coordinate (bd) at ({sqrt(2)/2}, {-sqrt(6)/2}, -1);
    \coordinate (cd) at ({-sqrt(2)}, 0, -1);

    
    \draw[red, very thick] (a1) .. controls (o) .. (d1);
    \draw[red, very thick] (b1) .. controls (o) .. (c1);

    \filldraw[red] (a1) circle (0.15em);
    \filldraw[red] (b1) circle (0.15em);
    \filldraw[red] (c1) circle (0.2em);
    \filldraw[red] (d1) circle (0.2em);

    \draw[very thick] (a) -- (b);
    \draw[very thick] (a) -- (c);
    \draw[very thick, blue] (a) -- (d);
    \draw[very thick, blue] (b) -- (c);
    \draw[very thick] (b) -- (d);
    \draw[very thick, dashed] (c) -- (d);

    \filldraw (a) circle (0.05em);
    \filldraw (b) circle (0.05em);
    \filldraw (c) circle (0.05em);
    \filldraw (d) circle (0.05em);

    \node[anchor=east,  yshift=2em, xshift = -0.3em] at (ab){$\pi$};
    \node[anchor=north, xshift = -1em, yshift = -0.3em] at (cd){$\pi$};
    \node[anchor=south west] at (ac){$0$};
    \node[anchor=south east] at (ad){$0$};
    \node[anchor=north west] at (bc){$0$};
    \node[anchor=north east] at (bd){$0$};
\end{scope}
\end{tikzpicture}
}}
$
\caption{Positive (left) and negative (right) compatible tangles drawn in red.  The edges they link are drawn in blue.  The sign is the sign of the angle at a vertex from a $\pi$ edge to the linked (blue) edge, viewed from outside the tetrahedron.}
\label{fig:tetrahedron_resolutions}
\end{figure}

Given a taut structure, we produce a marking (and hence a tangle in $M$, and associated smooth branched double cover) specifying a sign $\{+, -\}$ on each tetrahedron, and following the local rules of Figure \ref{fig:tetrahedron_resolutions}.  We will say a taut triangulation equipped with such a choice of signs is a \emph{signed taut triangulation}.  In this case, as usual, we denote the smooth branched cover 
$\widetilde{M} \rightarrow M$ and its branch locus $\widetilde{\sigma}\subset\widetilde{M}$ over $\sigma\subset M$. 

We write $\Sk(\widetilde{M}^{\mathrm{sing}})_\Theta$ for specialization of the singular skein by using $\Theta$ to evaluate the angle variables. 

\begin{theorem} \label{smoothing skein module}
Fix a taut structure $\Theta$, and a signing so that the corresponding marking is effective.  Then
there is a unique morphism $\Sk(\widetilde{M}^{\mathrm{sing}}; \widetilde{\tau})_\Theta \to \widehat{\Sk}^{\zeta}(\widetilde{M}; \widetilde{\sigma})$ commuting with the map of skeins induced by the inclusion $\widetilde{M}^{\mathrm{sing}} \setminus \widetilde \tau \hookrightarrow \widetilde{M} \setminus \widetilde{\sigma}$. 
\end{theorem}
\begin{proof}
Such a morphism is unique if it exists, since any skein in $\Sk(\widetilde{M}^{\mathrm{sing}}; \widetilde{\tau})_\Theta$ can be represented as a linear combination of links in $\widetilde{M}^{\mathrm{sing}} \setminus \widetilde \tau$. 

Consider $\widetilde{M}^\circ$, the complement in $\widetilde{M}^{\mathrm{sing}}$ of small neighborhoods of the singular points. 
This is a manifold with one boundary torus near each singular point. 
Then $(\widetilde{M}, \widetilde{\sigma})$ is obtained from $(\widetilde{M}^\circ, \widetilde{\tau})$ by gluing in solid tori with appropriate sign lines. 
This is illustrated in Figure \ref{fig:solid_torus}. 

\begin{figure}
\centering
\[
\vcenter{\hbox{
\begin{tikzpicture}
\begin{scope}[scale=0.4]
\draw[very thick, dotted] (0, 0) ellipse (1 and 0.5);
\draw[very thick] (3, 0) arc (0:-180:3 and 1.5);
\draw[very thick, dotted] (3, 0) arc (0:180:3 and 1.5);
\draw[very thick] (3, 0) -- (3, 4);
\draw[very thick] (-3, 0) -- (-3, 4);
\draw[very thick, dotted] (1, 0) -- (1, 4);
\draw[very thick, dotted] (-1, 0) -- (-1, 4);
\filldraw[red] (2, 0) circle (0.1);
\filldraw[red] (2, 4) circle (0.1);
\draw[red, very thick] (2, 0) -- (2, 4);
\filldraw[red] (-2, 0) circle (0.1);
\filldraw[red] (-2, 4) circle (0.1);
\draw[red, very thick] (-2, 0) -- (-2, 4);
\draw[very thick] (0, 4) ellipse (1 and 0.5);
\draw[very thick] (0, 4) ellipse (3 and 1.5);
\end{scope}
\end{tikzpicture}
}}
\quad\overset{2:1}{\rightarrow}\;\;
\vcenter{\hbox{
\tdplotsetmaincoords{65}{52}
\begin{tikzpicture}[tdplot_main_coords]
\begin{scope}[scale = 0.5, tdplot_main_coords]
    \newcommand*{\defcoords}{
        \coordinate (o) at (0, 0, 0);
        \coordinate (a) at (3, 0, 0);
        \coordinate (b) at ({3*cos(90)}, {3*sin(90)}, 0);
        \coordinate (c) at ({3*cos(2*90)}, {3*sin(2*90)}, 0);
        \coordinate (d) at ({3*cos(3*90)}, {3*sin(3*90)}, 0);
        \coordinate (abc) at ($1/3*(a)+1/3*(b)+1/3*(c)$);
        \coordinate (abd) at ($1/3*(a)+1/3*(b)+1/3*(d)$);
        \coordinate (bcd) at ($1/3*(b)+1/3*(c)+1/3*(d)$);
        \coordinate (acd) at ($1/3*(a)+1/3*(c)+1/3*(d)$);
    }
    \defcoords
    \draw[very thick] (a) -- (b);
    \draw[very thick, dotted] (b) -- (c);
    \draw[very thick, dotted] (c) -- (d);
    \draw[very thick] (d) -- (a);
    \draw[very thick] (a) -- (3, 0, 3);
    \draw[very thick] (b) -- ({3*cos(90)}, {3*sin(90)}, 3);
    \draw[very thick, dotted] (c) -- ({3*cos(2*90)}, {3*sin(2*90)}, 3);
    \draw[very thick] (d) -- ({3*cos(3*90)}, {3*sin(3*90)}, 3);
    \filldraw (a) circle (0.05em);
    \filldraw (b) circle (0.05em);
    \filldraw (c) circle (0.05em);
    \filldraw (d) circle (0.05em);
    \draw[white, line width=5] (0.5,0.5,0) -- ($(0.5,0.5,0)+(0,0,3)$);
    \draw[ultra thick, red] (0.5,0.5,0) -- ($(0.5,0.5,0)+(0,0,3)$);
    \filldraw[red] (0.5,0.5,0) circle (0.2em);
    \draw[white, line width=5] (-0.5,-0.5,0) -- ($(-0.5,-0.5,0)+(0,0,3)$);
    \draw[ultra thick, red] (-0.5,-0.5,0) -- ($(-0.5,-0.5,0)+(0,0,3)$);
    \filldraw[red] (-0.5,-0.5,0) circle (0.2em);
    \begin{scope}[shift={(0, 0, 3)}]
        \defcoords
        \draw[very thick] (a) -- (b);
        \draw[very thick] (b) -- (c);
        \draw[very thick] (c) -- (d);
        \draw[very thick] (d) -- (a);
        \filldraw (a) circle (0.05em);
        \filldraw (b) circle (0.05em);
        \filldraw (c) circle (0.05em);
        \filldraw (d) circle (0.05em);
        \filldraw[red] (0.5,0.5,0) circle (0.2em);
        \filldraw[red] (-0.5,-0.5,0) circle (0.2em);
    \end{scope}
\end{scope}
\end{tikzpicture}
}}
\]
\caption{Solid torus double-covering the 3-ball. The top and the bottom faces of the right-hand side each correspond to a pair of faces of the tetrahedron sharing a $\pi$-edge.}
\label{fig:solid_torus}
\end{figure}
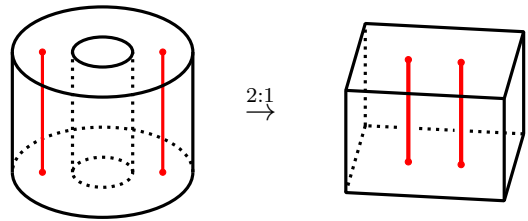

\begin{figure}
\centering
\[
P_{1,0} := 
\vcenter{\hbox{
\begin{tikzpicture}
\begin{scope}[scale=0.4]
\draw[very thick, dotted] (0, 0) ellipse (1 and 0.5);
\draw[very thick] (3, 0) arc (0:-180:3 and 1.5);
\draw[very thick, dotted] (3, 0) arc (0:180:3 and 1.5);
\draw[very thick] (3, 0) -- (3, 4);
\draw[very thick] (-3, 0) -- (-3, 4);
\draw[very thick, dotted] (1, 0) -- (1, 4);
\draw[very thick, dotted] (-1, 0) -- (-1, 4);
\filldraw[red] (2, 0) circle (0.1);
\filldraw[red] (2, 4) circle (0.1);
\filldraw[red] (-2, 0) circle (0.1);
\filldraw[red] (-2, 4) circle (0.1);
\draw[very thick] (0, 4) ellipse (1 and 0.5);
\draw[very thick] (0, 4) ellipse (3 and 1.5);
\draw[ultra thick] ($(0, 4) - ({1/sqrt(2)}, {0.5/sqrt(2)})$) -- ($(0, 4) - ({3/sqrt(2)}, {1.5/sqrt(2)})$) -- ($(0, 0) - ({3/sqrt(2)}, {1.5/sqrt(2)})$) -- ($(0, 0) - ({1/sqrt(2)}, {0.5/sqrt(2)})$) -- cycle;
\draw[ultra thick, ->] ($(0, 0) - ({3/sqrt(2)}, {1.5/sqrt(2)})$) -- ($(0, 2) - ({3/sqrt(2)}, {1.5/sqrt(2)})$);
\end{scope}
\end{tikzpicture}
}}
\;,\quad
P_{0,1} := 
\vcenter{\hbox{
\begin{tikzpicture}
\begin{scope}[scale=0.4]
\draw[very thick, dotted] (0, 0) ellipse (1 and 0.5);
\draw[very thick] (3, 0) arc (0:-180:3 and 1.5);
\draw[very thick, dotted] (3, 0) arc (0:180:3 and 1.5);
\draw[very thick] (3, 0) -- (3, 4);
\draw[very thick] (-3, 0) -- (-3, 4);
\draw[very thick, dotted] (1, 0) -- (1, 4);
\draw[very thick, dotted] (-1, 0) -- (-1, 4);
\filldraw[red] (2, 0) circle (0.1);
\filldraw[red] (2, 4) circle (0.1);
\filldraw[red] (-2, 0) circle (0.1);
\filldraw[red] (-2, 4) circle (0.1);
\draw[very thick] (0, 4) ellipse (1 and 0.5);
\draw[very thick] (0, 4) ellipse (3 and 1.5);
\draw[ultra thick, <-] (0, 3.25) arc (-90:270:1.5 and 0.75);
\end{scope}
\end{tikzpicture}
}}
\]
\caption{Skein operators on the boundary of the solid torus.}
\label{fig:meridian-longitude}
\end{figure}
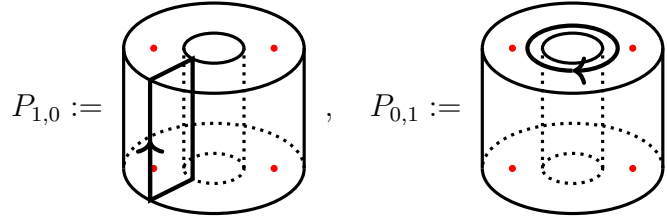

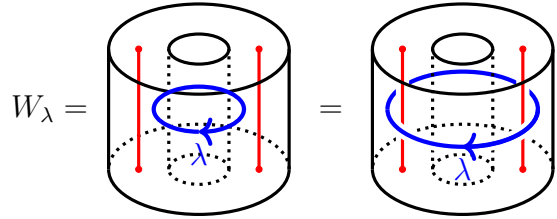
\begin{figure}
\centering
\[
W_{\lambda} = 
\vcenter{\hbox{
\begin{tikzpicture}
\begin{scope}[scale=0.4]
\draw[very thick, dotted] (0, 0) ellipse (1 and 0.5);
\draw[very thick] (3, 0) arc (0:-180:3 and 1.5);
\draw[very thick, dotted] (3, 0) arc (0:180:3 and 1.5);
\draw[very thick] (3, 0) -- (3, 4);
\draw[very thick] (-3, 0) -- (-3, 4);
\draw[very thick, dotted] (1, 0) -- (1, 4);
\draw[very thick, dotted] (-1, 0) -- (-1, 4);
\filldraw[red] (2, 0) circle (0.1);
\filldraw[red] (2, 4) circle (0.1);
\filldraw[red] (-2, 0) circle (0.1);
\filldraw[red] (-2, 4) circle (0.1);
\draw[very thick, red] (2, 0) -- (2, 4);
\draw[very thick, red] (-2, 0) -- (-2, 4);
\draw[very thick] (0, 4) ellipse (1 and 0.5);
\draw[very thick] (0, 4) ellipse (3 and 1.5);
\draw[blue, ultra thick, <-] (0, 1.25) arc (-90:270:1.5 and 0.75);
\node[blue, below] at (0, 1.25){$\lambda$};
\end{scope}
\end{tikzpicture}
}}
\;=\;
\vcenter{\hbox{
\begin{tikzpicture}
\begin{scope}[scale=0.4]
\draw[blue, ultra thick] (0, 2) [partial ellipse = 0 : 180 :2.5 and 1.25];
\draw[very thick, dotted] (0, 0) ellipse (1 and 0.5);
\draw[very thick] (3, 0) arc (0:-180:3 and 1.5);
\draw[very thick, dotted] (3, 0) arc (0:180:3 and 1.5);
\draw[very thick] (3, 0) -- (3, 4);
\draw[very thick] (-3, 0) -- (-3, 4);
\draw[very thick, dotted] (1, 0) -- (1, 4);
\draw[very thick, dotted] (-1, 0) -- (-1, 4);
\draw[white, line width=5] (2, 0) -- (2, 4);
\draw[very thick, red] (2, 0) -- (2, 4);
\draw[white, line width=5] (-2, 0) -- (-2, 4);
\draw[very thick, red] (-2, 0) -- (-2, 4);
\filldraw[red] (2, 0) circle (0.1);
\filldraw[red] (2, 4) circle (0.1);
\filldraw[red] (-2, 0) circle (0.1);
\filldraw[red] (-2, 4) circle (0.1);
\draw[very thick] (0, 4) ellipse (1 and 0.5);
\draw[very thick] (0, 4) ellipse (3 and 1.5);
\draw[white, line width=5] (0, 2) [partial ellipse = 0 : -180 :2.5 and 1.25];
\draw[blue, ultra thick] (0, 2) [partial ellipse = 0 : -180 :2.5 and 1.25];
\draw[blue, ultra thick, ->] (0, 2) [partial ellipse = 0 : -90 :2.5 and 1.25];
\node[blue, below] at (0, 0.75){$\lambda$};
\end{scope}
\end{tikzpicture}
}}
\]
\caption{$W_\lambda$ in the solid torus with sign lines}
\label{fig:W_lambda-sign-line}
\end{figure}

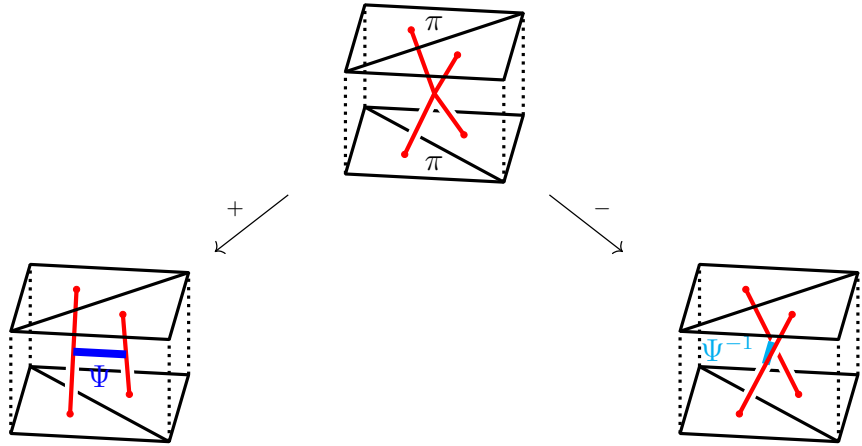
\begin{figure}
\centering
\begin{math}
\begin{tikzcd}
 &
\vcenter{\hbox{
\tdplotsetmaincoords{65}{52}
\begin{tikzpicture}[tdplot_main_coords]
\begin{scope}[scale = 0.5, tdplot_main_coords]
    \newcommand*{\defcoords}{
        \coordinate (o) at (0, 0, 0);
        \coordinate (a) at (3, 0, 0);
        \coordinate (b) at ({3*cos(90)}, {3*sin(90)}, 0);
        \coordinate (c) at ({3*cos(2*90)}, {3*sin(2*90)}, 0);
        \coordinate (d) at ({3*cos(3*90)}, {3*sin(3*90)}, 0);
        \coordinate (abc) at ($1/3*(a)+1/3*(b)+1/3*(c)$);
        \coordinate (abd) at ($1/3*(a)+1/3*(b)+1/3*(d)$);
        \coordinate (bcd) at ($1/3*(b)+1/3*(c)+1/3*(d)$);
        \coordinate (acd) at ($1/3*(a)+1/3*(c)+1/3*(d)$);
    }
    \defcoords
    \draw[very thick] (a) -- (b);
    \draw[very thick] (b) -- (c);
    \draw[very thick] (c) -- (d);
    \draw[very thick] (d) -- (a);
    \draw[very thick] (a) -- (c);
    \draw[very thick, dotted] (a) -- (3, 0, 3);
    \draw[very thick, dotted] (b) -- ({3*cos(90)}, {3*sin(90)}, 3);
    \draw[very thick, dotted] (c) -- ({3*cos(2*90)}, {3*sin(2*90)}, 3);
    \draw[very thick, dotted] (d) -- ({3*cos(3*90)}, {3*sin(3*90)}, 3);
    \filldraw (a) circle (0.05em);
    \filldraw (b) circle (0.05em);
    \filldraw (c) circle (0.05em);
    \filldraw (d) circle (0.05em);
    \draw[white, line width=5] (abc) -- (0, 0, 1.5);
    \draw[white, line width=5] (acd) -- (0, 0, 1.5);
    \draw[ultra thick, red] (abc) -- (0, 0, 1.5);
    \draw[ultra thick, red] (acd) -- (0, 0, 1.5);
    \filldraw[red] (abc) circle (0.2em);
    \filldraw[red] (acd) circle (0.2em);
    \node[below] at (o){$\pi$};
    \begin{scope}[shift={(0, 0, 3)}]
        \defcoords
        \draw[ultra thick, red] (abd) -- (0, 0, -1.5);
        \draw[ultra thick, red] (bcd) -- (0, 0, -1.5);
        \draw[very thick] (a) -- (b);
        \draw[very thick] (b) -- (c);
        \draw[very thick] (c) -- (d);
        \draw[very thick] (d) -- (a);
        \draw[very thick] (b) -- (d);
        \filldraw (a) circle (0.05em);
        \filldraw (b) circle (0.05em);
        \filldraw (c) circle (0.05em);
        \filldraw (d) circle (0.05em);
        \filldraw[red] (abd) circle (0.2em);
        \filldraw[red] (bcd) circle (0.2em);
        \node[above] at (o){$\pi$};
    \end{scope}
\end{scope}
\end{tikzpicture}
}}
 \arrow[dl, "+", swap] \arrow[dr, "-"] & \\
\vcenter{\hbox{
\tdplotsetmaincoords{65}{52}
\begin{tikzpicture}[tdplot_main_coords]
\begin{scope}[scale = 0.5, tdplot_main_coords]
    \newcommand*{\defcoords}{
        \coordinate (o) at (0, 0, 0);
        \coordinate (a) at (3, 0, 0);
        \coordinate (b) at ({3*cos(90)}, {3*sin(90)}, 0);
        \coordinate (c) at ({3*cos(2*90)}, {3*sin(2*90)}, 0);
        \coordinate (d) at ({3*cos(3*90)}, {3*sin(3*90)}, 0);
        \coordinate (abc) at ($1/3*(a)+1/3*(b)+1/3*(c)$);
        \coordinate (abd) at ($1/3*(a)+1/3*(b)+1/3*(d)$);
        \coordinate (bcd) at ($1/3*(b)+1/3*(c)+1/3*(d)$);
        \coordinate (acd) at ($1/3*(a)+1/3*(c)+1/3*(d)$);
    }
    \defcoords
    \draw[very thick] (a) -- (b);
    \draw[very thick] (b) -- (c);
    \draw[very thick] (c) -- (d);
    \draw[very thick] (d) -- (a);
    \draw[very thick] (a) -- (c);
    \draw[very thick, dotted] (a) -- (3, 0, 3);
    \draw[very thick, dotted] (b) -- ({3*cos(90)}, {3*sin(90)}, 3);
    \draw[very thick, dotted] (c) -- ({3*cos(2*90)}, {3*sin(2*90)}, 3);
    \draw[very thick, dotted] (d) -- ({3*cos(3*90)}, {3*sin(3*90)}, 3);
    \filldraw (a) circle (0.05em);
    \filldraw (b) circle (0.05em);
    \filldraw (c) circle (0.05em);
    \filldraw (d) circle (0.05em);
    \draw[white, line width=5] (abc) -- ($(abd)+(0,0,3)$);
    \draw[ultra thick, red] (abc) -- ($(abd)+(0,0,3)$);
    \filldraw[red] (abc) circle (0.2em);
    \draw[white, line width=5] (acd) -- ($(bcd)+(0,0,3)$);
    \draw[ultra thick, red] (acd) -- ($(bcd)+(0,0,3)$);
    \filldraw[red] (acd) circle (0.2em);
    \draw[blue, line width=3] ($1/2*(abc) + 1/2*(abd) + (0, 0, 3/2)$) -- ($1/2*(acd) + 1/2*(bcd) + (0, 0, 3/2)$);
    \node[blue, anchor=north] at ($1/4*(abc) + 1/4*(abd) + 1/4*(acd) + 1/4*(bcd) + (0, 0, 3/2)$){$\Psi$};
    \begin{scope}[shift={(0, 0, 3)}]
        \defcoords
        \draw[very thick] (a) -- (b);
        \draw[very thick] (b) -- (c);
        \draw[very thick] (c) -- (d);
        \draw[very thick] (d) -- (a);
        \draw[very thick] (b) -- (d);
        \filldraw (a) circle (0.05em);
        \filldraw (b) circle (0.05em);
        \filldraw (c) circle (0.05em);
        \filldraw (d) circle (0.05em);
        \filldraw[red] (abd) circle (0.2em);
        \filldraw[red] (bcd) circle (0.2em);
    \end{scope}
\end{scope}
\end{tikzpicture}
}}
& & \vcenter{\hbox{
\tdplotsetmaincoords{65}{52}
\begin{tikzpicture}[tdplot_main_coords]
\begin{scope}[scale = 0.5, tdplot_main_coords]
    \newcommand*{\defcoords}{
        \coordinate (o) at (0, 0, 0);
        \coordinate (a) at (3, 0, 0);
        \coordinate (b) at ({3*cos(90)}, {3*sin(90)}, 0);
        \coordinate (c) at ({3*cos(2*90)}, {3*sin(2*90)}, 0);
        \coordinate (d) at ({3*cos(3*90)}, {3*sin(3*90)}, 0);
        \coordinate (abc) at ($1/3*(a)+1/3*(b)+1/3*(c)$);
        \coordinate (abd) at ($1/3*(a)+1/3*(b)+1/3*(d)$);
        \coordinate (bcd) at ($1/3*(b)+1/3*(c)+1/3*(d)$);
        \coordinate (acd) at ($1/3*(a)+1/3*(c)+1/3*(d)$);
    }
    \defcoords
    \draw[very thick] (a) -- (b);
    \draw[very thick] (b) -- (c);
    \draw[very thick] (c) -- (d);
    \draw[very thick] (d) -- (a);
    \draw[very thick] (a) -- (c);
    \draw[very thick, dotted] (a) -- (3, 0, 3);
    \draw[very thick, dotted] (b) -- ({3*cos(90)}, {3*sin(90)}, 3);
    \draw[very thick, dotted] (c) -- ({3*cos(2*90)}, {3*sin(2*90)}, 3);
    \draw[very thick, dotted] (d) -- ({3*cos(3*90)}, {3*sin(3*90)}, 3);
    \filldraw (a) circle (0.05em);
    \filldraw (b) circle (0.05em);
    \filldraw (c) circle (0.05em);
    \filldraw (d) circle (0.05em);
    \draw[white, line width=5] (abc) -- ($(bcd)+(0,0,3)$);
    \draw[ultra thick, red] (abc) -- ($(bcd)+(0,0,3)$);
    \filldraw[red] (abc) circle (0.2em);
    \draw[white, line width=5] (acd) -- ($(abd)+(0,0,3)$);
    \draw[cyan, line width=3] ($1/2*(abc) + 1/2*(bcd) + (0, 0, 3/2)$) -- ($1/2*(acd) + 1/2*(abd) + (0, 0, 3/2)$);
    \node[cyan, anchor=east] at ($1/4*(abc) + 1/4*(bcd) + 1/4*(acd) + 1/4*(abd) + (0, 0, 3/2)$){$\Psi^{-1}$};
    \draw[ultra thick, red] (acd) -- ($(abd)+(0,0,3)$);
    \filldraw[red] (acd) circle (0.2em);
    \begin{scope}[shift={(0, 0, 3)}]
        \defcoords
        \draw[very thick] (a) -- (b);
        \draw[very thick] (b) -- (c);
        \draw[very thick] (c) -- (d);
        \draw[very thick] (d) -- (a);
        \draw[very thick] (b) -- (d);
        \filldraw (a) circle (0.05em);
        \filldraw (b) circle (0.05em);
        \filldraw (c) circle (0.05em);
        \filldraw (d) circle (0.05em);
        \filldraw[red] (abd) circle (0.2em);
        \filldraw[red] (bcd) circle (0.2em);
    \end{scope}
\end{scope}
\end{tikzpicture}
}}
\end{tikzcd}
\end{math}
\caption{Resolving the singular point by inserting a skein dilogarithm}
\label{fig:holomorphic_disks}
\end{figure}

Now, by definition, $\Sk(\widetilde{M}^{\mathrm{sing}}; \widetilde{\tau})$ is the quotient of 
$\Sk(\widetilde{M}^{\circ}; \widetilde{\tau})$
by the 3-term relation \eqref{eq:skeinrel5} (imposed near each singularity). 
Thus a map from $\Sk(\widetilde{M}^{\mathrm{sing}}; \widetilde{\tau})$ is the same as a map from $\Sk(\widetilde{M}^{\circ}; \widetilde{\tau})$ which annihilates this relation. 

But with the specialization of angles, the 3-term skein relation \eqref{eq:skeinrel5} becomes the 3-term recurrence relation \eqref{eq:dilog-relative-recurrence} satisfied by the skein dilogarithm in case of the positive resolution, or the corresponding recurrence relation for its inverse in case of the negative resolution, inserted along the distinguished cycle.\footnote{The sign lines don't change the recursion since we can pull the sign lines away symmetrically. See Figure \ref{fig:meridian-longitude} for the meridian and longitude operators drawn on the boundary torus, and Figure \ref{fig:W_lambda-sign-line} for the elements $W_\lambda$ in the presence of sign lines in the solid torus. 
The orientation of the longitude $P_{0,1}$ is chosen to match that of the distinguished cycle. 
}
Thus, we construct our desired map by gluing a solid torus $D^2\times S^1$ containing a skein dilogarithm $\Psi\in\Sk(D^2\times S^1)$ (or its inverse $\Psi^{-1}\in\Sk(D^2\times S^1)$) cabling the distinguished cycle to each boundary of $\widetilde{M}^\circ$; 
see Figure \ref{fig:holomorphic_disks}. 
\end{proof}

\section{A skein trace for double covers of 3-manifolds}
\label{combinatorial skein trace}

Building on ideas from \cite{Gaiotto-Moore-Neitzke-WKB, Gaiotto-Moore-Neitzke-spectral, Freed-Neitzke} and especially \cite{Neitzke-Yan-nonabelianization, Neitzke-Yan-gl3}, we will prove: 

\begin{thm}\label{thm:UV-IRmap}
Let $M$ be a 3-manifold with ideal triangulation $\Delta$ and associated (singular) branched double cover $\widetilde{M}^{\mathrm{sing}}$ with branch locus $\widetilde{\tau}$.  
Then there is a well-defined $\Z[a^\pm, z^\pm]$-module homomorphism 
\[
[\,\cdot\,]_{\widetilde{M}^{\mathrm{sing}}}^{\mathrm{net}} \colon \mathrm{Sk}_{a^2,z}(M) \rightarrow \mathrm{Sk}_{a, z}(\widetilde{M}^{\mathrm{sing}}; \widetilde{\tau}). 
\]
\end{thm}

Here, ``net'' is to remind us of spectral networks. 
The morphism is an explicit formula for lifting links, described in \eqref{def : F(K)} below.  We prove it is well defined, in Section \ref{sec:q-UV-IR-map}, by checking by hand that the explicit formula respects the skein relations.   The fact that the morphism is nontrivial can be seen already by calculations for the trivial double cover, which we give in Section \ref{sec: trivial double cover}.  By the results of the previous section, we derive the following consequence for smooth double covers: 

\begin{cor}\label{thm:UV-IRmap smooth}
Let $M$ be a 3-manifold equipped with effective signed taut ideal triangulation, 
and let $\widetilde{M} \to M$ be the branched double cover associated to the tangle $\sigma$ it specifies. 
Then there is a well-defined map
\[
[\,\cdot\,]_{\widetilde{M}}^{\mathrm{net}} \colon \mathrm{Sk}_{a^2,z}(M) \rightarrow \widehat{\mathrm{Sk}}_{a, z}(\widetilde{M};\widetilde{\sigma}). 
\]
\end{cor}
\begin{proof}
Compose the map of Theorem \ref{thm:UV-IRmap} with that of Theorem \ref{smoothing skein module}. 
\end{proof}

\subsection{Turning numbers}

We describe additional geometric data on a 3-manifold $M$ that allows us to compute framing data numerically. Assume that the tangent bundle of $M$ is decomposed as $TM=\R\oplus E$, where $\R$ is the trivial bundle trivialized by some vector field $\xi$ and $E$ is an $SO(2)$-bundle, then links in $M$ that are nowhere tangent to $\xi$ gets an induced framing as follows. If the tangent to $K$ is $t$, let $e_1$ denote the projection of $t$ to $E$ and frame $K$ by the unit normal vector field $e_2$ in $E$ that is perpendicular to $e_1$ in $E$ and such that $(e_1,e_2)$ is a positively oriented basis. If furthermore $E$ is equipped with a flat connection $\nabla$ then the section $e_2$ along $K$ can be compared to a parallel section in $E$, if $s\in [0,L]$ parametrizes $K$ then 
\begin{equation}\label{eq : turning general}
\int_0^L \langle\nabla_s e_2,-e_1\rangle \,ds \ = \ \int_0^L \langle\nabla_s e_1,e_2\rangle \,ds \ \in \ 2\pi\Z,
\end{equation}
measures to total turning of $e_2$, or equivalently $e_1$, with respect to the parallel section along $K$. We thus have an integer that measures the total turning of the induced framing of a link nowhere parallel to $\xi$. Consider now a general framing of $K$, given by a vector field $n$ generic with respect to $\xi$. Then $n$ is parallel to $\xi$ only at a finite number of points along $K$, we call such a point a \emph{framing tangency}, and if $n_E$ denotes the projection of $n$ to $E$ and $p$ is a point where $n_E$ vanishes then $\nabla_s n_E(p)$ is linearly independent with $e_1(p)$. We take the sign $\sigma(p)=\pm 1$ of the framing tangency $p$ to be the orientation sign of $(e_1(p),\nabla_s n_E(p))$. It is easy to see that the framing given by $n$ is homotopic to the framing given by $\hat n_E=\pm e_2$, the unit vector in direction $n_E$ outside small intervals of framing tangencies and a $\sigma(p)\pi$-rotation
$$
\hat n_E(\text{before}) \ \to \ \sigma(p)\xi \ \to \ \hat n_E(\text{after})
$$ 
inside the small intervals. 
\begin{lemma}\label{l:turning and angles}
Two generic framings of a knot $K$ are homotopic if and only if their
sums of the turning numbers and discrete $\pi$-turns at framing tangencies:
\[
\int_0^L\langle\nabla_s e_1,e_2\rangle \,ds + \sum_{p} \sigma(p)\pi \ \in \ 2\pi\Z
\]
agree.
\end{lemma}
\begin{proof}
The sum of the turning number and the $\pi$-turns is clearly invariant under generic homotopy of framings. Further one can homotope any framing to a framing given by $e_2$ outside a small arc where all $\pi$ rotations take place and it is clear that any two such framings with the same algebraic number of $\pi$-turns are homotopic.
\end{proof}

We now consider an analogue of this for branched double covers. Consider an ideal triangulation $\Delta$ on $M$ equipped with a generalized angle structure. The generalized angle structure gives a geometric polygonal decomposition of the leaf space of $\Delta$, or equivalently the 2-skeleton $\Delta^\vee_{(2)}$ of the dual polygonal decomposition, as follows. Consider a $2$-dimensional polygon $\eta$ of $\Delta^{\vee}_{(2)}$ inside a tetrahedron of $\Delta$. It has four corners, two at barycenters of the boundary triangles, one at the barycenter of an edge and one at the barycenter of the tetrahedron itself. We represent $\eta$ as a planar polygon with angles at the corners corresponding to barycenters of the boundary triangles equal to $\frac{\pi}{2}$, the angle at the corner of the barycenter of the edge equal to $\theta$ as determined by the generalized angle structure and the angle at the corner of the barycenter equal to $\pi-\theta$, see Figure \ref{fig:leafspace_Euclideanstructure}. 
\begin{figure}
\centering
$\vcenter{\hbox{
\tdplotsetmaincoords{60}{80}
\begin{tikzpicture}[tdplot_main_coords]
\begin{scope}[scale = 0.8, tdplot_main_coords]
    \coordinate (o) at (0, 0, 0);
    \coordinate (a) at (0, 0, 3);
    \coordinate (a1) at (0, 0, -1);
    
    \coordinate (b) at ({2*sqrt(2)}, 0, -1);
    \coordinate (b1) at ({-2*sqrt(2)/3}, 0, 1/3);
    
    \coordinate (c) at ({-sqrt(2)}, {sqrt(6)}, -1);
    \coordinate (c1) at ({sqrt(2)/3}, {-sqrt(6)/3}, 1/3);
    
    \coordinate (d) at ({-sqrt(2)}, {-sqrt(6)}, -1);
    \coordinate (d1) at ({sqrt(2)/3}, {sqrt(6)/3}, 1/3);
    
    \coordinate (ab) at ({sqrt(2)}, 0, 1);
    \coordinate (ac) at ({-sqrt(2)/2}, {sqrt(6)/2}, 1);
    \coordinate (ad) at ({-sqrt(2)/2}, {-sqrt(6)/2}, 1);
    \coordinate (bc) at ({sqrt(2)/2}, {sqrt(6)/2}, -1);
    \coordinate (bd) at ({sqrt(2)/2}, {-sqrt(6)/2}, -1);
    \coordinate (cd) at ({-sqrt(2)}, 0, -1);
    
    \draw[very thick] (c) -- (d);

    \draw[white, ultra thick] (o) -- (a1);
    
    \filldraw [very thin, fill= orange, opacity = 0.2] (o) -- (a1) -- (cd) -- (b1) -- (o);
    \filldraw [very thin, fill= blue, opacity = 0.2] (o) -- (b1) -- (ac) -- (d1) -- (o);
    \node[anchor = south west] at (ac) {$\theta$};
    \filldraw [very thin, fill= orange, opacity = 0.2] (o) -- (b1) -- (ad) -- (c1) -- (o);
    \filldraw [very thin, fill= orange, opacity = 0.2] (o) -- (a1) -- (bc) -- (d1) -- (o);
    \filldraw [very thin, fill= orange, opacity = 0.2] (o) -- (a1) -- (bd) -- (c1) -- (o);
    
    \draw[red, very thick] (o) -- (a1);
    \draw[red, very thick] (o) -- (b1);
    \draw[red, very thick] (o) -- (c1);
    \draw[red, very thick] (o) -- (d1);

    \filldraw[red] (o) circle (0.15em);
    \filldraw[red] (a1) circle (0.15em);
    \filldraw[red] (b1) circle (0.15em);
    \filldraw[red] (c1) circle (0.15em);
    \filldraw[red] (d1) circle (0.15em);

    \filldraw [very thin, fill= orange, opacity = 0.2] (o) -- (c1) -- (ab) -- (d1) -- (o);

    \draw[white, ultra thick] (a) -- (b);
    \draw[very thick] (a) -- (b);
    \draw[very thick] (a) -- (c);
    \draw[very thick] (a) -- (d);
    \draw[very thick] (b) -- (c);
    \draw[very thick] (b) -- (d);
    
    \filldraw (a) circle (0.05em);
    \filldraw (b) circle (0.05em);
    \filldraw (c) circle (0.05em);
    \filldraw (d) circle (0.05em);
\end{scope}
\end{tikzpicture}
}}
\;\;\rightarrow\;\;
\vcenter{\hbox{
\begin{tikzpicture}
\filldraw[fill= blue, opacity = 0.2] (0, 0) -- (0, 1) -- ({sqrt(3)}, 1) -- ({sqrt(3)/2}, -1/2) -- cycle;
\node[anchor=west] at ({sqrt(3)}, 1){$\theta$};
\node[anchor=east] at (0, 1){$\frac{\pi}{2}$};
\node[anchor=north] at ({sqrt(3)/2}, -1/2){$\frac{\pi}{2}$};
\node[anchor=north east] at (0, 0){$\pi - \theta$};
\end{tikzpicture}
}}$
\caption{Euclidean structure on each sheet of the leaf space}
\label{fig:leafspace_Euclideanstructure}
\end{figure}
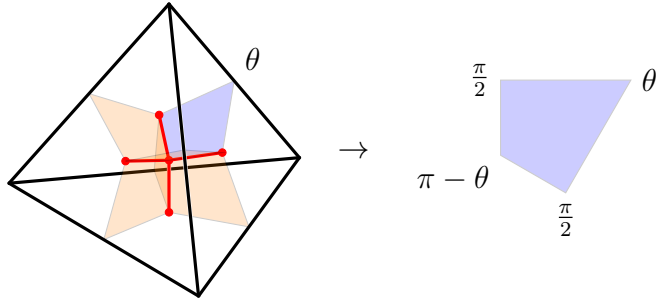

Consider the tangent bundles $T\eta$ of such polygonal representatives of the 2-dimensional polygons $\eta$ of $\Delta^\vee$ and equip them with the $SO(2)$-structure induced by the flat metric on $\eta\subset\R^2$ as above. Note that these bundles glue naturally along the edges of the polygons in $\eta$ and thus gives a flat $SO(2)$-bundle over $\Delta^{\vee}_{(2)}\setminus \Delta^{\vee}_{(0)}$. Pulling back along the foliation gives a flat $SO(2)$-bundle $E$ over $\widetilde{M}^{\mathrm{sing}}\setminus (\widetilde{\tau}\cup \widetilde{\Delta}_{(1)})$.

\begin{lemma}
The $SO(2)$-bundle $E$ extends over the $1$-skeleton as a flat bundle over $\widetilde{M}^{\mathrm{sing}}\setminus \widetilde{\tau}$ and gives a decomposition $T(\widetilde{M}^{\mathrm{sing}} \setminus \widetilde{\tau})= E\oplus \R$, where $\R$ is trivial bundle with section the tangent $\xi$ to the foliation of $\widetilde{M}^{\mathrm{sing}}$ induced by $\Delta$. Moreover, the monodromy in $E$ around the a loop in the three faces of $\Delta^{\vee}$ adjacent to a vertex of a tetrahedron in $\Delta$ equals $2\pi$. 
\end{lemma}

\begin{proof}
The first statement follows from the condition that angles sum to $2\pi$ around edges of the triangulation. To see the second statement, note that $E$ can be viewed as the quotient of the tangent bundle by the foliation. The final statement follows from adding the angles at the barycenter of the three faces of the leaf space that are adjacent to a vertex of the tetrahedron:
\[
(\pi - \theta)+(\pi - \theta')+(\pi - \theta'') = 3\pi - (\theta + \theta' + \theta'') = 2\pi.
\]
\end{proof}

The orientation on the foliation together with the orientation of $\widetilde{M}^{\mathrm{sing}}$ induces an orientation on $E$. We use this flat bundle $E$ to define the turning number of a link diagram on the leaf space:
\begin{defn}
Let $\widetilde{K}$ be an oriented link in $\widetilde{M}^{\mathrm{sing}}\setminus \widetilde{\tau}$ with tangent vectors everywhere transverse to the foliation. Equip $\tilde K$ with the framing given by the tangent $\xi$ to the foliation. Then projecting to the leaf space, we get a link diagram framed by a vector in the projection direction.  
Consider the $SO(2)$-bundle $E$ along $\widetilde{K}$ with the section $e_1$ which is the normalized projection of the tangent vector to $\widetilde{K}$ to $E$. Define the \emph{turning number} of $\widetilde{K}$, $\operatorname{turn}(\widetilde{K}) \in \mathbb{R} = T_1 SO(2)$,  to be the monodromy of this global section with respect to the flat connection $\nabla$ determined by the generalized angle structure, with notation as in \eqref{eq : turning general}:
\[
\operatorname{turn} (\widetilde{K}) \ = \ \frac{1}{2\pi} \int_0^L\langle\nabla_s e_1,e_2\rangle \,ds \ \in \ \R.
\]
\end{defn}

We next show how to compute turning numbers diagrammatically. Let $\widetilde{K}$ be a knot that is generic with respect to the triangulation $\Delta$. Project $\widetilde{K}$ first to $K\subset M$ and then to the leaf space $\kappa\subset\Delta^{\vee}_{(2)}$. Let $\eta$ be a face in $\Delta^{\vee}_{(2)}$. Then $\kappa\cap\eta$ is a collection of curves $\kappa[\eta]$ in $\eta$ the lift of the curve $\kappa[\eta]$ to $\widetilde{K}$ then subdivides $\kappa[\eta]$ into $\kappa[\eta]_1\cup\kappa[\eta]_2$ where $\kappa[\eta]_j$ has lift that lies in sheet $j\in\{1,2\}$ (with respect to the branch cut). We write 
\[
\operatorname{turn}_{ij} (\widetilde{K}_i) = \sum_\eta \operatorname{turn}_{ij} (\kappa[\eta]_i)
,\quad i\ne j\in\{1,2\},
\]
where $\operatorname{turn}_{ij}(\kappa[\eta]_i)$ denotes the turning number, in the usual elementary sense, of $\kappa[\eta]_i$ projected to that face of the $ij$-leaf space, oriented in such a way that the $i$- (resp., $j$-) direction of the foliation is pointing away (resp., toward) the point of projection. 

\begin{lemma}\label{l:generalturningproperties}
We have     
\[
\operatorname{turn} (\widetilde{K}) = \operatorname{turn}_{12} (\widetilde{K}_1) + \operatorname{turn}_{21} (\widetilde{K}_2).
\]
Furthermore, $\operatorname{turn}(\widetilde{K})$ is invariant under isotopies of $\widetilde{K}$ that are nowhere tangent to $\xi$ and changes by $2\pi$ ($-2\pi$) at positive (negative) generic tangencies with $\xi$.  
\end{lemma}

\begin{proof}
By definition the contribution to the turning number from a subarc of the link diagram that lies in a polygonal face equals the usual turning number of the tangent of the arc considered as plane curve. The first formula follows. The invariance statement is immediate from the definition. To see the last statement note that positive and negative tangencies correspond to the appearance or disappearance of positive or negative kinks in the diagram.  
\end{proof}

\begin{remark}
The turning number is a $\mathbb{Z}$-linear combination of ratios $\frac{\theta}{2\pi}$, where $\theta$ comes from the generalized angle structure, and $\frac{1}{2}$, corresponding to half-turns near the seams of the leaf space, so the turning number is in general \emph{not} an integer. 
\end{remark}

\begin{remark}
A generic framing of a knot $\widetilde{K}$ is tangent to $\xi$ at a finite number of points. Adding $\pi$-turns at such tangencies to the turning number exactly as in Lemma \ref{l:turning and angles} gives a numerical invariant of framed knots in $\widetilde{M}$ that remains invariant under framed isotopies by Lemma \ref{l:generalturningproperties}.
\end{remark}

\subsection{Proof of Theorem \ref{thm:UV-IRmap}}\label{sec:q-UV-IR-map}
Let $M$ be a 3-manifold with ideal triangulation $\Delta$, associated branched double cover $\widetilde{M}^{\mathrm{sing}}$, with branch locus $\tau\subset M$ and correspondingly $\widetilde{\tau}\subset\widetilde{M}^{\mathrm{sing}}$. In this section, we construct a $\Z[a^\pm, z^\pm]$-module homomorphism:
\[
[\,\cdot\,]_{\widetilde{M}^{\mathrm{sing}}}^{\mathrm{net}}\colon \mathrm{Sk}_{a^2,z}(M) \rightarrow \mathrm{Sk}_{a, z}(\widetilde{M}^{\mathrm{sing}},\widetilde{\tau}). 
\]
The construction is almost identical to that of \cite{Neitzke-Yan-nonabelianization, Panitch-Park-compatibility}, except that we are working with HOMFLYPT skeins, instead of $\mathfrak{gl}_2$-skeins. 

Let $K \subset M\setminus \tau$ be a framed link in general position with respect to the foliation. Projecting $K$ to the leaf space of the foliation gives a link diagram on the leaf space. 
Thinking of $K$ as a ribbon link, the link diagram will be flat on the leaf space, except at finitely many points where it is half-twisted (i.e., where the framing is parallel to the foliation). 

We will define $[K]_{\widetilde{M}^{\mathrm{sing}}}^{\mathrm{net}}$ as a finite linear combination $\sum_\beta c(\widetilde{K}_\beta) [\widetilde{K}_\beta]$, where the summands $\widetilde{K}_\beta \subset \widetilde{M}^{\mathrm{sing}}\setminus\widetilde{\tau}$ ranges over all possible lifts of $K$ that can be constructed by applying combinations of three local lifts: \emph{direct lifts}, \emph{detours}, and \emph{exchanges} defined as follows: 
\begin{itemize}
\item \emph{Direct lift:} an arc of $K$ can be lifted to either sheet of the branched double cover, we label the arc by the number of the sheet to where it lifts: 
\[
\vcenter{\hbox{
\begin{tikzpicture}
    \draw[very thick, ->] (0, 0) to [out=60,in=-120] (1, 1);
    \node at (1.5, 0.5) {$\rightsquigarrow$};
    \draw[very thick, ->] (2, 0) to [out=60,in=-120] (3, 1);
    \node[anchor = south] at (2.5, 0.5) {$i$};
    \node at (4.5, 0.5) {$i \in \{1,2\}$};
\end{tikzpicture}
}}
\]
\item \emph{Detour:} arcs in $K$ that intersect a critical leaf  (i.e.\ that cross the binding of the leaf space), can be lifted as follows, where the numbering indicates to which sheet we lift: 
\[
\vcenter{\hbox{
\begin{tikzpicture}
    \draw[red, very thick] (-0.5,-0.3) to (-0.5,1.3);
    
    \draw[green] (-0.5, 0.5) -- (0.5, 0.5);
    \draw[green, latex-latex] (-0.3, 0.5) -- (0.3, 0.5);
    \node[green, anchor = south] at (-0.2, 0.5) {\tiny $i$};
    \node[green, anchor = south] at (0.2, 0.5) {\tiny $j$}; 

    \draw[very thick, ->] (0, 0) to [out=60,in=-120] (1, 1);

    \draw[green] (0.5, 0.5) -- (0.8, 0.5);

    \node at (1.5, 0.5) {$\rightsquigarrow$};

    \draw[red, very thick] (2.5,-0.3) to (2.5,1.3);

    \draw[very thick] (3, 0) to [out=60,in=-145] (3.4, 0.4);
    \draw[very thick, ->] (3.6, 0.6) to [out=35,in=-120] (4, 1);
    
     \draw[white, line width=.15cm] (2.6, 0.4) -- (2.4, 0.4);
    \draw[very thick] (3.4, 0.4) -- (2.4, 0.4);
    \draw[very thick] (3.6, 0.6) -- (2.6, 0.6);
    \draw[very thick] (2.4, 0.4) to [out=180, in=-135] (2.37, 0.5) to [out=45, in=180] (2.6, 0.6);

    \filldraw (3.6, 0.6) circle (0.04em);
    \filldraw (3.4, 0.4) circle (0.04em);

     \draw[white, line width=.15cm] (2.5, 0.5) to (2.5,1.3);
    \draw[red, very thick] (2.5, 0.5) to (2.5,1.3);

    \node[anchor = west] at (3.2, 0.1) {$i$};
    \node[anchor = west] at (3.7, 0.6) {$j$};

\end{tikzpicture}
}}
\]
\item \emph{Exchange:} two arcs in $K$ that lie on the same line in the foliation (i.e.\ crossings in the link diagram of $K$ on the leaf space) can be lifted in the following way, where the numbering indicates to which sheet we lift: 
\[
\vcenter{\hbox{
\begin{tikzpicture}
    \draw[green] (-0.8, 0.5) -- (-0.5, 0.5); 
    
    \draw[very thick, ->] (-0.5,1.3) to [out=-75, in=105] (-0.5,-0.3);
    
    \draw[green] (-0.5, 0.5) -- (0.5, 0.5);
    \draw[green, latex-latex] (-0.3, 0.5) -- (0.3, 0.5);
    \node[green, anchor = south] at (-0.2, 0.5) {\tiny $i$};
    \node[green, anchor = south] at (0.2, 0.5) {\tiny $j$}; 

    \draw[very thick, ->] (0, 0) to [out=60, in=-120] (1, 1);

    \draw[green] (0.5, 0.5) -- (0.8, 0.5);

    \node at (1.5, 0.5) {$\rightsquigarrow$};

    \draw[very thick] (2.5, 1.3) to [out=-75, in=80] (2.5, 0.6);
    \draw[very thick, ->] (2.5, 0.4) to [out=-100, in=105] (2.5, -0.3);

    \draw[very thick] (3, 0) to [out=60,in=-145] (3.4, 0.4);
    \draw[very thick, ->] (3.6, 0.6) to [out=35,in=-120] (4, 1);
    
    \draw[very thick] (3.4, 0.4) -- (2.5, 0.4);
    \draw[very thick] (3.6, 0.6) -- (2.5, 0.6);

    \filldraw (3.6, 0.6) circle (0.04em);
    \filldraw (3.4, 0.4) circle (0.04em);

    \filldraw (2.5, 0.6) circle (0.04em);
    \filldraw (2.5, 0.4) circle (0.04em);

    \node[anchor = west] at (3.2, 0.1) {$i$};
    \node[anchor = west] at (3.7, 0.6) {$j$};

\end{tikzpicture}
}}
\]
\end{itemize}
It is easy to see that there are only finitely many possible lifts. 

We will associate a weight $c(\widetilde{K}_{\beta}) \in R^{\Theta}$ to each lift $\widetilde{K}_\beta\subset M^{\mathrm{sing}}\setminus\widetilde{\tau}$ of $K\subset M\setminus\tau$ that is constructed by the these local lifts. 
The weight is the product of \emph{turning factors}, \emph{exchange factors}, and \emph{framing factors} defined as follows.
\begin{itemize}
\item \emph{Turning factor:} 
\[
a^{\operatorname{turn}(\widetilde{K}_\beta)} = \prod_{\substack{1\leq i,j \leq 2\\ i \neq j}}a^{\operatorname{turn}_{ij}\widetilde{K}_{\beta;i}},
\]
where $\widetilde{K}_{\beta;i}$ is the part of the link diagram of $\widetilde{K}_\beta$ that lies in sheet $i$ over a polygon in the leaf space and where the product ranges over all polygons in the leaf space. 
\item \emph{Exchange factor:} 
\[
\operatorname{exch}(\widetilde{K}_{\beta})=\prod_{\text{exchanges}} \pm z
\]
where the product ranges over all exchanges of $\widetilde{K}_\beta$ and where the sign is the sign of the crossing on the leaf space. 
\item \emph{Framing factor:} 
\[
a^{\operatorname{frame}(\widetilde{K}_\beta)}=\prod_{\text{framing tangencies}} a^{\pm\frac12},
\]
where the product ranges over all framing tangencies and the sign in the exponent is the sign of the corresponding half-twist in the ribbon link diagram of $K$. 
\end{itemize}
We then define the weight of $\widetilde{K}_\beta$ as the monomial in $a^\pm,z^\pm$ given by
\[
c(\widetilde{K}_\beta)=a^{\operatorname{turn}(\widetilde{K}_\beta)}a^{\operatorname{frame}(\widetilde{K}_\beta)}\operatorname{exch}(\widetilde{K}_{\beta}).
\]
Finally, we define
\begin{equation}\label{def : F(K)}
[K]_{\widetilde{M}^{\mathrm{sing}}}^{\mathrm{net}} = \sum_{\beta} c(\widetilde{K}_\beta) [\widetilde{K}_\beta] \;\in \mathrm{Sk}_{a, z}(\widetilde{M}^{\mathrm{sing}};\widetilde{\tau}),
\end{equation}
where the sum is over all possible lifts $\widetilde{K}_\beta$. 

We show that $[K]_{\widetilde{M}^{\mathrm{sing}}}^{\mathrm{net}}\in \mathrm{Sk}_{a, z}(\widetilde{M}^{\mathrm{sing}}; \widetilde{\tau})$ depends only on the class $[K] \in \mathrm{Sk}_{a^2,z}(M)$; i.e., that it is invariant under framed isotopy of $K$ and that it respects the skein relations. Any framed isotopy of $K$ can be decomposed into isotopies that induce isotopies of of the diagram of $K$ on the leaf space, which clearly does not affect $[K]_{\widetilde{M}^{\mathrm{sing}}}^{\mathrm{net}}$, and a finite composition of diagram moves that are isotopies of the following types:
\begin{enumerate}[label= (\Roman*)]
\item\label{item:Reidemeister} 3 types of framed Reidemeister moves
\item\label{item:MovesNearBinding} 4 types of moves near the bindings:
\begin{enumerate}
    \item moving an exchange across a critical leaf
    \[
    \vcenter{\hbox{
    \tdplotsetmaincoords{60}{40}
    \begin{tikzpicture}[tdplot_main_coords]
    \begin{scope}[scale = 0.7, tdplot_main_coords]
        \coordinate (o) at (0, 0, 1.5);
        \coordinate (a) at (2, 0, 1.5);
        \coordinate (b) at ({2*cos(120)}, {2*sin(120)}, 1.5);
        \coordinate (c) at ({2*cos(240)}, {2*sin(240)}, 1.5);
        \coordinate (o1) at (0, 0, -1.5);
        \coordinate (a1) at (2, 0, -1.5);
        \coordinate (b1) at ({2*cos(120)}, {2*sin(120)}, -1.5);
        \coordinate (c1) at ({2*cos(240)}, {2*sin(240)}, -1.5);
        
        \filldraw [very thin, fill= orange, opacity = 0.2] (o) -- (a) -- (a1) -- (o1) -- cycle;
        \filldraw [very thin, fill= orange, opacity = 0.2] (o) -- (b) -- (b1) -- (o1) -- cycle;
        \filldraw [very thin, fill= orange, opacity = 0.2] (o) -- (c) -- (c1) -- (o1) -- cycle;
        
        \draw[red, very thick] (o) -- (o1);

        \draw[very thick] ({1.5*cos(240)}, {1.5*sin(240)}, {-1/2}) -- ({0.87*cos(240)}, {0.87*sin(240)}, {-0.08});
        \draw[very thick] ({0.63*cos(240)}, {0.63*sin(240)}, {0.08}) -- (0, 0, 0.5) -- ({1.5*cos(0)}, {1.5*sin(0)}, {1/2});
        \draw[very thick] ({1.5*cos(240)}, {1.5*sin(240)}, {1/2}) -- (0, 0, -0.5) -- ({1.5*cos(0)}, {1.5*sin(0)}, {-1/2});
    \end{scope}
    \end{tikzpicture}
    }}
    \;\;\;\leftrightarrow
    \vcenter{\hbox{
    \tdplotsetmaincoords{60}{40}
    \begin{tikzpicture}[tdplot_main_coords]
    \begin{scope}[scale = 0.7, tdplot_main_coords]
        \coordinate (o) at (0, 0, 1.5);
        \coordinate (a) at (2, 0, 1.5);
        \coordinate (b) at ({2*cos(120)}, {2*sin(120)}, 1.5);
        \coordinate (c) at ({2*cos(240)}, {2*sin(240)}, 1.5);
        \coordinate (o1) at (0, 0, -1.5);
        \coordinate (a1) at (2, 0, -1.5);
        \coordinate (b1) at ({2*cos(120)}, {2*sin(120)}, -1.5);
        \coordinate (c1) at ({2*cos(240)}, {2*sin(240)}, -1.5);
        
        \filldraw [very thin, fill= orange, opacity = 0.2] (o) -- (a) -- (a1) -- (o1) -- cycle;
        \filldraw [very thin, fill= orange, opacity = 0.2] (o) -- (b) -- (b1) -- (o1) -- cycle;
        \filldraw [very thin, fill= orange, opacity = 0.2] (o) -- (c) -- (c1) -- (o1) -- cycle;
        
        \draw[red, very thick] (o) -- (o1);

        \draw[very thick] ({1.5*cos(240)}, {1.5*sin(240)}, {-1/2}) -- (0, 0, -0.5) -- ({0.63*cos(0)}, {0.63*sin(0)}, {-0.08});
        \draw[very thick] ({0.87*cos(0)}, {0.87*sin(0)}, 0.08) -- ({1.5*cos(0)}, {1.5*sin(0)}, {1/2});
        \draw[very thick] ({1.5*cos(240)}, {1.5*sin(240)}, {1/2}) -- (0, 0, 0.5) -- ({1.5*cos(0)}, {1.5*sin(0)}, {-1/2});
    \end{scope}
    \end{tikzpicture}
    }}
    \]
    \item height exchange for detours
    \[
    \vcenter{\hbox{
    \tdplotsetmaincoords{60}{40}
    \begin{tikzpicture}[tdplot_main_coords]
    \begin{scope}[scale = 0.7, tdplot_main_coords]
        \coordinate (o) at (0, 0, 1.5);
        \coordinate (a) at (2, 0, 1.5);
        \coordinate (b) at ({2*cos(120)}, {2*sin(120)}, 1.5);
        \coordinate (c) at ({2*cos(240)}, {2*sin(240)}, 1.5);
        \coordinate (o1) at (0, 0, -1.5);
        \coordinate (a1) at (2, 0, -1.5);
        \coordinate (b1) at ({2*cos(120)}, {2*sin(120)}, -1.5);
        \coordinate (c1) at ({2*cos(240)}, {2*sin(240)}, -1.5);
    
        \draw[very thick] ({1.5*cos(120)}, {1.5*sin(120)}, {0}) -- (0, 0, 0.5);
        
        \filldraw [very thin, fill= orange, opacity = 0.2] (o) -- (a) -- (a1) -- (o1) -- cycle;
        \filldraw [very thin, fill= orange, opacity = 0.2] (o) -- (b) -- (b1) -- (o1) -- cycle;
        \filldraw [very thin, fill= orange, opacity = 0.2] (o) -- (c) -- (c1) -- (o1) -- cycle;
        
        \draw[red, very thick] (o) -- (o1);

        \draw[very thick] (0, 0, 0.5) -- ({1.5*cos(0)}, {1.5*sin(0)}, {1/2});
        \draw[very thick] ({1.5*cos(240)}, {1.5*sin(240)}, {0}) -- (0, 0, -0.5) -- ({1.5*cos(0)}, {1.5*sin(0)}, {-1/2});
    \end{scope}
    \end{tikzpicture}
    }}
    \;\;\;\leftrightarrow
    \vcenter{\hbox{
    \tdplotsetmaincoords{60}{40}
    \begin{tikzpicture}[tdplot_main_coords]
    \begin{scope}[scale = 0.7, tdplot_main_coords]
        \coordinate (o) at (0, 0, 1.5);
        \coordinate (a) at (2, 0, 1.5);
        \coordinate (b) at ({2*cos(120)}, {2*sin(120)}, 1.5);
        \coordinate (c) at ({2*cos(240)}, {2*sin(240)}, 1.5);
        \coordinate (o1) at (0, 0, -1.5);
        \coordinate (a1) at (2, 0, -1.5);
        \coordinate (b1) at ({2*cos(120)}, {2*sin(120)}, -1.5);
        \coordinate (c1) at ({2*cos(240)}, {2*sin(240)}, -1.5);

        \draw[very thick] ({1.5*cos(120)}, {1.5*sin(120)}, {0}) -- (0, 0, -0.5);
        
        \filldraw [very thin, fill= orange, opacity = 0.2] (o) -- (a) -- (a1) -- (o1) -- cycle;
        \filldraw [very thin, fill= orange, opacity = 0.2] (o) -- (b) -- (b1) -- (o1) -- cycle;
        \filldraw [very thin, fill= orange, opacity = 0.2] (o) -- (c) -- (c1) -- (o1) -- cycle;
        
        \draw[red, very thick] (o) -- (o1);

        \draw[very thick] (0, 0, -0.5) -- ({0.63*cos(0)}, {0.63*sin(0)}, {-0.08});
        \draw[very thick] ({0.87*cos(0)}, {0.87*sin(0)}, 0.08) -- ({1.5*cos(0)}, {1.5*sin(0)}, {1/2});
        \draw[very thick] ({1.5*cos(240)}, {1.5*sin(240)}, {0}) -- (0, 0, 0.5) -- ({1.5*cos(0)}, {1.5*sin(0)}, {-1/2});
    \end{scope}
    \end{tikzpicture}
    }}
    \]
    \item moving a strand across a critical leaf
    \[
    \vcenter{\hbox{
    \tdplotsetmaincoords{60}{40}
    \begin{tikzpicture}[tdplot_main_coords]
    \begin{scope}[scale = 0.7, tdplot_main_coords]
        \coordinate (o) at (0, 0, 1.5);
        \coordinate (a) at (2, 0, 1.5);
        \coordinate (b) at ({2*cos(120)}, {2*sin(120)}, 1.5);
        \coordinate (c) at ({2*cos(240)}, {2*sin(240)}, 1.5);
        \coordinate (o1) at (0, 0, -1.5);
        \coordinate (a1) at (2, 0, -1.5);
        \coordinate (b1) at ({2*cos(120)}, {2*sin(120)}, -1.5);
        \coordinate (c1) at ({2*cos(240)}, {2*sin(240)}, -1.5);
        
        \filldraw [very thin, fill= orange, opacity = 0.2] (o) -- (a) -- (a1) -- (o1) -- cycle;
        \filldraw [very thin, fill= orange, opacity = 0.2] (o) -- (b) -- (b1) -- (o1) -- cycle;
        \filldraw [very thin, fill= orange, opacity = 0.2] (o) -- (c) -- (c1) -- (o1) -- cycle;
        
        \draw[red, very thick] (o) -- (o1);

        \draw[very thick] ({1.5*cos(240)}, {1.5*sin(240)}, {-1/2}) .. controls (0, 0, -0.3) and (0, 0, 0.3) .. ({1.5*cos(240)}, {1.5*sin(240)}, {1/2});
    \end{scope}
    \end{tikzpicture}
    }}
    \;\;\;\leftrightarrow
    \vcenter{\hbox{
    \tdplotsetmaincoords{60}{40}
    \begin{tikzpicture}[tdplot_main_coords]
    \begin{scope}[scale = 0.7, tdplot_main_coords]
        \coordinate (o) at (0, 0, 1.5);
        \coordinate (a) at (2, 0, 1.5);
        \coordinate (b) at ({2*cos(120)}, {2*sin(120)}, 1.5);
        \coordinate (c) at ({2*cos(240)}, {2*sin(240)}, 1.5);
        \coordinate (o1) at (0, 0, -1.5);
        \coordinate (a1) at (2, 0, -1.5);
        \coordinate (b1) at ({2*cos(120)}, {2*sin(120)}, -1.5);
        \coordinate (c1) at ({2*cos(240)}, {2*sin(240)}, -1.5);
        
        \filldraw [very thin, fill= orange, opacity = 0.2] (o) -- (a) -- (a1) -- (o1) -- cycle;
        \filldraw [very thin, fill= orange, opacity = 0.2] (o) -- (b) -- (b1) -- (o1) -- cycle;
        \filldraw [very thin, fill= orange, opacity = 0.2] (o) -- (c) -- (c1) -- (o1) -- cycle;
        
        \draw[red, very thick] (o) -- (o1);

        \draw[very thick] ({1.5*cos(240)}, {1.5*sin(240)}, {-1/2}) -- (0, 0, -0.5);
        \draw[very thick] ({1.5*cos(240)}, {1.5*sin(240)}, {1/2}) -- (0, 0, 0.5);
        \draw[very thick] (0, 0, -0.5) .. controls (1.3, 0, -0.3) and (1.3, 0, 0.3) .. (0, 0, 0.5);

    \end{scope}
    \end{tikzpicture}
    }}
    \]
    \item moving a strand across the branch locus\footnote{The strands are drawn thicker here to emphasize the half-twist on the right-hand side. }
    \[
    \vcenter{\hbox{
    \tdplotsetmaincoords{60}{40}
    \begin{tikzpicture}[tdplot_main_coords]
    \begin{scope}[scale = 0.7, tdplot_main_coords]
        \coordinate (o) at (0, 0, 1.5);
        \coordinate (a) at (2, 0, 1.5);
        \coordinate (b) at ({2*cos(120)}, {2*sin(120)}, 1.5);
        \coordinate (c) at ({2*cos(240)}, {2*sin(240)}, 1.5);
        \coordinate (o1) at (0, 0, -1.5);
        \coordinate (a1) at (2, 0, -1.5);
        \coordinate (b1) at ({2*cos(120)}, {2*sin(120)}, -1.5);
        \coordinate (c1) at ({2*cos(240)}, {2*sin(240)}, -1.5);

        \filldraw[black] ({1.5*cos(120)}, {1.5*sin(120)}, {1/2 + 0.1}) -- ({1.5*cos(120)}, {1.5*sin(120)}, {1/2 - 0.1}) -- (0, 0, {0 - 0.1}) -- (0, 0, {0 + 0.1}) -- cycle;
        
        \filldraw [very thin, fill= orange, opacity = 0.2] (o) -- (a) -- (a1) -- (o1) -- cycle;
        \filldraw [very thin, fill= orange, opacity = 0.2] (o) -- (b) -- (b1) -- (o1) -- cycle;
        \filldraw [very thin, fill= orange, opacity = 0.2] (o) -- (c) -- (c1) -- (o1) -- cycle;
        
        \filldraw[black] ({1.5*cos(240)}, {1.5*sin(240)}, {-1/2 + 0.1}) -- ({1.5*cos(240)}, {1.5*sin(240)}, {-1/2 - 0.1}) -- (0, 0, {0 - 0.1}) -- (0, 0, {0 + 0.1}) -- cycle;

        \draw[red, very thick] (o) -- (o1);
        
    \end{scope}
    \end{tikzpicture}
    }}
    \;\;\;\leftrightarrow
    \vcenter{\hbox{
    \tdplotsetmaincoords{60}{40}
    \begin{tikzpicture}[tdplot_main_coords]
    \begin{scope}[scale = 0.7, tdplot_main_coords]
        \coordinate (o) at (0, 0, 1.5);
        \coordinate (a) at (2, 0, 1.5);
        \coordinate (b) at ({2*cos(120)}, {2*sin(120)}, 1.5);
        \coordinate (c) at ({2*cos(240)}, {2*sin(240)}, 1.5);
        \coordinate (o1) at (0, 0, -1.5);
        \coordinate (a1) at (2, 0, -1.5);
        \coordinate (b1) at ({2*cos(120)}, {2*sin(120)}, -1.5);
        \coordinate (c1) at ({2*cos(240)}, {2*sin(240)}, -1.5);

        \filldraw[black] ({1.5*cos(120)}, {1.5*sin(120)}, {1/2 + 0.1}) -- ({1.5*cos(120)}, {1.5*sin(120)}, {1/2 - 0.1}) -- (0, 0, {0.3 - 0.1}) -- (0, 0, {0.3 + 0.1}) -- cycle;
        
        \filldraw [very thin, fill= orange, opacity = 0.2] (o) -- (b) -- (b1) -- (o1) -- cycle;
        \filldraw [very thin, fill= orange, opacity = 0.2] (o) -- (c) -- (c1) -- (o1) -- cycle;

        \filldraw[black] (0, 0, {0.3 + 0.1}) -- (0, 0, {0.3 - 0.1}) -- (0.6, 0, {0 - 0.1}) -- (0.6, 0, {0 + 0.1}) -- cycle;
        \draw[white, line width=2] (0, 0, {-0.3 + 0.1}) -- (0.6, 0, {0 + 0.1});
        
        \filldraw [very thin, fill= orange, opacity = 0.2] (o) -- (a) -- (a1) -- (o1) -- cycle;
        \draw[red, very thick] (o) -- (o1);
        
        \filldraw[black] (0, 0, {-0.3 + 0.1}) -- (0, 0, {-0.3 - 0.1}) -- (0.6, 0, {0 - 0.1}) -- (0.6, 0, {0 + 0.1}) -- cycle;

        \filldraw[black] ({1.5*cos(240)}, {1.5*sin(240)}, {-1/2 + 0.1}) -- ({1.5*cos(240)}, {1.5*sin(240)}, {-1/2 - 0.1}) -- (0, 0, {-0.3 - 0.1}) -- (0, 0, {-0.3 + 0.1}) -- cycle;
        
    \end{scope}
    \end{tikzpicture}
    }}
    \]
\end{enumerate}
\item\label{item:3dmove} Moving a strand across a critical leaf ending on the 4-valent singular point
\[
\vcenter{\hbox{
\tdplotsetmaincoords{30}{80}
\begin{tikzpicture}[tdplot_main_coords]
\begin{scope}[scale = 0.7, tdplot_main_coords]
    \coordinate (o) at (0, 0, 0);
    \coordinate (a) at (0, 0, 3);
    \coordinate (a1) at (0, 0, -1);
    
    \coordinate (b) at ({2*sqrt(2)}, 0, -1);
    \coordinate (b1) at ({-2*sqrt(2)/3}, 0, 1/3);
    
    \coordinate (c) at ({-sqrt(2)}, {sqrt(6)}, -1);
    \coordinate (c1) at ({sqrt(2)/3}, {-sqrt(6)/3}, 1/3);
    
    \coordinate (d) at ({-sqrt(2)}, {-sqrt(6)}, -1);
    \coordinate (d1) at ({sqrt(2)/3}, {sqrt(6)/3}, 1/3);
    
    \coordinate (ab) at ({sqrt(2)}, 0, 1);
    \coordinate (ac) at ({-sqrt(2)/2}, {sqrt(6)/2}, 1);
    \coordinate (ad) at ({-sqrt(2)/2}, {-sqrt(6)/2}, 1);
    \coordinate (bc) at ({sqrt(2)/2}, {sqrt(6)/2}, -1);
    \coordinate (bd) at ({sqrt(2)/2}, {-sqrt(6)/2}, -1);
    \coordinate (cd) at ({-sqrt(2)}, 0, -1);

    \draw[white, ultra thick] (o) -- (a1);
    
    \filldraw [very thin, fill= orange, opacity = 0.2] (o) -- (a1) -- (cd) -- (b1) -- (o);
    \filldraw [very thin, fill= orange, opacity = 0.2] (o) -- (b1) -- (ac) -- (d1) -- (o);
    \filldraw [very thin, fill= orange, opacity = 0.2] (o) -- (b1) -- (ad) -- (c1) -- (o);
    \filldraw [very thin, fill= orange, opacity = 0.2] (o) -- (a1) -- (bc) -- (d1) -- (o);
    \filldraw [very thin, fill= orange, opacity = 0.2] (o) -- (a1) -- (bd) -- (c1) -- (o);
    
    \draw[red, very thick] (o) -- (a1);
    \draw[red, very thick] (o) -- (b1);
    \draw[red, very thick] (o) -- (c1);
    \draw[red, very thick] (o) -- (d1);

    \filldraw[red] (o) circle (0.15em);

    \filldraw [very thin, fill= orange, opacity = 0.2] (o) -- (c1) -- (ab) -- (d1) -- (o);

    \draw[very thick, ->] ({-sqrt(2)/4}, {-sqrt(6)/4}, {1/2}) -- ({-sqrt(2)/3}, 0, 1/6) -- ({-sqrt(2)/4}, {sqrt(6)/4}, {1/2});
\end{scope}
\end{tikzpicture}
}} 
\;\;\;\;\leftrightarrow
\vcenter{\hbox{
\tdplotsetmaincoords{30}{80}
\begin{tikzpicture}[tdplot_main_coords]
\begin{scope}[scale = 0.7, tdplot_main_coords]
    \coordinate (o) at (0, 0, 0);
    \coordinate (a) at (0, 0, 3);
    \coordinate (a1) at (0, 0, -1);
    
    \coordinate (b) at ({2*sqrt(2)}, 0, -1);
    \coordinate (b1) at ({-2*sqrt(2)/3}, 0, 1/3);
    
    \coordinate (c) at ({-sqrt(2)}, {sqrt(6)}, -1);
    \coordinate (c1) at ({sqrt(2)/3}, {-sqrt(6)/3}, 1/3);
    
    \coordinate (d) at ({-sqrt(2)}, {-sqrt(6)}, -1);
    \coordinate (d1) at ({sqrt(2)/3}, {sqrt(6)/3}, 1/3);
    
    \coordinate (ab) at ({sqrt(2)}, 0, 1);
    \coordinate (ac) at ({-sqrt(2)/2}, {sqrt(6)/2}, 1);
    \coordinate (ad) at ({-sqrt(2)/2}, {-sqrt(6)/2}, 1);
    \coordinate (bc) at ({sqrt(2)/2}, {sqrt(6)/2}, -1);
    \coordinate (bd) at ({sqrt(2)/2}, {-sqrt(6)/2}, -1);
    \coordinate (cd) at ({-sqrt(2)}, 0, -1);

    \draw[white, ultra thick] (o) -- (a1);
    
    \filldraw [very thin, fill= orange, opacity = 0.2] (o) -- (a1) -- (cd) -- (b1) -- (o);
    \filldraw [very thin, fill= orange, opacity = 0.2] (o) -- (b1) -- (ac) -- (d1) -- (o);
    \filldraw [very thin, fill= orange, opacity = 0.2] (o) -- (b1) -- (ad) -- (c1) -- (o);
    \filldraw [very thin, fill= orange, opacity = 0.2] (o) -- (a1) -- (bc) -- (d1) -- (o);
    \filldraw [very thin, fill= orange, opacity = 0.2] (o) -- (a1) -- (bd) -- (c1) -- (o);
    
    \draw[red, very thick] (o) -- (a1);
    \draw[red, very thick] (o) -- (b1);
    \draw[red, very thick] (o) -- (c1);
    \draw[red, very thick] (o) -- (d1);

    \filldraw[red] (o) circle (0.15em);

    \filldraw [very thin, fill= orange, opacity = 0.2] (o) -- (c1) -- (ab) -- (d1) -- (o);

    \draw[very thick, ->] ({-sqrt(2)/4}, {-sqrt(6)/4}, {1/2}) -- ({sqrt(2)/6}, {-sqrt(6)/6}, 1/6) .. controls +({sqrt(2)/6}, {sqrt(6)/6}, 1/6) and +({sqrt(2)/6}, {-sqrt(6)/6}, 1/6) .. ({sqrt(2)/6}, {sqrt(6)/6}, 1/6) -- ({-sqrt(2)/4}, {sqrt(6)/4}, {1/2});
\end{scope}
\end{tikzpicture}
}} 
\]
\end{enumerate}
The proof of invariance of $[K]_{\widetilde{M}^{\mathrm{sing}}}^{\mathrm{net}}$ under isotopies of type \ref{item:Reidemeister} and \ref{item:MovesNearBinding}, as well as the proof that $[\,\cdot\,]_{\widetilde{M}^{\mathrm{sing}}}^{\mathrm{net}}$ respects the skein relations of $\mathrm{Sk}_{a,z}(M)$, follows almost verbatim from the proof in \cite{Neitzke-Yan-nonabelianization} after the substitutions $(q-q^{-1})\mapsto z$ and $q\mapsto a$, see \cite[Section 7]{Neitzke-Yan-nonabelianization} for the details. 

It then only remains to check invariance under the move \ref{item:3dmove}, where 
the only non-trivial case is when there are detours. 
The image of the left hand side of \ref{item:3dmove} under $F$ that uses a detour is 
\[
\vcenter{\hbox{
\begin{tikzpicture}[scale=0.7]
\draw[ultra thick, ->] ({-sqrt(3)/2}, 1/2) -- ({sqrt(3)/2}, 1/2);
\draw[white, line width=5] (0, 0) -- (0, 1);
\draw[ultra thick, red] (0, 0) -- (0, 1);
\draw[ultra thick, red] (0, 0) -- ({-sqrt(3)/2}, -1/2);
\draw[ultra thick, red] (0, 0) -- ({sqrt(3)/2}, -1/2);
\draw[ultra thick, red, dotted] (0, 0) -- (0, -1);
\filldraw[orange, opacity=0.2] ({sqrt(3)}, 1) -- ({-sqrt(3)}, 1) -- (0, -2) -- cycle;
\node[anchor=north west] at ({-sqrt(3)}, 1){$\theta'$};
\node[anchor=north east] at ({sqrt(3)}, 1){$\theta$};
\node[anchor=south] at (0, -2){$\theta''$};
\node[anchor=south] at ({-sqrt(3)/2}, 1/2){$1$};
\node[anchor=south] at ({sqrt(3)/2}, 1/2){$2$};
\end{tikzpicture}
}},
\]
whereas the corresponding image of the right hand side of \ref{item:3dmove} under $F$ is a sum of two terms, each using one detour: 
\[
a^{\frac{\theta}{\pi}}\;
\vcenter{\hbox{
\begin{tikzpicture}[scale=0.7]
\node[anchor=north west] at ({-sqrt(3)}, 1){$\theta'$};
\node[anchor=north east] at ({sqrt(3)}, 1){$\theta$};
\node[anchor=south] at (0, -2){$\theta''$};
\draw[ultra thick, ->] (0.1, -1/2) to[out=0, in=180] ($({sqrt(3)/2}, 1/2)$);
\draw[white, line width=5] (0, 0) -- ({sqrt(3)/2}, -1/2);
\draw[ultra thick, red] (0, 0) -- ({sqrt(3)/2}, -1/2);
\draw[ultra thick, red] (0, 0) -- (0, 1);
\draw[ultra thick, red, dotted] (0, 0) -- (0, -1);
\draw[ultra thick, red] (0, 0) -- ({-sqrt(3)/2}, -1/2);
\draw[white, line width=5] ({-sqrt(3)/2}, 1/2) to[out=0, in=180] (0.1, -1/2);
\draw[ultra thick] ({-sqrt(3)/2}, 1/2) to[out=0, in=180] (0.1, -1/2);
\node[anchor=south] at ({-sqrt(3)/2}, 1/2){$1$};
\node[anchor=south] at ({sqrt(3)/2}, 1/2){$2$};
\filldraw[orange, opacity=0.2] ({sqrt(3)}, 1) -- ({-sqrt(3)}, 1) -- (0, -2) -- cycle;
\end{tikzpicture}
}}
\;\;+\;\;
a^{-\frac{\theta'}{\pi}}\;
\vcenter{\hbox{
\begin{tikzpicture}[scale=0.7]
\node[anchor=north west] at ({-sqrt(3)}, 1){$\theta'$};
\node[anchor=north east] at ({sqrt(3)}, 1){$\theta$};
\node[anchor=south] at (0, -2){$\theta''$};
\draw[ultra thick] ({-sqrt(3)/2}, 1/2) to[out=0, in=180] (-0.1, -1/2);
\draw[white, line width=5] (0, 0) -- ({-sqrt(3)/2}, -1/2);
\draw[ultra thick, red] (0, 0) -- ({-sqrt(3)/2}, -1/2);
\draw[ultra thick, red] (0, 0) -- (0, 1);
\draw[ultra thick, red] (0, 0) -- ({sqrt(3)/2}, -1/2);
\draw[ultra thick, red, dotted] (0, 0) -- (0, -1);
\draw[white, line width=5] (-0.1, -1/2) to[out=0, in=180] ($({sqrt(3)/2}, 1/2)$);
\draw[ultra thick, ->] (-0.1, -1/2) to[out=0, in=180] ($({sqrt(3)/2}, 1/2)$);
\node[anchor=south] at ({-sqrt(3)/2}, 1/2){$1$};
\node[anchor=south] at ({sqrt(3)/2}, 1/2){$2$};
\filldraw[orange, opacity=0.2] ({sqrt(3)}, 1) -- ({-sqrt(3)}, 1) -- (0, -2) -- cycle;
\end{tikzpicture}
}},
\]
where the coefficients come from alteration of the turning factors. 
Thus, the difference between $[\,\cdot\,]_{\widetilde{M}^{\mathrm{sing}}}^{\mathrm{net}}$ acting on the left- and right-hand sides is exactly the 3-term relation \eqref{eq:skeinrel5} near the four-valent point of the branch locus at the barycenter of a tetrahedron of $\Delta$ in the definition of $\mathrm{Sk}_{a, z}(\widetilde{M}^{\mathrm{sing}}; \widetilde{\tau})$ (Definition \ref{defn:SkeinModuleOfBranchedDoubleCover}). The theorem follows. 
$\qed$

\subsection{The trivial double cover} \label{sec: trivial double cover}

In this section, we consider trivial double covers. If $M$ is a 3-manifold and $\alpha$ is an everywhere non-zero $1$-form on $M$, the trivial double cover of $M$ is the two component Lagrangian in $T^\ast M$ with one component the zero section and the other the graph of $\alpha$.   

We first calculate the skein trace map 
\[
\Sk_{a_1a_2, z}(\R^3) \rightarrow \Sk_{a_1, z}(\R^3) \otimes \Sk_{a_2, z}(\R^3)
\]
for the trivial double cover of $\R^3$. 
For triangulated 3-manifolds, this situation applies to links which lie in a ball in the interior of a tetrahedron that is disjoint from the brach locus. Since the double cover has two $\R^3$ components it is possible to use separate framing variables $a_1$ and $a_2$ for the two sheets. The framing variable of the underlying $\R^3$ is then $a_1a_2$, and the case with non-separated framing variables is obtained simply by setting $a_1=a_2=a$. We express the results below with separate framing variables $a_1$ and $a_2$. 
In this case, the turning factor should be replaced by
\[
a_1^{\mathrm{turn}_1(\widetilde{K}_\beta)}a_2^{\mathrm{turn}_2(\widetilde{K}_\beta)} \ = \ 
\prod_{\substack{1\leq i,j \leq 2\\ i \neq j}}a_j^{\mathrm{turn}_{ij}\widetilde{K}_{\beta;i}}.
\]
While this example may appear trivial -- since $\Sk(\R^3)$ is freely spanned by the empty skein $[\emptyset]$, and the skein trace in this case sends $[\emptyset] \mapsto [\emptyset] \otimes [\emptyset]$ -- it nonetheless reveals, as we illustrate below, intricate algebraic structures underlying the Hecke algebra. 

\begin{eg}
We discuss the skein lift of the unknot to the trivial double cover of $\R^3$. 
The skein trace map applied to the unknot is as follows:
\begin{table}[H]
    \centering
    \begin{tabular}{c c c}
        \begin{tikzpicture}[anchorbase, scale=.5]
        \draw[very thick, ->] (1,2) to [out=180,in=90] (0,1) to [out=-90,in=-180] (1,0)to [out=0,in=-90] (2,1)to [out=90,in=0] (1,2);
        \node at (-1.2, 1.8) {\tiny $\otimes 1$};
        \node at (-1.2, 1.4) {\tiny $\odot 2$};
        \end{tikzpicture} & 
        $\mapsto$ & 
        $a_{2}$ \begin{tikzpicture}[anchorbase, scale=.5]
        \draw[very thick, ->] (1,2) to [out=180,in=90] (0,1) to [out=-90,in=-180] (1,0)to [out=0,in=-90] (2,1)to [out=90,in=0] (1,2);
        \node at (2, 2) {$1$};
        \end{tikzpicture} $+$ $a_{1}^{-1}$
        \begin{tikzpicture}[anchorbase, scale=.5]
        \draw[very thick, ->] (1,2) to [out=180,in=90] (0,1) to [out=-90,in=-180] (1,0)to [out=0,in=-90] (2,1)to [out=90,in=0] (1,2);
        \node at (2, 2) {$2$};
        \end{tikzpicture}
        \\
         &  & \\
        $\parallel$ &  & $\parallel$ \\
         &  & \\
        $\frac{a_{1}a_{2} - a_{1}^{-1}a_{2}^{-1}}{z}\;\emptyset$ & $\mapsto$ 
         & 
        $\qty( a_{2}\frac{a_{1}-a_{1}^{-1}}{z} + a_{1}^{-1}\frac{a_{2}-a_{2}^{-1}}{z}) \; \emptyset$
    \end{tabular},
\end{table}
\noindent where $\otimes 1$ and $\odot2$ indicates the natural oriented lift to the two sheets of the vector field dual of $\alpha$.\footnote{That is, the foliation is orthogonal to the page, oriented in such a way that the orientation of the leaves in sheet 1 (resp., sheet 2) is pointed away from (resp., toward) our eyes.}   
\end{eg}

\begin{eg}
We compute the the skein lift of a positive kink:
\begin{table}[H]
    \centering
    \begin{tabular}{c c c}
        \begin{tikzpicture}[anchorbase, scale=.5]
        \draw [very thick, ->] (0,-0.3) to [out=180,in=-60] (-0.35,0) to [out=120,in=-90] (-0.5,1);
         \draw [white, line width=.15cm] (-0.5,-1) to [out=90,in=-120] (-0.35,0) to [out=60,in=180] (0,0.3);
        \draw [very thick] (-0.5,-1) to [out=90,in=-120] (-0.35,0) to [out=60,in=180] (0,0.3) to [out=0,in=90] (0.3,0) to [out=-90,in=0] (0,-0.3);
        \node at (-1.8, 1) {\tiny $\otimes 1$};
        \node at (-1.8, 0.6) {\tiny $\odot 2$};
        \end{tikzpicture} & 
        $\mapsto$ & 
        $a_{2}^{-1}$ \begin{tikzpicture}[anchorbase, scale=.5]
        \draw [very thick, ->] (0,-0.3) to [out=180,in=-60] (-0.35,0) to [out=120,in=-90] (-0.5,1);
         \draw [white, line width=.15cm] (-0.5,-1) to [out=90,in=-120] (-0.35,0) to [out=60,in=180] (0,0.3);
        \draw [very thick] (-0.5,-1) to [out=90,in=-120] (-0.35,0) to [out=60,in=180] (0,0.3) to [out=0,in=90] (0.3,0) to [out=-90,in=0] (0,-0.3);
        \node at (0.1, 0.8) {$1$};
        \end{tikzpicture} $+$ $a_{1}$
        \begin{tikzpicture}[anchorbase, scale=.5]
        \draw [very thick, ->] (0,-0.3) to [out=180,in=-60] (-0.35,0) to [out=120,in=-90] (-0.5,1);
         \draw [white, line width=.15cm] (-0.5,-1) to [out=90,in=-120] (-0.35,0) to [out=60,in=180] (0,0.3);
        \draw [very thick] (-0.5,-1) to [out=90,in=-120] (-0.35,0) to [out=60,in=180] (0,0.3) to [out=0,in=90] (0.3,0) to [out=-90,in=0] (0,-0.3);
        \node at (0.1, 0.8) {$2$};
        \end{tikzpicture}
         $+$ $a_{1}$ $z$ \begin{tikzpicture}[anchorbase, scale=.5]
        \draw [very thick, ->] (0,0) to (0,2);
        \draw[very thick, <-] (1,1.4) to [out=180,in=90] (0.6,1) to [out=-90,in=-180] (1,0.6)to [out=0,in=-90] (1.4,1)to [out=90,in=0] (1,1.4);
        \node at (0.5, 1.8) {$1$};
        \node at (1.6, 1.4) {$2$};
        \end{tikzpicture}\\
         &  & \\
         &  & $\parallel$ \\
         &  & \\
        $\parallel$ & 
         & 
        $a_{1}a_{2}^{-1}$ \begin{tikzpicture}[anchorbase, scale=.5]
        \draw [very thick, ->] (0,0) to (0,2);
        \node at (0.5, 1.8) {$1$};
        \end{tikzpicture}
        $+ a_{1}a_{2}$ \begin{tikzpicture}[anchorbase, scale=.5]
        \draw [very thick, ->] (0,0) to (0,2);
        \node at (0.5, 1.8) {$2$};
        \end{tikzpicture}
        $+ a_{1}z \frac{a_{2}-a_{2}^{-1}}{z}$ \begin{tikzpicture}[anchorbase, scale=.5]
        \draw [very thick, ->] (0,0) to (0,2);
        \node at (0.5, 1.8) {$1$};
        \end{tikzpicture}\\
         &  & \\
         &  & $\parallel$ \\
         &  & \\
        $a_{1}a_{2}$ \begin{tikzpicture}[anchorbase, scale=.5]
        \draw [very thick, ->] (0,0) to (0,2);
        \end{tikzpicture} &
        $\mapsto$ & 
        $a_{1}a_{2} \left( \right.\;\;$ \begin{tikzpicture}[anchorbase, scale=.5]
        \draw [very thick, ->] (0,0) to (0,2);
        \node at (0.5, 1.8) {$1$};
        \end{tikzpicture}
        $+$
        \begin{tikzpicture}[anchorbase, scale=.5]
        \draw [very thick, ->] (0,0) to (0,2);
        \node at (0.5, 1.8) {$2$};
        \end{tikzpicture}
        $\left. \right)$
    \end{tabular}
\end{table}
\noindent This illustrates that skein lifting respects change of framing. 
\end{eg}

We next consider lifting skein elements $\gamma$ which are of the form: a framed link $K$ with components colored by representations (i.e., Young diagrams) other than the fundamental representation. To use our definition of the skein lifting map $F$ above, we would first express $\gamma$ as a linear combination of links in a tubular neighborhood of $K$ all colored by the fundamental representation and then apply $F$ from \eqref{def : F(K)}.
This is a rather complicated procedure, and one could ask for some way to carry out the calculation of the skein trace map acting on $\gamma$ that bypasses going to the fundamental representation. 
In the case of trivial double cover and when $K$ has no crossings (i.e., when $K$ is an unlink), there is such a way: 
\begin{thm}\label{thm:coproduct}
Under the skein trace map applied to the trivial double cover of the thickened annulus $S^1\times[0,1]\times\R$, with turning numbers corresponding to the product metric on $S^1\times[0,1]$ where the turning number of the central circle equals zero,   
\begin{equation}\label{eq:coproduct}
W_\lambda \mapsto \sum_{\mu, \nu}c_{\mu\nu}^{\lambda} W_\mu \otimes W_\nu,
\end{equation}
where $c_{\mu\nu}^{\lambda}$ denotes the Littlewood-Richardson coefficients. (Note that this 
is just the standard coproduct of the self-dual Hopf algebra of symmetric functions.)
\end{thm}
\begin{proof}
Gluing the inner boundary of an annulus with the outer boundary of another gives the usual product on $\SkAlg_{a,z}^+(\mathrm{Ann}) \cong \Lambda$, where $\Lambda$ is the algebra of symmetric functions. 
Since the leaves of the foliation are along the $I$-direction of $\mathrm{Ann} \times I$, it follows that the skein trace map 
\[
\Lambda \cong \SkAlg_{a_1 a_2,z}^+(\mathrm{Ann}) \rightarrow \SkAlg_{a_1,z}^+(\mathrm{Ann}) \otimes \SkAlg_{a_2,z}^+(\mathrm{Ann}) \cong \Lambda \otimes \Lambda
\]
is an algebra map, and hence $\Lambda$ is a bialgebra. 
We need to show that this coproduct is the standard coproduct \eqref{eq:coproduct}. 
Since the power sum elements $p_n$ generate $\Lambda$, it suffices to show that the corresponding elements in $\SkAlg^+(\mathrm{Ann})$, 
\[
P_n := \frac{1}{[n]}\qty(A_{0, n-1} + A_{1, n-2} + \cdots + A_{n-1, 0}),
\]
where $A_{i,j}$ is the closure of the braid $\sigma_1 \cdots \sigma_i \sigma_{i+1}^{-1} \cdots \sigma_{i+j}^{-1}$ read from bottom to top \cite{Aiston-powersum, Morton_powersum} and $[n]=(q^{n/2}-q^{-n/2})/(q^{1/2}-q^{-1/2})$ is the quantum integer $n$, 
are primitive with respect to our coproduct, i.e., 
\[
P_n \mapsto P_n \otimes 1 + 1 \otimes P_n. 
\]
In other words, we need to show that all the non-trivial lifts of $P_n$ cancel out so that only the two direct lifts survive. 

By direct calculation, observe that 
\begin{align*}
A_{i,j} &\mapsto A_{i,j} \otimes 1 + 1 \otimes A_{i,j} \\
&\quad +z \sum_{0\leq k\leq i-1} A_{k,0} \otimes A_{i-1-k,j}  \;-\; z \sum_{0\leq l \leq j-1} A_{0, l} \otimes A_{i,j-1-l}  \\
&\quad - z^2 \sum_{\substack{0\leq k\leq i-1 \\ 0\leq l\leq j-1}} A_{k,0}A_{0,l}\otimes A_{i-1-k, j-1-l},
\end{align*}
so that
\begin{align}
P_n &\mapsto P_n \otimes 1 + 1 \otimes P_n \nonumber\\
&\quad +\frac{z}{[n]} \left( \sum_{\substack{i_1, i_2, j_2 \geq 0 \\ i_1 + i_2 + j_2 = n-2}} A_{i_1, 0} \otimes A_{i_2, j_2} - \sum_{\substack{i_2, j_1, j_2 \geq 0 \\ i_2 + j_1 + j_2 = n-2}} A_{0, j_1} \otimes A_{i_2, j_2} \right. \label{eqn:P_n-calculation}\\
&\quad\hspace{10em} \left.- z \sum_{\substack{i_1, i_2, j_1, j_2 \geq 0 \\ i_1 + i_2 + j_1 + j_2 = n-3}} A_{i_1, 0} A_{0, j_1} \otimes A_{i_2, j_2} \right). \nonumber
\end{align}
Since
\begin{align*}
z \sum_{\substack{i_1, i_2, j_1, j_2 \geq 0 \\ i_1 + i_2 + j_1 + j_2 = n-3}} A_{i_1, 0} A_{0, j_1} \otimes A_{i_2, j_2} &= \sum_{\substack{i_1, i_2, j_1, j_2 \geq 0 \\ i_1 + i_2 + j_1 + j_2 = n-3}} \qty(A_{i_1+1, j_1} - A_{i_1, j_1 + 1}) \otimes A_{i_2, j_2} \\
&= \sum_{\substack{k, i_2, j_2 \geq 0 \\ k + i_2 + j_2 = n-2}} \qty(A_{k, 0} - A_{0, k}) \otimes A_{i_2, j_2},
\end{align*}
all the terms inside the brackets in \eqref{eqn:P_n-calculation} cancel out, and therefore 
\[
P_n \mapsto P_n \otimes 1 + 1 \otimes P_n.
\]
\end{proof}

We next take the knot $K$ as the framed unknot in $\R^3$ represented by a circle in a plane framed by the outward planar normal. Identify the thickened annulus in Theorem \ref{thm:coproduct} with a neighborhood of the unknot. Since the turning number of the unknot in the plane equals $\pm 1$ (depending on the orientation) and the turning number of the central circle in Theorem \ref{thm:coproduct} equals $0$, we must substitute $W_\mu \otimes W_\nu \mapsto a_2^{|\mu|}W_\mu \otimes a_1^{-|\nu|}W_\nu$ in the right hand side of \eqref{eq:coproduct} when the links are included in the plane. 
We get: 
\begin{cor}
Consider the unknot colored by $\lambda$. 
Then, under the skein trace map 
$\mathrm{Sk}_{a_1a_2, z}(\mathbb{R}^3) \rightarrow \mathrm{Sk}_{a_1, z}(\mathbb{R}^3) \otimes \mathrm{Sk}_{a_2, z}(\mathbb{R}^3)$,
\[
\vcenter{\hbox{
\begin{tikzpicture}[scale=.5]
    \draw[very thick, ->] (1,2) to [out=180,in=90] (0,1) to [out=-90,in=-180] (1,0)to [out=0,in=-90] (2,1)to [out=90,in=0] (1,2);
    \node at (-1.2, 1.8) {\tiny $\otimes 1$};
    \node at (-1.2, 1.4) {\tiny $\odot 2$};
    \node[anchor=north] at (1, 0){$\lambda$};
\end{tikzpicture}
}}
\;\;\mapsto\;\;
\sum_{\mu, \nu} c_{\mu \nu}^{\lambda}
\;
a_2^{|\mu|}
\vcenter{\hbox{
\begin{tikzpicture}[scale=.5]
    \draw[very thick, ->] (1,2) to [out=180,in=90] (0,1) to [out=-90,in=-180] (1,0)to [out=0,in=-90] (2,1)to [out=90,in=0] (1,2);
    \node[anchor=north] at (1, 0){$\mu$};
    \node at (2, 2) {$1$};
\end{tikzpicture}
}}
\otimes
a_1^{-|\nu|}
\vcenter{\hbox{
\begin{tikzpicture}[scale=.5]
    \draw[very thick, ->] (1,2) to [out=180,in=90] (0,1) to [out=-90,in=-180] (1,0)to [out=0,in=-90] (2,1)to [out=90,in=0] (1,2);
    \node[anchor=north] at (1, 0){$\nu$};
    \node at (2, 2) {$2$};
\end{tikzpicture}
}}
\;,
\]
and taking HOMFLYPT polynomials gives
\begin{align*}
\prod_{\square \in \lambda} &\frac{a_1a_2 q^{c(\square)/2}  - a_1^{-1}a_2^{-1}q^{-c(\square)/2}}{q^{h(\square)/2} - q^{-h(\square)/2}}\\
& \qquad = \sum_{\mu,\nu}c_{\mu,\nu}^{\lambda}a_2^{|\mu|} a_1^{-|\nu|} 
\prod_{\square \in \mu}\frac{a_1 q^{c(\square)/2} - a_1^{-1} q^{-c(\square)/2}}{q^{h(\square)/2} - q^{-h(\square)/2}} 
\prod_{\square \in \nu}\frac{a_2 q^{c(\square)/2} - a_2^{-1}q^{-c(\square)/2}}{q^{h(\square)/2} - q^{-h(\square)/2}}. \qed
\end{align*}
\end{cor}

\section{Conormals} \label{sec: conormals}

The basic construction of Definition \ref{main definition} requires a non-exact push-off of a link conormal.  Our proof of Theorem \ref{skein trace} will involve considering a family of such conormals along the crossing change for the link.  Here we give explicit constructions of such push-offs and discuss their basic properties. 

\begin{figure}
    \centering
    \includegraphics[width=0.55\linewidth]{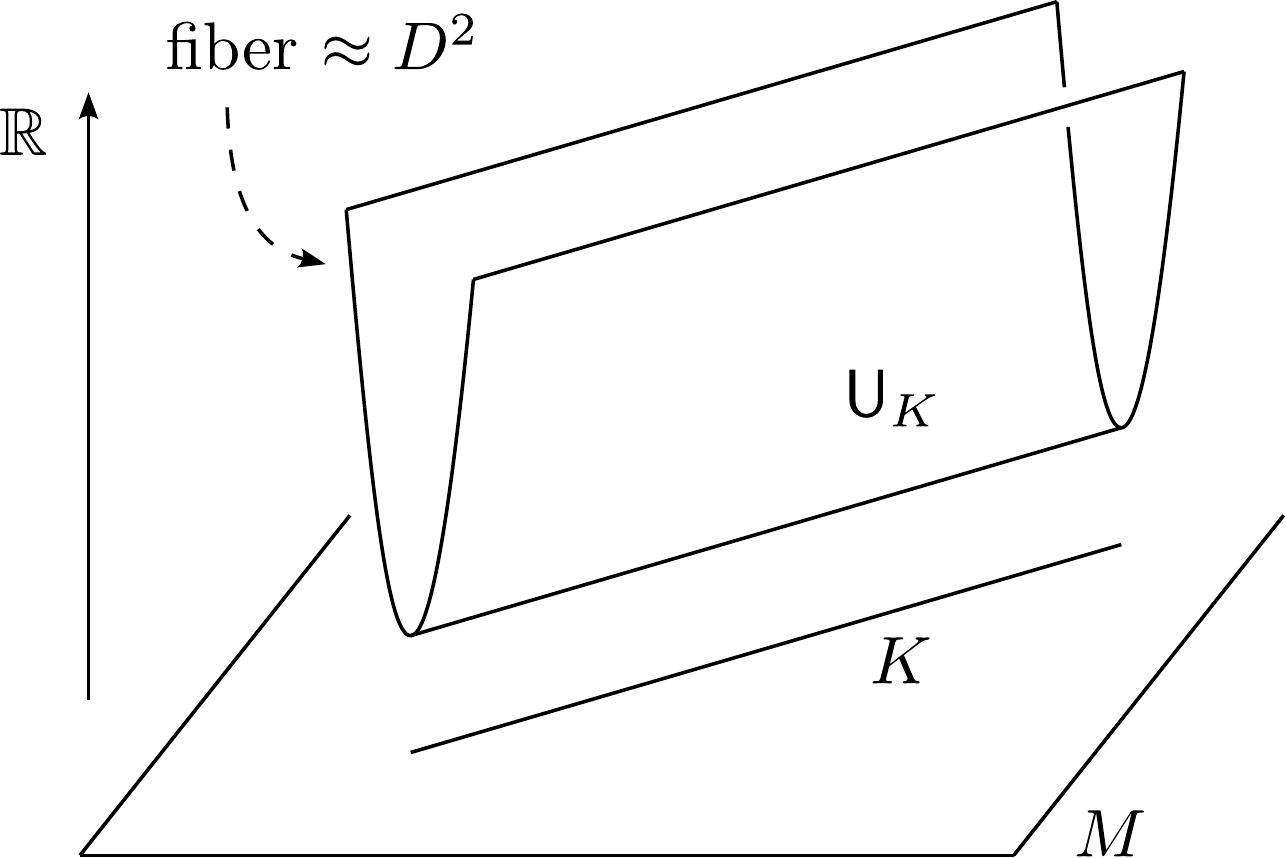}
    \caption{The sharp front in $J^0(M)\approx M\times \R$ of the conormal of a knot.}
    \label{fig:frontofconormal}
\end{figure}
We first give a front description of the conormal of the central curve $I\times 0 \subset I\times D^2$, where $I$ is a 1-manifold. 

Let $I$ denote either an interval $I\approx [0,1]$ or a circle $I\approx S^1$. 
Let $D^2$ denote the unit disk in $\R^2$, $D^2=\{v=(v_1,v_2)\colon v_1^2+v_2^2<1\}$. 
Consider the product $I \times D^2$. 
Take coordinates $(t,v)\in I\times D^2$. Let $c>0$ and consider a proper function $h\colon [0,1)\to \R$ with the following properties
\begin{itemize}
\item $h(r)=c r^2$ for $0\le r\le \frac12$,
\item $r=0$ is the only critical point of the convex function $h$, and 
\item $h(r)\to\infty$ as $|r|\to 1$. 
\end{itemize}
Let $f\colon I\times D^2 \rightarrow \mathbb{R}$ be given by 
\begin{equation}\label{eq : def function f}
f(t,v)=f(v_1,v_2)= h\left(\sqrt{v_1^2+v_2^2}\right)
\end{equation}
and consider a closed $1$-form $\tau$ along $I\times D^2$ with everywhere non-zero $dt$-component. 
Then the graph $\Gamma\subset T^\ast (I\times D^2)$ of $df + \tau$ is a graphical Lagrangian push-off of the conormal to $I \times 0 \subset T^\ast (I\times D^2)$. 
We denote it $\mathsf{U}_\tau$. 

Given any embedding $K\colon S^1 \to M$ with an (embedded) tubular neighborhood $N\colon D^2 \times S^1 \to M$, $N(0\times S^1)=K$, we may transport the above perturbed shifted conormal to $T^*M$, giving a graphical Lagrangian push-off of the conormal of $K$. 
We denote it by $\mathsf{U}_{N,\tau}$. 
When the choice of $N$ is implicit or irrelevant and $\tau=dt$, we sometimes write $\mathsf{U}_K$ instead. 
The construction is local along $K$, so makes sense when $K$ is only immersed, or when (even if $K$ is embedded) its tubular neighborhood is only immersed; though in these cases the resulting conormal push-off need not be graphical or even embedded.

We will now fix specific choices for tubular neighborhoods and shifting 1-forms for a model family of links participating in the first skein relation, and make some geometric observations regarding the resulting conormals.  

Consider coordinate cubes 
\[
B_\rho=\{(x_1,x_2,x_3)\colon |x_j|\le \rho; j=1,2,3\},
\] 
and a 1-parameter family of immersed intervals, $k(\sigma)$, $-\epsilon<\sigma<\epsilon$, in $B_4$:  
\begin{equation}\label{eq: def K(t) with phi}
k(\sigma) \ := \ k_1(\sigma) \ \cup \ k_2(\sigma)   \ :=  \{x_2=x_3=0\} \ \cup \ \{x_1=0,x_3=\sigma\phi(x_2)\},
\end{equation}
where $\phi\colon [-4,4]\to[0,1]$ is a smooth cut-off function equal to $0$ outside $[-\frac34,\frac34]$ and equal to $1$ in $[-\frac12,\frac12]$. 
We think of the intervals as oriented in the positive $x_1$- and $x_2$-directions, respectively.
Then the sign of $\sigma$ determines the sign of the crossing, see Figure \ref{fig:three pieces} (leftmost and middle links), and when $\sigma=0$, $k(0)$ is immersed with a double point at the origin and when $\sigma\ne 0$, $k(\sigma)$ is embedded. 
Note that $k(\sigma)\cap \partial B_1$ is the four points
\[
p_{1,\pm}= (\pm 1,0,0),\qquad p_{2,\pm} = (0,\pm 1,0),
\]
for all $\sigma$, where the subscript sign is $-$ where $k(\sigma)$ enters the ball and $+$ where it exits. 

Fix tubular neighborhoods of the components $k_1(\sigma), k_2(\sigma)$.
Let $\delta>0$, $0<\eta\ll\delta$, and let $\phi\colon [-1,1]\to [0,1]$ be as above. 
We choose closed $1$-forms in these tubular neighborhoods:
\begin{align*}
    \tau_1 = (\delta+\eta x_2 \phi'(x_1)) dx_1 + \eta\phi(x_1) dx_2,\\
    \tau_2 = (\delta-\eta x_1 \phi'(x_2)) dx_2 - \eta\phi(x_2) dx_1.
\end{align*}
Let $\mathsf{U}_{k(\sigma)}$ denote the union of the corresponding two graphical Lagrangian conormals shifted by $\tau_1$ and $\tau_2$, 
\begin{equation}\label{eq : def U_k(sigma)}
\mathsf{U}_{k(\sigma)}=\mathsf{U}_{k_1(\sigma),\tau_1}\cup \mathsf{U}_{k_2(\sigma),\tau_2}.
\end{equation}
We use coordinates $(x,y)$ on $T^\ast\R^3$, $x=(x_1,x_2,x_3)$ and $y=(y_1,y_2,y_3)$.

\begin{lemma}\label{l: intersection L_1 and L_2}
In the above situation, $\mathsf{U}_{k(\sigma)}$, $\sigma\ne 0$ is embedded and $\mathsf{U}_{k(0)}$ has a clean self-intersection along the line ($\approx\R$):
$$
(x_1,x_2,x_3,y_1,y_2,y_3)=\left(\xi,\xi,x_3,\delta,\delta,\frac{\partial f}{\partial v_2}(\xi,x_3)\right),\quad |x_3|<\sqrt{1-\xi^2},
$$ 
where $\xi=\frac{\delta-\eta}{c}$ and $f$ is as in \eqref{eq : def function f}.
\end{lemma}
\begin{proof}
 The Lagrangians $\mathsf{U}_{k_j(\sigma),\tau_j}$ are graphs of $1$-forms $\lambda_j$, $j=1,2$, where $\lambda_1$ and $\lambda_2$ are supported in
\[
\{(x,y)\colon |x_1|\le 1, |(x_2,x_3)|<r\} \quad\text{ and }\quad
\{(x,y)\colon |x_2|\le 1, |(x_1,x_3-\sigma\phi(x_2))|<r\},
\]
respectively. 
By definition $\lambda_j=df_j + \tau_j$, and for sufficiently small $|x_1|$ and $|x_2|$ we have
\begin{align}\label{eq: forms for immersed}
&\tau_1 = \delta dx_1+\eta dx_2,\quad \tau_2=\delta dx_2-\eta dx_1,\\\notag 
&df_1 = \left(\frac{\partial f}{\partial v_1}(x_2,x_3)dx_2 + \frac{\partial f}{\partial v_2}(x_2,x_3)dx_3\right)
=cx_2dx_2 + \frac{\partial f}{\partial v_2}(x_2,x_3)dx_3,\\\notag
&df_2 = \left(\frac{\partial f}{\partial v_1}(x_2,x_3-\sigma)dx_2 + \frac{\partial f}{\partial v_2}(x_1,x_3-\sigma)dx_3\right)
=cx_1dx_1 + \frac{\partial f}{\partial v_2}(x_1,x_3-\sigma)dx_3, 
\end{align}
see \eqref{eq : def function f}, for small $\delta>0$, $0<\eta\ll\delta$, and large $c>0$. 

Intersections correspond to zeros of the difference of the defining forms. Now, $\lambda_1-\lambda_2=0$ only if
\[
((\delta-\eta)- c x_1)dx_1 + (c x_2-(\delta-\eta))dx_2=0, 
\]    
which holds exactly when $x_1=x_2=\frac{\delta-\eta}{c}=\xi$ and then $y_1=y_2=\delta$. 
It remains to consider the $(x_3,y_3)$-coordinates where the Lagrangians are given by
\begin{align*}
\mathsf{U}_{k_1(\sigma),\tau_1}&\colon \ (x_3,y_3) = \left(x_3,\frac{\partial f}{\partial v_2}(\xi,x_3)\right),\quad |x_3|<\sqrt{1-\xi^2},\\
\mathsf{U}_{k_2(\sigma),\tau_2}&\colon \ (x_3,y_3) = \left(x_3,\frac{\partial f}{\partial v_2}(\xi,x_3-\sigma)\right),\quad |x_3-\sigma|<\sqrt{1-\xi^2}.    
\end{align*}
By convexity of $f$, the partial derivative $\frac{\partial f}{\partial v_2}(\xi,\cdot)$ is a strictly increasing function. 
Consequently, the curves are distinct for $\sigma\ne 0$ and agree for $\sigma=0$. 
The lemma follows. 
\end{proof}

The curves $k(\sigma)$ and the corresponding Lagrangians $\mathsf{U}_{k(\sigma)}$ give the local model for links in a family with a crossing.  

The link corresponding to the spliced crossing meets our box in an embedded curve $k_\asymp\subset B_1\cap\{x_3=0\}$ connecting $p_{1,-}$ to $p_{2,+}$ and $p_{2,-}$ to $p_{1,+}$, as in Figure \ref{fig:three pieces}. 
Let $N$ be an embedded tubular neighborhood of $k_\asymp$, and $dt$ a pullback to $N$ of a 1-form on $k_\asymp$ dual to the tangent vector.  Then we take the (embedded, graphical) conormal 
\begin{equation}\label{eq : def U_asymp}
\mathsf{U}_{k_\asymp}:=\mathsf{U}_{N, dt}. 
\end{equation}

\begin{figure}
    \centering
    \includegraphics[width=0.75\linewidth]{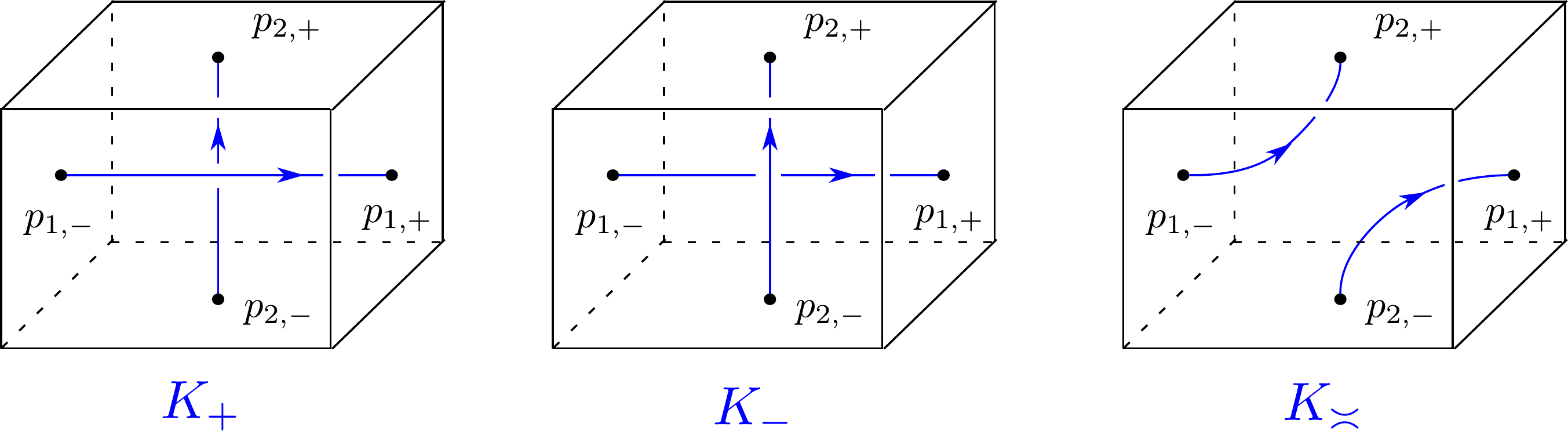}
    \caption{Three terms of the skein relation}
    \label{fig:three pieces}
\end{figure}

\begin{figure}
    \centering
    \includegraphics[width=0.65\linewidth]{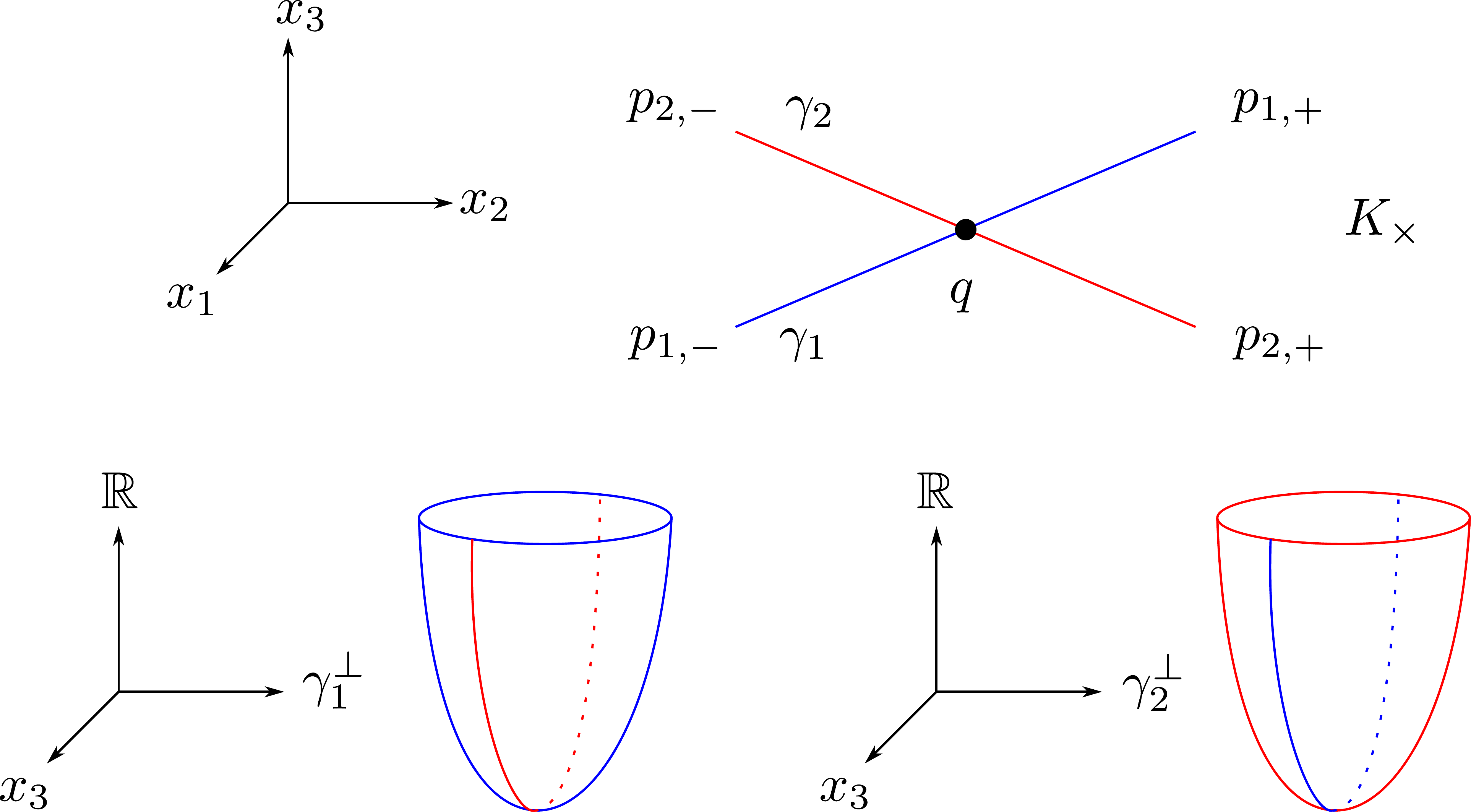}
    \caption{The link $k(0)$ and the front conormal fibers at the intersection point, with the clean intersection line. }
    \label{fig:intersction}
\end{figure}

It will later be necessary to localize the difference between the model knots close to the self-intersection of the original link. We achieve this by scaling the above models as follows.

We first consider the links $k(\sigma)\subset B_4$, $-\epsilon<\sigma<\epsilon$ and $k_\asymp\subset B_4$. For $0<\rho\le 1$ consider the scaling map
\[
s_\rho\colon B_4\to B_{4\rho},\quad s_\rho(x_1,x_2,x_3)=(\rho x_1,\rho x_2,\rho x_3).
\]
Let $k(\sigma;\rho)$ and $k_\asymp(\rho)$ denote the links in $B_4$ that consists of $s_\rho(k(\sigma))$ and $s_\rho(k_\asymp)$, respectively, with the segments of the $x_1$ and $x_2$ coordinate axes in $B_4\setminus B_{4\rho}$ added.

We next consider the Lagrangians. Consider a tubular neighborhood $N_\rho$ of $k_\asymp(\rho)$ that agrees with $N$ near $\partial B_4$ and agrees with $s_\rho(N)$ in $B_{4\rho}$. 
Note that the symplectomorphism $S_\rho\colon T^*B_1\to T^*B_\rho$, 
\[
S_\rho(x,y)=(\rho\cdot x, \rho^{-1}\cdot y)
\]
takes $\mathsf{U}_{k(\sigma)}\cap T^\ast (B_4\setminus B_{4(1-\epsilon)})$ to $\mathsf{U}_{s_\rho(k(\sigma))}\cap T^\ast (B_{4\rho}\setminus B_{4\rho(1-\epsilon)})$ and similarly for $\mathsf{U}_{k_\asymp}$ and $\mathsf{U}_{s_\rho(k_\asymp)}$. We take 
\begin{align}\label{eq : def U cross}
\mathsf{U}_{k(\sigma;\rho)} \ &= \ \left(\mathsf{U}_{N_\rho}\cap T^\ast({B_4\setminus B_{4\rho}})\right) \cup S_\rho(\mathsf{U}_{k(\sigma)}),\\\label{eq : def U splice}
\mathsf{U}_{k_\asymp(\rho)} \ &= \ \left(\mathsf{U}_{N_\rho}\cap T^\ast({B_4\setminus B_{4\rho}})\right) \cup S_\rho(\mathsf{U}_{k_\asymp}).
\end{align}

Finally, consider an immersed link $K_\times\subset M$ with a single double point where the tangent vectors of the branches are linearly independent.  We will choose conormals of the corresponding three links participating in the first skein relation by choosing models which agree with the above standard constructions inside some box \begin{equation} \label{coordinate box} 
b\colon B_4 = [-4,4]^3 \to M \qquad \qquad b^{-1}(K_\times)=k(0).
\end{equation}
Let $K_{\mathrm{out}}=K_\times\setminus b(B_4)$. Define links
\begin{align*}
    K(\sigma) &= K_{\mathrm{out}}\cup b(k(\sigma)),\quad -\epsilon<\sigma<\epsilon,\\
    K_\asymp &= K_{\mathrm{out}}\cup b(k_\asymp).
\end{align*}
Fix a tubular neighborhood $N_{\mathrm{out}}$ of $K_\times\cap (M\setminus b(B_{3}))$ that agrees with the tubular neighborhood $b(N)$ of $b(k(0))$ in $b(B_4\setminus B_3)$. Define 
\begin{align}\label{eq : conormal cross parameters}
\mathsf{U}(\sigma;\rho) &= \mathsf{U}_{N_{\mathrm{out}}}\cup b_\ast(\mathsf{U}_{k(\sigma;\rho)}),\quad -\epsilon<\sigma<\epsilon,\; 0<\rho\le 1,\\\label{eq : conormal splice parameters}
\mathsf{U}_\asymp(\rho) &= \mathsf{U}_{N_{\mathrm{out}}}\cup b_\ast(\mathsf{U}_{k_\asymp(\rho)}),\quad  0<\rho\le 1.
\end{align}

Then $\mathsf{U}(\sigma;\rho)$ are Lagrangian conormals of $K(\sigma)$ and $\mathsf{U}_\asymp(\rho)$ are Lagrangian conormals of $K_\asymp(\rho)$.

If $K_\times$ is connected, then $\mathsf{U}(\sigma,\rho)$, $\sigma\ne 0$, and $\mathsf{U}_\asymp(\sigma)$ have either one or two connected components and that $\mathsf{U}(\sigma,\rho)$ has one component exactly when $\mathsf{U}_\asymp(\rho)$ has two and vice versa. We call the homology classes in $\mathsf{U}(\sigma;\rho)$ and $\mathsf{U}_\asymp{\rho}$ represented by the natural lifts of the oriented links $K(\sigma;\rho)$ and $K_\asymp(\rho)$, respectively the \emph{basic homology class}.

\section{Flow graphs and the first skein relation}\label{sec : Lag conormals and flow graphs}

The purpose of this section is to describe the behavior of flow graph counts associated to curves stretching between shifted knot conormals and the zero section, as the knot undergoes a crossing change.  That is, we study the flow graphs for the family $\mathsf{U}(\sigma;\rho)$
for $|\rho|$ small.  

\subsection{Flow graph review}

Consider the cotangent bundle $T^\ast M$ of a manifold $M$ and let $S\subset T^\ast M$ be a Lagrangian. We take $S$ to be generic with respect to the projection $\pi\colon T^\ast M\to M$, meaning that $\pi|_S$ is an immersion outside a codimension one submanifold and outside the image of this locus $S$ is locally given by the graph of a finite number of $1$-forms, $\Gamma_{\beta_1}\cup\dots\cup \Gamma_{\beta_m}$.  (The numbering of the 1-forms can only be locally defined.)  Having fixed a metric on $M$, the duals of the differences $(\beta_i-\beta_j)^\ast$ give local vector fields along $M$. Consider a flow line of $(\beta_i-\beta_j)^\ast$ and its cotangent lift, the flow line lifted to the graphs of the $1$-forms that defines it. A flow graph is a union of flow lines such that the closure of a corresponding union of cotangent lifts of its flow lines gives a closed oriented curve in $S\subset T^\ast M$, see \cite[Definition 2.10]{Ekholm-morse} for details. 

Given a flow graph we associate a continuous map of a surface $\Sigma$ into $S\subset T^\ast M$ with boundary in $S$ as follows. Consider its \emph{flow strips}, i.e., the union of the lines in the fibers of $T^\ast M$ over the flow lines of the flow graph. The union of all its flow strips is a surface with two types of boundary, the cotangent lift of the flow graph and residual boundary components in the cotangent fiber over vertices of valency $3$ or higher (flow graphs in general position have no vertices of valency $> 3$). The residual boundary components are closed polygonal curves in a cotangent fiber, we fill them by polygonal disk. The filled surface is the surface $\Sigma$ associated to the flow graph. Note that $\Sigma$ is naturally oriented by the complex orientation of its flow strips.  We will refer to $\Sigma$ as the {\em flow surface}.  

The following is a trivial but important example of a flow graph. 
Let $K\subset M$ be a knot, $N\colon S^1\times D^2 \to M$ a fixed tubular neighborhood, and $\mathsf{U}_N\subset T^\ast M$ the corresponding shifted conormal.  

\begin{lemma}\label{l : basic flow loop}
There is a unique embedded flow graph for $(T^*M, \mathsf{U}_N \sqcup M)$. 
It is a loop, and the boundary of its flow surface is $S^1\times 0$ in $\mathsf{U}_N$ and $K$ in $M$. 
\end{lemma}

\begin{proof}
By definition of $\mathsf{U}_N$, the difference vector field in $(S^1\times D^2)$-coordinates is $\partial_t - \nabla f$ and $f\colon D^2\to [0,\infty)$ has only one critical point, a non-degenerate minimum at the origin.
\end{proof}

Flow graphs are related to holomorphic curves \cite{FloerMorseWitten,FukayaOhZeroloop,HutchingsLee}. 
This relation was studied in \cite{Ekholm-morse} for Legendrians in 1-jet spaces that project to immersed exact Lagrangians in the corresponding cotangent bundles. 
Under fiber scaling, holomorphic curves ending on these Lagrangians (Gromov-) limit to holomorphic maps into the zero-section (necessarily constant), but may be recovered by rescaling leading to a one-to-one correspondence between rigid holomorphic disks and rigid flow trees. 
A more general situation was considered in \cite{EENS} where the limit Lagrangian does carry non-constant holomorphic curves; in this case, the one-to-one correspondence relates the curves in the original geometry with holomorphic curves on the limit Lagrangian with flow trees attached. 
Here we consider somewhat more general situations but the results we need can be extracted from the analysis in the papers mentioned. 
We review and extract the relevant results relating holomorphic curves and flow graphs in Appendix \ref{flow graphs and holomorphic curves}.

\subsection{Description of the result} \label{ssec: description}

Let us outline the result of the calculations in the remainder of this section.  We study the family $\mathsf{U}(\sigma;\rho)$.  For $\sigma \gg 0$, these conormals are graphical, hence their flow objects are described entirely by Lemma \ref{l : basic flow loop}.  However, this is not the case in a neighborhood of $\sigma=0$. 

For these conormals, we will call a flow graph {\em basic} if the boundary of the flow surface traces out the basic class on the conormal. 

We will further perturb the conormal to the singular knot to a $\widetilde{\mathsf{U}}(0;\rho)$, which we will show carries exactly two basic flow graphs.  One is a flow loop tracing out a knot isotopic to $K_-$, and this flow loop persists for similarly perturbed $\widetilde{\mathsf{U}}(\sigma;\rho)$ with $|\sigma|$ sufficiently small (Lemma \ref{l: flow loops near cross}). 

The other is a flow graph `with corners' corresponding to a holomorphic curve with two corners along the self-intersection of $\widetilde{\mathsf{U}}(\sigma;0)$ (Lemma \ref{l: flow loops = flows with corners}).   This graph `glues up' (as would the curve) to a flow graph on  
$\widetilde{\mathsf{U}}(\sigma;\rho)$ for $0 < \sigma < \epsilon$.  The boundary of the holomorphic curve corresponding to this flow graph traces out $K_{\asymp}$ along the zero section.  

This behavior illustrates the well-definedness (and need for) skein-valued curve counting.  Indeed, we should have $\psi(\widetilde{\mathsf{U}}(\sigma;\rho))$ independent of $\sigma > 0$.  For $\sigma \gg 0$, the conormal is graphical, and so the only flow graph is the loop tracing out $K_-$ itself.  For $\sigma < \epsilon$, we have flow graphs tracing out $K_+$ and $K_{\asymp}$; we will see later (confer the proof of Theorem \ref{explicit description} for the sign effect of attaching a flow line) that both come with positive sign.  The holomorphic curve corresponding to the nontrivial graph has Euler characteristic $-1$.  Then with $\sigma_+\gg 0$ and $0<\sigma<\epsilon$, invariance demands: 
\begin{equation}[K_+]  = \psi(\widetilde{\mathsf{U}}(\sigma_+;\rho)) = \psi(\widetilde{\mathsf{U}}(\sigma;\rho)) = [K_-] + z[K_\asymp].
\end{equation}
This is precisely the first skein relation.   

More relevant to work in the present article is a {\em different} appearance of the first skein relation.  The above description of flow graphs (and the correspondence between flow graphs and holomorphic curves) amounts to the assertion, for $0 < \sigma < \epsilon$, 
\begin{eqnarray*}
    \psi(\widetilde{\mathsf{U}}(\sigma;\rho)) & = & [K_-] + z [K_\asymp] \\
    \psi(\widetilde{\mathsf{U}}(-\sigma;\rho)) & = & [K_-].
\end{eqnarray*}
We could have deduced these formulas by comparing to the corresponding $\psi$ for $|\sigma|$ large, but our analysis here has shown something sharper, which will be crucial for our later work: the terms in the above formulas are a direct description of the holomorphic curves which appear. 
But in any case, the promised appearance of the skein relation, following from the above, is: 
$$\psi(\widetilde{\mathsf{U}}(\sigma;\rho)) - \psi(\widetilde{\mathsf{U}}(-\sigma;\rho)) = [K_\asymp] = \psi(\mathsf{U}_\asymp).$$

\subsection{A choice of perturbation}

Recall we have fixed in \eqref{coordinate box} a choice of coordinates $B_4 = [-4,4]^3 \to M$; all the knots of interest agree in the complement of the image of  $B_1 = [-1,1]^3$.   We keep our notation for rescalings: $s_\rho(x_1, x_2, x_3) = (\rho x_1, \rho x_2, \rho x_3)$, and $S_\rho := T^* s_\rho$.  

Let
\begin{equation}\label{eq : U_out}
\mathsf{U}^{\mathrm{out}}(\rho)= \mathsf{U}(\sigma;\rho)\cap T^\ast (M\setminus B_\rho) = \mathsf{U}_\asymp(\rho)\cap T^\ast (M\setminus B_\rho).
\end{equation}
Note that the first intersection is independent of $\sigma$. Also, let
\begin{align}
&\mathsf{U}^{\mathrm{in}}(\sigma,\rho) = \mathsf{U}(\sigma;\rho)\cap S_\rho(T^\ast B_1),\\ 
&\mathsf{U}_\asymp^{\mathrm{in}}(\rho) = \mathsf{U}_\asymp(\rho)\cap S_\rho(T^\ast B_1).
\end{align}
Then $S_\rho$ identifies $\mathsf{U}^{\mathrm{in}}(\sigma,\rho)$ and $\mathsf{U}_\asymp^{\mathrm{in}}(\rho)$ with $\mathsf{U}_{k(\sigma)}\subset T^\ast B_1$, see \eqref{eq : def U_k(sigma)} and $\mathsf{U}_{k_\asymp}\subset T^\ast B_1$, see \eqref{eq : def U_asymp}, respectively. Similarly, $S_\rho$ identifies $\mathsf{U}^{\mathrm{out}}(\rho)\cap T^\ast B_{3\rho}\setminus B_{2\rho}$ with $\mathsf{U}_{k(\rho)}\cap T^\ast (B_{3}\setminus B_{2})=\mathsf{U}_{k_\asymp}\cap T^\ast (B_{3}\setminus B_{2})$.

Recall the intersection points
\[
p_{1,\pm}= (\pm 1,0,0),\qquad p_{2,\pm} = (0,\pm 1,0).
\]
of $k(\sigma)$ and $k_\asymp$ with $\partial B_1$. Identify $\partial B_1$ with $\partial B_\rho \subset M$. 

The Lagrangian $\mathsf{U}^{\mathrm{out}}(\rho)$ is the graph of a closed $1$-form in a tubular neighborhood of $K$ in $M\setminus B_\rho$. 
Let $V^{\mathrm{out}}(\rho)$ denote the vector field dual to this 1-form.  Then the flow $\Phi^{\mathrm{out}}(\rho)$ of $V^{\mathrm{out}}(\rho)$ has flow lines that start at $p_{1,+}$ and $p_{2,+}$ and end at $p_{1,-}$ and $p_{2,-}$. The flow correspondingly takes a neighborhood $P_+\subset\partial B_1$ of $\{p_{1,+},p_{2,+}\}$ to a neighborhood $P_-\subset \partial B_1$ of $\{p_{1,-},p_{2,-}\}$. 

We write
$P_\pm(\rho)=P_{1,\pm}(\rho)\cup P_{2,\pm}(\rho)$ for the components of these neighborhoods around the corresponding $\{p_{1,\pm},p_{2,\pm}\}$. 
Note that when $K_\asymp$ is connected, $\Phi^{\mathrm{out}}(\rho)$ takes $P_{1,+}(\rho)$ to $P_{1,-}(\rho)$ and $P_{2,+}(\rho)$ to $P_{2,-}(\rho)$; when $K_\asymp$ has  two components, $\Phi^{\mathrm{out}}(\rho)$ takes $P_{1,+}(\rho)$ to $P_{2,-}(\rho)$ and $P_{2,+}(\rho)$ to $P_{1,-}(\rho)$.

The vector field $V^{\mathrm{out}}(\rho)$ has the form
$$
V^{\mathrm{out}}= \partial_t -\nabla f,
$$
in coordinates $(t,v)\in I\times D^2$ on a tubular neighborhood, where $f$ is as in \eqref{eq : def function f}. This implies that its flow contracts the fibers. More precisely, if we pick coordinates $\xi\in D^2$ on $P_{1,\pm}$ and $P_{2,\pm}$, respectively, centered at the flow line where $\nabla f=0$ we find that the flow map acts as multiplication by $e^{-kT_j}$, $k>0$, where $T_j$ is the flow time (the flow time $T_j$ is proportional to $\delta^{-1}$, where $\delta$ is the size of the shift of the conormal),
\begin{equation}\label{eq : flow contractions}
\Phi^{\mathrm{out}}(\xi) = e^{-kT_j}\xi, \quad j=1,2 
\end{equation}
where $\xi$ in the left-hand lies in $P_{j,+}$ and the right-hand side lies in $P_{i,-}$, where $i=j$, if $K_\asymp$ is connected and $i\ne j$ is if it has two components. 

We will consider small perturbations of the flow, corresponding to 
a fixed model perturbation of $\mathsf{U}_{k(\sigma)}|_{T^*(B_3\setminus B_2) }=\mathsf{U}_{k_\asymp}|_{T^*(B_3\setminus B_2) }$, supported in a neighborhood of the negative $x_1$-axis. 
We will denote perturbed objects as e.g., $\widetilde{\mathsf{U}}, \widetilde{V}$.  Note that $\widetilde{V}^{\mathrm{out}}(\rho)$ then differs 
from $V^{\mathrm{out}}(\rho)$ only on $B_{3 \rho} \setminus B_{2\rho}$.  
The exact form of the perturbation will not be important but its effect on the $x_3$-coordinate of the return map will. 

\begin{definition} \label{positive negative perturbation}
    We say that the perturbation is positive (negative) if: 
    \begin{itemize}
        \item  $K_\asymp$ is connected and 
        $x_3(\widetilde{\Phi}^{\mathrm{out}}(0 \in P_{1,+}))$ is positive (negative). 
        \item  $K_\asymp$ is disconnected and 
        $x_3(\widetilde{\Phi}^{\mathrm{out}}(0 \in P_{2,+}))$ is positive (negative). 
    \end{itemize}    
    We will write $\kappa>0$ for the size of the perturbation with respect to the standard metric on $B_4$. 
\end{definition}

\begin{lemma}
    For given perturbation of size $\kappa>0$, if the shift $\delta$ is sufficiently small, then for small enough choice of the original disk $P_{2,+}$ (or $P_{1,+}$ as the case may be), its image under $\widetilde{\Phi}^{\mathrm{out}}(\rho)$ in $P_{1,-}$ lies in the region $x_3>0$ ($x_3<0$) if the perturbation is positive (negative). 
\end{lemma}

\begin{proof}
Note that the flow time $T_j\to\infty$ as $\delta\to 0$. Thus, for sufficiently small shift $\delta$, if $x_3(\widetilde{\Phi}^{\mathrm{out}}(0))=\eta>0$ then $\widetilde{\Phi}^{\mathrm{out}}$ takes the whole disk $P_{j,+}$ to the region $\frac12\eta<x_3<\frac32\eta$, see \eqref{eq : flow contractions}.    
\end{proof}

\subsection{Flow loops}

As the Lagrangian $\mathsf{U}_\asymp(\rho)$ is graphical, Lemma \ref{l : basic flow loop} shows that it supports exactly one flow graph, the basic flow loop(s), in the basic homology class. 
The same is true for the perturbed $\widetilde{\mathsf{U}}_\asymp(\rho)$, so we have: 
\begin{lemma}\label{l : L_=(rho)}
 The flow surface of the basic flow loop(s) of $\widetilde{\mathsf{U}}_\asymp(\rho)$ has boundary which is at $C^1$-distance $\mathcal{O}(\kappa\rho)$ to $S^1\times 0$ in all components of $\widetilde{\mathsf{U}}_\asymp(\rho)$ and of $C^1$-distance $\mathcal{O}(\kappa\rho)$ to $K_\asymp(\rho)$ in $M$. In particular, the boundary component of the flow loop(s) in $M$ converges to $K_\times$ as $\rho\to 0$.   
\end{lemma}
\begin{proof}
Uniqueness of flow loops under small perturbation follows from the transversality of the original flow loop (which we discussed in  Lemma \ref{l : basic flow loop}). To see the convergence, note that $K_{\asymp}(\rho)$ agrees with $K_\times$ in $M\setminus s_\rho(B_1)$ by construction, and that $s_\rho(B_1)$ tends to a point as $\rho\to 0$.
\end{proof}

We will re-describe the determination of the above loop as a fixed point calculation.  Inside $B_1$, the conormal $\mathsf{U}_{k_\asymp}\subset T^\ast B_1$ is the graph of a 1-form; we denote its dual vector field as  $V^{\mathrm{in}}_\asymp$ and the corresponding flow map 
\[\Phi_\asymp^{\mathrm{in}}\colon P_-\to P_+.
\]
It is clear from the definition of $\mathsf{U}_{k_\asymp}$ that the map $\Phi^{\mathrm{in}}_\asymp$ is a contraction of the $2$-disks in $P_\pm$.  

We denote by $\Phi^{\mathrm{in}}_\asymp(\rho)$ the corresponding map when rescaled and transplanted to the image of $B_\rho$ in $M$, which gives the corresponding flow map for $\mathsf{U}^{\mathrm{in}}(\rho)$.  

\begin{lemma} \label{fixed point treatment smoothing}
If $K_\asymp(\rho)$ has one component then its flow loop appears as the unique fixed point of 
\[
\widetilde{\Phi}^{\mathrm{out}}(\rho)\circ \Phi^{\mathrm{in}}_\asymp(\rho)\circ \widetilde{\Phi}^{\mathrm{out}}(\rho)\circ \Phi^{\mathrm{in}}_\asymp(\rho)
\]
and if it has two, then the two components of the basic flow loop correspond to the  unique fixed points of 
\[
\widetilde{\Phi}^{\mathrm{out}}(\rho)\circ \Phi^{\mathrm{in}}_\asymp.
\] 
\end{lemma}
\begin{proof}
    Certainly any flow loop is such a fixed point.  Uniqueness holds because both $\widetilde{\Phi}^{\mathrm{out}}$ and $\Phi^{\mathrm{in}}_\asymp$ are contractions. 
\end{proof}

We turn to study the flow loops for $\widetilde{\mathsf{U}}(\sigma,\rho)$. As above we will consider these as fixed points of flow maps. Outside $B_\rho$, $\widetilde{\mathsf{U}}(\sigma,\rho)=\widetilde{\mathsf{U}}(\rho)$ and we have the corresponding flow map $\widetilde{\Phi}^{\mathrm{out}}(\rho)\colon P_+\to P_-$ exactly as before. Inside $B_\rho\approx B_1$ we have instead the flow $\Phi^{\mathrm{in}}(\sigma,\rho)\colon P_{j,-}\to P_{j,+}$, $j=1,2$, defined by $\mathsf{U}_{k(\sigma;\rho)}$, see \eqref{eq : conormal cross parameters} and \eqref{eq : def U_k(sigma)}. As for the flows discussed above also this flow gives a contraction $P_{j;-}\to P_{j,+}$, $j=1,2$.

\begin{lemma}\label{l: flow loops near cross}
For all sufficiently small $|\sigma|$, 
there is a unique basic flow loop (or union of flow loops) for $\mathsf{U}(\sigma,\rho)$.
This loop is isotopic to $K_+$ for positive perturbations and to $K_-$ for negative perturbations.
\end{lemma}
\begin{proof}
The flow loops correspond to fixed points of flow maps $\widetilde{\Phi}^{\mathrm{out}}(\rho)\circ \Phi^{\mathrm{in}}(\sigma,\rho)$ in  case $K_\pm$ has two components and  $\widetilde{\Phi}^{\mathrm{out}}(\rho)\circ \Phi^{\mathrm{in}}(\sigma,\rho)\circ \widetilde{\Phi}^{\mathrm{out}}(\rho)\circ \Phi^{\mathrm{in}}(\sigma,\rho)$ in case it is connected.

Consider first the two component case. Then the flow loop of the conormal of the component that contains the $x_2$-axis lies in the $x_3=0$ plane near the crossing: the map $\widetilde{\Phi}^{\mathrm{out}}(\rho)\colon P_{2,+}\to P_{2,-}$ was not perturbed, $\widetilde{\Phi}^{\mathrm{out}}(\rho)=\Phi^{\mathrm{out}}(\rho)\colon P_{2,+}\to P_{2,-}$, the knot agrees with the $x_2$-axis near the origin, and the flow loop is just the original standard flow loop. 

Consider now the other flow loop as a fixed point of the map $\widetilde{\Phi}^{\mathrm{out}}(\rho)\circ\Phi^{\mathrm{in}}(\sigma;\rho)\colon P_{1,-}\to P_{1,-}$. This fixed point has positive or negative $x_3$-coordinate depending on the sign of the perturbation. Since the flow of the fiber disk that defines $\Phi^{\mathrm{in}}(\sigma)$ has the form $(x_2,x_3)=e^{-kt}(x_2^0,x_3^0)$, the sign of the $x_3$-coordinate does not change along the flow line and the result follows in this case. 

The case of a single knot is similar: The sign of the $x_3$-coordinate does not change along the flow lines near the crossing (the flow lines giving $\Phi^{\mathrm{in}}(\sigma;\rho)$). Along the flow lines corresponding to $\widetilde{\Phi}^{\mathrm{out}}(\rho)=\Phi^{\mathrm{out}}(\rho)\colon P_{1,+}\to P_{2,-}$, the $x_3$-coordinate is scaled along the flow by $e^{-kT}$, where $T$ is the flow time. It follows as before that the sign of the crossing agrees with that of the perturbation.  
\end{proof}

\subsection{Flow graphs with corners for $\widetilde{\mathsf{U}}(0,\rho)$}

In this subsection, we will study flow graphs for the immersed Lagrangian $\widetilde{\mathsf{U}}(0;\rho)$ which is our perturbed conormal  of the singular knot $K_\times$.  

Recall that the flow graphs appropriate to immersed Lagrangians may have edges asymptotic to the critical locus of differences of the local defining functions of the Lagrangians i.e., zeroes of differences of the local 1-forms of the Lagrangian. Such flow graphs correspond to  holomorphic curves with corners at the immersed points of the Lagrangian over the critical locus \cite{Ekholm-morse}. We will refer to these as flow graphs with corners. 

The local model near the immersed point of $K_\times$ was described in \eqref{eq : def U_k(sigma)}; it has two components  $\mathsf{U}_{k(0)}=\mathsf{U}_{k_1(0);\tau_1}\cup \mathsf{U}_{k_2(0);\tau_2}$.
We have the following.
\begin{lemma}\label{l : flow tree uniqueness}
For $\mathsf{U}_{k(0)}$ and the zero section, there is a unique flow graph which starts at $q_{-,1}\in P_{-,1}$ ($q_{-,2}\in P_{-,2}$) and has a corner corresponding to the intersection $\mathsf{U}_{k_1(0),\tau_1}\cap\mathsf{U}_{k_2(0),\tau_2}$.  It is a rigid flow tree with exactly one corner, and ends at some point $q_{+,2} \in P_{+, 2}$ ($q_{+,1} \in P_{+,1}$). 

We denote the resulting map as 
\begin{equation*}
    \Phi^{\mathrm{in}}_0 \colon P_-  \to  P_+,\quad\Phi^{\mathrm{in}}(q_{-, k}) = \Phi^{\mathrm{in}}(q_{+, k'}),
\end{equation*}
where $k,k'\in\{1,2\}$ and $k\ne k'$.
\end{lemma}
\begin{proof}
Consider flow graphs with one corner that starts at $q_{-,1}\in P_{-,1}$. The flow of the difference vector field dual to the difference of $1$-forms defining $L_1$ and $L_2$ oriented as the cotangent lift are shown in blue in Figure \ref{fig:flows}. In order for the one corner flow tree to have a corner at the intersection of $L_1$ and $L_2$, the initial flow line between $L_1$ and the zero-section must intersect the stable/un-stable manifold of the difference vector field that is oriented toward the intersection in $L_1$ (vertical in Figure \ref{fig:flows}). (In the language of \cite{Ekholm-morse}, the flow tree has a $Y_0$-vertex at this point.) Following the flow line to the intersection between $L_1$ and $L_2$, then jumping to $L_2$ and follow the difference flow line back to the $Y_1$-vertex we switch flows and then follow the flow line between $L_2$ and the zero section to $P_{2,+}$.    

Figure \ref{fig:flows} shows $\mathsf{U}_{k_j(0),\tau_j}$, $j=1,2$, with the flows of their difference together with the flows of their respective differences with the zero section. The first statement follows from the fact that for each flow line between $L_1$ ($L_2$) and the zero section there is a unique intersection point (which then gives the $Y_0$-vertex) with the stable/unstable manifolds of the (the projection to $B_1$) of the intersection line in $\mathsf{U}_{k_1(0);\tau_1}\cap \mathsf{U}_{k_2(0);\tau_2}$. (See Figure \ref{fig:flowtreemap} for an illustration of a flow tree.)

The case of flow graphs that start at $q_{-,2}\in P_{-,2}$ is similar using instead the stable/unstable manifold oriented toward the critical locus in $L_2$ (the horizontal line in Figure \ref{fig:flows}). 
\end{proof}

\begin{lemma}
The map $\Phi^{\mathrm{in}}_0\colon P_-\to P_+$ is a contraction. 
\end{lemma}
\begin{proof}
This follows from the properties of the flows of the vector fields dual to $d\tau_j+df$, which are the difference forms of $\mathsf{U}_{k_j(0),\tau_j}$ and the zero section $B_1$. The flows are depicted in Figure \ref{fig:flows}. Consider the flow map with initial condition in $P_{-,1}$ near $(x_1,x_2,x_3)=(-1,0,0)$. The map follows the flow line of $d\tau_1+df_1$ until it hits the stable/unstable manifold of $(d\tau_j+df_1)-(d\tau_2+df_2)$, then it follows the flow line of $d\tau_2+df_2$. We note that both flows contract the $x_3$-direction in the standard way. We check that the total flow map contracts the remaining direction in $P_{-,1}\to P_{+,2}$. The first flow is the solution to the differential equation
\[
\begin{cases}
\dot x_1 &= \delta,\\
\dot x_2 &= cx_2 + \eta
\end{cases}
\]
with general solution
\[
\begin{cases}
\dot x_1(t) = x_1(0) + \delta t,\\
\dot x_2(t)= \left(x_2(0)+\frac{\eta}{c}\right) e^{ct} -\frac{\eta}{c}
\end{cases}.
\]
The second is similar with the roles of $x_1$ and $x_2$ switched. This means the the flow lines of the first flow focuses around the central flow line starting at $x_2=\frac{\eta}{c}$, and the flow lines of the second focuses around the flow line ending at $x_1=\frac{\eta}{c}$. It is then straightforward to see that the flow map is a contraction.
\end{proof}

\begin{figure}
    \centering
    \includegraphics[width=0.75\linewidth]{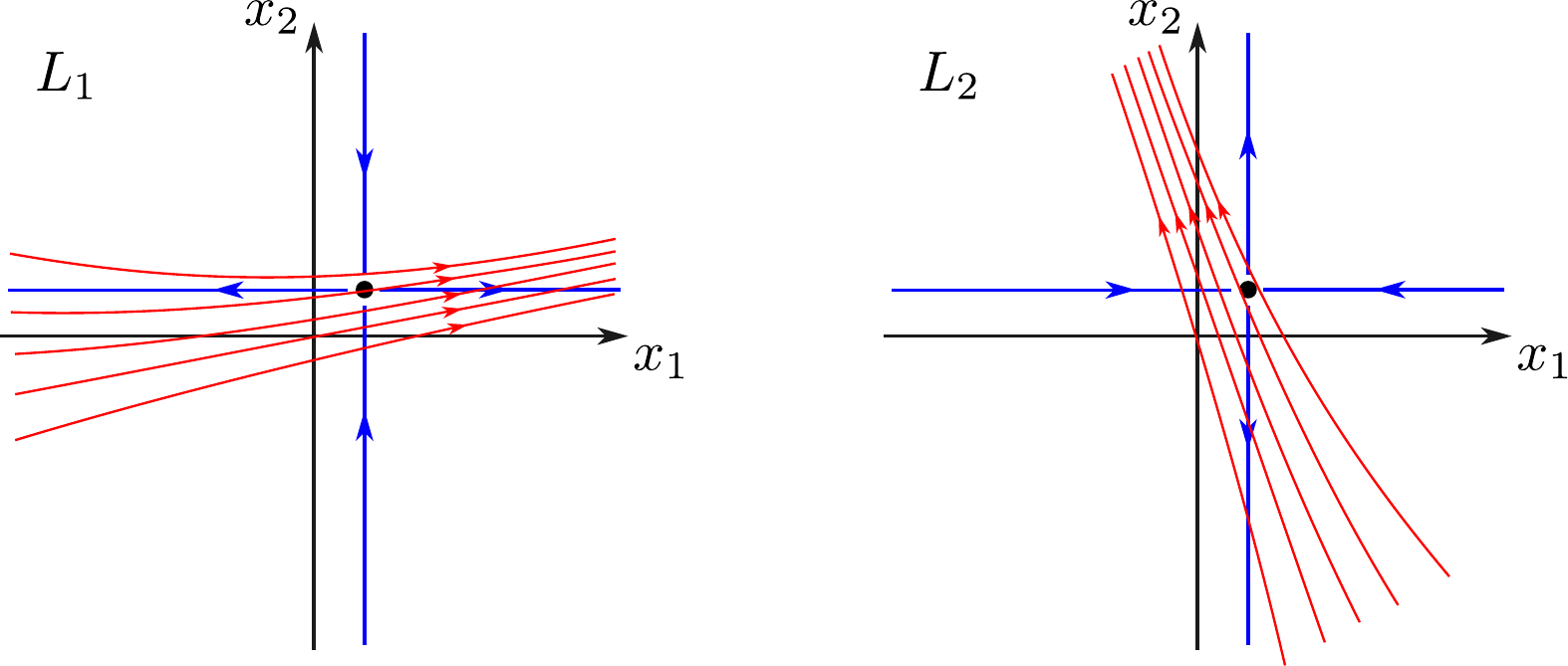}
    \caption{Flows along $L_1$ and $L_2$; flow lines in blue refer to the difference between $L_1$ and $L_2$; flow lines in red refer to flows between $L_j$ and the zero section. The small parameter $\eta>0$ in \eqref{eq: forms for immersed} makes the central flow lines of $L_j$ and the $0$-section non-parallel to the stable/unstable flow lines for the flow of the difference between $L_1$ and $L_2$. The critical point in the flow between $L_1$ and $L_2$ corresponds to the intersection line in $L_1\cap L_2$. }
    \label{fig:flows}
\end{figure}

\begin{figure}
    \centering
    \includegraphics[width=0.75\linewidth]{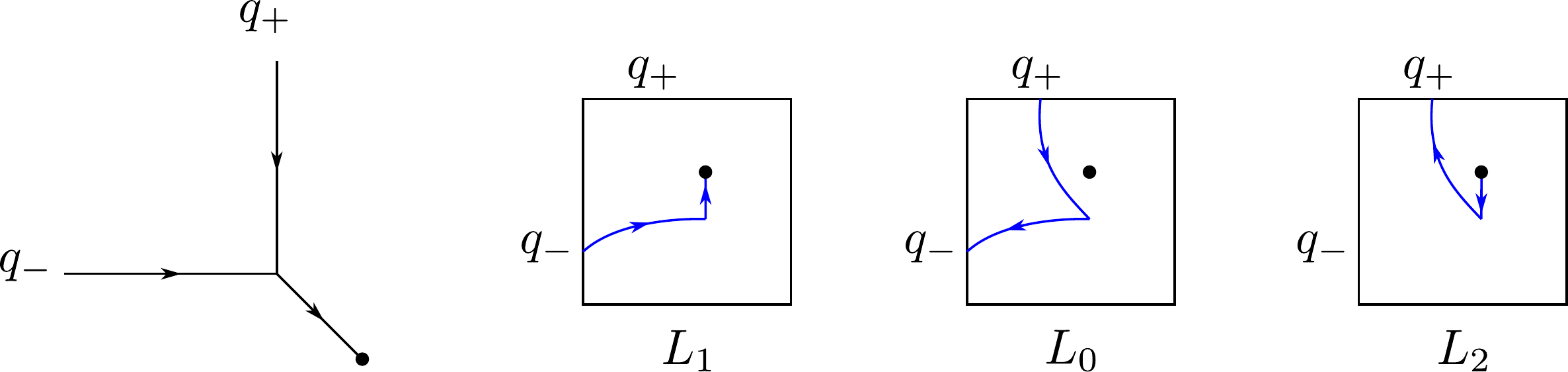}
    \caption{The unique flow tree with initial condition $q_{1,-}\in P_{1,-}$. To the left the tree itself, to the right, the boundary of its flow surface in the various Lagrangians.}
    \label{fig:flowtreemap}
\end{figure}

Recall that  $\widetilde{\mathsf{U}}_\asymp(\rho)$ also
had a unique flow loop; we write $\gamma_\asymp(\rho)$ for the boundary of the corresponding flow surface in the $0$-section $M$ (which is just the flow loop itself). Fix a tubular neighborhood $N$ of the central immersed curve $K_\times$ close to $\gamma_\asymp(\rho)$ lies in an $\mathcal{O}(\rho)$ disk sub-bundle of $N$ and $\gamma_\asymp(\rho)\cap N\setminus s_\rho(B_1)$ is given by a smooth section over $K_\times\setminus s_\rho(B_1)$ since it is a flow of a smooth differential equation. .

\begin{lemma}\label{l: flow loops = flows with corners}
There is a unique basic flow graph with corners for $\widetilde{\mathsf{U}}(0,\rho)$.  It has two corners.  
We write  $\gamma(0;\rho)$ for the boundary in the zero section $M$ of the corresponding flow surface. The curve $\gamma(0,\rho)$ lies in an $\mathcal{O}(\rho)$ disk subbundle of $N$ and is given by a smooth section over $K_\times\setminus s_\rho(B_1)$. Furthermore, the $C^1$-distance between the section of $\gamma_\asymp(\rho)$ and $\gamma(0;\rho)$ in $N\setminus s_\rho(B_1)$ is of the order of magnitude $\rho$.  
\end{lemma}

\begin{proof}
The flow graphs can be computed exactly as in Lemma \ref{fixed point treatment smoothing}, save with $\Phi^{\mathrm{in}}_\asymp(\rho)$ replaced with the appropriate rescaling $\Phi^{\mathrm{in}}_0(\rho)$
of $\Phi^{\mathrm{in}}_0$. 

Both the actual inside maps are defined by applying $s_\rho$ to fixed maps on the unit box.  This implies that their distance in $N$ is $\mathcal{O}(\rho)$. The fact that the outside map is a standard contraction given by flows of the same differential equation then gives the 1-1 correspondence with corresponding sections at $\mathcal{O}(\rho)$ $C^1$-distance as claimed.  
\end{proof}

\subsection{Flow graphs for $\widetilde{\mathsf{U}}(\sigma,\rho)$, $\sigma\ne 0$}
The family of Lagrangians $\widetilde{\mathsf{U}}(\sigma,\rho)$, $-\sigma_0<\sigma<\sigma_0$ constitutes a Lagrangian regular homotopy. For $\sigma\ne 0$, $\mathsf{U}(\sigma,\rho)$ are embedded and $\widetilde{\mathsf{U}}(0,\rho)$ self-intersects cleanly along a curve, i.e., has a Bott degenerate intersection along $\R$, see Lemma \ref{l: intersection L_1 and L_2}. Take $\sigma_0$ small, then far from the intersection curve the Lagrangian regular homotopy is simply a small Lagrangian isotopy, which corresponds to a shift along a closed $1$-form in the cotangent bundle of $\mathsf{U}(0,\rho)$.

From the point of view of the Bott intersection, Lagrangians nearby the Bott degenerate Lagrangian correspond to closed $1$-forms on the Bott manifold. We will next describe $\mathsf{U}(\sigma,\rho)$ from this perspective and then show that it allows us to determine flow graphs of $\mathsf{U}(\sigma;\rho)$ from those of $\mathsf{U}(0,\rho)$.

We use $T^\ast B_4$-coordinates near the self-intersection of $\mathsf{U}(0,\rho)$. 
Here, the components of the Lagrangian $\mathsf{U}(\sigma,\rho)$ have two components $\mathsf{U}_{k_1(\sigma),\tau_1}$ and $\mathsf{U}_{k_2(\sigma),\tau_2}$ given by graphs of the 1-forms in \eqref{eq: forms for immersed}. Lemma \ref{l: intersection L_1 and L_2} shows that the intersection curve $\mathsf{U}_{k_1(0),\tau_1}\cap\mathsf{U}_{k_2(0),\tau_2}$ lies over a line segment at $(x_1,x_2)=(\xi,\xi)$ parallel to the $x_3$-axis where the difference 1-form of $\mathsf{U}_{k_2(\sigma),\tau_2}$ and $\mathsf{U}_{k_1(\sigma),\tau_1}$ along this intersection is given by
\begin{equation}\label{eq: difference along intersection}
(d\tau_2 + df_2) - (d\tau_1 + df_1) = \frac{\partial f}{\partial v_2}(\xi,x_3-\sigma) - \frac{\partial f}{\partial v_2}(\xi,x_3)=\sigma g(x_3)dx_3 ,
\end{equation}
where $g(x_3)>0$ for all $x_3$. We use this line segment parallel to the $x_3$-axis to parameterize the intersection, the shifting $1$-form is then $\sigma g(x_3)dx_3$. 

Describing flow graphs of $\mathsf{U}(\sigma,\rho)$ from data on $\mathsf{U}(0,\rho)$ is a simple case of Morse theory for functions with Bott-degenerate critical loci, see e.g., \cite{FloerMorseWitten, pozniak, EkholmEtnyreSabloff} for applications to Lagrangians with clean intersections. Note first that any flow graph of $\mathsf{U}(\sigma,\rho)\cup M$ converges to a flow graph of $\mathsf{U}(0,\rho)$, possibly with corners at the self-intersection of $\mathsf{U}(0,\rho)$. (This convergence should be understood in the Gromov compactness sense: the convergence of flow lines is uniform on compact subsets of the domain of the flow graph.)
Further, if we rescale $\mathsf{U}(\sigma,\rho)$ in a neighborhood of its intersection we see that if the limit is a flow graph with corners, then the rescaled flow converges to a flow line of $\pm g(x_3)\partial_{x_3}$ connecting corners. We call such a flow line of $\pm g(x_3)\partial_{x_3}$ a \emph{flow line inside the intersection line}. Using standard finite dimensional Morse theory, it is easy to glue a flow graph with corners of $\mathsf{U}(0,\rho)$, connected by a flow line in the intersection line to a flow graph of $\mathsf{U}(\sigma,\rho)$ provided the direction of the flow line (given by $\pm g(x_3)\partial_{x_3}$) agrees with the shift given by $\sigma g(x_3)dx_3$, $\sigma\ne 0$. As in the case of flow graphs, we consider the cotangent lift of the flow line inside the intersection line, which is the flow line itself lifted to the components and oriented as cotangent lifts of edges of flow graphs.

We will use the following notation for flow graphs with corners connected by flow lines in the intersection line. Given a flow graph of $\mathsf{U}(0,\rho)$, possibly with  corners, a {\em positive (resp. negative) pregluing} is the data of flow lines of $g(x_3)dx_3$ (resp. of $-g(x_3)dx_3$) inside the intersection line that connect the corners in pairs such that the flow graph together with the flow line in the intersection admits a continuous cotangent lift to $\mathsf{U}(\sigma_0;\rho)$ (resp. to $\mathsf{U}(-\sigma_0;\rho)$) for $\sigma_0 > 0$. Note that flow graphs without corners tautologically admit both positive and negative gluing.

The correspondence between flow graphs with flow lines in the intersection line of the immersed Lagrangian and actual flow graph on nearby Lagrangians is then as follows. 

\begin{lemma}\label{l : flow graphs U(sigma,rho)}
As $\sigma\to 0$, flow graphs determined by $\mathsf{U}(\sigma;\rho)$, $\sigma\ne 0$, converge (in the sense of Gromov compactness  to flow graphs on $\mathsf{U}(0;\rho)$ with corners at the intersection. 

Furthermore, for all sufficiently small $|\sigma|$ there is a $1-1$ correspondence between flow graphs on $\mathsf{U}(\sigma;\rho)$, $\sigma<0$ ($\sigma>0$) 
and flow graphs of $\mathsf{U}(0,\rho)$ with positive (negative) pregluing. The cotangent lifts of corresponding flow graphs are arbitrarily $C^0$-close and $C^1$-close outside $\mathcal{O}(\sigma)$-neighborhoods of the corners.  
\end{lemma}

\begin{proof}
The first statement is standard as discussed above. The second statement of $C^0$- and $C^1$-convergence follows easily from properties of the flow of the ordinary differential equation
\[
\begin{cases}
\dot x_1 &= cx_1,\\
\dot x_2 &= cx_2,\\
\dot x_3 &= \sigma
\end{cases},
\]
for small $|\sigma|\ge 0$. See Figure \ref{fig:flow tree gluing} for illustration.
\end{proof}

\begin{figure}
    \centering
    \includegraphics[width=0.5\linewidth]{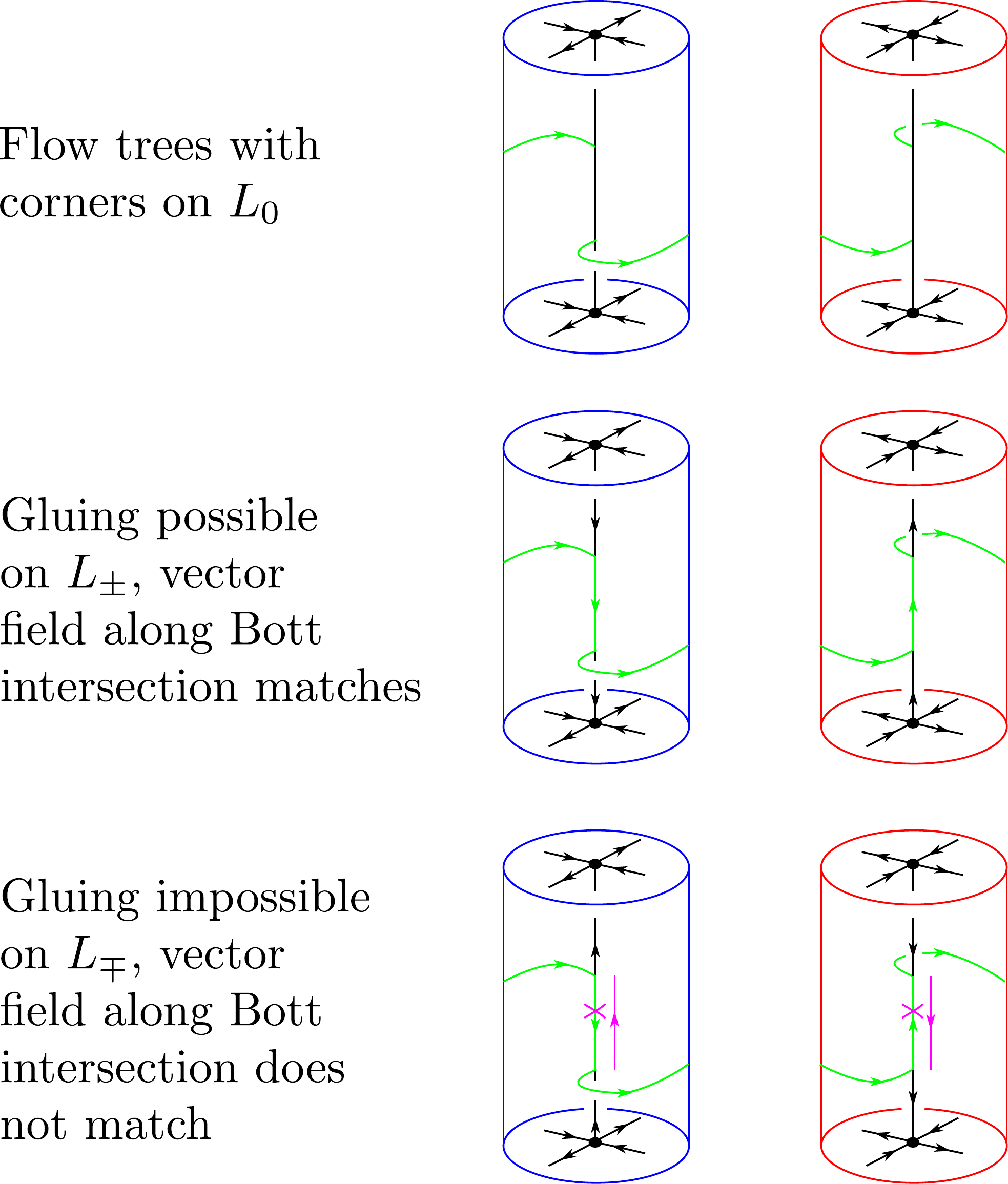}
    \caption{Flow tree gluing}
    \label{fig:flow tree gluing}
\end{figure}

\begin{corollary}\label{c: flow graphs U(sigma,rho)}
    For $|\rho| \ll 0$, there is at most one basic flow graph with (a non-zero number of) vertices  for 
    $\widetilde{\mathsf{U}}(\sigma;\rho)$. For positive (negative) perturbation, this graph appears only for
    $\sigma>0$ ($\sigma<0$). It has two trivalent vertices and the components of its cotangent lift in $M$ is $C^0$-close to $K_\asymp(\rho)$ and the corresponding holomorphic curve has boundary smoothly isotopic to it. 
\end{corollary}

\begin{proof}
    By Lemma \ref{l: flow loops = flows with corners}, there is a unique flow graph with corners for $\widetilde{\mathsf{U}}(\sigma;\rho)$.  It has exactly two corners, hence admits either positive or negative pregluing (there is a flow line between them going one way or the other).  To determine which, the key point is the following fact: if the perturbation of the outside flow is positive (negative) then the corner of the flow tree emanating from $P_{1,-}$ lies above (below) the corner in the flow tree emanating from $P_{2,-}$.
    
    To see this consider the positive case, the negative case is similar. Note first that the $x_3$-component of along flow line of the difference of the two components of $L_{k(0)}$ (the flow line that asymptotes to the intersection) is constant, set $\sigma=0$ in \eqref{eq: forms for immersed}. It follows that the corner in the flow tree of the flow map $\Phi^{\mathrm{in}}_0$ emanating from the fixed point in $P_{1,-}$ has $x_3$-coordinate given by scaling the $x_3$-coordinate of the fixed point by  $e^{-kT}$ of the fixed point, where $T$ is the flow time from the fixed point to the splitting $Y_0$-vertex. Similarly the following corner of the flow tree of the flow tree map $\Phi^{\mathrm{in}}_0$ emanating from the fixed point in $P_{2,-}$ lies at further scaling by $e^{-kT'}$, where $T'$ is the flow time along the unperturbed outer arc and the short inside flow time to the splitting. It follows that the $x_3$-coordinates of the corners are as claimed. 

    Then Lemma \ref{l : flow graphs U(sigma,rho)} asserts the existence of the desired flow graph for $\widetilde{\mathsf{U}}(\sigma;\rho)$, and the non-existence of any others.  (Note we already studied the flow graphs without vertices and/or corners in Lemma \ref{l: flow loops near cross}.)  

    Finally, the boundary of the cotangent lift in $M$ are the flow lines through the fixed point that intersect the  $\widetilde{\mathsf{U}}(0,\rho)$ with the stable/unstable manifolds of the line of intersections. For small perturbation this is $C^0$-close to $K_\asymp(\rho)$. The isotopy statement follows for the local gluing model for $Y_0$-vertices, see \cite[Section 6.1.5]{Ekholm-morse}.
\end{proof}

\section{Review of skein-valued curve counting} \label{svcc review}

Here we briefly recall the setup and main results of \cite{SOB, ghost, bare}, with particular attention to the various topological `brane data' involved (as these brane data choices will play an important role in our subsequent calculations).  Let $X$ be a symplectic 6-manifold with trivial(ized) first Chern class, and $L \subset X$ a (not necessarily connected) smooth Lagrangian submanifold with trivial(ized) Maslov class.  Fix an almost complex structure $J$, standard in a neighborhood of $L$.  In case $X$ and $L$ are non-compact, we require that $(X, L, J)$ at infinity is such as to ensure that Gromov compactness arguments remain valid.  
Fix (as usual for holomorphic curve counting) a spin structure on $L$, or more generally, a background class in $H^2(X, \Z/2\Z)$ and a relative spin structure.

The skein valued curve counting requires additional `brane data' to control various framings.  We
fix a vector field $\xi$ on $L$ with transverse zeroes, and a 4-chain $\Xi$ with $\partial \Xi = 2 L$, such that, in a Weinstein neighborhood of $L$, the neighborhood of $\partial \Xi$ in $\Xi$ is $\overline{\R_{\ge 0} \cdot \pm J \xi}$.  Fix a bounding chain $b$ for $\gamma := (\Xi \setminus \partial \Xi) \cap L$.  Given a map $u\colon (C, \partial C) \to (X, L)$ which is holomorphic and immersed near $\partial L$,  the {\em corrected framing} of $\partial C$ is 
\begin{equation} \label{corrected framing}
\xi + u(C) \cap \Xi + u(\partial C) \cap b.
\end{equation} 
Here, the first term is a vector field; for generic $u$ it will be nowhere parallel to $\partial C$ and hence define a framing; the next two terms are integers, and the sum is understood via the fact that topological types of framings of a knot in a 3-manifold form a $\Z$-torsor. 
As explained in \cite{SOB}, this corrected framing is invariant under deformations of 
$\xi, \Xi, b$.  The space of vector fields is connected, and the dependence on global aspects of the choices $(\Xi, b)$ is as follows: the difference between two such choices gives an element of $H_4^{BM}(X) \oplus H_2^{BM}(L)$, and the corresponding difference in corrected framings is the intersection of this class with $[u] \in H_2(X, L)$. 

Fix a class $d \in H_2(X, L)$.  Recall that we call a map $u\colon (C, \partial C) \to (X, L)$ {\em bare} if $u^{\ast}(\omega)$ gives positive symplectic area to every irreducible component of $C$.  We write $\chi(u)$ for Euler characteristic of a smoothing of the domain of $u$.  Given some system of (weighted, multivalued) perturbations for the holomorphic curve equation, if $u$ is a transverse solution to the perturbed equation, we write $\mathrm{wt}(u)$ for the local contribution of this solution.  
The main result of \cite{SOB, ghost, bare} is that there is a system of perturbations to the holomorphic curve equation so that the following sum over bare solutions is well defined and deformation invariant: 
\begin{equation}
    Z_{X, L, d} = \sum_{[u] = d} \mathrm{wt}(u) \cdot z^{-\chi(u)} \cdot [u(\partial \Sigma)] \in \Sk(L) \otimes_{\Z[z]} \Q[[z]]
\end{equation}

To sum over $d$, one generally must extend scalars by the Novikov ring $\Q[ H_2(X, L)]$.  However, when (as will be the case throughout this article) the symplectic form is exact, $\omega = d \lambda$, we detect the symplectic area by integrating $\lambda$ over the image of $[u]$ in $H_1(L)$, which in turn can be read from the class $[u(\partial \Sigma)] \in \Sk(L)$.  Thus we may instead simply complete $\Sk(L)$ along the cone of $\lambda$ positive classes. Going forward, we will omit this completion from the notation, and write $Z_{X, L} = \sum_d Z_{X, L, d} \in \Sk(L)$.

\subsection{Spin structures -- remarks and extensions} \label{sec: remarks spin}

\begin{enumerate}

\item Recall that the choice of spin structure on $L$ is used to orient the moduli space of maps, so ultimately to determine signs.  If we change the spin structure on $L$, we change the signs by a factor determined by the degree $d \in H_2(X, L)$.  In our exact setting, this factors through $H_1(L)$ and hence through the skein.  That is, changing the spin structure is intertwined with a certain automorphism of the skein, multiplying each curve by a sign depending on its homology class. 
\item Rather than choose a spin structure on $L$, we may fix a link $\eta \subset L$ and choose a spin structure on $L \setminus \eta$ which does not extend across $\eta$.  Generically, the boundaries of holomorphic curves will not meet $\eta$, so this causes no problem for defining the count.  As far as invariance is concerned, we pick up a sign when the boundary of a curve crosses $\eta$.  So, the skein-valued curve count is well defined in the skein of $L$ with a sign line, see \eqref{eq:skeinrel4}, along $\eta$. 
\end{enumerate}

\subsection{4-chains -- remarks and extensions} \label{sec: remarks 4-chain}

\begin{enumerate} 
\item The theory depends on the vector field $\xi$ on $L$ only through the corresponding section of $\underline{\xi}$ of the sphere bundle $SL$.  The condition that $\xi$ has transverse zeroes translates to asking $\underline{\xi}$ is defined away from finitely many points $p_i$, and near each such point, $\underline{\xi}$ defines a degree $\pm 1$ map from boundary of some $\epsilon$ ball around $p_i$ to the tangent sphere. (Note that the extra boundary fiber disks of the 4-chains of $\xi$ and of $-\xi$ at zeros of $\xi$ cancel in the 4-chain of $\pm\xi$ with boundary $2L$.)  We could work throughout specifying only $\underline{\xi}$ rather than $\xi$. 
\item 
Rather than choose a class $b$ with $\gamma = (\Xi \setminus L) \cap L$, we may work with a skein with a {\em framing line} along $\gamma$.  That is, we take framed links in $L \setminus \gamma$, and impose the usual skein relations, along with the relation: 
\begin{gather}
\vcenter{\hbox{
\begin{tikzpicture}[scale=0.7]
\draw[dotted] (0,0) circle (1);
\draw[ultra thick, <- , green] (-1, 0) -- (1, 0);
\draw[white, line width=2.5mm] (0, -1) -- (0, 1);
\draw[ultra thick, ->] (0, -1) -- (0, 1);
\node[text=green] at (1.5,0) {$\gamma$};
\end{tikzpicture}
}}
\;\;=\;\;
a\cdot
\vcenter{\hbox{
\begin{tikzpicture}[scale=0.7]
\draw[dotted] (0,0) circle (1);
\draw[ultra thick, ->] (0, -1) -- (0, 1);
\draw[white, line width=2.5mm] (-1, 0) -- (1, 0);
\draw[ultra thick, <- , green] (-1, 0) -- (1, 0);
\node[text=green] at (1.5,0) {$\gamma$};
\end{tikzpicture}
}}
\;. \label{eq:skeinrel6}
\end{gather}
The sign $\pm 1$ in the exponent of $a$ in \eqref{eq:skeinrel6} reflects the change in local linking number with $\gamma$.
\item If $L = \bigsqcup L_i$ and compatibly $\Xi = \sum \Xi_i$ (i.e. $\partial \Xi_i = 2 L_i$), then we may use a different framing variables $a_i$ on each $L_i$.
\item 
Similarly, we may take several 4-chains $\xi_i, \Xi_i$ for one given Lagrangian $L$.  Then we may consider the 4-chain $\Xi_L := \frac{1}{n} \sum \Xi_i$.  Of course we should count 4-chain intersections with each $\Xi_i$ by $a_L^{1/n}$.   
\item \label{4-chain with monodromy} Given a local system of sets of $n$ vector fields on $L$, the collection of which however may undergo monodromy around $L$, we may consider a 4-chain $\Xi$ for $L$ which is locally near $L$ of the form $\Xi_L := \frac{1}{n} \sum \Xi_i$. Counting holomorphic curves in the skein using such a 4-chain gives a deformation invariant result since the relationship between the 4-chain/vector field pair and the skein relations in \cite{SOB} are all local over $L$. 

\item We may specialize at $a=1$, or, if there are many components as above, specialize some of the $a$ variables to $1$.  Evidently, for a component $L_i$ for which we set $a_i = 1$, the invariant is independent of the choice of $\xi_i, \Xi_i$; in fact, an inspection of \cite{SOB} shows that we do not need to make such choices at all to define the $a_i=1$ specialization of the invariant. In the $a=1$ skein of $S^3$, the class of every nontrivial knot or link is zero, but in the solid torus that is not the case, see Section \ref{basic classes}. 
\end{enumerate}

\subsection{Basic classes} \label{basic classes}
Suppose $X$ is an exact symplectic manifold, and $L \subset X$ is a Maslov zero Lagrangian with the topology of the solid torus; assume $L$ is equipped with vector field and 4-chain as appropriate to define skein-valued curve counts per \cite{SOB}.   
Fix 
a generic longitude of $L$.  The longitude acquires a framing from the vector field, and an orientation by asking that the fixed symplectic primitive be positive on the longitude. 

A choice of framed oriented longitude determines expansions as in \eqref{torus skein expansion} for elements in the skein of $L$.  
We will write $\psi(L)$ for the coefficient of $\psi_{\square, \emptyset}$ in the skein-valued curve count for $(X, L)$.    
More generally, if $M$ is another Maslov zero Lagrangian, again equipped with appropriate data and disjoint from $L$, then we write  $\psi(L) \in \Sk(M)$ for the element obtained by taking the original curve count in $\Sk(M \sqcup L) = \Sk(M) \otimes \Sk(L)$ and extracting the coefficient of $W_{\square, \emptyset}$ in the second factor.

Similarly, when $L$ is a disjoint union of $n$ solid tori, we similarly write $\psi(L)$ for the coefficient of $W_{\square, \emptyset}^{\otimes n}$ in the skein-valued curve count; again $\psi(L)$ is valued in the skein of the other Lagrangians.

For $\Psi$ which are known a priori to admit an expansion in $W_{\lambda, \emptyset}$, the foregoing discussion and definition of $\psi(L) \in \Sk(M)$ remain valid in the $a=1$ specialization.

\subsection{Normalization of the basic cylinder} \label{cylinder normalization}

Let $M$ be a 3-manifold.  Let $K \subset M$ be a link.  Then $\psi(\mathsf{U}_K)$ is the count of one embedded cylinder for each component of $K$; their geometric boundary in $M$ is $K$.  Here we explain how to normalize the 4-chain on $M$ so that, if we put the $a=1$ skein on $\mathsf{U}_K$, then $$\psi(\mathsf{U}_K) = [K] \in \Sk(M).$$

Fix a vector field $\xi$ on $M$ and take the 4-chain for the zero section $M \subset T^*M$
to be $\Xi := \overline{\R_{\ge 0} \cdot \pm J \xi}$.   

Fix a metric on $M$ such that $\xi$ is perpendicular to $K$ only at isolated points and that $\langle\xi,\tau\rangle$ has transverse zeros at these points.  Then:

\begin{lemma}
    Let 
    $K_{\xi > 0} := \langle \xi, TK \rangle^{-1}(\R_{> 0})$ and similarly 
    $K_{\xi < 0} := \langle \xi, TK \rangle^{-1}(\R_{< 0})$.  Then, for small push-offs $\mathsf{U}_K$, the locus $\Xi \cap \mathsf{U}_K$ is obtained by smoothing the loci
    $$K_{\xi> 0} \cup \R_{\le 0}\xi|_{\partial K_{\xi > 0}} \qquad \mathrm{and} \qquad K_{\xi< 0} \cup \R_{\ge 0}\xi|_{\partial K_{\xi < 0}}.$$
    In particular, $\Xi \cap \mathsf{U}_K = 0 \in H_1^{BM}(\mathsf{U}_K, \Z)$. $\qed$
\end{lemma}
 Note that the transverse disk generating $H_2^{BM}(\mathsf{U}_K, \Z)$ pairs as $1$ with the longitude $\ell$ of the solid torus; thus, we may choose a Borel-Moore chain $b$ with $db = \Xi \cap \mathsf{U}_K$ so that 
$\ell \cap b \in \Z$ takes any value.  We choose $b$ so that $\ell \cap b = 0$.

\begin{remark}
    A version of the above discussion, identifying $\psi(\mathsf{U}_K) = K$ was the first step of the proof of the Ooguri-Vafa conjecture in \cite{SOB}.  In that article, we were free to use the full skein (not specialized at $a=1$) on $\mathsf{U}_K$ because we were working with $M = S^3$, so $K$ necessarily bounds and $\mathsf{U}_K$ admits a 4-chain.   
    Indeed, since $T^*M \cong M \times \R^3$, pulling back gives an isomorphism $H^{BM}_\star(M) \cong  H^{BM}_{\star + 3}$ carrying $K \mapsto \mathsf{U}_K$.  That is, $\mathsf{U}_K$ admits a 4-chain if and only if  $[K]$ is in the image of $H_1(\partial M) \to H_1(M)$.  This is why in the present more general setting, we will use the $a=1$ skein on $\mathsf{U}_K$ to avoid restricting attention to links which are null-homologous rel. $\partial M$. 
\end{remark}

\section{Skein trace from curve counting}\label{sec : curves graphs skein}

In this section, we will prove Theorem \ref{skein trace} on the well-definedness of the skein trace.  First we must clarify the statement, or more precisely Definition \ref{main definition}, by giving the choice of 4-chain on $L \subset T^*M$ used to define the skein-valued curve counting.  Then we check that the prescription of Definition \ref{main definition} for the map $K \mapsto [K]_L \in \Sk(L)$ satisfies the three skein relations in $M$, hence factors through $\Sk(M)$.  

\subsection{4-chains for branched covers} \label{sec: brane data}

Let $M$ be a 3-manifold; fix a vector field with transverse zeros $\xi_M$ on $M$ and 4-chain $\Xi_M := \overline{\R_{\ge 0} \cdot \pm J \xi}$.  We write $\Xi=\Xi[J]$ when we want to emphasize the dependence of the almost complex structure. 

Consider a Lagrangian $L \subset T^*M$ so that $\pi\colon L\to M$ is a degree $n$ cover.  
We will say that a 4-chain $\Xi_L$ for $L$ is {\em compatible} with $\Xi_M$ if, in the complement of a fiberwise compact subset of $T^*M$, we have $\Xi_L = n \Xi_M$. 

If $B\subset M$ is an open ball where $\pi$ is not branched then, over $B$, $L$ is given by the graphs of the differentials $df_1,\dots, df_n$ of $n$ functions. We parameterize the graph of $df_j$ by a map $\gamma_{j}\colon B\to T^\ast M$: for $q\in B$, $\gamma_j(q)=(q,df_j(q))\in T^\ast M$ and extend $\gamma_j$ to a symplectomorphism $\Gamma_j\colon T^\ast B\to T^\ast B$ fiberwise, $\Gamma_j(q,p)=(q,df_j(q)+p)$.

\begin{definition}\label{def : standard over B}
Choose local coordinates $M \supset B \to \R^3$, on a region where $\xi$ has no zeroes and  $L \to M$ is a (not branched) cover. We say that $L$ is {\em standard over $B$} if, inside some fiberwise compact region of $T^*B$ containing $L \cap T^*B$, one has, in local coordinates: 
\begin{enumerate}
    \item The 4-chain for $L$ is $\Xi_L = \sum \Xi_{L_i}$ where $\Xi_{L_i}$ corresponds to the push-forward vector field on the graph of $df_i$:
    \begin{align}\label{eq : push forward 4chain}
    \xi_i(\gamma_i(q)) \ &:= \ \left(d\gamma_i(q)\right) \xi(q),\\\notag
    \Xi_{L_i} \ &= \ \Gamma_i\left(\Xi[J_\Gamma]\right) \ = \ \bigcup_{q\in B} \left(\gamma_i(q) \pm J(\gamma_i(q)) \cdot \R_{> 0} \xi_i(\gamma_i(q))\right),
    \end{align}
    where $J_\Gamma=d\Gamma^{-1}\circ J\circ d\Gamma$.
    \item $\Xi_{L_i}$ is disjoint from the graph of $df_j$ for all $j \ne i$. 
\end{enumerate} 
\end{definition}

We will show that Theorem \ref{skein trace} holds for any 4-chain which is standard over some ball $B$.  In the remainder of this subsection, we establish the existence of such chains. 

Let $\xi$ be a vector fields on $M$ and let $B\subset M$ be a ball with compact closure where $\xi$ has no zeros. Let $\Xi=\sum \Xi_{L_i}$, where $\Xi_{L_i}$ is as in \eqref{eq : push forward 4chain}. Let $S_\lambda\colon T^\ast M\to T^\ast M$ denote the fiber scaling by $\lambda$.
\begin{lemma}\label{l : not parallel no intersection}
If $df_i-df_j$, $1\le i<j\le n$, is not a multiple of the co-vector $\xi^\ast$ dual to $\xi$ then condition $(2)$ of Definition \ref{def : standard over B} holds for $S_\lambda(L)$ for all sufficiently small $\lambda>0$.
\end{lemma}

\begin{proof}
Since $d\Gamma= 1 + \mathcal{O}(\lambda)$, $J_\Gamma \xi_i= J \xi +\mathcal{O}(\lambda)$. Then, if, for some $s\in\R$, 
\[
\lambda df_j(q) = \lambda df_i(q) + s J\xi_i(q),
\]
we have
\[
df_j(q)-df_i(q) = \lambda^{-1} s\left(\xi^\ast + \mathcal{O}(\lambda) \right)
\]
and in the limit $\lambda\to 0$ we find that $df_j-df_i$ is a multiple of $\xi^\ast$. The lemma follows. 
\end{proof}

We next consider how condition $(2)$ fails generically. Consider a vector field $\xi$ in general position with respect to $df_i-df_j$. Then the locus where $df_i-df_j$ is parallel to $\xi$ is a transversely cut out $1$-submanifold $K_{ij}\subset B$.
\begin{lemma}\label{l : parallel along link}
For all sufficiently small $\lambda>0$, the intersection $\Xi_{L_i}\cap L_j$ is a $1$-submanifold isotopic to $\Gamma_j(K_{ij})$; in fact the intersection $C^1$-converges to $\Gamma_j(K_{ij})$ as $\lambda\to 0$. 
\end{lemma}

\begin{proof}
As above, the intersection $\Xi_{L_i}\cap L_j$ corresponds to the solutions $q\in B$ to the equation
\[
df_j(q)-df_i(q) = \lambda^{-1} s\left(\xi^\ast + \mathcal{O}(\lambda) \right)
\]
and $K_{ij}$ corresponds to solutions to the equation
\[
df_j(q)-df_i(q) = \sigma \xi^\ast. 
\]
The lemma follows by transversality of the solutions to the last equation.
\end{proof}

We next consider compatibility of the 4-chain $\Xi$. 
\begin{lemma}\label{l : interopolate outside}
Let $\Xi=\sum_i\Xi_{L_i}$ be a 4-chain given by the formula \eqref{eq : push forward 4chain} in $D_1^\ast M|_B$. Then, for all sufficiently small $\lambda>0$, $\Xi$ can be continued to a $4$-chain in $T^\ast M|_B$ that agrees with $n\Xi_M$ outside $D_2^\ast M|_B$.  
\end{lemma}

\begin{proof}
Since $\Xi\to n\Xi_M$ as $\lambda\to 0$, there is a homotopy of chains that take $\Xi\cap \partial D^\ast_1 M|_B$ to $n\Xi_M\cap\partial D_1^\ast M|_B$. We use this homotopy to interpolate in $D_2^\ast M|_B \setminus D^\ast_1 M|_B$. The lemma follows.
\end{proof}

In fact, 4-chains which are standard in the complement of some neighborhood of the branch locus are generic for all sufficiently small scalings of a given branched cover. 

Consider a degree $n$ branched cover $\pi\colon L\to M$ where $L\subset T^\ast M$ is a Lagrangian. Let $\tau\subset M$ be the link which is the branch locus of $\pi$. Then $L$ defines a system of $n$ closed differential forms $(\alpha_1,\dots,\alpha_n)$ over $M\setminus \tau$. Let $K\subset M$ denote the subset where the dual of $\xi^\ast$ is proportional to some difference $\alpha_i-\alpha_j$. For $\xi$ in general position $K$ is a transversely cut out 1-submanifold. Let $\widetilde{K}\subset K$ denote the natural lift of $K$ that lifts each arc of $K$ where $\xi^\ast$ is parallel to $\alpha_i-\alpha_j$ to the corresponding sheet of $L$. 
We assume that $\xi$ is generic and that the closure of $K$ is disjoint from $\tau$. Then $K\subset M$ is a closed submanifold. 

\begin{corollary} \label{existence of 4-chain} 
    Let $N(\tau)$ and $N(K)$ be a tubular neighborhoods of $\tau$ and $K$ respectively. Then for all sufficiently small $\lambda>0$ there is a $4$-chain $\Xi_L$ on $S_\lambda(L)$ that is compatible, and standard over any ball $B\subset M\setminus (N(\tau)\cup N(K))$. Furthermore, the intersection $\mathrm{int}(\Xi_L)\cap L$ is isotopic to $\widetilde{K}\subset L$ and $C^1$-converges to $\widetilde K$ as $\lambda\to 0$.   
\end{corollary}
\begin{proof}
For a fixed ball in $B\subset M\setminus (N(\tau)\cup N(K))$ this follows from Lemmas \ref{l : not parallel no intersection}, \ref{l : parallel along link}, and \ref{l : interopolate outside}. To see that it holds for any ball first observe that the we can get compatibility near any vertex. The closure of the complement of the vertices in $M\setminus (N(\tau)\cup N(K))$ is compact and can then be covered by a finite number of balls. The argument for a fixed ball then implies the lemma. 
\end{proof}

\begin{remark}
    The analogous proposition holds with the identical proof if in place of single vector fields $\xi$, one has local systems of vector fields as in \eqref{4-chain with monodromy} in Section \ref{sec: remarks 4-chain}. 
\end{remark}

\subsection{Proof of Theorem \ref{skein trace}} \label{existence of skein trace}

\subsubsection{First skein relation}
Consider a Lagrangian contained in a unit cotangent ball, $L \subset B^* M \subset T^\ast M$.  In particular, under the fiber scaling, $\lambda L \to M$ as $\lambda \to 0$.    We have the following general fact:

\begin{lemma} \label{limit configurations}
    Consider $L \subset T^*M$ as above, and a Lagrangian $U \subset T^*M \setminus B^*M$. 
    
    Then in the flow tree limit, holomorphic curves for $(T^*M, U \sqcup \lambda L)$ converge to holomorphic curves for $(T^*M, U \sqcup M)$ with attached flow graphs for $\lambda L$ in $T^*M$.  
\end{lemma}
\begin{proof}
    This is a special case of Lemma \ref{l : flow graph converegnce 3}. 
\end{proof}

We also have: 

\begin{lemma}\label{l : no graph at the intersection}
Consider $L \subset T^\ast M$. Let $K$ be an immersed link in $M$ with conormal $U_{\times}$ with clean intersection and let $U_{\pm}(\rho)$ and $U_\asymp(\rho)$ be the nearby conormals of the links $K_\pm$ and $K_\asymp$ as above. Let $p$ be the intersection point of $K$. Then for generic $L$ (open dense set of $L$) there is a neighborhood of $V$ around $p$ in $M$ such that no rigid flow graph in the basic class on $U_\times$ has any flow line of $L$ inside $V$. It follows that the same holds for generic resolutions $U_\pm(\rho)$ and $U_\asymp(\rho)$ for all sufficiently small $\rho$. 
\end{lemma}

\begin{proof}
The genericity statement is a consequence of general position: the formal dimension of a rigid flow graph of $L$ on $K_\times$ with the additional constraint that some flow line in the graph is attached at a particular point along $K_\times$ equals $-1$. The statement about $U_\times$ follows. 

For $U_\pm(\rho)$ and $U_\asymp(\rho)$ rigid flow graphs on $K_\pm(\rho)$ and $K_\asymp(\rho)$ converge to rigid flow graphs on $K_\times$ as $\rho\to 0$. Thus, if $K_\times$ is generic there is a neighborhood $V$ of $p$ such that for large enough $\rho$, rigid flow graphs of $K_\pm(\rho)$ ($K_\asymp(\rho)$) do not have any flow graphs attached to $K_\pm(\rho)$ ($K_\asymp(\rho)$) in $V$ that leave $V$. 

It remains to show that they also do not have flow lines entirely contained in $V$. Such a flow line would converge to the constant flow line at the intersection point, the limit of its tangent vector is the limit of a gradient difference of sheets of $L$ at $p$. Thus, as long as all gradient differences are not parallel to the shifts of the two resolution directions of $K_{\pm}(\rho)$ and $K_\asymp(\rho)$ there are no such flow graphs. The lemma follows.
\end{proof}

We now prove: 

\begin{proposition} \label{first skein relation}
    Let $K_\times \subset M$ be a link with an immersed point, nowhere parallel to (and hence framed by) the fixed vector field $v_M$.  Let 
    $K_+, K_-, K_\asymp \subset M$ be the three corresponding links for the first skein relation (so $K_+ - K_- = z K_\asymp$ in $\Sk(M)$).  Then $$[K_+]_L - [K_-]_L = z [K_\asymp]_L \in \Sk(L)$$
\end{proposition}
\begin{proof}
    We should compare holomorphic curves between $L$ and (any choice of, see Section \ref{basic classes}) perturbed conormals 
    to $K_+, K_-, K_{\asymp}$. Isotope $K_\times$ so that its singular point is in the ball in $M$ where the 4-chain of $L$ satisfies the `standardness-condition', and moreover is in the open set $U$ of Lemma \ref{l : no graph at the intersection}.  We will choose $\rho$ sufficiently small that the $B_\rho$ neighborhood around the singular point remains in $U$. 

    We use the corresponding $\widetilde{\mathsf{U}}_\asymp(\rho)$, $\widetilde{\mathsf{U}}(\pm \sigma;\rho)$ for sufficiently small $\sigma,\rho >0$, and negative perturbation in the sense of Definition \ref{positive negative perturbation} and choose them so that Lemma \ref{l : no graph at the intersection} holds. 

    Then, as described in Lemma \ref{l : flow graphs U(sigma,rho)} and Corollary \ref{c: flow graphs U(sigma,rho)},  there is one holomorphic curve each in the basic class for $(T^*M, M \sqcup \widetilde{\mathsf{U}}(-\sigma, \rho))$
    and $(T^*M, M \sqcup \widetilde{\mathsf{U}}_{\asymp}(\rho))$.  Denote these curves as $C_-$ and $C_\asymp$, respectively.  There are two holomorphic curves for $(T^*M, M \sqcup \widetilde{\mathsf{U}}(\sigma, \rho))$, which we will denote $C'_-, C'_\asymp$. 
    These curves have the following properties: 
    \begin{itemize}
        \item $C'_-$ is $C^1$ close to $C_-$ near $M$
        \item $C'_\asymp$ is $C^1$ close to $C_\asymp$ at least near $M \setminus B_\rho$
        \item there is an isotopy $\partial C'_\asymp \sim \partial C'_\asymp$ which is $C^1$ small in $M \setminus B_\rho$ and standard in $B_\rho$, see Lemma \ref{l : flow graphs U(sigma,rho)}.
        \item The domains of $C_-, C'_-, C_\asymp$ are (possibly unions of) cylinders, while $C'$ is a three-holed sphere, $\chi(C'_{\asymp}) = -1$. 
    \end{itemize}

    By Lemma \ref{limit configurations}, the holomorphic curves ending on $L \sqcup \widetilde{\mathsf{U}}_{\asymp}(\rho)$ limit to flow trees for $M \cup L$ attached to $C_\asymp$.  The limit is not necessarily transverse, but 
    since no tree enters the region $U$ where $C_\asymp$ and $C_\asymp'$ differs and since perturbations can be taken supported in a small neighborhood of these flow trees it follows, using the same perturbation for the gluing in the two cases that the solutions for $C_\asymp$ are in canonical bijection with the solutions to the corresponding problem for $C'_\asymp$. Let us write respectively $\mathrm{Count}(C_\asymp)$ and $\mathrm{Count}(C'_\asymp)$ for the skein count of the resulting glued curves.  

    The same argument applied to $C_-$ and $C'_-$ again gives canonical identifications of moduli spaces and we use analogous notation. We conclude: 
    \begin{align*}
    \psi(\widetilde{\mathsf{U}}(\sigma, \rho)) &= \mathrm{Count}(C'_-) + \mathrm{Count}(C'_\asymp) = \mathrm{Count}(C_-) + z \cdot \mathrm{Count}(C_\asymp)\\ 
    &= \psi(\widetilde{\mathsf{U}}(-\sigma, \rho)) + z\psi(\widetilde{\mathsf{U}}_\asymp(\rho)).
    \end{align*}
    Here the $z$ power appears because, as noted above, $\chi(C'_\asymp) = \chi(C_\asymp) - 1$, and we count curves by $z^{-\chi}$. 
    No $a$ powers appear because the geometries match away from $B_\rho$, and the 4-chain in $B_\rho$ is standard. 
    This completes the proof. 
\end{proof}

\subsubsection{Second skein relation}\label{subsubsec:second-skein-rel}

We first study a model geometry.  

\begin{definition} \label{model geometry}
Consider $M = \R^3$.  Consider a trivial cover $L = L_1 \sqcup \cdots \sqcup L_n \subset T^* \R^3$ where 
each $L_i$ is a constant Lagrangian of the form $\R^3 \times (\ell_i^1, \ell_i^2, \ell_i^3) \subset \R^3 \times (\R^3)^* = T^* \R^3$.  We assume all the $\ell_i^3$ are distinct and order the $L_i$ such that $\ell_1^3 >  \ell_2^3 > \ldots > \ell_n^3$.

We fix the vector field $v_M = \partial/\partial x_3$, and the corresponding linear 4-chain $W_M =  \pm \R_{>0} J \cdot \partial/\partial x_3$.  (Here $J$ is determined by the choice of a metric on $\R^3$; we take the standard Euclidean metric.) 

We choose 4-chains for the $L_i$ by taking the linear 4-chain whose corresponding vector field $v_{L_i}$ is $(\pi_i: L_i \to M)^* v_i = v_i + (\ell_i^1, \ell_i^2, \ell_i^3)$.  
We write $a_i$ for the `a' variable associated to the component $L_i$. 
\end{definition}

\begin{lemma} \label{unknot model geometry}
Let $U$ be a standard unknot given by $(\cos \theta, \sin \theta, 0) \subset \R^3$ in the model geometry of Def. \ref{model geometry}. 
Then 
$$[U]_L = \frac{a_1\dots a_n-a_1^{-1}\dots  a_n^{-1}}{z} \in \Sk(L)$$
\end{lemma}
\begin{proof}
    Let $L_U$ be the shifted conormal of the unknot.  It is straightforward to see that there is a unique holomorphic cylinder between $L_U$ and each $L_i$, meeting $L_i$ in the lift of $U$ to $L_i$.  By construction, $U$ is the standardly framed unknot with respect to $v_{L_i}$ (i.e. projects to an embedded unknot in the leaf space of the vector field), hence the boundary contribution of this holomorphic curve is $(a_i - a_i^{-1})/z$.  
    
    It remains to compute the 4-chain contributions.  Because the Lagrangians are constant, we may compute the intersection after projecting to the cotangent fiber coordinates.  The image of the holomorphic cylinder is   
    $(\ell_i^1, \ell_i^2, \ell_i^3) + t (\cos(\theta), \sin(\theta), 0)$ with $t > 0$.  The image of the 4-chain for $L_j$ is the line $(\ell_j^1, \ell_j^2, \ell_j^3 + s)$.  There is a unique $(t \cos(\theta), t \sin(\theta))$ so that the first two coordinates match.  Considering the last coordinate, we see there is a unique intersection, which is in the positive part of the 4-chain if  $s = \ell_i^3 - \ell_j^3 > 0$, and the negative part of the 4-chain otherwise.   That is, the contribution of the 4-chain of $L_j$ to the cylinder ending on $L_i$ is $a_j$ if $j > i$ and $a_j^{-1}$ if $j < i$. 

    All in all, we find
    $$[U]_L =  \sum_{j=1}^n a_1^{-1} a_2^{-1}\dots a_{j-1}^{-1}\, \left(\frac{a_j - a_j^{-1}}{z} \right) \, a_{j+1}\dots a_{n} \ = \ \frac{a_1\dots a_n-a_1^{-1}\dots  a_n^{-1}}{z}$$
    as desired. 
 \end{proof}

\begin{corollary} \label{second skein relation}
    Fix a Lagrangian $L \subset T^*M$ such that $L \to M$ is a degree $n$ cover, and $L$ is equipped with a somewhere standard 4-chain.  Then we have, for a standardly framed unknot $U$ in $M$, 
    $$[U]_L = \left(\frac{a_L^n - a_L^{-n}}{z} \right) [\emptyset]_L \in \Sk(L)$$
\end{corollary}
\begin{proof}
    We may put our unknot $U$ anywhere on $M$, and take it arbitrarily small. 
    
    We fix attention at an open ball of the standard neighborhood through which no flow trees pass (which exists per Lemma \ref{l : no graph at the intersection}), and there use $v_M$ to choose coordinates so that $v_M = \partial/\partial x_3$.  Taking an arbitrarily small neighborhood of a point, we may arbitrarily closely approximate the model geometry of Definition \ref{model geometry}. 

    When we compute $[U]_L$, each contribution is a disjoint union of holomorphic curves corresponding to $L$ flow graphs (which are by hypothesis disjoint from the neighborhood) and the model cylinders.  The contribution of the external flow graphs are identical for each model cylinder, and total, by definition, $[\emptyset]_L$.  

    The result then follows from Lemma \ref{unknot model geometry}. 
\end{proof}

\subsubsection{Third skein relation}\label{subsubsec:third-skein-rel}

\begin{proposition}
    Fix a Lagrangian $L \subset T^*M$ such that $L \to M$ is a degree $n$ cover, and $L$ is equipped with a somewhere standard 4-chain.
    Let $\beta \subset M$ be a knot framed by $v_M$, and $\beta^+$ an isotopic knot whose $v_M$ framing is one greater, so in particular $[\beta^+] = a_M [\beta] \in \Sk(M)$.  Then 
    $$[\beta^+]_L = a_L^n [\beta]_L.$$
\end{proposition}
\begin{proof}
    We isotope $\beta, \beta^+$ so that they differ only in the standard neighborhood and in a region $\Omega$ in the complement of all flow graphs.  We take this region small enough to be arbitrarily close to the model geometry \ref{model geometry} and choose coordinates so that $\beta$ and $\beta^+$ are as follows. Let $\psi\colon \R\to[0,1]$ be a smooth function that equals $0$ on $[-2,2]$ and equals $1$ on $\R\setminus [-3,3]$ and let $\epsilon,\delta>0$ be small, then 
    \begin{itemize}
    \item $\beta$ is parameterized by
    \[
    t\mapsto (1,t,0),
    \]
    \item and $\beta^+$ is parametrized by
    \[
    t\mapsto \left((1-\psi(t))t^2 +\psi(t)\; ,\; t((1-\psi(t))(t^2-\delta^2) +\psi(t))\; ,\; (1-\psi(t))\epsilon t \right)
    \]
    \end{itemize}
    
    We take graphical perturbed conormals for $\beta, \beta^+$.  As in the proof of Proposition \ref{first skein relation}, the total curve count is given by gluing the basic cylinders for $\beta$ (resp. $\beta^+$) to flow graphs ending on $\beta$ (resp. $\beta^+$).  By general position, compare Lemma \ref{l : no graph at the intersection}, all flow graphs either attach outside our neighborhood, or lie entirely contained in our neighborhood.  
    The former kind are exactly the same for $\beta, \beta^+$, and, while they may not necessarily be transverse, their gluing problems are similarly identical so they will contribute identically. 
    Of the latter kind, there is only one: the interval $\{ (\delta^2, 0, \epsilon s)\,|\, -\delta\le s \le \delta\}$.  

    Let us fix some choice of external contributions (which again, will be the same in both cases), and analyze how the internal geometry contributes.  We now pass to the model geometry.  Let us write $\beta_1, \ldots, \beta_n$ for the lifts of $\beta$ to $L_i$; similarly $\beta^+_i$ for the corresponding lifts of $\beta^+$.  In addition to these trivial lifts, we have the contributions corresponding to $\beta^+$ with the attached flow line; there is one such contribution for each $i > j$.  The corresponding holomorphic curve has boundary approximately the union of $\beta_i$ and a standard-framed unknot $U_j$ on $L_j$; the curve itself has Euler characteristic one less (corresponding to the flow line) than the trivial contributions. 

    The $a$ power computation is (for small enough $\epsilon$) similar to that of
    Lemma \ref{unknot model geometry}. (See the end of Section \ref{sec:4chain and turning} for a detailed discussion of flow line gluing and framings). 
    We compute: 
    \begin{align*}
[\beta^+|_\Omega]_L &= 
\beta_1^+ a_2 \ldots a_n \\ & + 
\left( a_1^{-1} \beta_2^+ a_3 \ldots a_n + z \beta_2 U_1 a_2 \ldots a_n \right) \\
&+ \left( a_1^{-1} a_2^{-1} \beta_3^+ a_4 \ldots a_n + z \beta_3 (a_1^{-1} U_2 a_3 \ldots a_n + U_1 a_2 \ldots a_n) \right) \\ & + \cdots \\
&=
(a_1 \ldots a_n) (\beta_1 + \cdots + \beta_n) \\ &= (a_1 \ldots a_n) [\beta|_\Omega]_L
\end{align*}
    The above equality would make sense in a skein-with-boundary-marked-points for $(L|_\Omega, \partial L_\Omega)$, but what we really mean by it here is that it holds after attaching the remainder of $\beta$ and any fixed external flow graph contribution.  The result  follows.
\end{proof}

\subsection{A Symplectic Field Theory perspective}

Theorem \ref{skein trace} and the results it is built on can be used to express related skein-valued curve counts in many situations. 
We discuss some of them in this section but start out with the symplectic field theory perspective that underlies all such applications.

\subsubsection{Sketch of another proof} \label{sketch}
Let us note another approach to the skein trace result; the idea is to stretch along a cosphere bundle $S^*M \subset T^*M$ which separates $L$ from $\mathsf{U}_K$.  

Recall the Abbondondalo-Schwarz isomorphism comparing the symplectic homology of a cotangent bundle with the chains on the loop space of the base \cite{Abbondandolo-Schwarz}; recall it has also an $S^1$-equivariant version.  We will give elsewhere a higher genus analogue for three dimensional cotangent bundles, which compares a `bare curve' version of the degree zero (all genus) SFT of $S^*M$ with the skein module $\Sk(M)$.   
The isomorphism will be constructed by counting curves in $(T^*M, 0_M)$ that are asymptotic to Reeb orbits at infinity, by their boundary in the skein of $M$.  Accept for now (or regard as conjectural) this isomorphism; we denote it:
\[
\Omega: \mathrm{SFT}^0_{\mathrm{bare}}(S^*M) \xrightarrow{\sim} \Sk(M).
\]
Given this isomorphism, Theorem \ref{skein trace} can be deduced by SFT stretching from the well-definedness of skein-valued curve counting.
Indeed, consider the geometry $(T^*M, \mathsf{U}_K \sqcup L)$.  By definition, we set $[K]_L := \psi_{T^*M, L}(\mathsf{U}_K) \in \Sk(L)$.  However, let us now consider the geometry 
$(S^*M \times \R, \mathsf{U}_K)$.  Counting holomorphic curves here which may end at $\mathsf{U}_K$ and at the negative end of $S^*M \times \R$, we may define correspondingly 
$\psi_{S^*M \times \R}(\mathsf{U}_K) \in \mathrm{SFT}^0_{\mathrm{bare}}(S^*M)$.   Finally, let us define a similar map (now not an isomorphism)
$$\Omega_L: \mathrm{SFT}^0_{\mathrm{bare}}(S^*M) \to \Sk(L)$$
by counting holomorphic curves in $T^*M$ which are asymptotic to Reeb orbits and end on L. 

Now, by SFT stretching,
\begin{eqnarray*} 
    \psi_{T^*M, M} & =& \Omega \circ \psi_{S^*M} \\
    \psi_{T^*M, L} & =& \Omega_L \circ \psi_{S^*M} 
\end{eqnarray*} 
Now, 
$$\Omega \circ \psi_{S^*M}(\mathsf{U}_K)  = \psi_{T^*M, M}(\mathsf{U}_K) = [K] \in \Sk(M)$$
and, by definition, 
$$\Omega_L \circ \psi_{S^*M}(\mathsf{U}_K)  = \psi_{T^*M, L}(\mathsf{U}_K) = [K]_L \in \Sk(L).$$
Thus we find that the skein trace is given by the formula 
$$[K]_L = \Omega_L \Omega^{-1}[K]$$
and, in particular, $[K]_L$ depends only on the class of $K$ in the skein of $M$, as desired. 

\subsubsection{The skein trace map and Lagrangian cabling} 
Now, let $X$ be any Calabi-Yau 3-fold, and $M\subset X$ a (Maslov zero) Lagrangian submanifold.  Choose $L \subset T^\ast M$ a Lagrangian which branch covers $M$; after fiber rescaling, of the cotangent bundle, we may implant $L$ into $X$.  Fix, if desired, also some other Lagrangian $\Lambda \subset X$ away from a neighborhood of $M$ containing $L$.  Fix brane data for all involved objects, and assume that the 4-chain for $L$ is compatible with that of $M$, and somewhere standard.

\begin{corollary}\label{cor : skein trace and curves on nearby Lags}
The skein-valued curve counts $Z(X,\Lambda \sqcup M) \in \Sk(\Lambda) \otimes \Sk(M)$ and $Z(X,\Lambda \sqcup L) \in \Sk(\Lambda) \otimes \Sk(L)$ are related by the skein trace: 
$$Z(X,\Lambda \sqcup L)= [(Z(X,\Lambda \sqcup M))]_L.$$ 
\end{corollary}
\begin{proof}
Any bare holomorphic curve with boundary on $M$ is arbitrarily $C^1$-close to the basic cylinder on its boundary in a small neighborhood of the zero section. 
The result then follows from Lemma \ref{l : flow graph converegnce 3}.
\end{proof}

\begin{remark} 
    There is also an argument along the lines of the sketch above: we may count curves in $X \setminus M$ ending on $\Lambda$ and asymptotic to Reeb orbits in the negative $S^*M$ end; this will take value in $Z(X \setminus M; \Lambda) \in \Sk(\Lambda) \otimes \mathrm{SFT}^0_{bare}(S^*M)$.  Now, by SFT stretching,
    \begin{eqnarray*}
        Z(X,\Lambda \sqcup M) = \Omega (Z(X \setminus M; \Lambda)) \\
        Z(X,\Lambda \sqcup L) = \Omega_L (Z(X \setminus M; \Lambda))
    \end{eqnarray*}
    so the result follows by inverting $\Omega$.  
\end{remark}

\begin{example}\label{rmk:arbitrary-color}
For a framed knot $K \subset M$ and any skein $\sigma$ in the standard solid torus, we write $\sigma(K) \in \Sk(M)$ for the cabling of $K$ by $\sigma$, i.e.\ we implant the solid torus in a neighborhood of $K$ and take the image of $\sigma$.  It is well known that 
$W_{\mu, \emptyset}(K)$ is the skein counterpart of the coloring of a knot $K$ by the representation corresponding to the partition $\mu$.  

We recall from \cite{ekholm-shende-colored} the formula: 
\[
Z(T^\ast M, \mathsf{U}_K \sqcup M) = \sum_\mu W_{\mu,\emptyset} \otimes W_{\mu,\emptyset}(K) \in \Sk(\mathsf{U}_K) \otimes \Sk(M),
\]  
Applying Corollary \ref{cor : skein trace and curves on nearby Lags}, 
\[
Z(T^\ast M, \mathsf{U}_K \sqcup L) = \sum_\mu W_{\mu,\emptyset} \otimes [W_{\mu,\emptyset}(K)]_L \in \Sk(\mathsf{U}_K) \otimes \Sk(L),
\]  
This is a geometric interpretation of the skein traces of all colorings $W_{\mu, \emptyset}(K)$: they are obtained by extracting the coefficients of the curve count $Z(T^\ast M, \mathsf{U}_K \sqcup L)$ in the $W_{\lambda, \mu}$ basis of $\Sk(\mathsf{U}_K)$.  

This argument leads to an alternative proof of Theorem \ref{thm:coproduct} after we observe that the only holomorphic curves there are the two annuli with boundary in the two sheets and their multiple covers:
\begin{align*}
\sum_{\mu, \nu} W_{\mu,\emptyset} W_{\nu,\emptyset} \otimes W_{\mu,\emptyset}(K) \otimes W_{\nu,\emptyset}(K) &= \sum_{\lambda} W_{\lambda, \emptyset} \otimes \qty(\sum_{\mu, \nu} c_{\mu, \nu}^{\lambda} W_{\mu, \emptyset}(K) \otimes W_{\nu, \emptyset}(K)) 
\end{align*}
in $\Sk(\mathsf{U}_K) \otimes \Sk(\mathrm{Ann}) \otimes \Sk(\mathrm{Ann})$, where $K$ is the core of the thickened annulus. 
\end{example}

\section{Branched double covers, skein counts, and flow graphs}
This section considers holomorphic curves on Lagrangians associated to a branched double cover. In Sections \ref{ssec:SingLagofdc} and \ref{subsec:smooth-cone} we construct Lagrangians from branched double covers. In Section \ref{subsec: double cover 4-chain} we equip them with 4-chains. The Lagrangian give flow graphs which we determine in Section \ref{ssec:flowgraphsondc} . The flow graphs correspond to holomorphic curves and we give combinatorial formulas for the intersections of such curves with the 4-chain in Section \ref{sec:4chain and turning}. Finally in Section \ref{subsec:moduli-interpretation-leafspace}, we show that the foliation induced by an ideal triangulation has a direct interpretation in terms of moduli spaces of holomorphic disks with two positive punctures on the Lagrangian cobordism corresponding to a double cover.   

\subsection{Lagrangian double covers away from the cone points}\label{ssec:SingLagofdc}
We recall from \cite{Treumann-Zaslow, Schrader-Shen-Zaslow} certain constructions of Lagrangian double covers in $T^*M$ from fronts in $M$.  

Let us first consider the case of a surface $M$ and cubic graph $\Gamma \subset M$.  Then there is a singular hypersurface $F_\Gamma \subset M  \times \R$  such that (1) it is a trivial two sheet cover over $M \setminus \Gamma$, (2) it has transverse self-crossings over $\Gamma$, and (3) over the vertices of $\Gamma$, it has the $D_4^-$ singularity of Arnol'd. It is the front of a smooth Legendrian $\Lambda_\Gamma \subset J^1 M$, see Figure \ref{fig:Lagrangianbranchpt}. The figure also indicates a choice of perturbation of the Lagrangian to one whose front has generic singularities (only double points, cusps, and swallowtails), which we will later use in flow tree calculations.  

\begin{figure}
\centering
\setlength{\tabcolsep}{2em}
\begin{tabular}{c c c}
 & \underline{Symmetric cover} & \underline{Perturbed cover}\\
\noalign{\vskip 2em}
In $L$ & 
$
\vcenter{\hbox{
\begin{tikzpicture}[scale=0.7]
\draw[orange, very thick] (0, -2) -- (0, 2);
\draw[orange, very thick] ({-sqrt(3)}, -1) -- ({sqrt(3)}, 1);
\draw[orange, very thick] ({-sqrt(3)}, 1) -- ({sqrt(3)}, -1);
\filldraw[red] (0, 0) circle (0.1);
\end{tikzpicture}
}}
$
& 
$
\vcenter{\hbox{
\begin{tikzpicture}[scale=0.7]
\begin{scope}[shift={(0.6, 0)}]
\draw[orange, very thick] (0, -2) -- (0, 2);
\filldraw (0, 0) circle (0.05);
\node[right] at (0, 0){$1$};
\end{scope}
\begin{scope}[shift={(-0.3, {0.3*sqrt(3)})}]
\draw[orange, very thick] ({-sqrt(3)}, -1) -- ({sqrt(3)}, 1);
\filldraw (0, 0) circle (0.05);
\node[above left] at (0, 0){$2$};
\end{scope}
\begin{scope}[shift={(-0.3, {-0.3*sqrt(3)})}]
\draw[orange, very thick] ({-sqrt(3)}, 1) -- ({sqrt(3)}, -1);
\filldraw (0, 0) circle (0.05);
\node[below left] at (0, 0){$3$};
\end{scope}
\filldraw[red] (0, 0) circle (0.05);
\draw[very thick] (0, 0) circle (0.6);
\end{tikzpicture}
}}
$
\\
\noalign{\vskip 1em}
In $M$ & 
$
\vcenter{\hbox{
\begin{tikzpicture}[scale=0.7]
\draw[orange, very thick] (-2, 0) -- (0, 0);
\draw[orange, very thick] (0, 0) -- (1, {sqrt(3)});
\draw[orange, very thick] (0, 0) -- (1, {-sqrt(3)});
\filldraw[red] (0, 0) circle (0.1);
\end{tikzpicture}
}}
$ 
& 
$
\vcenter{\hbox{
\begin{tikzpicture}[scale=0.7]
\draw[orange, very thick] (-2, 0) -- (0.6, 0);
\draw[orange, very thick] (-0.3, {-0.3*sqrt(3)}) -- (1, {sqrt(3)});
\draw[orange, very thick] (-0.3, {0.3*sqrt(3)}) -- (1, {-sqrt(3)});
\filldraw[red] (0, 0) circle (0.05);
\draw [very thick] plot [domain=0:360, samples=200, smooth] ({0.2*(cos(2*\x)+2*cos(-\x))}, {0.2*(sin(2*\x)+2*sin(-\x))});
\filldraw (0.6, 0) circle (0.05);
\node[right] at (0.6, 0){$1$};
\filldraw (-0.3, {0.3*sqrt(3)}) circle (0.05);
\node[above left] at (-0.3, {0.3*sqrt(3)}){$3$};
\filldraw (-0.3, {-0.3*sqrt(3)}) circle (0.05);
\node[below left] at (-0.3, {-0.3*sqrt(3)}){$2$};
\end{tikzpicture}
}}
$
\end{tabular}
\caption{Left: the front of the Lagrangian near the vertex of a cubic vertex projected to the base.  Right: the front determining our choice of perturbed Lagrangian.
The projection of the perturbed Lagrangian to  the zero section  is modeled by the map $z \mapsto w=z^2 + \epsilon \bar{z}$. 
The circle $|z| = |\epsilon|/2$ is the critical set, and the deltoid $w = \frac{\epsilon^2}{4}(e^{i\theta} + 2e^{-i\theta})$ is the caustic. 
}
\label{fig:Lagrangianbranchpt}
\end{figure}
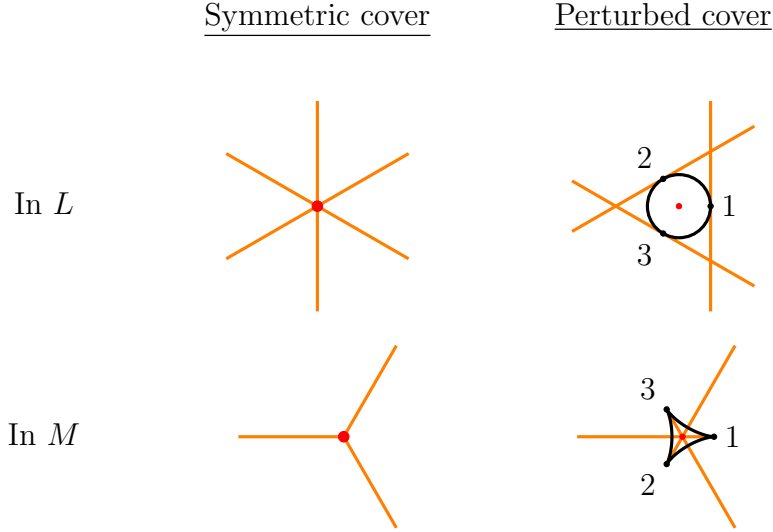

Assuming the connected components of $M \setminus \Gamma$ (the faces of the graph) are contractible, we may ensure $\Lambda_\Gamma$ has exactly one Reeb chord over each such component, corresponding to a maximum of the difference between the local functions whose jets give the two branches of the Legendrian.  So, the corresponding Lagrangian immersion similarly has one double point over each face.  This construction was described and explored in \cite{Treumann-Zaslow}.  
An analogous discussion goes through for any $M$ and any singular hypersurface $\Gamma \subset M$ which is locally a product (cubic graph in surface) times (smooth manifold).  

A motivating instance for this construction is the case when $\Gamma$ is the 1-skeleton $\mathbb{T}$ of a tetrahedron, viewed as a graph on a sphere.  In this case, the corresponding $$\Lambda_\mathbb{T} \subset J^1 (S^2 = \partial_\infty \R^3) \subset S^5 = \partial_\infty T^* \R^3 $$ is isotopic to the ideal contact boundary of the Harvey-Lawson cone $L_{\mathrm{HL}}$.  In coordinates, $L_{\mathrm{HL}}\approx \mathbb{T} \times\R$ 
is the cone on the torus 
$$
\Lambda_{\mathrm{HL};\theta} = \{(e^{i \theta_1}, e^{i\theta_2}, e^{i\theta_3})\, | \,  \theta_1 + \theta_2 + \theta_3 = \theta\} \subset S^5 \subset \C^3;
$$
we sometimes write $\Lambda_{\mathrm{HL}}:=\Lambda_{\mathrm{HL};\frac{\pi}{2}}$. 

Let us now turn to to a 3-manifold $M$ equipped with an ideal triangulation $\Delta$; we write $M^\circ$ for a complement of the neighborhood of the barycenters of the tetrahedra of $\Delta$.  We will write $\Delta^\vee$ for the dual `polyhedral' decomposition; note however that the maximal cells will be topologically of the form $S \times \R$ for some component $S$ of the ideal boundary $\partial_\infty M$.  
Consider the hypersurface $\Delta^\vee_{\le 2}$ given by the 2-skeleton of $\Delta^\vee$.  Then $(M^\circ, \Delta^\vee_{\le 2})$ is a manifold equipped with a hypersurface with cubic planar singularities, and so we may form the corresponding Lagrangian $L_{\Delta}^\circ := L_{\Delta^\vee_{\le 2}} \subset T^*(M^\circ)$.  We arrange that in each cell of $\Delta^\vee$, the difference of local defining functions has no maximum and grows as one approaches $\partial_\infty M$ and decreases as one approaches the vertices of $\Delta^\vee_{\le 2}$.  

We will also write $\Lambda_{\partial \Delta}$ for the Legendrian in $J^1 \partial_\infty M$ corresponding to the cubic graph given by $\Delta^\vee|_{\partial_\infty M}$.  We glue in Harvey-Lawson cones at the negative end and obtain a singular Lagrangian $L_{\Delta} \subset T^*M$.

We collect some properties of this construction: 

\begin{proposition}\label{prp: exact branched double cover}
    Let $M$ be a 3-manifold with ideal triangulation $\Delta$. 
    Let $(T^*M)^\circ$ be the complement in $T^*M$ of the barycenters of the tetrahedra of $\Delta$. We regard this as a Weinstein cobordism with positive end $\partial_\infty T^*M$ and negative end $\bigsqcup S^5\times[0-\infty)$.  

    Then $L_\Delta^\circ$ is an embedded exact Lagrangian cobordism with positive end $\Lambda_{\partial \Delta} \subset J^1 (\partial_\infty M) \subset \partial_\infty T^*M$ and negative end identified with $\bigsqcup \Lambda_{\mathbb{T}} \subset \bigsqcup S^5$. 
    
    Furthermore, the positive end $\Lambda_{\partial \Delta}$ is a two-component Legendrian link with components obtained by shifting $\partial_\infty M$ in the positive and negative Reeb direction. There are thus natural $\partial_\infty M$ Bott families of Reeb chords connecting the two shifted components. At the negative end, each component $\Lambda_{\mathbb{T}}\approx \Lambda_{\mathrm{HL}}$ has $\mathbb{T}$ family of minimal action Reeb chords of action $\frac{2\pi}{3}$, see \cite{Ekholm-Shende-unknot}.   
\end{proposition}
\begin{proof}
    By inspection of the above construction. 
\end{proof}

\subsection{Smoothing the cones}\label{subsec:smooth-cone}

The Harvey-Lawson cone $L_{\mathrm{HL}}$ has three distinguished smoothings to non-exact Lagrangian solid tori $L_{\mathrm{HL}}^\star$, known also as the `toric Aganagic-Vafa branes' for $\C^3$.  They are distinguished from each other by the 1-cycle in $\Lambda_{\mathrm{HL}}$ which is collapsed; with respect to some choice of basis of $H_1(\Lambda_{\mathrm{HL}},\Z)$, these cycles are $(0,1), (1,0), (-1,-1)$.  In the branched double cover setting, these three smoothings are topologically the smoothings we associated to a marking of the tetrahedron in Section \ref{ssec:smoothcoversfromtriang}. Furthermore, the pull-back of of the Liouville form on $\C^3$ to the smoothing $L_{\mathrm{HL}}^\star$ gives a $1$-form that is positive on the distinguished cycle of the tetrahedron. i.e., it gives an effective marking in the sense of Definition \ref{defn: positivity condition}.  

Consider a 3-manifold $M$ with a marked ideal triangulation $\Delta$ with corresponding smooth branch locus $\tau$ and corresponding distinguished cycle in each tetrahedron of $\Delta$, see Section \ref{sec: combinatorial criterion for 1-form}. Proposition \ref{prp: exact branched double cover} gives an exact Lagrangian cobordism $L^\circ_\Delta\subset(T^\ast M)^\circ$ with negative conical ends on the Legendrian tori $\Lambda_{\mathbb{T}}\subset S^5$ near the barycenter of each tetrahedron, identified with $\Lambda_{\mathrm{HL}}\subset S^5$. We think of $T^\ast M$ as obtained from $(T^\ast M)^\circ$ by gluing symplectic $\C^3$'s at the negative ends of $(T^\ast M)^\circ$.

We say that a Lagrangian $L_\tau\subset T^\ast M$ is a \emph{smoothing of $L_\Delta^\circ$ respecting $\tau$} if it agrees with smoothings $L_{\mathrm HL}^\star$ in balls $B_j(\delta)$ around the barycenter of each tetrahedron $\delta\in \Delta$ that are positive on the distinguished cycle over $\delta$ determined by the smoothing $\tau$, and if $L_\tau \setminus \bigcup_j B_j$ is graphical over $L_\Delta^\circ\setminus \bigcup_j B_j$. We then have the following result connecting smoothings and effective $1$-forms.

\begin{proposition} \label{prop: 1-form criterion}
The following are equivalent: 
\begin{enumerate}
    \item\label{itm:1-form criterion 1} 
    There is a smoothing $L_\tau$ of $L^\circ_\Delta$ respecting $\tau$.
    \item\label{itm:1-form criterion 2} 
    The marking of $\Delta$ that gives the smoothing $\tau$ is effective in the sense of Definition \ref{defn: positivity condition}.
\end{enumerate}
\end{proposition}
\begin{proof}
    If \eqref{itm:1-form criterion 1} holds then along $L_\Delta^\circ\setminus \bigcup_j B_j$, the desired 1-form in \eqref{itm:1-form criterion 2} is the $1$-form of the graphical submanifold $L_\tau\subset T^\ast L_\Delta^\circ$. Since $L_\tau$ is a smoothing of $L_\Delta^\circ$, the 1-form is a pull-back of a closed 1-form on $\Lambda_{\mathbb{T}}$ which then extends and gives a $1$-form positive on all distinguished cycles. 
    
    Conversely, given a $1$-form $\zeta$ positive on distinguished cycles we find, for all sufficiently small $\epsilon>0$, (graphical) exact Lagrangian isotopies near each negative end such that the 1-form $\epsilon\zeta$  agrees with the form of the smoothing $L_{\mathrm{HL}}^\star$ of the Harvey Lawson cone. Let $L_\tau$ be the Lagrangian that agrees with the smoothing of the Harvey Lawson cones near the negative ends and that outside the negative end is the graph of $\epsilon\zeta$ over $L_\Delta^\circ$. 
\end{proof}
    
Note, the only condition on the shift form $\zeta$ in each torus $\Lambda_{\mathbb{T}}$ is that, it vanishes along the collapsing cycle and is positive along the distinguished cycle. 
Therefore, \eqref{itm:1-form criterion 2} is equivalent to the existence of a closed 1-form $\zeta$ used in Definition \ref{defn: completed skein} to define a completion of the skein module.

\subsection{4-chains for branched double covers} \label{subsec: double cover 4-chain}

We will use a 4-chain associated to a local system $\xi_L$ of pairs of vector fields over the double cover $L$, as in \eqref{4-chain with monodromy} of Section \ref{sec: remarks 4-chain}.  To construct such a chain, it suffices to exhibit the local system of vector fields satisfying the hypotheses of Proposition \ref{existence of 4-chain}.

Recall that over a ball in $M$ where $\pi\colon L \to M$ is not branched, there are $\frac12 n(n-1)$ vector fields on $M$ given by 
$(\alpha_i - \alpha_j)^\ast$, $i<j$ where $\alpha_i$ is a local 1-form on $M$ whose graph is the $i^{\rm th}$ local sheet of $L$.  In the present case of interest, $n=2$, so there is only one such local vector field. In fact, outside the branch locus difference $1$-form is defined up to sign and hence gives a field of involutive line segments. Note that this field of line segments has a natural lift to a vector field $v$ on $L\setminus \widetilde{\tau}$: if for $p\in M$ if $=\pi^{-1}(p)=\{p_1,p_2\}\subset L$ then $v(p_1)=d\Gamma_{\alpha_1}(\alpha_1(p_1)-\alpha_2(p_2))^\ast$ and $v(p_2)=d\Gamma_{\alpha_2}(\alpha_2(p_2)-\alpha_1(p_1))^\ast$. In order to extend $v$ over the branch locus, fix a tubular neighborhood of the branch locus $\tau\subset M$. Let $r$ denote the radius of the tubular neighborhood and note that the difference vector fields decay as $\sqrt{r}$ as we approach $\tau$. In order to get a $C^k$-extension we scale the difference vector field by $r^{k+1}$ in the fiber direction of the tubular neighborhood. We write $\xi_{\mathrm{Mor}}$ for the corresponding vector field on $L$; extended by zero along the branch locus.  

We write $\overset{\leftrightarrow}{\xi}_{\mathrm{Mor}}$ for the field of involutive intervals on $M$ with positive and negative halves that pullback to $\xi_{\mathrm{Mor}}$ on $L$. 
We will use a small perturbation of $\xi_{\mathrm{Mor}}$ to construct a 4-chain for $L$ using Proposition \ref{existence of 4-chain}. Consider a $3$-manifold $M$ with a taut triangulation $\Delta$ and corresponding branched double cover $L\subset T^\ast M$. Then the line field $\overset{\leftrightarrow}{\xi}_{\mathrm{Mor}}$ is tangent to the foliation with leaf space given by the 2-cells of the polygonal decomposition dual to $\Delta$, see Section \ref{ssec:brancheddoublecoversfoliations}, or equivalently, the vector field ${\xi}_{\mathrm{Mor}}$ is tangent to the foliation of $L$ by boundaries of holomorphic strips in $\mathcal{M}(L^\circ_\Delta)$, see Section \ref{subsec:moduli-interpretation-leafspace}. 

\begin{lemma}
With $M$ and $\Delta$ as above, there exists a line field $\overset{\leftrightarrow}{\xi}_{\mathrm{pert}}$ on $M$ which has the following properties: 
\begin{itemize}
\item $\overset{\leftrightarrow}{\xi}_{\mathrm{pert}}$ is invariant under the flow of the foliation of $\Delta$ and induces a line field on the leaf space,
\item $\overset{\leftrightarrow}{\xi}_{\mathrm{pert}}$ is perpendicular to the faces of all tetrahedra in $\Delta$,
\item $\overset{\leftrightarrow}{\xi}_{\mathrm{pert}}$ is everywhere non-vanishing (see Remark \ref{r: angle condition and framing lines}),
\item $\overset{\leftrightarrow}{\xi}_{\mathrm{pert}}$ is parallel to the branch locus $\tau$,
\item away from $\tau$ and the edges of $\Delta$, $\overset{\leftrightarrow}{\xi}_{\mathrm{pert}}$ is nowhere parallel to $\overset{\leftrightarrow}{\xi}_{\mathrm{Mor}}$, 
\end{itemize}
\end{lemma}
\begin{proof}
We may define the line field in the complement of the edges using only the condition that the dihedral angles of each individual tetrahedron are $0,0, \pi$.  We do so as follows: 
Choose one of the two $\pi$-edges and call the two adjacent faces ``incoming'' and call the other two faces ``outgoing''. 
Then, upon choosing an orientation, the line field $\overset{\leftrightarrow}{\xi}_{\mathrm{pert}}$ looks like a vector field coming in through the incoming faces and going out through the outgoing faces; see Figure \ref{fig:line_field}.  It is clear from the picture that we may choose the line field parallel to the branch locus.  
This line field glues along the faces to give a line field on $M$ minus the edges.
Finally, the line field extends over the edges as a nonzero line field if the sum of angles around edges is $2\pi$
, i.e., when the ideal triangulation is taut. 
\end{proof}
\begin{remark}\label{r: angle condition and framing lines}
If the condition on the triangulation $\Delta$ that the angle sum around edges is relaxed from being $2\pi$ to being just a multiple of $\pi$, then the line field vanishes along edges with angle sum $\ne 2\pi$, and there is a weighted framing line (compare \eqref{eq:skeinrel6}) there. 
\end{remark}

\begin{figure}
\centering
\begin{math}
\vcenter{\hbox{
\tdplotsetmaincoords{65}{52}
\begin{tikzpicture}[tdplot_main_coords]
\begin{scope}[scale = 0.7, tdplot_main_coords]
    \newcommand*{\defcoords}{
        \coordinate (o) at (0, 0, 0);
        \coordinate (a) at (3, 0, 0);
        \coordinate (b) at ({3*cos(90)}, {3*sin(90)}, 0);
        \coordinate (c) at ({3*cos(2*90)}, {3*sin(2*90)}, 0);
        \coordinate (d) at ({3*cos(3*90)}, {3*sin(3*90)}, 0);
        \coordinate (abc) at ($1/3*(a)+1/3*(b)+1/3*(c)$);
        \coordinate (abd) at ($1/3*(a)+1/3*(b)+1/3*(d)$);
        \coordinate (bcd) at ($1/3*(b)+1/3*(c)+1/3*(d)$);
        \coordinate (acd) at ($1/3*(a)+1/3*(c)+1/3*(d)$);
    }
    \defcoords
    \draw[very thick] (a) -- (b);
    \draw[very thick] (b) -- (c);
    \draw[very thick] (c) -- (d);
    \draw[very thick] (d) -- (a);
    \draw[very thick] (a) -- (c);
    \draw[very thick, dotted] (a) -- (3, 0, 3);
    \draw[very thick, dotted] (b) -- ({3*cos(90)}, {3*sin(90)}, 3);
    \draw[very thick, dotted] (c) -- ({3*cos(2*90)}, {3*sin(2*90)}, 3);
    \draw[very thick, dotted] (d) -- ({3*cos(3*90)}, {3*sin(3*90)}, 3);
    \filldraw (a) circle (0.05em);
    \filldraw (b) circle (0.05em);
    \filldraw (c) circle (0.05em);
    \filldraw (d) circle (0.05em);
    \draw[white, line width=5] (abc) -- ($(abd)+(0,0,3)$);
    \draw[ultra thick, red] (abc) -- ($(abd)+(0,0,3)$);
    \filldraw[red] (abc) circle (0.2em);
    \draw[white, line width=5] (acd) -- ($(bcd)+(0,0,3)$);
    \draw[ultra thick, red] (acd) -- ($(bcd)+(0,0,3)$);
    \filldraw[red] (acd) circle (0.2em);
    \begin{scope}[shift={(0, 0, 3)}]
        \defcoords
        \draw[very thick] (a) -- (b);
        \draw[very thick] (b) -- (c);
        \draw[very thick] (c) -- (d);
        \draw[very thick] (d) -- (a);
        \draw[very thick] (b) -- (d);
        \filldraw (a) circle (0.05em);
        \filldraw (b) circle (0.05em);
        \filldraw (c) circle (0.05em);
        \filldraw (d) circle (0.05em);
        \filldraw[red] (abd) circle (0.2em);
        \filldraw[red] (bcd) circle (0.2em);
    \end{scope}
    \draw[very thick, <->] (-2.5, -2.5, 0) -- (-2.5, -2.5, 3);
    \node[anchor=east] at (-2.5, -2.5, 1.5){$\overset{\leftrightarrow}{\xi}_{\mathrm{pert}}$};
    \node[anchor=north] at (0, 0, 0){$\pi$};
    \node[anchor=south] at (0, 0, 3){$\pi$};
\end{scope}
\end{tikzpicture}
}}
\end{math}
\caption{The line field $\overset{\leftrightarrow}{\xi}_{\mathrm{pert}}$, and the branch locus of the double cover (drawn in red).}
\label{fig:line_field}
\end{figure}
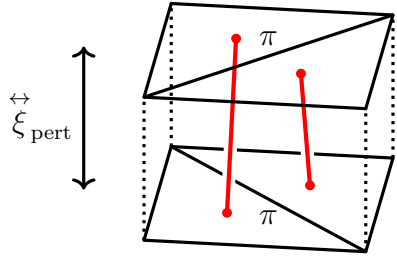

Because $\xi_{\mathrm{pert}}$ is parallel to the branch locus $\tau$, it lifts to a line field on $L$ which is tangent to $\widetilde{\tau}$. 

Consider an ideal tetrahedron $\delta\in \Delta$ and its preimage in $\widetilde{\delta}\subset L$. Let $\overset{\rightarrow}{\xi}_{\mathrm{pert}}$ and $\overset{\leftarrow}{\xi}_{\mathrm{pert}}$ denote the two vector fields with opposite orientations on $\widetilde{\delta}$ that lift the line field $\overset{\leftrightarrow}{\xi}_{pert}$ on $\delta$.
Let 
$$
\xi_{\mathrm{pert}} := \tfrac{1}{2}(\overset{\rightarrow}{\xi}_{\mathrm{pert}} \oplus \overset{\leftarrow}{\xi}_{\mathrm{pert}})
$$ 
be the \emph{formal} average of these two vector fields. 
Then $\xi_{\mathrm{pert}}$ glues smoothly over the faces of tetrahedra in $\Delta$ to give a local system of two vector fields along $L$. 

Define $C_L = C_{\xi_L}$ to be the compatible 4-chain of $L$ constructed as in Proposition \ref{existence of 4-chain} from the vector field 
$\pm\left({\xi}_{\mathrm{Mor}} \pm \epsilon{\xi}_{\mathrm{pert}}\right)$, i.e., so that the pushforward vector fields of the $4$-chain along $L$ is the system of vector fields 
\[
\xi_{L} := \xi_{\mathrm{Mor}} + \epsilon\xi_{\mathrm{pert}}. 
\]
Here we require that $\epsilon > 0$ is sufficiently small in the following sense. Fix a tubular neighborhood of $\widetilde{\tau}$ such that the tangent vector of the lift of any flow tree near the branch locus is linearly independent from $\xi_{\mathrm{pert}}$, then take $\epsilon>0$ sufficiently small compared to the separation of the two branches of the double cover outside the tubular neighborhood, so that here $|\epsilon \xi_{\mathrm{pert}}| \ll |\xi_{\mathrm{Mor}}|$.

\begin{rmk}\label{rmk:transverse-orientation}
If we are given a taut ideal triangulation in the sense of Lackenby, i.e., if we are also given a compatible choice of transverse orientation on each tetrahedron, then the local system of vector fields constructed above has a global section, namely the vector field $\vec{\xi}_{\mathrm{pert}}$ along the direction of the transverse orientation. 
In this case, we can simply use the vector field $\vec{\xi}_{\mathrm{pert}}$ to get a genuine 4-chain (instead of a local system of it). 
\end{rmk}

\subsection{Intersections of holomorphic curves with the 4-chain}\label{sec:4chain and turning}
In this section, we give a diagrammatic calculation of the $a$-powers in the skein lifting map. To this end, we consider a holomorphic curve $u$ in $T^\ast M$ with boundary on $M$. Assume that $M$ has an ideal triangulation $\Delta$ and let $L$ be the corresponding branched double cover. Then by Lemma \ref{l : flow graph converegnce 3}, $u$ gives rise to holomorphic curves $\widetilde{u}$ with boundary on $L$ which are obtained from $u$ by attaching flow graphs to $u$. Here we will give a combinatorial formula for the intersection numbers between $\widetilde{u}$ and the 4-chain $C_L$ constructed in Section \ref{subsec: double cover 4-chain} provided $L$ is sufficiently close to $M$. 

Note that there are three different local models for points in $\widetilde{u}$:
\begin{itemize}
\item points where $\tilde u$ is simply a lift of $u$, 
\item points where $\tilde u$ agrees with the holomorphic strip of a flow line, 
\item and junction points where a flow line is joined to $u$.  
\end{itemize}

We consider first simple lifts. Consider an arc $\gamma$ in the boundary $\partial u$ that maps to a region in a tetrahedron of $\Delta$ in the complement of the 3d spectral network, see Section \ref{ssec:brancheddoublecoversfoliations}. We pick local coordinates $x=(x_1,x_2,x_3)\in\R^2\times\R$ in this region of the tetrahedron in such a way that the foliation corresponds to the foliation of lines in the $x_3$-direction (with constant flow speed around $\gamma$) and the leaf space corresponds to the $x_1x_2$-plane. We write $\pi_{\R^2}$ for the projection to the $x_1x_2$-plane and $\pi_\R$ for the projection to the $x_3$-axis. The coordinates $x$ give corresponding coordinates $(x,y)$ on the cotangent bundle $T^\ast M$, where $y=(y_1,y_2,y_3)$ are momenta dual to $x$.  

With this choice of coordinates we then have the two sheets of the $\lambda$-fiber scaled Lagrangian $\lambda\cdot L$ 
given by
\[
\lambda\cdot L_{\pm} = \Gamma_{\pm\lambda\alpha},\quad \alpha=\lambda \left(\pm\tfrac12\delta dy_3 + \mathcal{O}(\lambda)\right),
\]
for some $\delta>0$, see Section \ref{ssec:brancheddoublecoversfoliations}. On $\lambda\cdot L_\pm$, the associated vector field, the dual of the difference covector pushed forward, is
\[
\xi_{\mathrm{Mor}} = \lambda\left(\pm\delta\partial_{x_3} + \mathcal{O}(\lambda)\right).
\]

Consider next the perturbation vector field $\xi_{\mathrm{pert}}$. Since it is invariant under the flow of the foliation and is small compared to $\xi_{\mathrm{Mor}}$, we write
\[
\xi_{\mathrm{pert}} = \pm\lambda\epsilon\; v(x_1,x_2), 
\]
where $\epsilon\ll\delta$ and where $v(x_1,x_2)$ is a unit length vector field on $\R^2$ pulled back to $\R^3$.

The framing vector field (i.e. the 4-chain vector field) on $L\pm$ is then given by: 
\[
\xi_{L_\pm} = \xi_{\mathrm{Mor}} + \xi_{\mathrm{pert}}=\lambda\left(\pm\delta\partial_{x_3} +\epsilon v +\mathcal{O}(\lambda)\right).
\]
In fact, we have a local system of two such vector fields, the second component is obtained by changing the sign of $v$, see Section \ref{subsec: double cover 4-chain}.
Let $C_L=C_{L_+}+C_{L_-}$ be the compatible 4-chain with $\partial C_{L_\pm}=2L_\pm$ constructed as in Section \ref{existence of 4-chain} using this data. Consider a holomorphic curve $u$ with boundary arc $\gamma\subset\partial u$ on $M$ and the corresponding lifts $\widetilde{u}_\pm$ with boundary in $L_\pm$.

\begin{lemma}\label{l: a power easy lift}
For all $\lambda>0$ sufficiently small, there is exactly one intersection point in $\widetilde{u}_\pm\cap C_{L_\mp}$ in a neighborhood of any point on $\gamma$ where the projection of $\gamma$ to the leaf space $\R^2$ has tangent vector $\pi_{\R^2}\dot\gamma$ equal to a negative multiple of $v$, and no other intersections. Furthermore, the orientation sign of the intersection point in $\widetilde{u}_\pm\cap C_{L_\mp}$ agrees with the sign of the tangency, i.e. the orientation sign of $\pi_{\R^2}\ddot\gamma\wedge \pi_{\R^2}\dot\gamma$ in the orientation of $\R^2$ induced by $\xi_{\mathrm{Mor}}=\pm\delta\partial_{x_3}$ and the orientation of $L_\pm$. 
\end{lemma}

\begin{proof}
Parameterize $u$ by $\sigma+i\tau$ in a neighborhood of the origin in the upper half plane, so that its boundary curve $\gamma$ has projected tangent vector field
$\pi_{\R^2} (\dot \gamma(\sigma))$ of length $1$. 

Near $\lambda\cdot L_{\pm}$, the holomorphic curve is then as follows in $\R^3_x\times \R^3_y$-coordinates on $T^\ast M$
\[
\tilde{u}_\pm (\sigma,\tau)= \left(\gamma(\sigma), (0,0,\pm\lambda\delta/2) + \left(\tau \dot\gamma(\sigma)+\mathcal{O}(\tau^2)\right)\right)+\mathcal{O}(\lambda^2).
\]
Similarly, the 4-chain $C_{L_\pm}$ is
\begin{align*}
C_{L_\pm} &= \left\{\left((x_1,x_2,x_3), \pm (0,0,\lambda\delta/2) + 
t\lambda\left((\epsilon v, \pm\delta)+\mathcal{O}(\lambda)\right)\right);\;t\ge 0\right\} \\ 
&\cup 
\left\{\left((x_1,x_2,x_3), \pm (0,0,\lambda\delta/2) + 
t\lambda\left(-(\epsilon v, \pm \delta)+\mathcal{O}(\lambda)\right)\right);\;t\ge 0\right\},
\end{align*}
oriented so that
\[
\partial C_{L_{\pm}} = 2L_{\pm}. 
\]

The intersections of $\widetilde{u}_\pm$ and $C_{L_\mp}$ then correspond to solutions of the following equation 
\[
(0,0,\pm\lambda\delta/2) + \tau \dot\gamma(\sigma) +\mathcal{O}(\tau^2)= (0,0,\mp\lambda\delta/2)+t\lambda\left(\beta(\epsilon v ,\mp\delta)+\mathcal{O}(\lambda)\right),\quad t\ge 0,
\] 
where $\beta=\pm 1$. For any solution $\tau=\mathcal{O}(\lambda)$, and hence 
for all sufficiently small $\lambda>0$ solutions correspond to solutions to  
\begin{align*}
\tau \pi_{\R^2}\dot\gamma(\sigma)  &= t\lambda\beta\epsilon \;v(\pi(\gamma(\sigma))), \\
\tau \pi_{\R}\dot\gamma(\sigma)  &= -\lambda\beta\delta t - \lambda\delta.
\end{align*}
From the first equation, we see that $\tau = \lambda\epsilon |t|$. The second equation then gives $\beta=-1$: if $\beta=1$ the right hand side is of size $|\lambda\delta t|$, whereas the left hand side is of the smaller size $|\lambda\epsilon t|$. For $\beta=-1$ we get   
\[
\tau \pi_{\R^2}\dot\gamma  = -t \lambda\epsilon v,\quad\quad
\tau \pi_{\R} \dot\gamma = \lambda\delta(t-1) .
\]
It follows that any point where $\pi_{\R^2}\dot\gamma$ is negatively parallel to $v$ is an intersection point, take
\[
\tau=t\lambda\epsilon>0, \quad t=\frac{\lambda\delta}{\lambda\delta-\lambda\epsilon\pi_\R\dot\gamma} >0.
\]
Note, that the intersection lies very close to the boundary of the holomorphic curve, close to the boundary point where $\pi_{\R^2}\dot\gamma$ is negatively parallel to $v$;
see Figure \ref{fig:intersection_points}. 
\begin{figure}
\centering
\begin{tikzpicture}
\node at (-1.2, 1.4){$\otimes$};
\node[anchor=south] at (-1.2, 1.5){$\xi_{\mathrm{Mor}}$};
\node[anchor=south] at (0, 1.4){$\xi_{\mathrm{pert}}$};
\draw[->] (-0.2, 1.4) -- (0.2, 1.4);
\draw[very thick] (2, 0) to[out=90, in=0] (1.5, 1.2) to[out=180, in=0] (0.7, 0.3) to[out=180, in=0] (0, 1) to[out=180, in=90] (-0.5, 0.3) to[out=-90, in=180] (0.8, -1) to[out=0, in=-90] (2, 0);
\draw[very thick, ->] (2, 0) -- (2, 0);
\node[anchor=west] at (2, 0){$\pi_{\mathbb{R}^2}\gamma$};
\filldraw[gray] (0, 1) circle (0.07);
\filldraw[gray] (0.7, 0.3) circle (0.07);
\filldraw[gray] (1.5, 1.2) circle (0.07);
\node[gray, anchor=north] at (0, 0.9){$+1$};
\node[gray, anchor=north] at (0.7, 0.3){$-1$};
\node[gray, anchor=north] at (1.5, 1.1){$+1$};
\end{tikzpicture}
\caption{Projected to the leaf space of the Morse flow, intersection points with the 4-chain are in 1-1 correspondence with the points of $\gamma$ where the tangent vector is negatively parallel to the perturbation vector field.}
\label{fig:intersection_points}
\end{figure}
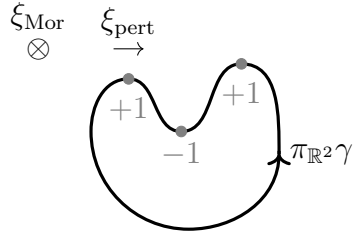
The sign of the intersection point is given by the sign of the orientation of
\[
\partial_{\sigma}u \wedge \partial_{\tau} u \wedge J\xi_{\mathrm{Mor}}
\]
at that point, which is the same as the sign of the orientation of
\[
\pi_{\R^2}\ddot\gamma  \wedge \pi_{\R^2}\dot\gamma  \wedge \xi_{\mathrm{Mor}},
\]
see Figure \ref{fig:intersection_signs}. 
\begin{figure}
\centering
\begin{tikzpicture}
\node at (-1.2, 1.0){$\otimes$};
\node[anchor=south] at (-1.2, 1.1){$\xi_{\mathrm{Mor}}$};
\node[anchor=south] at (-1.2, 0.2){$\xi_{\mathrm{pert}}$};
\draw[->] (-1.2, 0.2) -- (-1.2, -0.2);
\draw[very thick, ->] ({-2+sqrt(3)},-1) arc (-30:30:2);
\node[anchor=west] at ({-2+sqrt(3)},1){$\pi_{\mathbb{R}^2}\gamma$};
\filldraw[gray] (0, 0) circle (0.07);
\node[gray, anchor=west] at (0, 0){$+1$};
\draw[very thick, ->] ({2+2-sqrt(3)},-1) arc (210:150:2);
\node[anchor=east] at ({2+2-sqrt(3)},1){$\pi_{\mathbb{R}^2}\gamma$};
\filldraw[gray] (2, 0) circle (0.07);
\node[gray, anchor=east] at (2, 0){$-1$};
\end{tikzpicture}
\caption{Sign of the intersection points}
\label{fig:intersection_signs}
\end{figure}
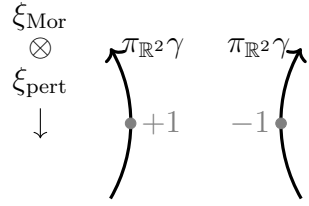
The lemma follows.
\end{proof}

\begin{remark}
In Lemma \ref{l: a power easy lift}, we computed intersections from one local 4-chain. Our actual 4-chain is built from a local system of vector fields and the intersection is the weighted sum of two contributions. 
\end{remark}

We next consider contributions from Morse flow line parts of the lift $\widetilde{u}$. Let $u_{\mathrm{strip}}$ denote the holomorphic strip close to a flow line in $\widetilde{u}$. 
\begin{lemma}\label{l:strip4chain}
    For all $\lambda>0$ sufficiently small, $C_L\cap u_{\mathrm{strip}}=\emptyset$.
\end{lemma}
\begin{proof}
The holomorphic strip is within distance $\mathcal{O}(\lambda^2)$ of the fiber strip over the flow line, with fibers in direction $\xi_{\mathrm{Mor}}^\ast+\mathcal{O}(\lambda)$. Since $C_L$ has inner normal $\xi_{\mathrm{Mor}}+\epsilon v$, where $v$ is not parallel to $\xi_{\mathrm{Mor}}$, the lemma follows.
\end{proof}

We finally consider the contribution at junction points, where a flow line strip $u_{\mathrm{strip}}$ is attached to $u$. Let $\gamma$ denote the are in $\partial u$ where $u_{\mathrm{strip}}$ is attached.
\begin{lemma}\label{l: junction4chain}
    If the tangent $2$-plane containing $\dot\gamma$ and $\overset{\leftrightarrow}{\xi}_{\mathrm{Mor}}$ at the junction point does not contain $\overset{\leftrightarrow}{\xi}_{\mathrm{pert}}$ then there are no $4$-chain intersections of $\widetilde{u}$ near the junction point. 
\end{lemma}
\begin{proof}
The holomorphic curve $\widetilde{u}$ with boundary on $L$ close to this configuration is constructed by Floer gluing as in \cite{Ekholm-morse}. It follows in particular that $\widetilde{u}$ is arbitrarily close to the pregluing of the explicit local models of the curves given in \cite[Section 6]{Ekholm-morse}. Thus, outside a small neighborhood of the point where the flow line is joined to the curve, the tangent vectors of the curve on $L$ are proportional to $\xi_{\mathrm{Mor}}$ and $\dot\gamma$, where $\gamma$ is the arc in $u$ where $u_{\mathrm{strip}}$ is attached. Inside a small region, the tangent vector is arbitrarily close to the linear interpolation of these vectors and in particular lies arbitrarily close to the plane they span. It follows that, as long as $\xi_{\mathrm{pert}}$ at the junction point does not lie in this plane, there are no 4-chain intersections near the boundary. 
\end{proof}

The results above lead to the following combinatorial description of the $a$-power of the skein lift.
\begin{corollary}
    The $a$-power of the skein lift is given by the turning number of the lifted curve projected to the leaf space. 
\end{corollary}
\begin{proof}
Contributions from points where $\tilde u$ is simply a lift of $u$, and points where $\tilde u$ agrees with the holomorphic strip of a flow line are given by the turning number by Lemmas \ref{l: a power easy lift} and \ref{l:strip4chain}. Consider next a junction point, by Lemma \ref{l: junction4chain} there are no local 4-chain intersections near the junction point. However, as the boundary of the glued curve is arbitrarily close to the pregluing we find that there are kinks in the projection to the leaf space. It is straightforward to check that these give the same contribution as turning numbers of smoothings, see Remark \ref{r:framingturning}.  
\end{proof}

\begin{remark}\label{r:framingturning}
As in the proofs above, to check the relation between actual gluing and smoothing near junctions for all small $\lambda>0$, it suffices to carry out the calculation in a model situation. We take $\R^3$ coordinates with leaf space $\R^2$ of the first two coordinates and use the model described in Figure \ref{fig:junctions}. 
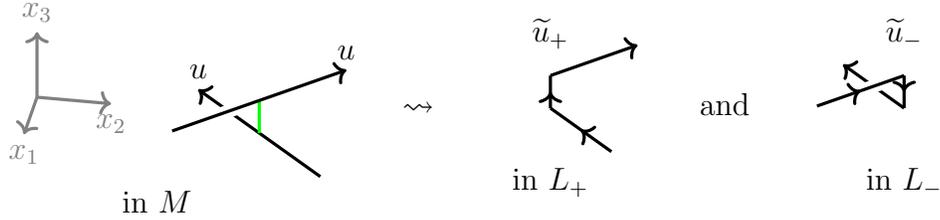
\begin{figure}
\centering
\[
\vcenter{\hbox{
\tdplotsetmaincoords{60}{100}
\begin{tikzpicture}[tdplot_main_coords]
\newcommand*{\defcoords}{
    \coordinate (o) at (0, 0, 0);
    \coordinate (x1) at (1, 0, 0);
    \coordinate (x2) at (0, 1, 0);
    \coordinate (x3) at (0, 0, 1);
    \coordinate (UI) at (1, 1, -0.25);
    \coordinate (UO) at (-1, -1, -0.25);
    \coordinate (OI) at (1, -1, 0.25);
    \coordinate (OO) at (-1, 1, 0.25);
    \coordinate (UM) at ($1/2*(UI) + 1/2*(UO)$);
    \coordinate (OM) at ($1/2*(OI) + 1/2*(OO)$);
}
\defcoords
\draw[gray, very thick, ->] (o) -- (x1);
\draw[gray, very thick, ->] (o) -- (x2);
\draw[gray, very thick, ->] (o) -- (x3);
\node[gray, below] at (x1){$x_1$};
\node[gray, below] at (x2){$x_2$};
\node[gray, above] at (x3){$x_3$};
\node[right] at (0, 1, -1.5){$\text{in }M$};
\begin{scope}[shift={(0, 3, 0)}]
    \defcoords
    \draw[very thick, ->] (UI) -- (UO);
    \draw[white, line width = 5] (OI) -- (OO);
    \draw[green, very thick] (UM) -- (OM);
    \draw[very thick, ->] (OI) -- (OO);
    \node[above] at (UO){$u$};
    \node[above] at (OO){$u$};
\end{scope}
\end{tikzpicture}
}}
\;\;\rightsquigarrow\;\;
\vcenter{\hbox{
\tdplotsetmaincoords{60}{100}
\begin{tikzpicture}[tdplot_main_coords]
\newcommand*{\defcoords}{
    \coordinate (o) at (0, 0, 0);
    \coordinate (x1) at (1, 0, 0);
    \coordinate (x2) at (0, 1, 0);
    \coordinate (x3) at (0, 0, 1);
    \coordinate (UI) at (1, 1, -0.25);
    \coordinate (UO) at (-1, -1, -0.25);
    \coordinate (OI) at (1, -1, 0.25);
    \coordinate (OO) at (-1, 1, 0.25);
    \coordinate (UM) at ($1/2*(UI) + 1/2*(UO)$);
    \coordinate (OM) at ($1/2*(OI) + 1/2*(OO)$);
}
\defcoords
\begin{scope}[very thick,decoration={markings, mark=at position 0.5 with {\arrow{>}}}] 
    \draw[postaction={decorate}] (UI) -- (UM);
    \draw[postaction={decorate}] (UM) -- (OM);
    \draw[->] (OM) -- (OO);
\end{scope}
\node[above] at (0, 0, 0.5){$\widetilde{u}_+$};
\node[below] at (0, 0, -1){$\text{in }L_+$};
\end{tikzpicture}
}}
\;\quad\text{ and }\;
\vcenter{\hbox{
\tdplotsetmaincoords{60}{100}
\begin{tikzpicture}[tdplot_main_coords]
\newcommand*{\defcoords}{
    \coordinate (o) at (0, 0, 0);
    \coordinate (x1) at (1, 0, 0);
    \coordinate (x2) at (0, 1, 0);
    \coordinate (x3) at (0, 0, 1);
    \coordinate (UI) at (1, 1, -0.25);
    \coordinate (UO) at (-1, -1, -0.25);
    \coordinate (OI) at (1, -1, 0.25);
    \coordinate (OO) at (-1, 1, 0.25);
    \coordinate (UM) at ($1/2*(UI) + 1/2*(UO)$);
    \coordinate (OM) at ($1/2*(OI) + 1/2*(OO)$);
}
\defcoords
\begin{scope}[very thick,decoration={markings, mark=at position 0.5 with {\arrow{>}}}] 
    \draw[->] (UM) -- (UO);
    \draw[white, line width = 5] (OI) -- (OM);
    \draw[postaction={decorate}] (OI) -- (OM);
    \draw[postaction={decorate}] (OM) -- (UM);
\end{scope}
\node[above] at (0, 0, 0.5){$\widetilde{u}_-$};
\node[below] at (0, 0, -1){$\text{in }L_-$};
\end{tikzpicture}
}}
\]
\caption{A local model for $\widetilde{u}$ with a holomorphic strip attached. Here, $\xi_{\mathrm{Mor}} = \pm \partial_{x_3}$ in $L_{\pm}$ and $\xi_{\mathrm{pert}} = -\epsilon \partial_{x_1}$.}
\label{fig:junctions}
\end{figure}
The same model projected along the framing vector field $\xi_{\mathrm{Mor}} + \xi_{\mathrm{pert}}$ is shown in Figure \ref{fig:junctions-a-power}. 
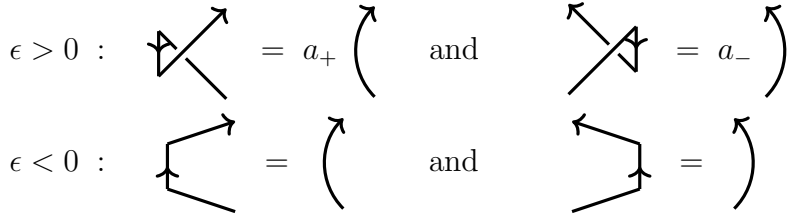
\begin{figure}
\centering
\begin{align*}
\epsilon > 0 \;: &\quad
\vcenter{\hbox{
\begin{tikzpicture}[scale=0.6]
\begin{scope}[very thick, decoration={markings, mark=at position 0.5 with {\arrow{>}}}] 
    \draw (1, -1) -- (-0.5, 0.5);
    \draw[white, line width=5] (-0.5, -0.5) -- (1, 1);
    \draw[->] (-0.5, -0.5) -- (1, 1);
    \draw[postaction={decorate}] (-0.5, 0.5) -- (-0.5, -0.5);
\end{scope}
\end{tikzpicture}
}}
\;=\;
a_+
\vcenter{\hbox{
\begin{tikzpicture}[scale=0.6]
\draw[very thick, ->] (1, -1) to[out=135, in=-135] (1, 1);
\end{tikzpicture}
}}
\quad\text{ and }
\qquad
\vcenter{\hbox{
\begin{tikzpicture}[scale=0.6]
\begin{scope}[very thick, decoration={markings, mark=at position 0.5 with {\arrow{>}}}] 
    \draw[->] (0.5, -0.5) -- (-1, 1);
    \draw[white, line width=5] (-1, -1) -- (0.5, 0.5);
    \draw (-1, -1) -- (0.5, 0.5);
    \draw[postaction={decorate}] (0.5, 0.5) -- (0.5, -0.5);
\end{scope}
\end{tikzpicture}
}}
\;=\;
a_-
\vcenter{\hbox{
\begin{tikzpicture}[scale=0.6]
\draw[very thick, ->] (-1, -1) to[out=45, in=-45] (-1, 1);
\end{tikzpicture}
}}
\\
\epsilon < 0 \;: &\quad 
\;
\vcenter{\hbox{
\begin{tikzpicture}[scale=0.6]
\begin{scope}[very thick, decoration={markings, mark=at position 0.5 with {\arrow{>}}}] 
    \draw (1, -1) -- (-0.5, -0.5);
    \draw[->] (-0.5, 0.5) -- (1, 1);
    \draw[postaction={decorate}] (-0.5, -0.5) -- (-0.5, 0.5);
\end{scope}
\end{tikzpicture}
}}
\;=\;
\vcenter{\hbox{
\begin{tikzpicture}[scale=0.6]
\draw[very thick, ->] (1, -1) to[out=135, in=-135] (1, 1);
\end{tikzpicture}
}}
\quad\quad\text{ and }
\qquad\;
\vcenter{\hbox{
\begin{tikzpicture}[scale=0.6]
\begin{scope}[very thick, decoration={markings, mark=at position 0.5 with {\arrow{>}}}] 
    \draw[->] (0.5, 0.5) -- (-1, 1);
    \draw (-1, -1) -- (0.5, -0.5);
    \draw[postaction={decorate}] (0.5, -0.5) -- (0.5, 0.5);
\end{scope}
\end{tikzpicture}
}}
\;=\;
\vcenter{\hbox{
\begin{tikzpicture}[scale=0.6]
\draw[very thick, ->] (-1, -1) to[out=45, in=-45] (-1, 1);
\end{tikzpicture}
}}
\end{align*}
\caption{The same model projected along the framing vector field $\xi_{\mathrm{Mor}} + \xi_{\mathrm{pert}}$}
\label{fig:junctions-a-power}
\end{figure}
Note that while the framing of the link $\partial \widetilde{u}_{\pm}$ changes depending on the sign of $\epsilon$, the extra $a$-powers are exactly compensated by the 4-chain intersections of $\widetilde{u}_{\mp}$: 
the extra $a$-powers appearing in the case $\epsilon > 0$ can be interpreted in two skein-equivalent ways, either as the framing factor for $\partial \widetilde{u}_{\pm}$ (left-hand sides) or as the 4-chain intersection of $\widetilde{u}_{\mp}$ (right-hand sides). 
This shows that, up to skein equivalence, the lifted link obtained by attaching a flow line can be taken to be the resolution of the crossing, in the projection along the Morse flow vector field $\xi_{\mathrm{Mor}}$. 
\end{remark}

\subsection{Spin structures} \label{double cover spin}

We will make the following spin structure choices.  We fix a spin structure on the base 3-manifold $M$.\footnote{A 3-manifold is spin if orientable.  We can generalize to the case of non-orientable $M$ as follows. 
On $T^*M$, there is a choice of symplectic background class coming from the Lagrangian polarization by cotangent fibers, and there is a canonical choice of relative spin structure on the base manifold with respect to the background class.  A spin structure on the base manifold gives an identification of this choice of background class with the standard symplectic background class.}
On $L$, we pull back the spin structure along $L \to M$.  This makes sense away from the branch locus, and defines a spin structure on the complement of the branch locus which does not extend across the branch locus. As noted in Section \ref{sec: remarks spin} above, we may still do skein-valued curve counting using this broken spin structure, but must correspondingly introduce a sign line into the skein of $L$ along the branch locus. 

The significance of this particular spin structure choice is that for flow graphs corresponding holomorphic curves along $L$, because both edges of the holomorphic curve corresponding to a given flow line are along pieces of Lagrangian with canonically identified spin structures, the signs introduced by the spin structure cancel, save possibly where the flow graph ends along image of the branch locus.

\subsection{Flow graphs on a branched double cover}\label{ssec:flowgraphsondc}

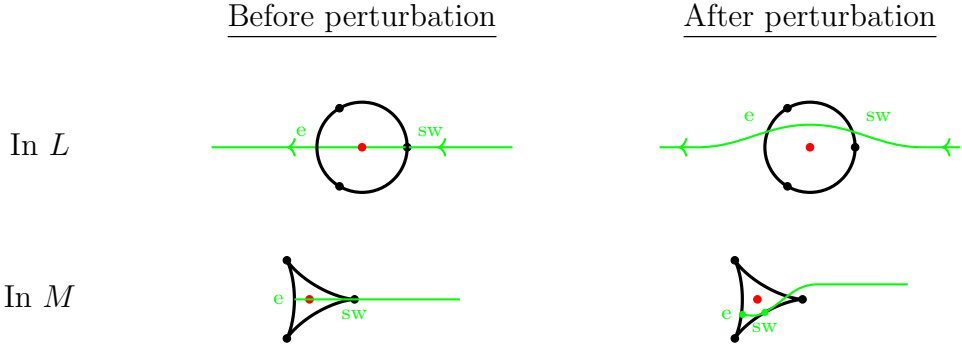
\begin{figure}
\centering
\setlength{\tabcolsep}{2em}
\begin{tabular}{c c c}
 & \underline{Before perturbation} & \underline{After perturbation}\\
\noalign{\vskip 2em}
In $L$ & 
$
\vcenter{\hbox{
\begin{tikzpicture}[scale=1.0]
\begin{scope}[shift={(0.6, 0)}]
\filldraw (0, 0) circle (0.05);
\end{scope}
\begin{scope}[shift={(-0.3, {0.3*sqrt(3)})}]
\filldraw (0, 0) circle (0.05);
\end{scope}
\begin{scope}[shift={(-0.3, {-0.3*sqrt(3)})}]
\filldraw (0, 0) circle (0.05);
\end{scope}
\begin{scope}[thick, decoration={markings, mark=at position 0.5 with {\arrow{>}}}]
    \draw[green, postaction={decorate}] (2, 0) -- (0, 0);
    \draw[green, postaction={decorate}] (0, 0) -- (-2, 0);
\end{scope}
\filldraw[red] (0, 0) circle (0.05);
\draw[very thick] (0, 0) circle (0.6);
\node[green, above left] at (-0.6, 0){\tiny $\mathrm{e}$};
\node[green, above right] at (0.6, 0){\tiny $\mathrm{sw}$};
\end{tikzpicture}
}}
$
& 
$
\vcenter{\hbox{
\begin{tikzpicture}[scale=1.0]
\begin{scope}[shift={(0.6, 0)}]
\filldraw (0, 0) circle (0.05);
\end{scope}
\begin{scope}[shift={(-0.3, {0.3*sqrt(3)})}]
\filldraw (0, 0) circle (0.05);
\end{scope}
\begin{scope}[shift={(-0.3, {-0.3*sqrt(3)})}]
\filldraw (0, 0) circle (0.05);
\end{scope}
\filldraw[red] (0, 0) circle (0.05);
\draw[very thick] (0, 0) circle (0.6);
\begin{scope}[thick, decoration={markings, mark=at position 0.5 with {\arrow{>}}}]
    \draw[green, postaction={decorate}] (2, 0) -- (1.5, 0);
    \draw[green] (1.5, 0) to[out=180, in=0] (0, 0.3) to[out=180, in=0] (-1.5, 0);
    \draw[green, postaction={decorate}] (-1.5, 0) -- (-2, 0);
\end{scope}
\node[green, above right] at (0.6, 0.2){\tiny $\mathrm{sw}$};
\node[green, above left] at (-0.6, 0.2){\tiny $\mathrm{e}$};
\end{tikzpicture}
}}
$
\\
\noalign{\vskip 2em}
In $M$ & 
$
\vcenter{\hbox{
\begin{tikzpicture}[scale=1.0]
\filldraw[red] (0, 0) circle (0.05);
\draw [very thick] plot [domain=0:360, samples=200, smooth] ({0.2*(cos(2*\x)+2*cos(-\x))}, {0.2*(sin(2*\x)+2*sin(-\x))});
\filldraw (0.6, 0) circle (0.05);
\filldraw (-0.3, {0.3*sqrt(3)}) circle (0.05);
\filldraw (-0.3, {-0.3*sqrt(3)}) circle (0.05);
\draw[green, thick] (-0.2, 0) -- (2, 0);
\node[green, left] at (-0.2, 0){\tiny $\mathrm{e}$};
\node[green, below] at (0.6, 0){\tiny $\mathrm{sw}$};
\end{tikzpicture}
}}
$
& 
$
\vcenter{\hbox{
\begin{tikzpicture}[scale=1.0]
\filldraw[red] (0, 0) circle (0.05);
\draw [very thick] plot [domain=0:360, samples=200, smooth] ({0.2*(cos(2*\x)+2*cos(-\x))}, {0.2*(sin(2*\x)+2*sin(-\x))});
\filldraw (0.6, 0) circle (0.05);
\filldraw (-0.3, {0.3*sqrt(3)}) circle (0.05);
\filldraw (-0.3, {-0.3*sqrt(3)}) circle (0.05);
\draw[green, thick] (2, 0.2) -- (0.8, 0.2) to[out=180, in=30] (0.1, {-0.1*sqrt(3)}) to[out=-150, in=-10] (-0.2, -0.2);
\filldraw[green] (-0.2, -0.2) circle (0.04);
\node[green, left] at (-0.2, -0.2){\tiny $\mathrm{e}$};
\filldraw[green] (0.1, {-0.1*sqrt(3)}) circle (0.04);
\node[green, below] at (0.1, {-0.1*sqrt(3)}){\tiny $\mathrm{sw}$};
\end{tikzpicture}
}}
$
\end{tabular}
\caption{Flow tree near the branch locus with a 2-valent switch vertex and a 1-valent end vertex. ($\mathrm{e}$ denotes an end and $\mathrm{sw}$ denotes a switch.)}
\label{fig:SwitchEnd}
\end{figure}

Let $M$ be a 3-manifold with a signed taut triangulation $\Delta$ with $t$ tetrahedra. Let $L^\circ\subset T^\ast M^\circ$ be the singular double cover as in Proposition \ref{prp: exact branched double cover}. This is an exact Lagrangian cobordism connecting $t$ Legendrian tori $\Lambda_{\mathbb{T}}$ at the barycenter of each tetrahedron in $\Delta$ in the negative end $\approx \bigsqcup S^5$ of $T^\ast M^\circ$ to $b$ Legendrian tori at the vertices of $\Delta$ in the positive end of $T^\ast M^\circ$.

Let $K\subset M$ be a link and let $\mathsf{U}_K\subset T^\ast M$ be its pushed-off conormal. By Lemma \ref{l : flow graph converegnce 3}, holomorphic curves ending on $\mathsf{U}_K \sqcup L^\circ$ with boundary in $\mathsf{U}_K$ that go once around the longitude are in 1-1 correspondence with configurations $\Gamma$ of the standard holomorphic cylinder between $\mathsf{U}_K$ and $M$ with flow graphs for $L^\circ \sqcup M$ on it.
\begin{lemma}\label{l:flowgraphs on Lcirc}
In any flow configuration $\Gamma$ of the basic cylinder of $\mathsf{U}_K$ with flow graphs of $L^\circ$ the flow graph does not have 3-valent vertices. Furthermore, the flow graphs attached are of two types
\begin{itemize}
    \item[$(i)$] Flow line segments connecting two points on $K$.
    \item[$(ii)$] Flow line segments with an end and a switch connecting a point of $K$ to the branch locus, see Figure \ref{fig:SwitchEnd}. 
\end{itemize}
\end{lemma}

\begin{proof}
By exactness $L^\circ$ has no flow loops, and since the multiplicity of $L^\circ\to M^\circ$ is two outside a small neighborhood of the branch locus, there can be no trivalent vertices. Therefore, in this region, all flow graphs are simply flow lines. It remains to understand flow graphs near the branch locus. It is straightforward to check that any such flow graph of dimension zero has one switch and one end and that there are three such flow graphs along the branch locus, see Figure \ref{fig:SwitchEnd}. It follows that rigid flow graph components are either flow lines connecting distinct points on $K$ or intersections of $K$ with the flow manifolds of flow lines with a switch and an end emanating from the branch locus.     
\end{proof}

Assume next that the ideal triangulation $\Delta$ of $M$ is equipped with signed taut structure that satisfies the criterion of Proposition \ref{prop: existence of smoothing}. Let $L_\tau \subset T^\ast M$ denote the smoothing of $L^\circ$ with these markings and periods. 
\begin{theorem} \label{explicit description} 
      Fix the brane data on $L_\tau$ constructed as in Sections \ref{subsec: double cover 4-chain} and \ref{double cover spin}.  Then the maps from Theorems \ref{skein trace} and \ref{thm:UV-IRmap smooth} agree, i.e., for all knots $K \subset M$, we have
    $$[K]_{L} = [K]^{\mathrm{net}}_{L}  \in \Sk(L)$$
\end{theorem}
\begin{proof}
We will compute $[K]_{L}$ directly from its definition in the flow tree limit. Let $\mathsf{U}_K$ be the pushed-off conormal; we count holomorphic curves ending on $\mathsf{U}_K \sqcup L$ whose boundary in $\mathsf{U}_K$ goes once around the longitude. 

We first consider the cobordism $L^\circ$ with negative ends. By Lemmas \ref{l : flow graph converegnce 3} and \ref{l : simple implies simple}, there is 1-1 correspondence between flow graphs on the holomorphic cylinder of $\mathsf{U}_K$ and holomorphic curves with boundary on $L^\circ$. By Lemma \ref{l:flowgraphs on Lcirc} the flow graphs are exactly the configurations that contribute to $[K]^{\mathrm{net}}_L$. We need to check signs. The sign of a glued curve can be calculated from a gluing sequence as in \cite{EENS}, as there this calculation reduces to a comparison of evaluations of kernel/cokernel elements at the junction points and the resulting sign is the crossing sign for exchanges and the crossing sign with the detour manifold at detours. 

For exchanges, the kernel of the marked points gives the tangent vectors of the curve branches at the crossing and the kernels of the strip of the flow line gives the vertical vector. 

For detours, we use the following model. We think of the critical leaves near the branch locus as a moduli space of disks with Reeb chord asymptotics, see Figure \ref{fig:flow lines and disks}, where we stretch around the perturbed branch locus. 

We orient this moduli space as follows. Choosing an orientation of the branch locus orients the normal bundle of the branch locus which in turn gives oriented zero dimensional disks, see Figure \ref{fig:flow lines and disks}. We then orient the 1-dimensional moduli space using the branch locus orientation again. Note that this uses the branch locus orientation twice and hence the resulting orientation is well defined. This then gives a coorientation to the two-dimensional surface in $L$ that is the evaluation along the boundary of the disks in the moduli space.

We next use the gluing sequence to orient the disks that results from gluing flow graphs as discussed to holomorphic curves with boundary on $L$, compare \cite{ENS}. Here we should compare the oriented kernel vector fields of the gluing point evaluation maps along the moduli spaces and the orientation of $L$. As in \cite{ENS}, signs of relevant $L^2$ inner products can be computed using evaluation maps and here the gluing sign is given by the intersection sign of the boundary orientation of the big curve and the oriented moduli space. This matches the combinatorial calculation up to sign, see Figure \ref{fig:combvsholo}. 

To see that the holomorphic curve sign matches the combinatorial sign exactly, we use deformation invariance: 
Consider the boundary of a holomorphic curve that intersects the projection of the evaluation map of the moduli space of detour disks, i.e., a critical leaf transversely at one point. Consider its lifts, with the boundary condition that specifies the sheet to which we lift at the incoming and out going ends -- this corresponds to specifying a Reeb chord of the detour attached. We take these lifting conditions so that the only possible choice is the lift with one detour. On the other hand, if we isotope the boundary of the curve across the branch locus, then it intersects two critical leaves, but with the given boundary conditions, the only possible lift is the direct lift without using any detour and for the direct lift there is no extra sign. The result follows for lifting to the Lagrangian $L^\circ_\Delta$.

Consider now the smoothing $L_\tau$. Note the angular periods can be scaled $(z_1,\dots,z_t)\mapsto \epsilon(z_1,\dots,z_t)$ and still satisfy the condition in Proposition \ref{prop: existence of smoothing}. This means we can consider the SFT-stretched picture, where $L_\tau$ is obtained by gluing Harvey Lawson cones $\bigsqcup L_{\mathbb{T}}$ to $L^\circ$ in $(T^*M)^\circ$ to $\bigsqcup \Lambda_{\mathbb{T}}$ at the negative end, according to $\epsilon(z_1,\dots,z_t)$ for all sufficiently small $\epsilon>0$. Because $\Lambda_{\mathbb{T}} \subset S^5$ has no index $\le 0$ Reeb chords or orbits, the curve counting problem separates.  We have already determined the curves for $L^\circ$, and the curves in the smoothed Harvey-Lawson cone $L_{\mathbb{T}}$ were determined in \cite{Ekholm-Shende-unknot}. The theorem follows.
\end{proof}

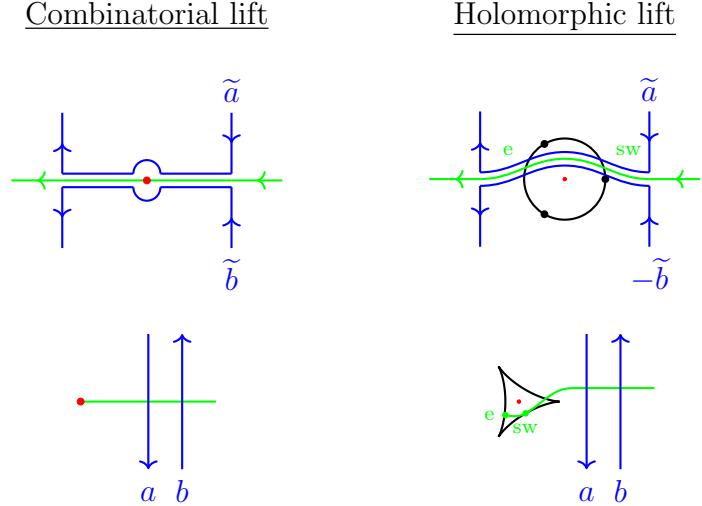
\begin{figure}
\centering
\setlength{\tabcolsep}{2em}
\begin{tabular}{c c}
\underline{Combinatorial lift} & \underline{Holomorphic lift} \\
\noalign{\vskip 1em}
$
\vcenter{\hbox{
\begin{tikzpicture}[scale = 0.9]
\begin{scope}[thick,decoration={markings, mark=at position 0.5 with {\arrow{>}}}]
    \draw[green, postaction={decorate}] (2, 0) -- (1.25, 0);
    \draw[green] (1.25, 0) -- (-1.25, 0);
    \draw[green, postaction={decorate}] (-1.25, 0) -- (-2, 0);
    \draw[blue, postaction={decorate}] (1.25, 1) -- (1.25, 0.1);
    \draw[blue] (1.25, 0.1) -- (0.2, 0.1);
    \draw[blue] (0, 0.1) [partial ellipse = 0 : 180 : 0.2];
    \draw[blue] (-0.2, 0.1) -- (-1.25, 0.1);
    \draw[blue, postaction={decorate}] (-1.25, 0.1) -- (-1.25, 1);
    \draw[blue, postaction={decorate}] (1.25, -1) -- (1.25, -0.1);
    \draw[blue] (1.25, -0.1) -- (0.2, -0.1);
    \draw[blue] (0, -0.1) [partial ellipse = -180 : 0 : 0.2];
    \draw[blue] (-0.2, -0.1) -- (-1.25, -0.1);
    \draw[blue, postaction={decorate}] (-1.25, -0.1) -- (-1.25, -1);
\end{scope}
\filldraw[red] (0, 0) circle (0.05);
\node[blue, above] at (1.25, 1){$\widetilde{a}$};
\node[blue, below] at (1.25, -1){$\widetilde{b}$};
\end{tikzpicture}
}}
$ 
& 
$
\vcenter{\hbox{
\begin{tikzpicture}[scale=0.9]
\begin{scope}[shift={(0.6, 0)}]
\filldraw (0, 0) circle (0.05);
\end{scope}
\begin{scope}[shift={(-0.3, {0.3*sqrt(3)})}]
\filldraw (0, 0) circle (0.05);
\end{scope}
\begin{scope}[shift={(-0.3, {-0.3*sqrt(3)})}]
\filldraw (0, 0) circle (0.05);
\end{scope}
\filldraw[red] (0, 0) circle (0.025);
\draw[thick] (0, 0) circle (0.6);
\begin{scope}[thick, decoration={markings, mark=at position 0.5 with {\arrow{>}}}]
    \draw[green, postaction={decorate}] (2, 0) -- (1.25, 0);
    \draw[green] (1.25, 0) to[out=180, in=0] (0, 0.3) to[out=180, in=0] (-1.25, 0);
    \draw[green, postaction={decorate}] (-1.25, 0) -- (-2, 0);
\end{scope}
\node[green, above right] at (0.6, 0.2){\tiny $\mathrm{sw}$};
\node[green, above left] at (-0.6, 0.2){\tiny $\mathrm{e}$};
\begin{scope}[shift={(0, 0.1)}, thick, decoration={markings, mark=at position 0.5 with {\arrow{>}}}]
    \draw[blue, postaction={decorate}] (1.25, 0.9) -- (1.25, 0);
    \draw[blue] (1.25, 0) to[out=180, in=0] (0, 0.3) to[out=180, in=0] (-1.25, 0);
    \draw[blue, postaction={decorate}] (-1.25, 0) -- (-1.25, 0.9);
\end{scope}
\begin{scope}[shift={(0, -0.1)}, thick, decoration={markings, mark=at position 0.5 with {\arrow{>}}}]
    \draw[blue, postaction={decorate}] (1.25, -0.9) -- (1.25, 0);
    \draw[blue] (1.25, 0) to[out=180, in=0] (0, 0.3) to[out=180, in=0] (-1.25, 0);
    \draw[blue, postaction={decorate}] (-1.25, 0) -- (-1.25, -0.9);
\end{scope}
\node[blue, above] at (1.25, 1){$\widetilde{a}$};
\node[blue, below] at (1.25, -1){$-\widetilde{b}$};
\end{tikzpicture}
}}
$
\\
\noalign{\vskip 1em}
$
\vcenter{\hbox{
\begin{tikzpicture}[scale = 0.9]
\draw[green, thick] (0, 0) -- (2, 0);
\filldraw[red] (0, 0) circle (0.05);
\draw[blue, <-, thick] (1, -1) -- (1, 1);
\draw[blue, ->, thick] (1.5, -1) -- (1.5, 1);
\node[blue, below] at (1, -1.1){$a$};
\node[blue, below] at (1.5, -1){$b$};
\end{tikzpicture}
}}
$ 
& 
$
\vcenter{\hbox{
\begin{tikzpicture}[scale = 0.9]
\draw [thick] plot [domain=0:360, samples=200, smooth] ({0.2*(cos(2*\x)+2*cos(-\x))}, {0.2*(sin(2*\x)+2*sin(-\x))});
\draw[green, thick] (2, 0.2) -- (0.8, 0.2) to[out=180, in=30] (0.1, {-0.1*sqrt(3)}) to[out=-150, in=-10] (-0.2, -0.2);
\filldraw[green] (-0.2, -0.2) circle (0.04);
\node[green, left] at (-0.2, -0.2){\tiny $\mathrm{e}$};
\filldraw[green] (0.1, {-0.1*sqrt(3)}) circle (0.04);
\node[green, below] at (0.1, {-0.1*sqrt(3)}){\tiny $\mathrm{sw}$};
\filldraw[red] (0, 0) circle (0.025);
\draw[blue, <-, thick] (1, -1) -- (1, 1);
\draw[blue, ->, thick] (1.5, -1) -- (1.5, 1);
\node[blue, below] at (1, -1.1){$a$};
\node[blue, below] at (1.5, -1){$b$};
\end{tikzpicture}
}}
$
\end{tabular}
\caption{The combinatorial and holomorphic lifts agree: isotope $\tilde b$ in the combinatorial lift across the sign line. }
\label{fig:combvsholo}
\end{figure}

\begin{remark}
It is interesting to see how the sign rules above are affected by an orientation reversal of $M$. 
For the exchanges, all crossings change sign, and hence there is a $-1$ for each strip attached. Since the strips also change the number of boundary components, we find that the change of sign matches the overall sign. 
For the detours, they do not change the number of boundary components, so the sign changes of boundary components of curves before and after lifting commute. 
Note also that the curves in the relation at negative ends are unchanged but that the $a$-powers changes sign which takes recursion relations of disks to those of anti-disks.
\end{remark}

\subsection{3d spectral network as a moduli space of holomorphic disks}\label{subsec:moduli-interpretation-leafspace} 
In this section, we show that the 2-skeleton $\Delta^\vee_{\le 2}$ of the polygonal decomposition dual to an ideal triangulation $\Delta$ on a 3-manifold $M$ has a natural interpretation as the compactified moduli space of certain holomorphic curves with boundary on $L_\Delta^\circ$. This gives a holomorphic curve interpretation of the foliation used to define the skein lifting map that will be useful if one considers similar questions for `large' deformations of the Lagrangian to which we lift in the cotangent bundle.

Let $M$ and $\Delta$ be as above and construct the Lagrangian cobordism $L_\Delta^\circ$. Consider the moduli space $\mathcal{M}(L_\Delta^\circ)$ of holomorphic disks in $T^\ast M$ with two positive punctures asymptotic to Reeb chords in $\Lambda_{\partial\Delta}$. 
\begin{lemma}\label{l: 2-dim strata mdli}
The interior of $\mathcal{M}(L_\Delta^\circ)$ is naturally diffeomorphic to the union of the interiors of the 2-cells of $\Delta^{\vee}$.
\end{lemma}

\begin{proof}
We use the correspondence between flow lines and holomorphic disks. There is exactly one smooth flow line in the foliation determined by difference of the sheets of $L_\Delta^\circ$ through each interior point of a 2-cell. This flow line connects vertices of $\Delta$ and has unique asymptotics in $\partial_\infty M$. By the correspondence, we then find associated holomorphic strips. The diffeomorphism is obtained by intersection with $\Delta^{\vee}_{2}$. It is clearly a diffeomorphism for the Morse flows, and since this is an open condition, also for the associated holomorphic curves.   
\end{proof}

We next consider the boundary of $\mathcal{M}(L_\Delta^\circ)$. There is one boundary configuration that corresponds to boundary breaking, where the strip with two positive punctures breaks into two disks with one positive puncture each. We denote this boundary component $\partial_{\times}\mathcal{M}$.

\begin{lemma}\label{l: 1-dim strata mdli}
The boundary $\partial_{\times}\mathcal{M}(L_\Delta^\circ)$ is naturally diffeomorphic to the union of the 1-cells of $\Delta^{\vee}$.
\end{lemma}

\begin{proof}
Consider first the flow lines. As a flow line goes to the boundary of a $2$-cell it breaks into two flow lines from the vertices at its end to the branch locus which is the $1$-cell at its boundary. Note that these flow lines correspond to flow trees with a switch and an end, see Figure \ref{fig:SwitchEnd}, 
and that their lifts intersect transversely in a single point and that these are the only such flow trees. We thus find that there are pairs of intersecting disks corresponding to the boundary configurations, and by transversality of the evaluation maps and uniqueness of gluing we find that the pairs of disks give the boundary $\partial_{\times}\mathcal{M}$ as claimed, see Figure \ref{fig:flow lines and disks}. 
\end{proof}

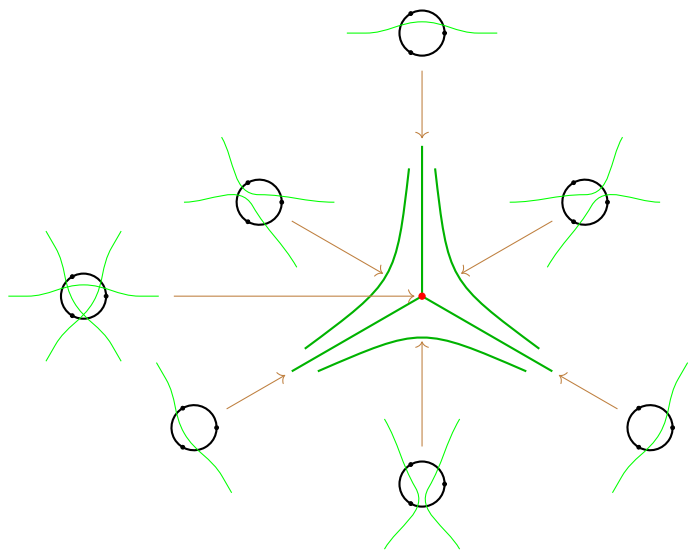
\begin{figure}
\centering
\[
\vcenter{\hbox{
\begin{tikzpicture}
\coordinate (o) at (0, 0);
\coordinate (a) at (0, 2);
\coordinate (b) at ({sqrt(3)}, -1);
\coordinate (c) at ({-sqrt(3)}, -1);
\draw[darkgreen, thick] (o) -- (a);
\draw[darkgreen, thick] (o) -- (b);
\draw[darkgreen, thick] (o) -- (c);
\filldraw[red] (o) circle (0.1em);
\draw[darkgreen, thick] ($0.9*(a)+0.1*(b)$) .. controls ($-0.2*(c)$) .. ($0.1*(a)+0.9*(b)$);
\draw[darkgreen, thick] ($0.9*(b)+0.1*(c)$) .. controls ($-0.2*(a)$) .. ($0.1*(b)+0.9*(c)$);
\draw[darkgreen, thick] ($0.9*(c)+0.1*(a)$) .. controls ($-0.2*(b)$) .. ($0.1*(c)+0.9*(a)$);
\draw[brown, ->] (-3.3, 0) -- ($(o)-(0.1, 0)$);
\begin{scope}[shift={(-4.5, 0)}, scale=0.5]
    \begin{scope}[shift={(0.6, 0)}]
    \filldraw (0, 0) circle (0.05);
    \end{scope}
    \begin{scope}[shift={(-0.3, {0.3*sqrt(3)})}]
    \filldraw (0, 0) circle (0.05);
    \end{scope}
    \begin{scope}[shift={(-0.3, {-0.3*sqrt(3)})}]
    \filldraw (0, 0) circle (0.05);
    \end{scope}
    \draw[thick] (0, 0) circle (0.6);
    \draw[green] (2, 0) -- (1.5, 0) to[out=180, in=0] (0, 0.3) to[out=180, in=0] (-1.5, 0) -- (-2, 0);
    \begin{scope}[rotate=120]
        \draw[green] (2, 0) -- (1.5, 0) to[out=180, in=0] (0, 0.3) to[out=180, in=0] (-1.5, 0) -- (-2, 0);
    \end{scope}
    \begin{scope}[rotate=240]
        \draw[green] (2, 0) -- (1.5, 0) to[out=180, in=0] (0, 0.3) to[out=180, in=0] (-1.5, 0) -- (-2, 0);
    \end{scope}
\end{scope}
\draw[brown, ->] (0, -2) -- (0, -0.6);
\begin{scope}[shift={(0, -2.5)}, scale=0.5]
    \begin{scope}[shift={(0.6, 0)}]
    \filldraw (0, 0) circle (0.05);
    \end{scope}
    \begin{scope}[shift={(-0.3, {0.3*sqrt(3)})}]
    \filldraw (0, 0) circle (0.05);
    \end{scope}
    \begin{scope}[shift={(-0.3, {-0.3*sqrt(3)})}]
    \filldraw (0, 0) circle (0.05);
    \end{scope}
    \draw[thick] (0, 0) circle (0.6);
    \draw[green] (-1, {sqrt(3)}) to[out=-60, in=120] ($({-0.1*sqrt(3)}, -0.1)$) to[out=-60, in=60] (-1, {-sqrt(3)});
    \draw[green] (1, {sqrt(3)}) to[out=-120, in=60] ($({0.1*sqrt(3)}, -0.1)$) to[out=-120, in=120] (1, {-sqrt(3)});
\end{scope}
\begin{scope}[rotate=120]
    \draw[brown, ->] (0, -2) -- (0, -0.6);
    \begin{scope}[shift={(0, -2.5)}, scale=0.5]
        \begin{scope}[shift={(0.6, 0)}]
        \filldraw (0, 0) circle (0.05);
        \end{scope}
        \begin{scope}[shift={(-0.3, {0.3*sqrt(3)})}]
        \filldraw (0, 0) circle (0.05);
        \end{scope}
        \begin{scope}[shift={(-0.3, {-0.3*sqrt(3)})}]
        \filldraw (0, 0) circle (0.05);
        \end{scope}
        \draw[thick] (0, 0) circle (0.6);
        \draw[green] (-1, {sqrt(3)}) to[out=-60, in=120] ($({-0.1*sqrt(3)}, -0.1)$) to[out=-60, in=60] (-1, {-sqrt(3)});
        \draw[green] (1, {sqrt(3)}) to[out=-120, in=60] ($({0.1*sqrt(3)}, -0.1)$) to[out=-120, in=120] (1, {-sqrt(3)});
    \end{scope}
\end{scope}
\begin{scope}[rotate=240]
    \draw[brown, ->] (0, -2) -- (0, -0.6);
    \begin{scope}[shift={(0, -2.5)}, scale=0.5]
        \begin{scope}[shift={(0.6, 0)}]
        \filldraw (0, 0) circle (0.05);
        \end{scope}
        \begin{scope}[shift={(-0.3, {0.3*sqrt(3)})}]
        \filldraw (0, 0) circle (0.05);
        \end{scope}
        \begin{scope}[shift={(-0.3, {-0.3*sqrt(3)})}]
        \filldraw (0, 0) circle (0.05);
        \end{scope}
        \draw[thick] (0, 0) circle (0.6);
        \draw[green] (-1, {sqrt(3)}) to[out=-60, in=120] ($({-0.1*sqrt(3)}, -0.1)$) to[out=-60, in=60] (-1, {-sqrt(3)});
        \draw[green] (1, {sqrt(3)}) to[out=-120, in=60] ($({0.1*sqrt(3)}, -0.1)$) to[out=-120, in=120] (1, {-sqrt(3)});
    \end{scope}
\end{scope}
\draw[brown, ->] (0, 3) -- (0, 2.1);
\begin{scope}[shift={(0, 3.5)}, scale=0.5]
    \begin{scope}[shift={(0.6, 0)}]
    \filldraw (0, 0) circle (0.05);
    \end{scope}
    \begin{scope}[shift={(-0.3, {0.3*sqrt(3)})}]
    \filldraw (0, 0) circle (0.05);
    \end{scope}
    \begin{scope}[shift={(-0.3, {-0.3*sqrt(3)})}]
    \filldraw (0, 0) circle (0.05);
    \end{scope}
    \draw[thick] (0, 0) circle (0.6);
    \draw[green] (2, 0) -- (1.5, 0) to[out=180, in=0] (0, 0.3) to[out=180, in=0] (-1.5, 0) -- (-2, 0);
\end{scope}
\begin{scope}[rotate=120]
    \draw[brown, ->] (0, 3) -- (0, 2.1);
    \begin{scope}[shift={(0, 3.5)}, scale=0.5]
    \begin{scope}[shift={(0.6, 0)}]
    \filldraw (0, 0) circle (0.05);
    \end{scope}
    \begin{scope}[shift={(-0.3, {0.3*sqrt(3)})}]
    \filldraw (0, 0) circle (0.05);
    \end{scope}
    \begin{scope}[shift={(-0.3, {-0.3*sqrt(3)})}]
    \filldraw (0, 0) circle (0.05);
    \end{scope}
    \draw[thick] (0, 0) circle (0.6);
    \draw[green] (2, 0) -- (1.5, 0) to[out=180, in=0] (0, 0.3) to[out=180, in=0] (-1.5, 0) -- (-2, 0);
    \end{scope}
\end{scope}
\begin{scope}[rotate=240]
    \draw[brown, ->] (0, 3) -- (0, 2.1);
    \begin{scope}[shift={(0, 3.5)}, scale=0.5]
    \begin{scope}[shift={(0.6, 0)}]
    \filldraw (0, 0) circle (0.05);
    \end{scope}
    \begin{scope}[shift={(-0.3, {0.3*sqrt(3)})}]
    \filldraw (0, 0) circle (0.05);
    \end{scope}
    \begin{scope}[shift={(-0.3, {-0.3*sqrt(3)})}]
    \filldraw (0, 0) circle (0.05);
    \end{scope}
    \draw[thick] (0, 0) circle (0.6);
    \draw[green] (2, 0) -- (1.5, 0) to[out=180, in=0] (0, 0.3) to[out=180, in=0] (-1.5, 0) -- (-2, 0);
    \end{scope}
\end{scope}
\end{tikzpicture}
}}
\]
\caption{The disks in the boundary corresponding to flow lines ending in the branch locus.}
\label{fig:flow lines and disks}
\end{figure}

Finally, we consider the boundary of the moduli space where parts of the curve fall into the negative end. We call this boundary $\partial_-\mathcal{M}$. 

\begin{lemma}\label{l: 0-dim strata mdli}
The boundary $\partial_-\mathcal{M}$ corresponds to the barycenters in $\Delta$. 
Curves in $\partial_-\mathcal{M}$ consist of two strips with a positive and a negative puncture each and two disks in the negative end, one at each negative asymptotic. 
The curves in the negative end form a 2-dimensional space: each disk can be translated in the $\R$-direction.
\end{lemma}

\begin{proof}
Remaining flow lines connect vertices to barycenters. These correspond to strips with a positive and a negative Reeb chord asymptote. Remaining claims follow from the description of the moduli space of curve on $\Lambda_{\mathbb{T}}$, see \cite{Ekholm-Longhi-Shende}. 
\end{proof}

We summarize the results in Lemmas \ref{l: 2-dim strata mdli}, \ref{l: 1-dim strata mdli}, and \ref{l: 0-dim strata mdli}.
\begin{corollary}
The moduli space $\mathcal{M}(L_\Delta^\circ)$ of holomorphic disks in $T^\ast M^\circ$ with two positive punctures and boundary on $L_\Delta^\circ$ is naturally isomorphic to $\Delta^{\vee}_{\le 2}$, where 2-dimensional stratum correspond to one level unbroken curves, the 1-dimensional stratum to nodal disks with one boundary node, and the $0$-dimensional stratum to SFT broken curves with fixed positive end two disks with one positive and one negative puncture each and two negative end disks at the negative end Reeb chords, each parameterized by $\R$. 

In particular, evaluation at a boundary marked point $\mathrm{ev}\colon\mathcal{M}(L_\Delta^\circ)\to L_\Delta^\circ$ is an injective degree one map.\qed       
\end{corollary}

\section{Wall crossings} \label{section simple wall crossings}

In this section, we will prove Theorems \ref{general wall crossing} and \ref{skein valued cluster}.  The basic idea of the proofs of these results is the following. Suppose there was a Lagrangian $L \subset T^*(S \times \R)$ whose symplectic reductions over $t \in \R$ gave $\Sigma_t\subset T^\ast S$. For a knot $K \subset S$, we could then consider $[K]_L$ and compute it by putting $K$ at either very negative or very positive $t$.  Then by Theorem \ref{algebra structure}, we would deduce as desired: 
\[
[K]_{\Sigma_-} [\emptyset]_L = [K]_L = [\emptyset]_L [K]_{\Sigma_+}.
\]

Such a Lagrangian $L$ exists when the Lagrangian isotopy is exact, but the isotopies of interest here (for which $[\emptyset]_L \ne 0$) are not exact and the Lagrangian $L$ does not exist. However, as we will see below, it is possible to construct a totally real $\widetilde{L}$ which, despite not being Lagrangian, can play a role analogous to that of $L$ and using $\widetilde{L}$ we prove Theorem \ref{general wall crossing}. To then establish Theorem \ref{skein valued cluster} we show that under its hypotheses, $\widetilde{L}$ bounds a single disk, whose multiple cover count was previously determined in \cite{Ekholm-Shende-unknot}. 

\subsection{Lagrangian isotopies and proof of Theorem \ref{general wall crossing}}\label{wall crossing lifts}
We consider first a general construction of totally real submanifolds associated to `slowed down' Lagrangian isotopies and show that there exists almost complex structures for which these totally real submanifolds are good boundary conditions for holomorphic curves.  

Let $L\subset X$ be a Lagrangian, possibly non-compact in which case we assume it is asymptotic to a Legendrian. Let $\iota_t\colon L\to X$, $0\le t\le 1$ be a Lagrangian isotopy which we take to be constant near the ends. We will slow down the isotopy so that its trace is suitable for holomorphic curve counting.

Fix an integer $N>0$, which we will ultimately take sufficiently large. First, rescale the isotopy to $\phi'_t=\iota_{\frac{1}{N} t}$, $0\le t\le N$. Then $|\frac{\partial \phi_{t}'}{\partial t}|=\mathcal{O}(1/N)$. Second, transport the isotopy $\phi'_t$ to $[0,2N]$ and make it constant in intervals $[2j-1,2j]$. We call the resulting isotopy $\phi_t$, $t\in\R$, where we extend it to be constant in $(-\infty,0]$ and $[2N,\infty)$. More formally, let $\kappa\colon [0,1]\to [0,1]$ be an approximation of the identity map that is constant in small neighborhoods of the end points. Then
\[
\phi_t=
\begin{cases}
\iota_0, & t\in (-\infty,0],\\
\iota_{\frac{\kappa(t-2j)+j}{N}}, & t\in [2j,2j+1],\quad 0\le j\le N-1,\\
\iota_{\frac{j+1}{N}}, & t\in [2j+1,2j+2],\quad 0\le j\le N-1,\\
\iota_1, & t\in [2N,\infty).
\end{cases}
\]

Consider now the trace $\Phi\colon L\times\R\to X\times T^\ast\R$ of $\phi_t$:
\[
\Phi(q,t) = (\phi_t(q),t,0).
\]
\begin{lemma}\label{l : slow}
The trace is Lagrangian over $(-\infty,0]$, over $[2N,\infty]$, and over each interval $[2j+1,2j+2]$, $0\le j\le N-1$. Furthermore over each interval $[2j,2j+1]$, $0\le j\le N-1$, there is a diffeomorphism $\psi_{2j}\colon X\times T^\ast [2j,2j+1]\to X\times T^\ast [2j,2j+1]$ of $C^1$-distance $\mathcal{O}(1/N)$ from the identity such that $\psi_{2j}(\Phi(q,t))=(\phi_{2j+\frac12}(q),t,0)$.
\end{lemma}
\begin{proof}
Direct consequence of the construction of $\phi_t$.
\end{proof}

It follows from Lemma \ref{l : slow} that the trace $\Phi(L\times\R)$ is totally real for $N$ sufficiently large. Furthermore, if $J_X$ is an almost complex structure on $X$ compatible with its symplectic form $\omega$, then the complex structures $J_{2j}= d\psi_{2j}^{-1}\circ J_X\circ d\psi_{2j}$, $0\le j\le N$, are compatible with form $\omega$ provided $N$ is sufficiently large. Define an almost complex structure $J$ compatible with $\omega_0=\omega+dt\wedge d\tau$ on $X\times T^\ast \R$ by taking $J=J_{2j}\oplus J_\R$ over $[2j,2j+1]$, where $J_\R$ is the standard complex structure on $T^\ast\R$, and $J_t\oplus J_\R$, $t\in[2j+1,2j+2]$, where $J_t$ interpolates between $J_{2j}$ and $J_{2j+2}$ over $[2j+1,2j+2]$. More formally,
\[
\begin{cases}
J|_{X\times T^\ast(-\infty,0]} &= J_X\oplus J_\R,\\
J|_{X\times T^\ast[2j,2j+1]} &= J_{2j}\oplus J_\R,\quad 0\le j\le N-1,\\
J|_{X\times T^\ast[2j+1,2j+2]_t} &= J_{t}\oplus J_\R,\quad 0\le j\le N-1,\\
J|_{X\times T^\ast[2N,\infty)} &= J_X\oplus J_\R.
\end{cases}
\]

We next show that $J$-holomorphic curves in $X\times T^\ast\R$ with boundary on $\Phi(L\times\R)$ are confined to intervals of length $1$. 

\begin{lemma}\label{l : sits in slice}
Any connected $J$-holomorphic curve with boundary on $\Phi(L\times\R)$ that takes some point into $X\times T^\ast[2j,2j+1]$, into $X\times T^\ast (-\infty,0]$ or into $X\times T^\ast[2N,\infty)$ lies entirely in a slice $(t,\tau)=(c,0)$.
\end{lemma}

\begin{proof}
By our choice of complex structure and the definition of $\Phi$, the $J$-holomorphic curve is holomorphic for a split complex structure $J_t\oplus J_\R$ and has boundary on a Lagrangian of the form $L\times\R\subset T^\ast \R$. The projection of the curve to $T^\ast \R$ is then holomorphic and hence constant by the maximum principle.
\end{proof}

\begin{proof}[Proof of Theorem \ref{general wall crossing}]
We use notation as above for $L=\Sigma$ and $X=T^\ast S$. Assume that the moduli space of $\phi_0(\Sigma)$ and $\phi_{2N}(\Sigma)$ in $T^\ast S$ are empty. 

We claim then that there exists a perturbation for bare curves with boundary on $\Phi(\Sigma\times\R)\subset T^\ast S\times T^\ast\R$ and $\epsilon>0$ such that no bare curve intersects the regions $T^\ast S\times T^\ast [-\infty,0]$ and  $T^\ast S\times T^\ast [2N,\infty)$.

To see this, note that by Lemma \ref{l : sits in slice}, any connected $J$-holomorphic curve which intersects these regions in fact lies in the fiber over $(t,0)\in T^\ast\R$ and is $J_{0}$-holomorphic respectively $J_{2N}$-holomorphic. Since the corresponding moduli spaces are empty, moduli spaces are empty in $4\epsilon$-intervals around $t=0$ and $t=2N$ by Gromov compactness. Since bare curve perturbations are supported in a small neighborhood of the original unperturbed moduli spaces, the claim holds.

Equation \eqref{general skein mutation} then follows: take $\Omega(\phi)$ to be the skein-valued bare curve count of 
$$
\Phi(L)\cap T^\ast S\times T^\ast\left[\epsilon,2N-\epsilon\right]
$$ 
in $\Sk(L)$ and observe that the left and right hand sides correspond to including the conormal of $K$ on the far left and on the far right, respectively; they are equal by deformation isotopy invariance of the skein count where the isotopy translates the conormal in the $\R$-direction. 

We next consider homotopy invariance. Let $\phi_{s,t}$, $0\le s\le 1$, be a $1$-parameter families of isotopies. By continuity in $s$ we find $N$ so that after rescaling, the size of the derivative $\frac{\partial\phi_{s,t}}{\partial t}$ is sufficiently small for all $s$. As for $\Phi(L\times\R)$, the $1$-parameter family of totally real submanifolds $\Phi_s(L\times\R)$ provides an isotopy of good boundary conditions for bare curve counts, and hence  
the desired homotopy invariance of $\Omega(\phi)$ follows from deformation invariance of the skein-valued curve count. 
\end{proof}

\subsection{Transverse $(-1)$-disk and the proof of Theorem \ref{skein valued cluster}}\label{(-1)disks}
We next consider specific $1$-parameter families $\phi_t\colon\Sigma\to T^\ast S$ such that there are no holomorphic curves for $t\ne 0$ and such that there is a single embedded disk $D$ of index $-1$ and boundary on $\phi_0(\Sigma)$. We assume that this disk is transversely cut out in the sense of $1$-parameter families. In terms of Cauchy-Riemann operators, this means the following. 

Consider a cotangent neighborhood $T^\ast\Sigma$ of $\phi_0(\Sigma)$. The Lagrangian isotopy $\phi_t$ moves $\phi_0(\Sigma)$ inside $T^\ast\Sigma$ by shifting along closed $1$-forms. We extend it to an ambient isotopy by cutting off and denote it $\tilde\phi_t$. 
The transversality condition then means that the linearized $\bar\partial_J$ at $D$ has cokernel of dimension $1$ and that the linearized variation of $D$ given by $\frac{\partial\tilde\phi_t}{\partial_t}|_{t=0}$ spans the cokernel. We take the sign of the disk to be the sign induced by the standard orientation of the cokernel.

\begin{proof}[Proof of Theorem \ref{skein valued cluster}]
Consider the trace $\Phi\colon L\times\R\to T^\ast S\times T^\ast\R$ around $0$ and note that under the assumptions above, the disk $D$ is an embedded disk with boundary on $\Phi(L\times\R)$. 

The tangent bundle $T(T^\ast S)|_D$ splits as two complex line bundles $TD\oplus ND$, the tangent and normal bundles to $D$. Consider first the multiplication of $\phi_0(\Sigma)$ with $T^\ast\R$. In this case, the boundary condition splits as the sum of two real line bundles in $TD$ and $ND$, respectively. Trivializing these line bundles, the boundary conditions give rotating real lines along the boundary. Since the index of the disk equals $-1$ and the cokernel is $1$-dimensional it follows that the Maslov index of the line in the tangent direction equals $2$ and the line in the normal direction equals $-2$. The additional direction is not rotating at all, and the kernel consists of constant sections in this direction. This means that the linearized Cauchy-Riemann equation in the normal direction for this trivial product is conjugate to the standard $\bar\partial$-operator on the unit disk $D$ with values in $\C^2$ and Lagrangian boundary condition given by the columns in the matrix
\[
\left(\begin{matrix}
e^{-i\theta} & 0\\
0 & 1
\end{matrix}\right), \quad e^{i\theta}\in\partial D.
\]
Consider next the boundary condition of the trace of a deformation that spans the cokernel. Since it is a small deformation of the constant condition, Fourier expanding the boundary condition, it is straightforward to check that it is conjugate to 
\[
\frac{1}{1+r^2}\left(\begin{matrix}
e^{-i\theta} & - r e^{-i\theta/2}\\
r e^{-i\theta/2} & 1,
\end{matrix}\right), \quad e^{i\theta}\in\partial D,
\]
where $r>0$ is small. This in turn agrees with the boundary condition of the toric brane, up to $\mathcal{O}(r^2)$, studied in \cite{Ekholm-Shende-unknot}, where the multiple cover formula is shown to be the exponentiated skein dilogarithm. The theorem follows.
\end{proof}

\begin{remark} 
Using Morse flow graphs one can give another proof of the recursion relation for multiple covers of a disk corresponding to an isolated flow line as follows.
Consider a 3-manifold $M$ with a branched double cover $L$ with branch locus $\tau\subset M$. Consider a rigid basic disk $w$ corresponding to a flow line of $L\to M$ that connects two points on $\tau$. We study what happens when a moving big holomorphic curve $u_s$, $-\epsilon<s<\epsilon$ with boundary on $M$ crosses the boundary of $w$ at $s=0$. (The map $u_s$ could for example be the basic cylinder of a conormal of a knot $K_s$ that moves in such a way that $K_0$ intersects the boundary of $w$ and such that the tangent vector of $K_0$ together with the deformation vector and the tangent vector of $\partial w$ spans the tangent space to $M$.) 

The basic disk $w$ itself has multiple covers that correspond to flow lines with multiplicity. Thus, at the moment $s=0$ when $u_0$ hits the disk, there are broken flow graph and big curve configurations $u_0\cup w$, but at moments before or after there are not. As $u_s$ is embedded, Lemma \ref{l : simple implies simple} shows that there can be no curves with multiply covered edges attached in the limit. Looking at possible limits before (one flow line to the branch locus attached to $u$) and after the crossing (two flow lines to the branch locus attached) then reveals the recursion relation for the disk, where the operators correspond to the curves with simple flow lines attached. This gives another illustration of why there are simple recursion relations that control the many skein instances involved in a wall crossing: there are no very complicated big curves nearby.
\end{remark}

\section{Examples of skein-valued wall-crossing formulas} \label{sec:wall-crossing-examples}
In this section, we apply Theorems \ref{general wall crossing} and \ref{skein valued cluster} to derive skein-valued versions of well-known wall crossing formulas. 

Let $C$ be a Riemann surface. A quadratic differential $\phi \in H^0(C,K_C^{\otimes 2})$ on $C$ defines a two-fold ramified covering $\Sigma\to C$:
\begin{equation}\label{eq:spectral-cuurve}
	\lambda^2 = \phi
\end{equation}
with $\lambda$ the complex Liouville 1-form in $T^*C$.
The cover $\Sigma$ is smooth so long as $\phi$ has nondegenerate zeroes.  

We will consider such $\Sigma$ as real Lagrangians with respect to the real symplectic form $\omega = \mathrm{Re} \, d\lambda$.  We will count holomorphic curves in a complex structure $J$ tamed by $\omega$, {\em not} the complex structure on $T^*C$ coming from the complex structure of $C$.   

Let us recall the basic consequences of Theorem \ref{general wall crossing}. 
Suppose given  1-parameter family of such quadratic differentials, $\phi_t$ for $t$ varying in some interval $I = [0,1]$, such that the corresponding $\Sigma_0$ and $\Sigma_1$ bound no $J$-holomorphic curves.  Then, by Theorem \ref{general wall crossing}, this family determines an element $Z(\phi) \in \widehat{\Sk}(\Sigma)$, where here we write $\Sigma$ to indicate the total space of the family $\Sigma_t$.  
Moreover, suppose given a 2-parameter family $\phi_{t}[u]$ over some $(t, u) \in [0,1] \times I$ where the corresponding $\Sigma_0[\cdot]$ and $\Sigma_1[\cdot]$ never bound holomorphic curves.  Then, again by Theorem \ref{general wall crossing}, the element $Z(\phi[u])$ is independent of $u$.  (This independence should be regarded as a general ``wall crossing formula''.)

To actually evaluate $Z(\phi)$, we will use Theorem \ref{skein valued cluster}.  To do so we must show that our family $\Sigma_t$ bound holomorphic curves at some isolated collection of $t$, and each such instance is an embedded disk, transverse in an appropriate sense.  Choosing a metric on the base and passing to flow graphs, we must ask that the only flow graphs which appear are flow lines (and that they appear transversely in $t$). Such flow lines necessarily connect the branch points of the cover.  

\begin{lemma}For an appropriate choice of metric, Morse flow lines agree with 
the horizontal foliation on $C$ defined by the quadratic differential, i.e.
\begin{equation}
	\mathrm{Im}\,(e^{i\theta} \sqrt{\phi}) = 0.
\end{equation}
\end{lemma}
\begin{proof}
Indeed, this correspondence of flow lines and horizontal foliation is well known for the singular conformal metric associated to the quadratic differential, see e.g., \cite{strebel1984quadratic} and \cite[Example 2.23]{Casals-Nho-spectral}. This metric is flat except at a finite number of isolated points that are modeled on Euclidean cones. Rounding the tip of these cones we get a smooth metric that agrees with the flat metric outside arbitrarily small disk neighborhoods of the cone points. In particular, flow lines in the smooth metric converge to flow lines for the singular metric.  
\end{proof}

Flow lines connecting branch points are the `saddle connections' of such foliations, which are precisely what enter into the Gaiotto-Moore-Neitzke formulas \cite{Gaiotto-Moore-Neitzke-WKB}; correspondingly, these lines have already been determined in the examples we study below.

\subsection{The pentagon relation} 
\label{pentagon}

\begin{figure}
\begin{center}
\includegraphics[width=0.19\textwidth]{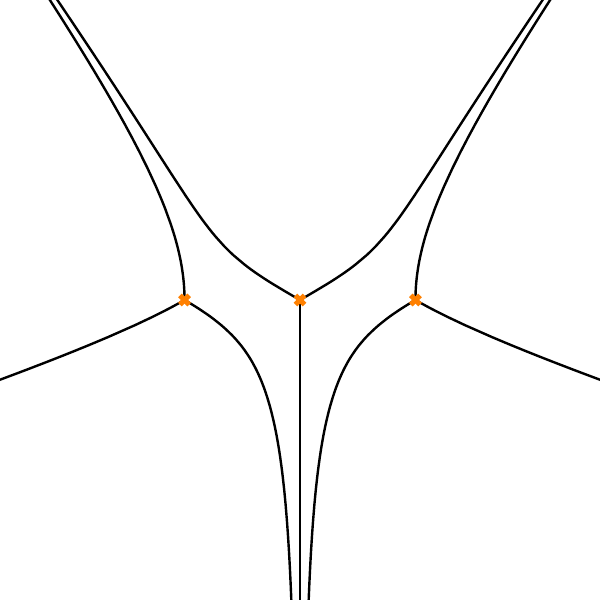}
\includegraphics[width=0.19\textwidth]{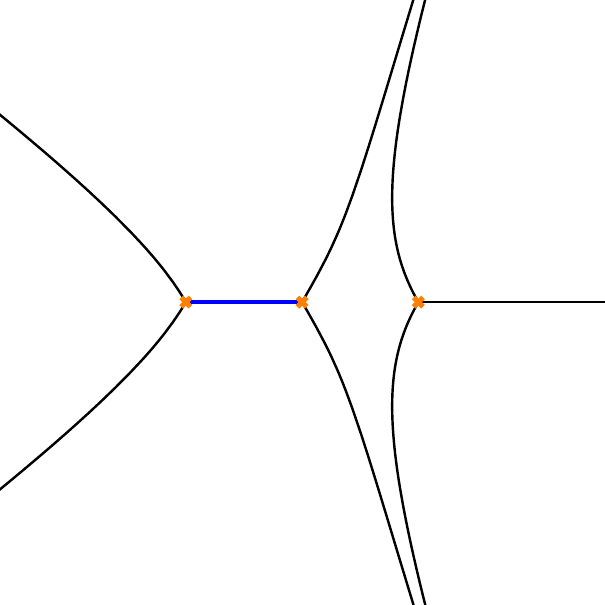}
\includegraphics[width=0.19\textwidth]{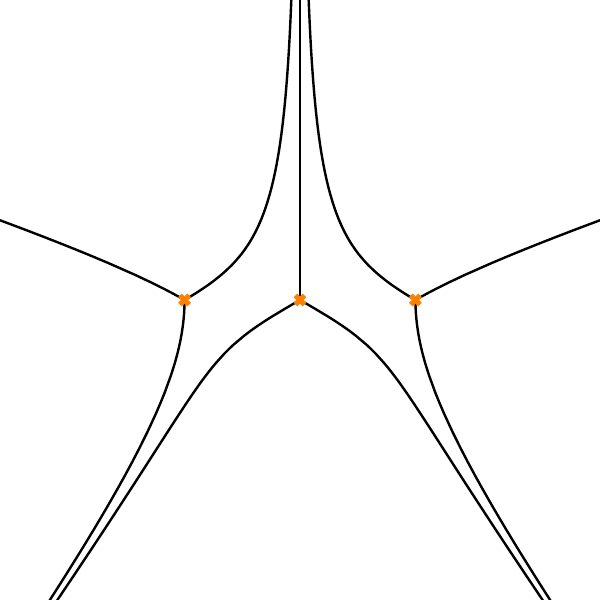}
\includegraphics[width=0.19\textwidth]{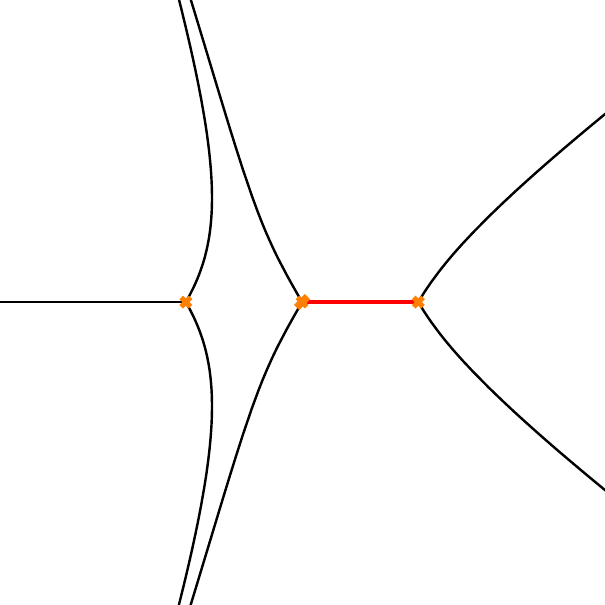}
\includegraphics[width=0.19\textwidth]{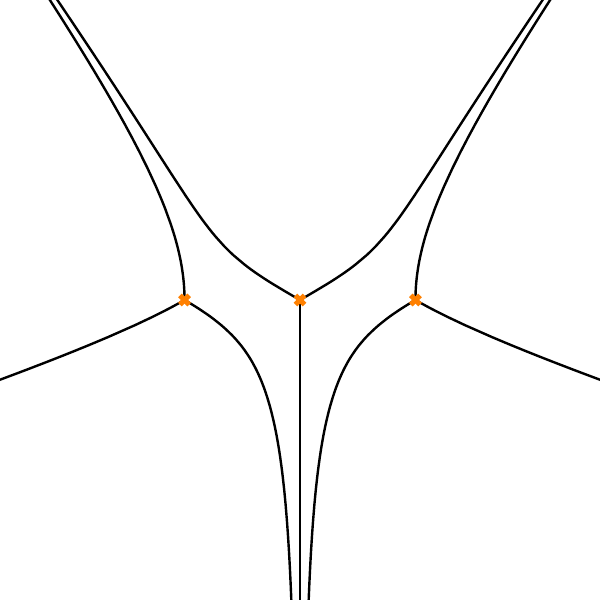}
\caption{Two holomorphic disks at $u=0$ for $0\leq \theta\leq \pi$.}
\label{fig:AD-strong}
\end{center}
\end{figure}

\begin{figure}
\begin{center}
\includegraphics[width=0.19\textwidth]{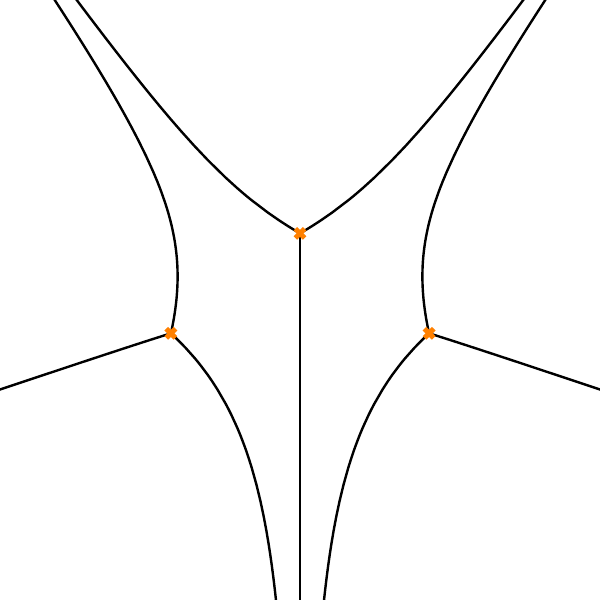}
\includegraphics[width=0.19\textwidth]{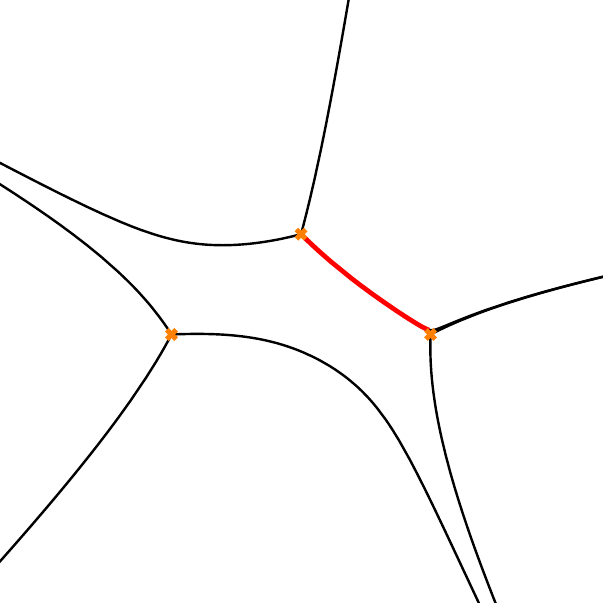}
\includegraphics[width=0.19\textwidth]{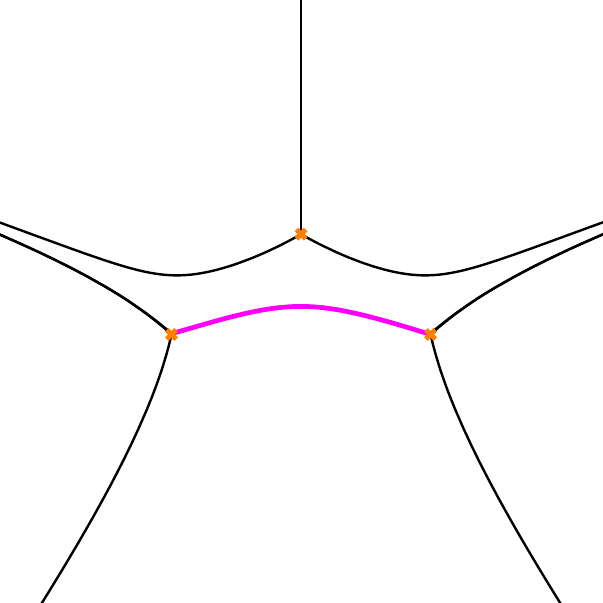}
\includegraphics[width=0.19\textwidth]{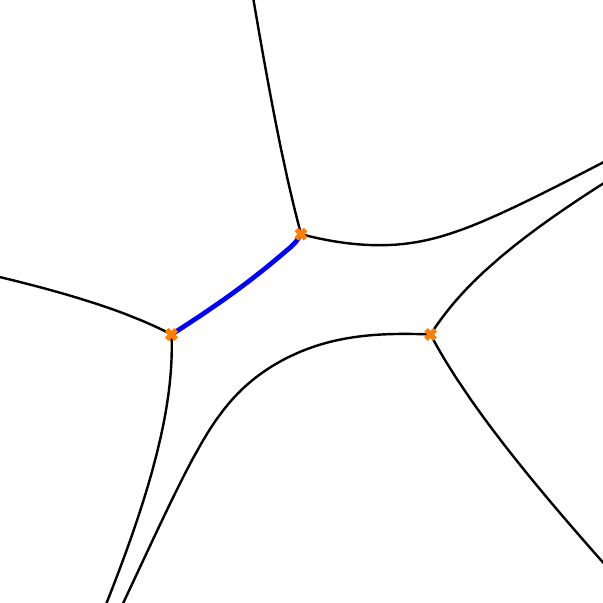}
\includegraphics[width=0.19\textwidth]{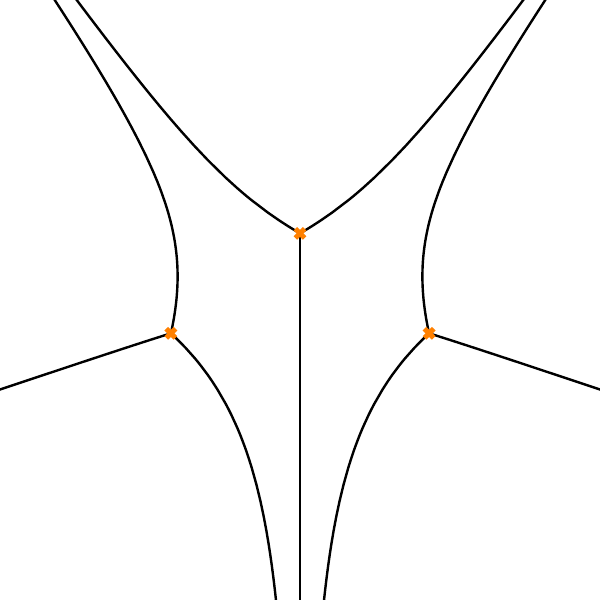}
\caption{Three holomorphic disks at $u=2i$ for $0\leq \theta\leq \pi$.}
\label{fig:AD-weak}
\end{center}
\end{figure}

Let $C=\mathbb{C}$ with quadratic differential 
\begin{equation}
	\phi_t[u](z) = e^{2\pi i t}(z^3 - 3 z+2 u)\, dz^2.
\end{equation}
This has no double zeros so long as  $u\in \mathbb{C} \setminus \{\pm 1\}$.
The variety in $T^*C$ given by \eqref{eq:spectral-cuurve} is the Seiberg-Witten curve for Argyres-Douglas theory introduced in \cite{Argyres:1995jj} and presented in this form in \cite{Gaiotto-Moore-Neitzke-WKB}. 
It is well-known that the count of holomorphic disks depends on $u$ piecewise continuously, and there are two regions corresponding to $|u|\lesssim 1$ and $|u|\gtrsim 1$.

For $u=0$ there are two holomorphic disks, corresponding to the trajectories shown in Figure~\ref{fig:AD-strong}.
Disk boundaries correspond to generators of $H_1(\Sigma,\mathbb{Z})$, respectively denoted $(0,1)$ for the first disk (with smaller value of $t$) and $(1,0)$ for the second one. 

Moving across the `line of marginal stability', approximately described by $|u|\approx 1$, causes the two disks to swap ordering in $t$.
For $u=2i$ there are three holomorphic disks, shown in Figure~\ref{fig:AD-weak}. In addition to disks with boundaries $(1,0)$ (now with largest value of $t$) and $(0,1)$ (with lowest value of $t$) there is now another disk whose boundary has class $(1,1)$.

\begin{figure}
\begin{center}
\includegraphics[width=0.4\textwidth]{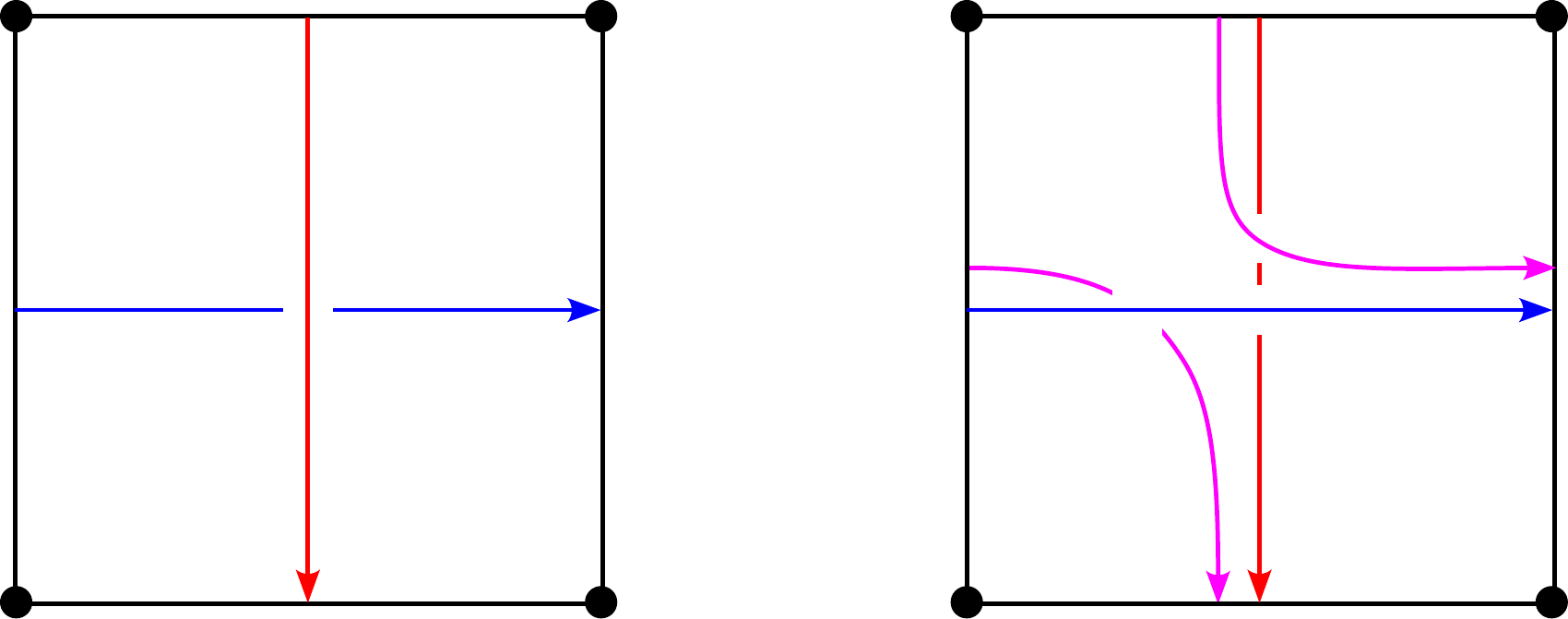}
\caption{The skein-valued pentagon relation on the punctured torus.}
\label{fig:pentagon}
\end{center}
\end{figure}

From  Theorems \ref{general wall crossing} and \ref{skein valued cluster}, we deduce the skein-valued pentagon relation: 
\begin{equation}\label{eq:pentagon-wcf}
	\Psi_{(1,0)} \Psi_{(0,1)} = \Psi_{(0,1)} \Psi_{(1,1)}[-1] \Psi_{(1,0)}.
\end{equation}
Here $\Psi_{(i,j)}$ is the skein-valued dilogarithm \eqref{eq:explicit-skein-dilog} with $P_d$ replaced by $P_{di,dj}$, corresponding to insertions on the boundary torus in homology class $(di,dj)$, see \cite[Definition 2.4]{Morton-Samuelson}.
Relation \eqref{eq:pentagon-wcf} is illustrated in Figure \ref{fig:pentagon}, and was originally derived in \cite{Nakamura, Hu}, and generalized in \cite{Scharitzer}.  As noted in those references, the skein-valued pentagon identity in the punctured torus implies, via \cite{Morton-Samuelson-DAHA, BCMN-DAHA}, the pentagon identity in the elliptic Hall algebra, originally proven in \cite{Garsia-Mellit, Zenkevich-pentagon}.

Specializing to the $\mathfrak{gl}_1$ skein algebra of $T^2$ with generators $\hat y\hat x = q \hat x\hat y$ is given by the replacements
\begin{equation}\label{eq:HOMLYPT-2-gl1}
	P_{i,j} \  \to\  q^{-ij/2}\hat y^i \hat x^j \,,
	\qquad
	\Psi_{i,j}[\xi] \  \to\  	(\xi q^{1/2-ij/2}\hat y^i \hat x^j ; q)^{-1}_\infty\,,
\end{equation}
and recovers the pentagon identity for the quantum dilogarithm, originally established in \cite{Faddeev:1993rs}. 
The pentagon identity has a celebrated derivation via a wall crossing in the space of stability conditions for the $A_2$ quiver \cite{Reineke-HN}, later reintepreted in terms of the Donaldson-Thomas invariants of an appropriate CY3 category \cite{Kontsevich:2008fj}.
We do not know a similar interpretation of the skein-valued pentagon identity. 

\begin{remark}
    The sign in $\Psi_{(1,1)}[-1]$ can be understood geometrically as follows. In Figures \ref{fig:AD-strong} and \ref{fig:AD-weak} disk boundaries projected to $C$ attach directly to zeroes of $\phi_2$, which correspond to projections of sign defects of the double covering.
    The direct lift of disk boundaries to the punctured torus is shown in the top row of Figure \ref{fig:pentagon-curves}, where sign defects are also marked.
    In the same figure, bottom row, a perturbation of disk boundaries is shown. Note that the perturbed boundstate disk (in purple) differs from the curve obtained by smoothing the intersection of generating curves (purple line in Figure \ref{fig:pentagon}), by crossing the sign defect in the center. This is the geometric origin of the sign in $\Psi_{(1,1)}[-1]$.

    Algebraically, the sign can be absorbed by twisting the skein algebra of the torus by a quadratic refinement, such as the one adopted in \cite{Gaiotto-Moore-Neitzke-WKB}.
    Let $\sigma:\mathbb{Z}^2\to \mathbb{C}$ such that $\sigma(i,j) \sigma(k,l) = \sigma(i+k, j+l)  (-1)^{jk-il} $. Note that this implies $\sigma(di, dj) = \sigma(i,j)^d$. 
    We twist the skein algebra of the torus by introducing $\tilde P_{i,j} := \sigma(i,j) P_{i,j}$. Replacing $P_{i,j}$ into \eqref{eq:explicit-skein-dilog} gives the twisted skein-valued dilogarithm $\tilde\Psi_{(i,j)} \equiv \Psi_{(i,j)}[\sigma(i,j)]$. 
    This obeys the pentagon identity without additional signs
\begin{equation}
	\tilde\Psi_{(1,0)} \tilde\Psi_{(0,1)} = \tilde\Psi_{(0,1)} \tilde \Psi_{(1,1)} \tilde\Psi_{(1,0)}.
\end{equation}
\end{remark}

\begin{figure}
\begin{center}
\includegraphics[width=0.65\textwidth]{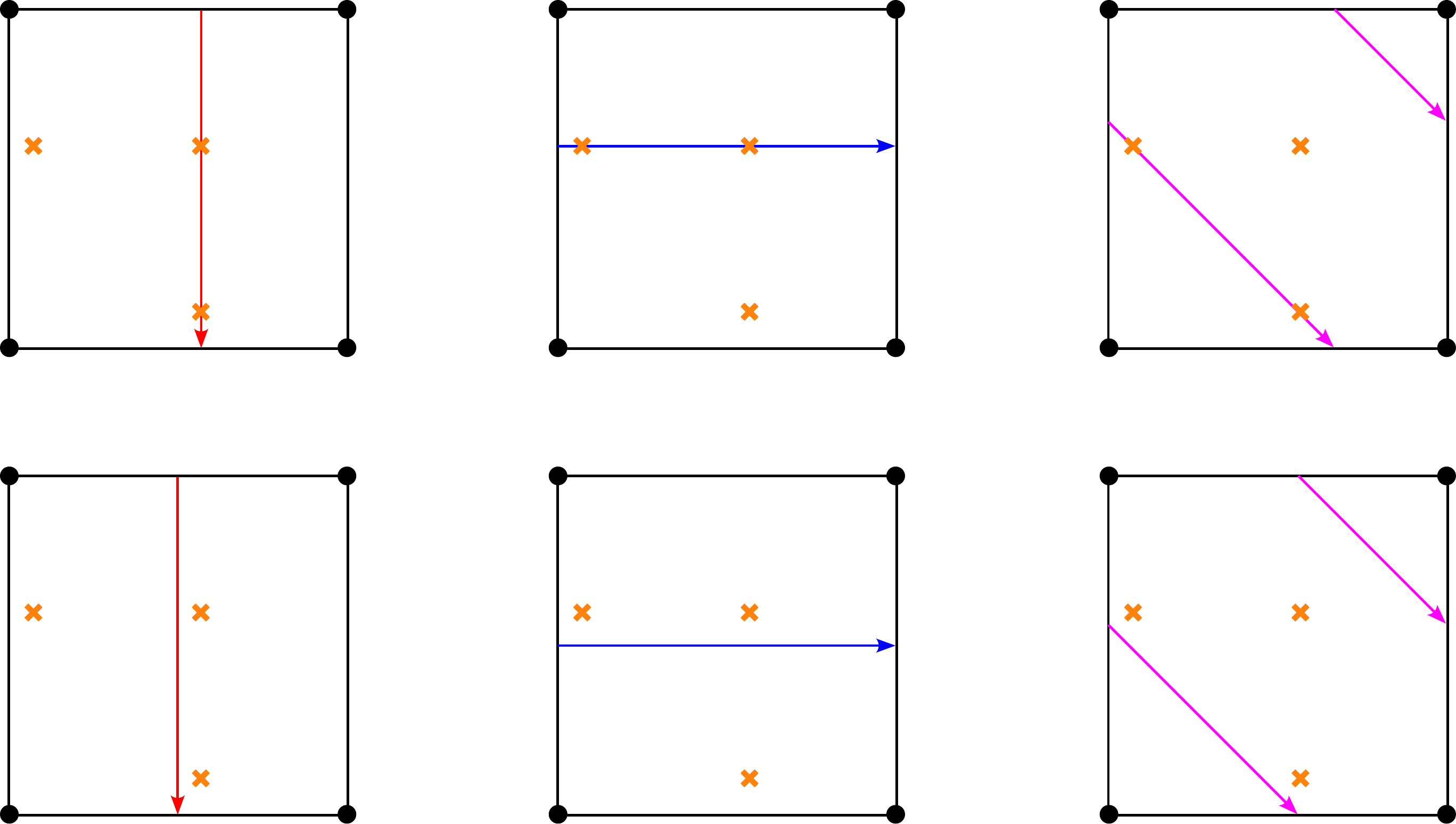}
\caption{Top row: direct lift of disk boundaries from Figure \ref{fig:AD-weak} to the torus double covering of $C$. Sign defects are marked by crosses. An additional sign defect is located in correspondence of the puncture, which is marked by a thick dot. Bottom row: a perturbation of disk boundaries away from sign defects.}
\label{fig:pentagon-curves}
\end{center}
\end{figure}

\subsection{The Seiberg-Witten wall-crossing}

Let $C=\mathbb{C}^*$ with a family of quadratic differentials
\begin{equation}
	\phi_t[u](z) = e^{2\pi i t} \left(\frac{1}{z^3} +\frac{2u}{z^2} +\frac{1}{z}\right)\, dz^2
\end{equation}
parameterized by $u\in \mathbb{C} \setminus \{\pm 1\}$.
The variety in $T^*C$ given by \eqref{eq:spectral-cuurve} is the Seiberg-Witten curve for $\mathcal{N}=2$ supersymmetric $SU(2)$  Yang-Mills theory introduced in \cite{Seiberg:1994rs} and presented in this form in \cite{Gaiotto-Moore-Neitzke-WKB}. 

Again there are two regions in parameter space, corresponding (roughly) to $|u|< 1$ and $|u|> 1$, with different curve counts.

Choosing $u=0$ gives two holomorphic disks. The corresponding flow lines on $C$ are shown in Figure~\ref{fig:YM-strong}.
In a suitable choice of basis for $H_1(\Sigma,\mathbb{Z})$, the boundaries of the two curves are in class
$(-1,1)$ for the first disk (with smaller value of $t$) and $(1,1)$ for the second one.

Deforming parameters to $|u|> 1$,  the two disks  swap order in $t$, but there are (in contrast to the previous example) infinitely many more embedded holomorphic disks, see Figure~\ref{fig:YM-weak}.

\begin{figure}
\begin{center}
\includegraphics[width=0.19\textwidth]{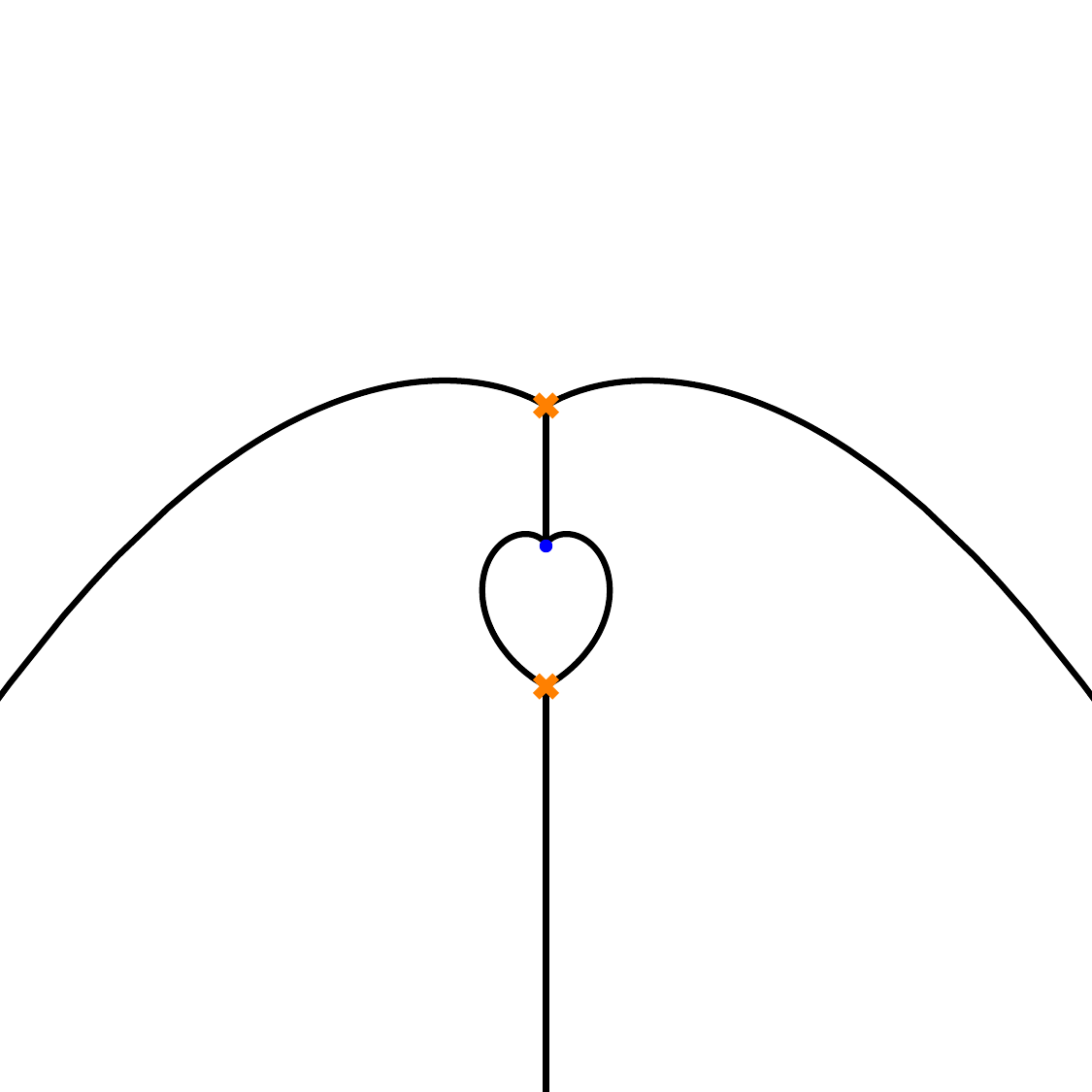}
\includegraphics[width=0.19\textwidth]{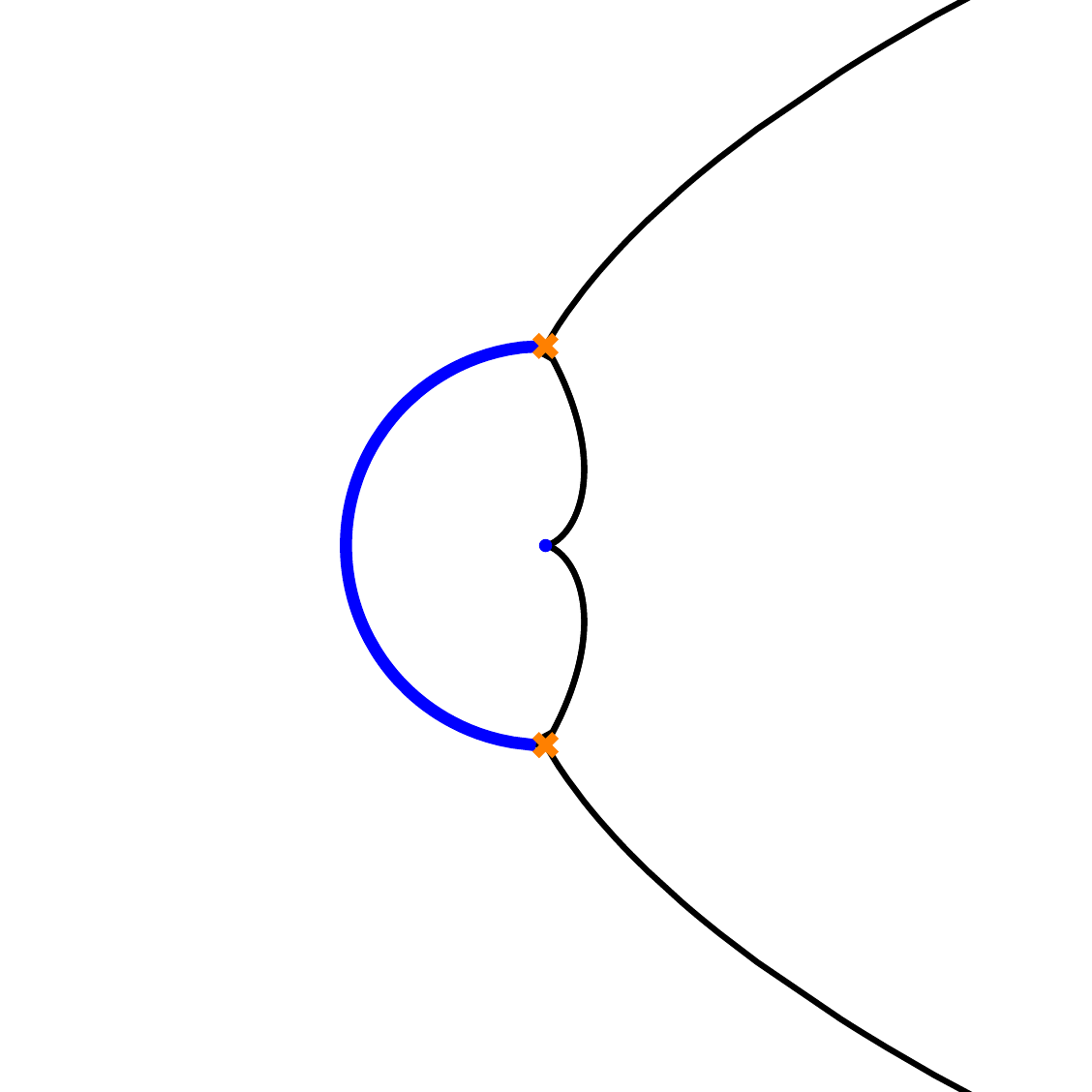}
\includegraphics[width=0.19\textwidth]{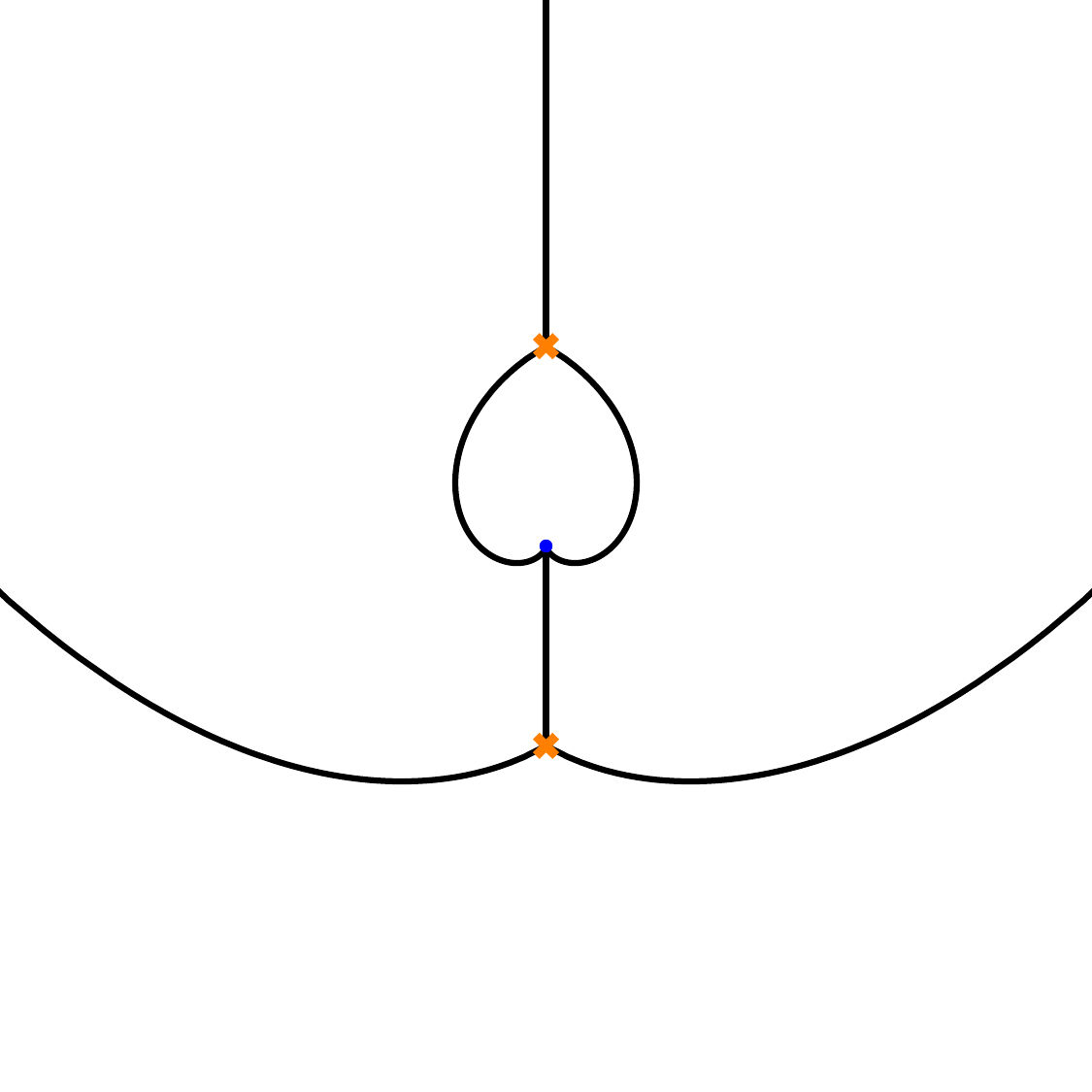}
\includegraphics[width=0.19\textwidth]{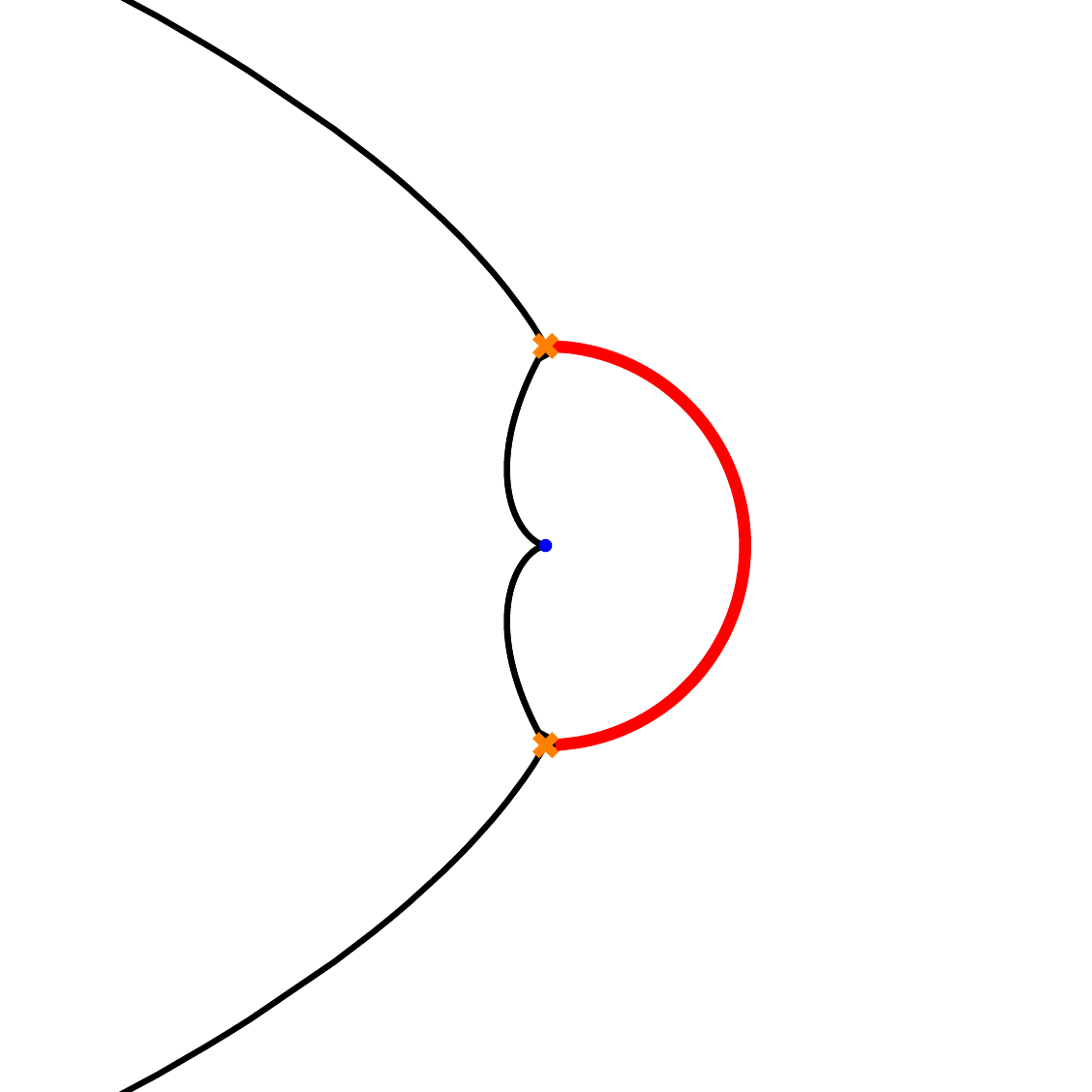}
\includegraphics[width=0.19\textwidth]{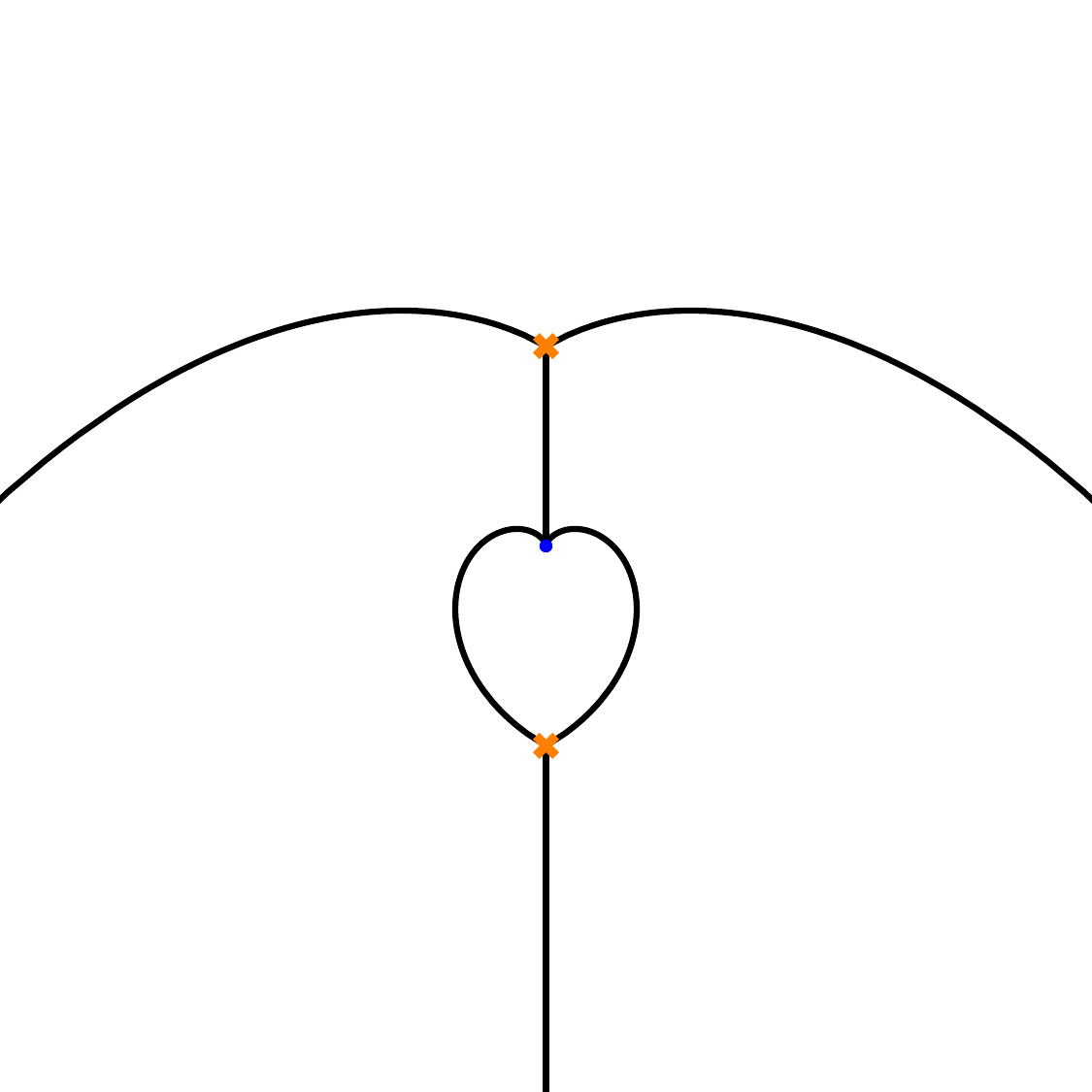}
\caption{Two holomorphic disks at $u=0$ for $0\leq \theta\leq \pi$.}
\label{fig:YM-strong}
\end{center}
\end{figure}

\begin{figure}
\begin{center}
\includegraphics[width=0.19\textwidth]{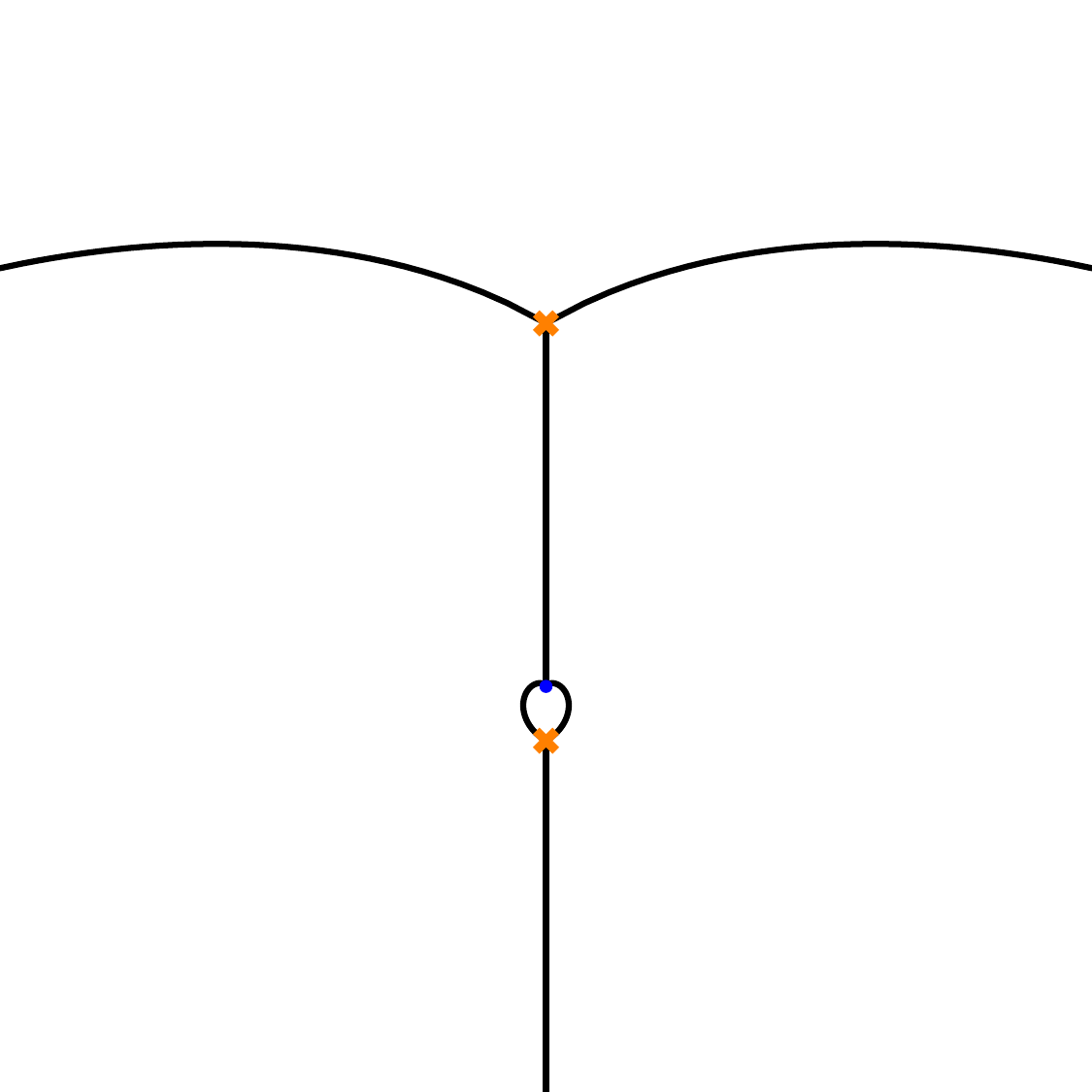}
\includegraphics[width=0.19\textwidth]{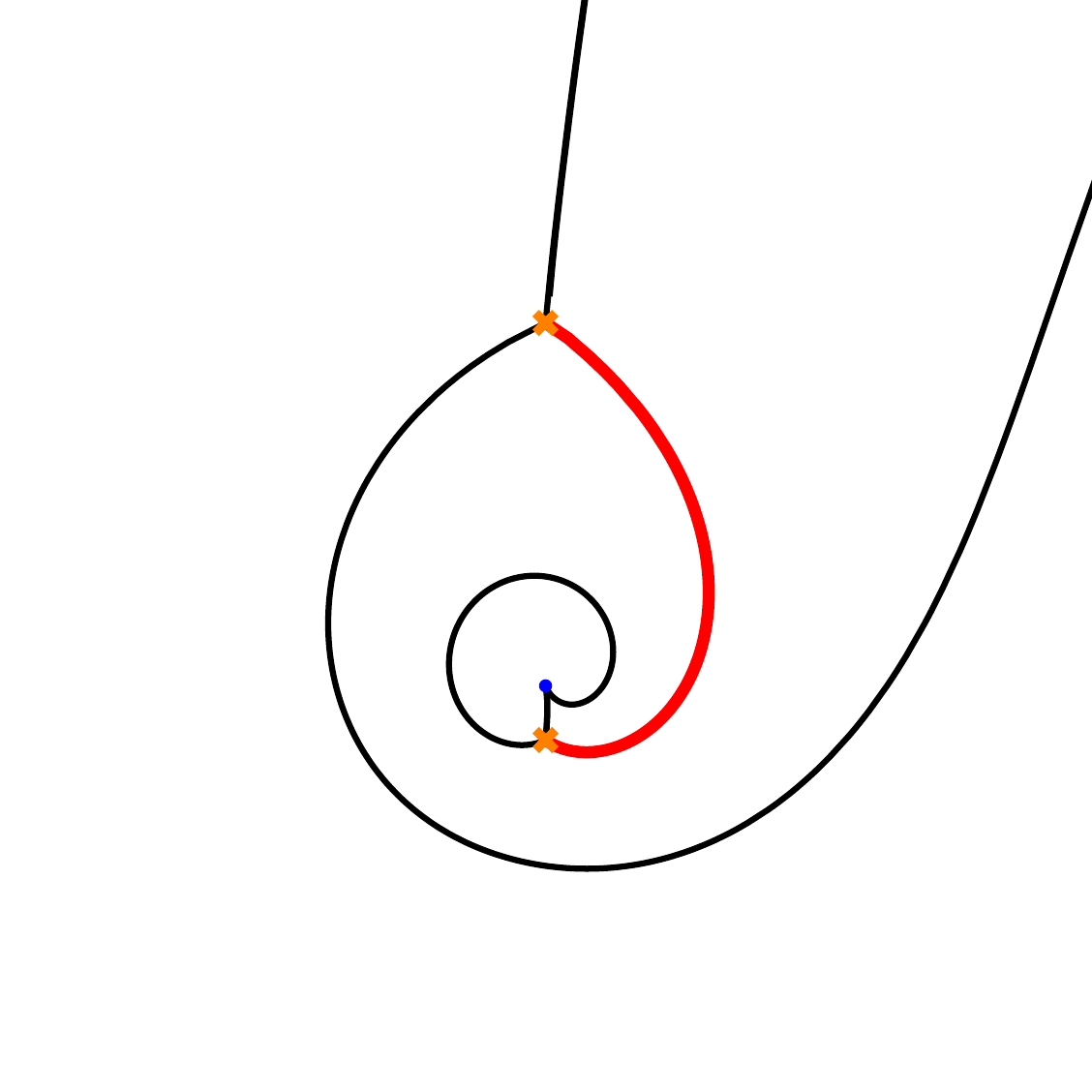}
\includegraphics[width=0.19\textwidth]{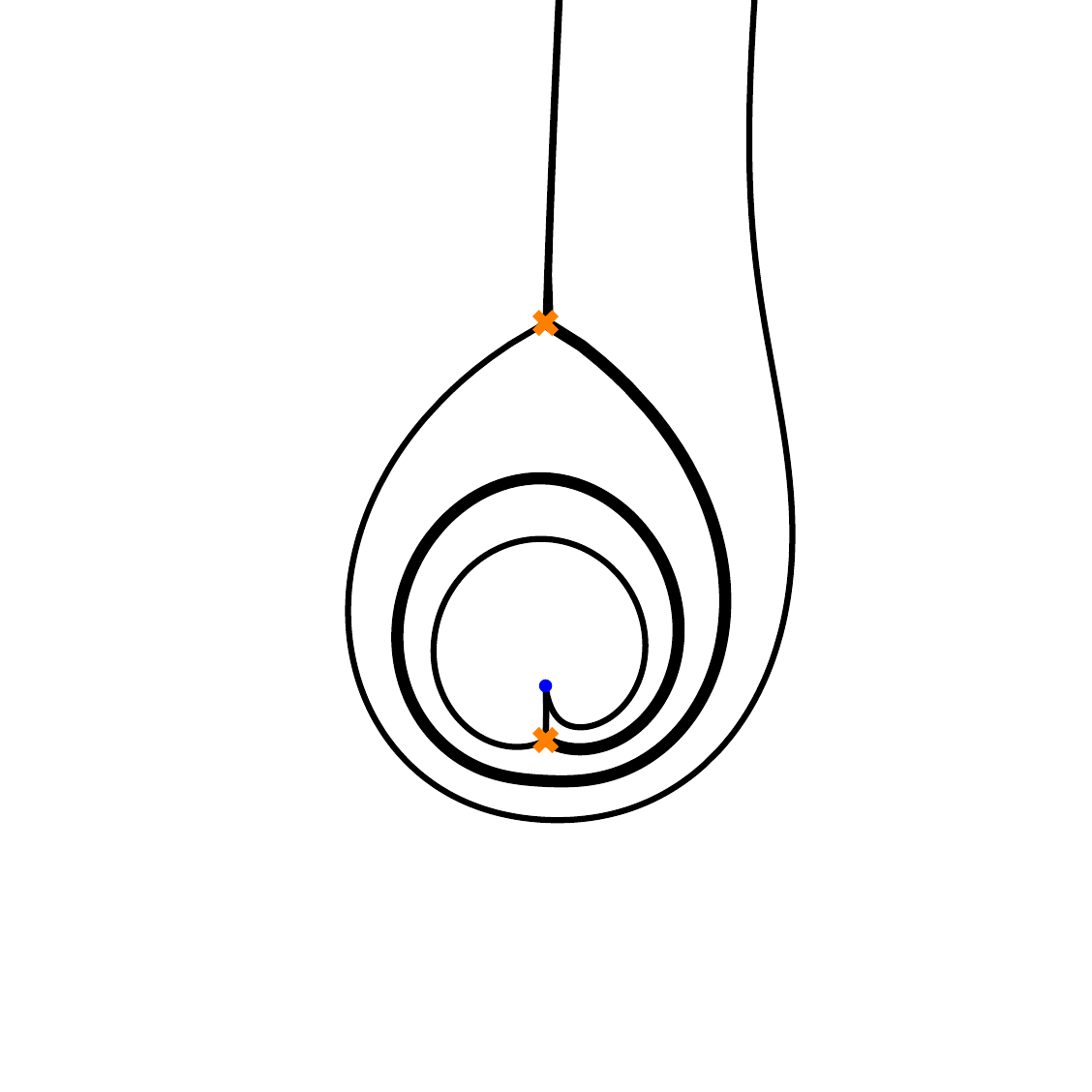}\\
\includegraphics[width=0.19\textwidth]{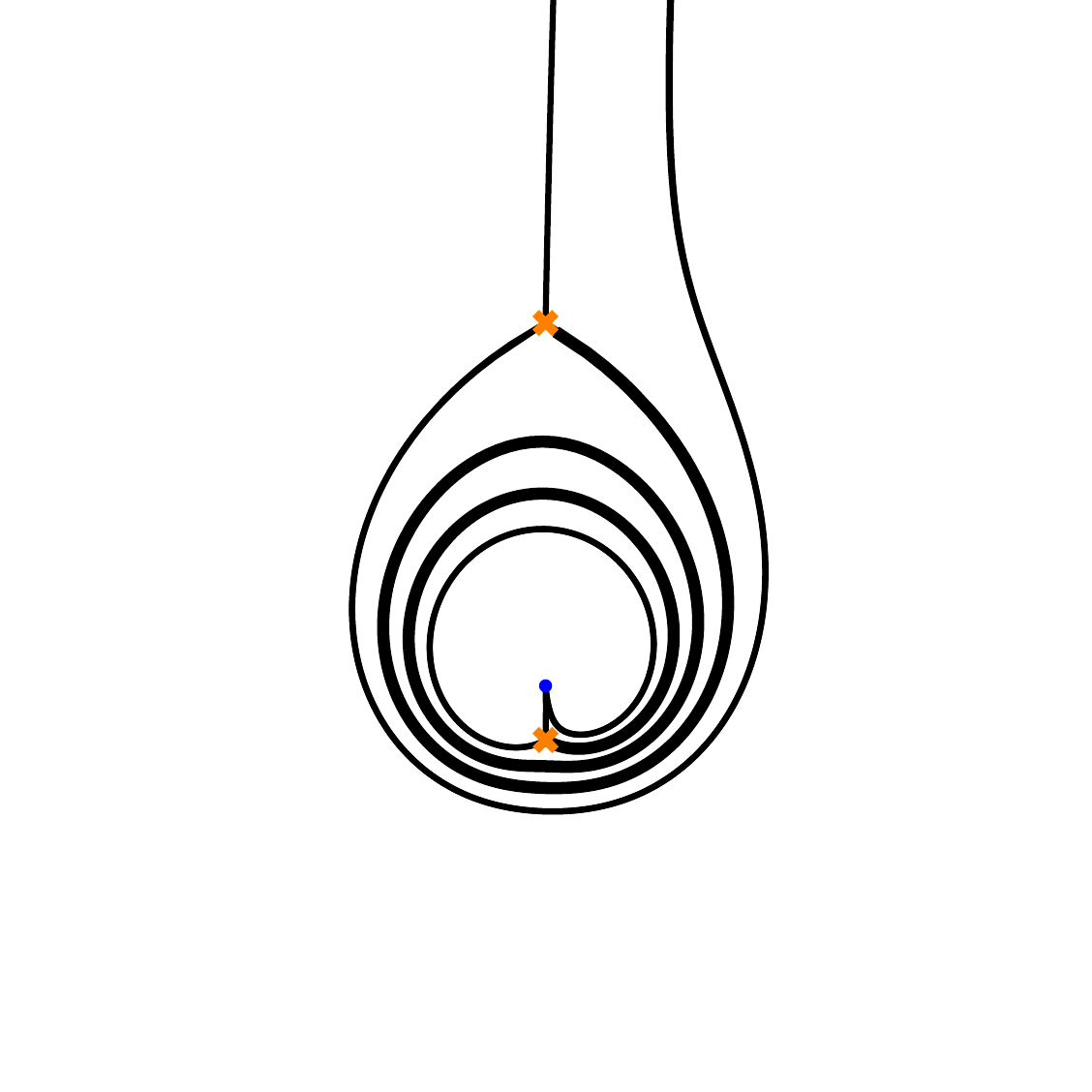}
\raisebox{40pt}{$\cdots$}
\includegraphics[width=0.19\textwidth]{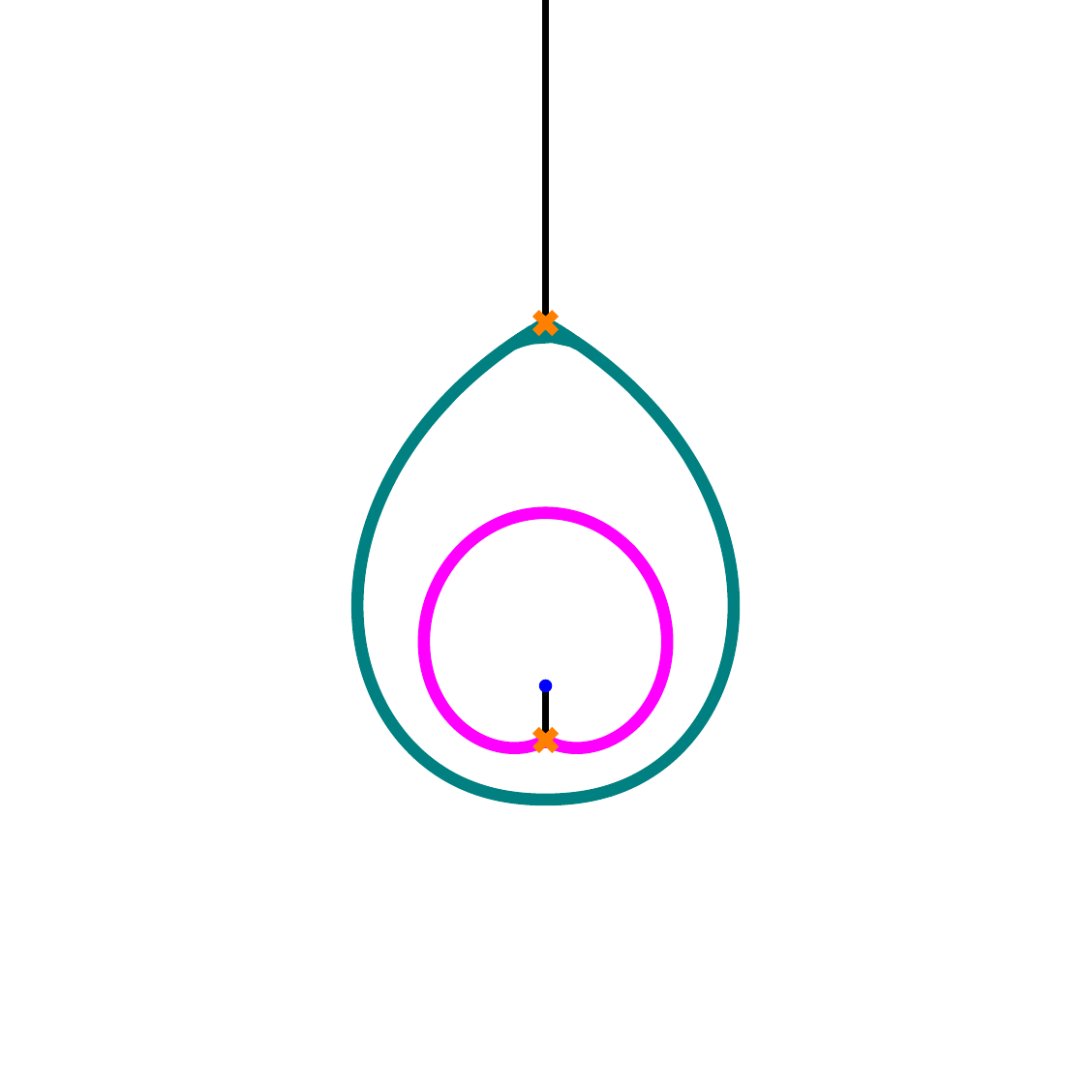}
\raisebox{40pt}{$\cdots$}
\includegraphics[width=0.19\textwidth]{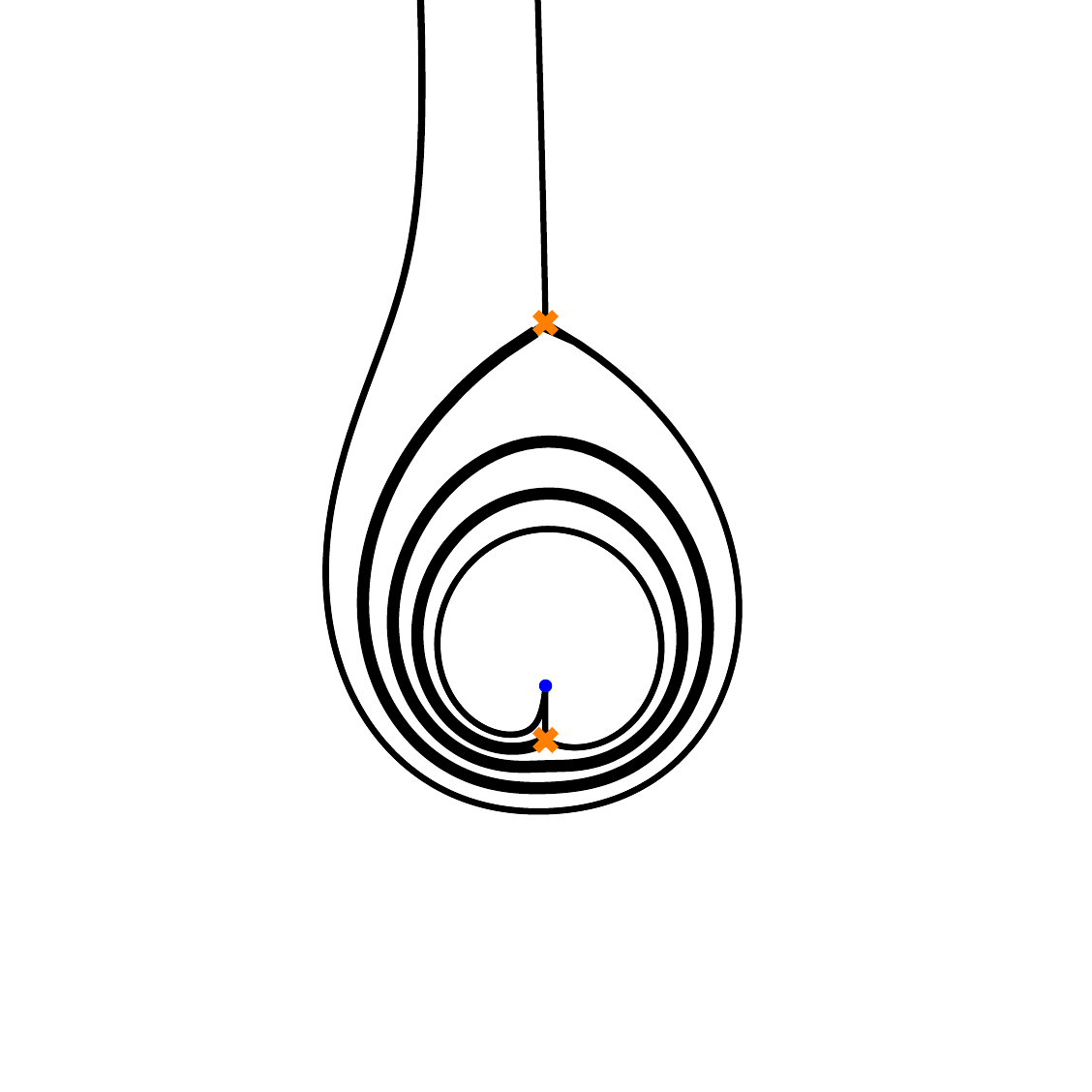}\\
\includegraphics[width=0.19\textwidth]{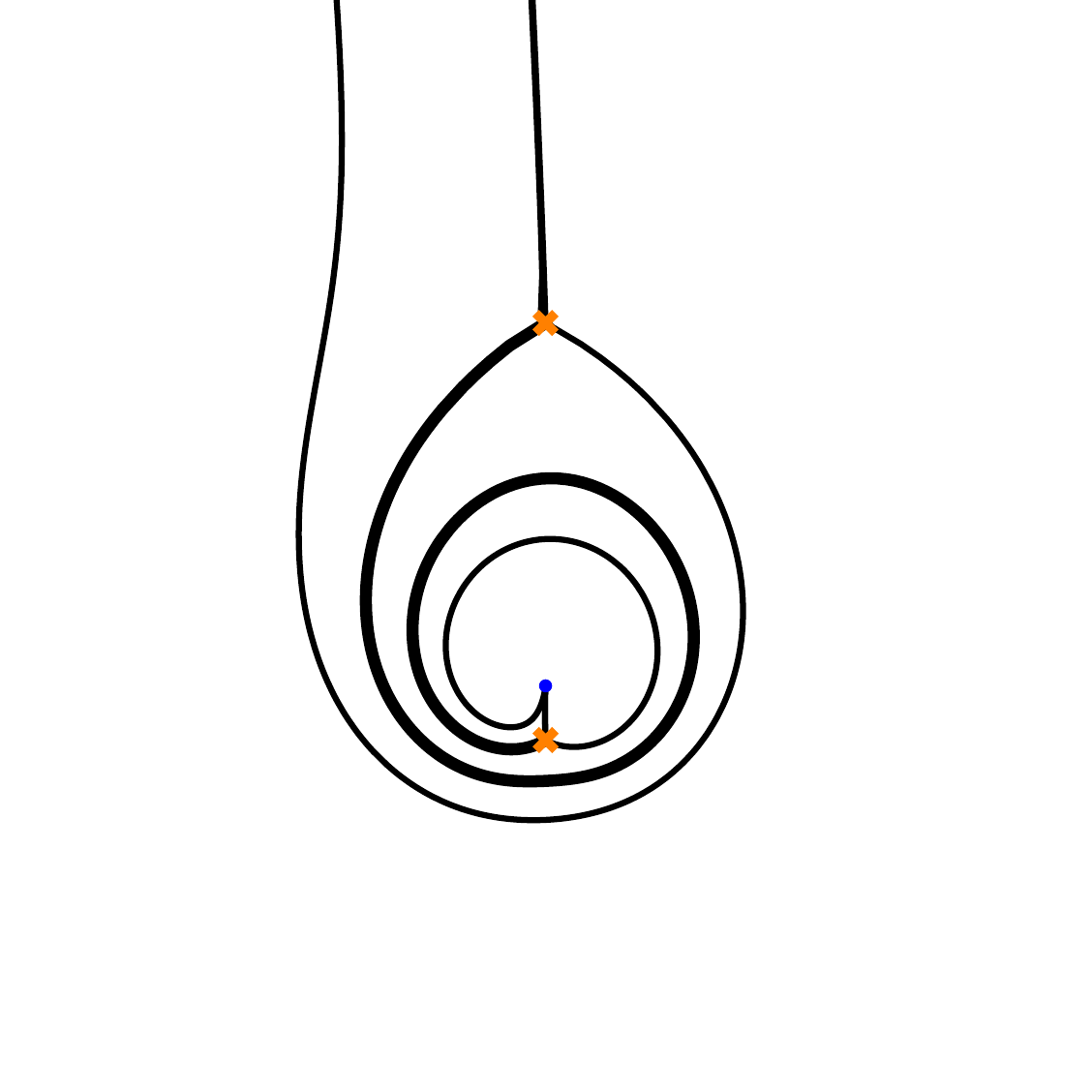}
\includegraphics[width=0.19\textwidth]{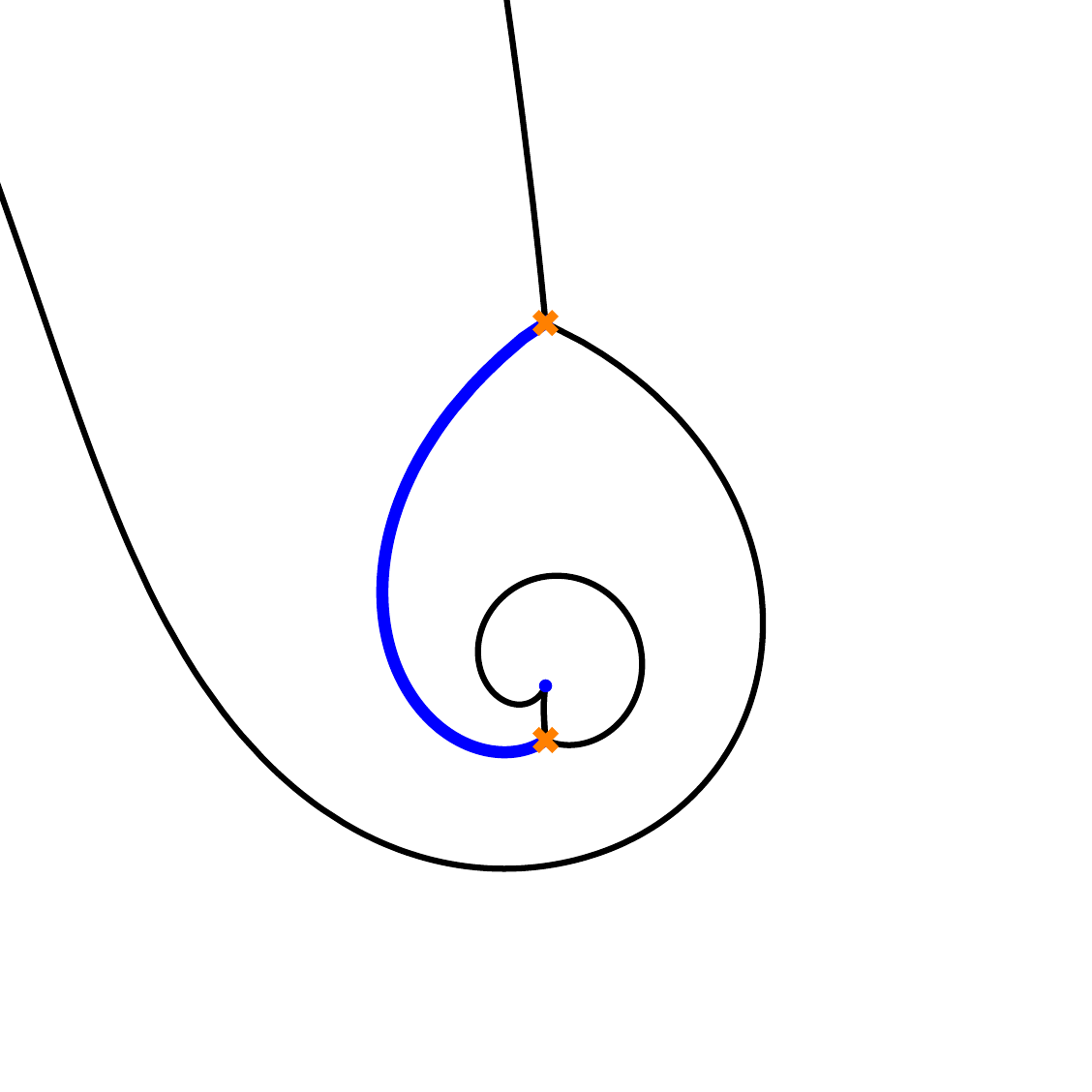}
\includegraphics[width=0.19\textwidth]{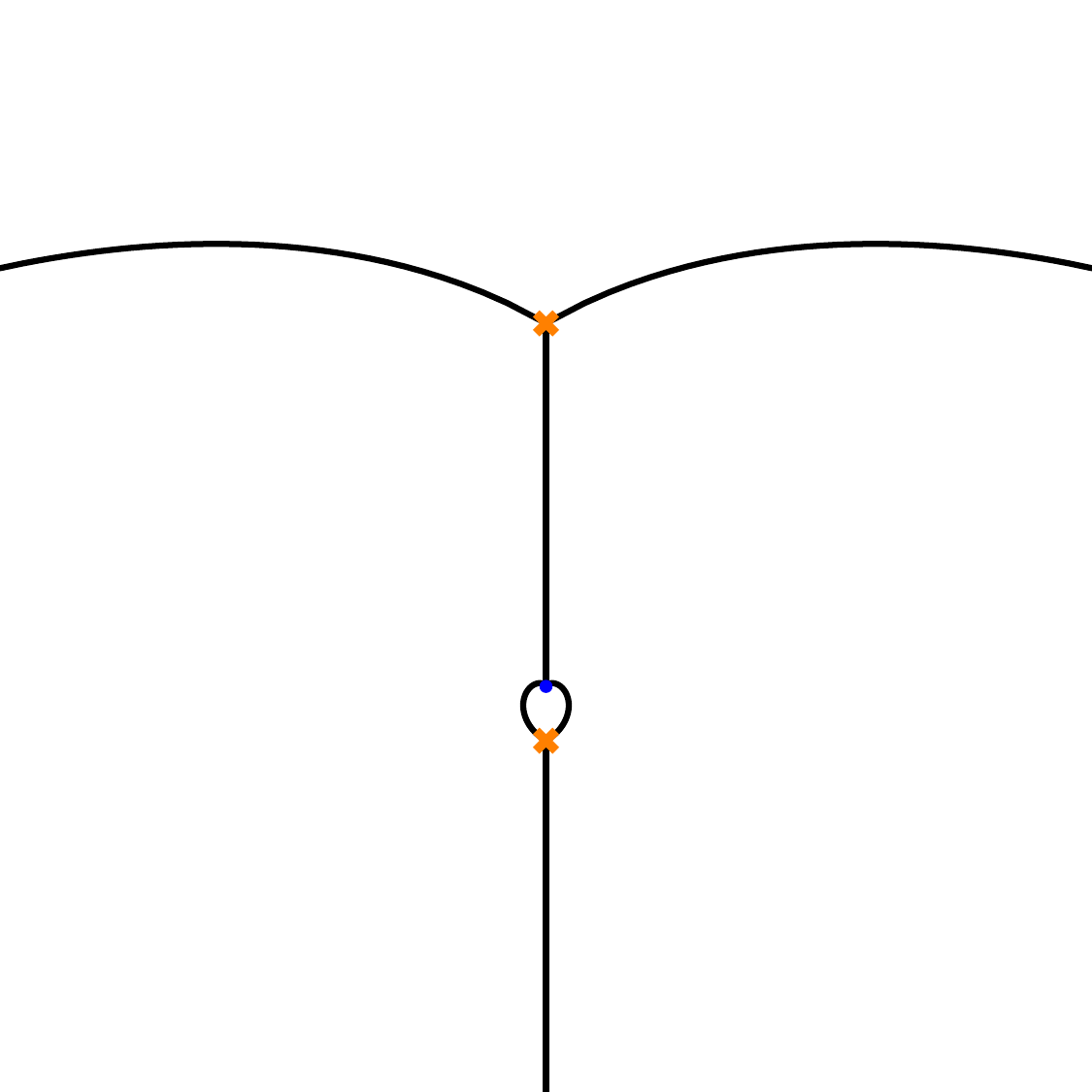}
\caption{Infinitely many holomorphic disks at $u=-2i$ for $0\leq \theta\leq \pi$.
}
\label{fig:YM-weak}
\end{center}
\end{figure}

From Theorems \ref{general wall crossing} and \ref{skein valued cluster}, we deduce: 
\begin{equation}\label{eq:SWwcf}
	\Psi_{(1,1)} \Psi_{(-1,1)} = 
	\Psi_{(-1,1)} \dots \Psi_{(-1,2n+1)} \dots
	\Psi_{A_{10}}^{-1} \Psi_{A_{01}}^{-1}
	\dots \Psi_{(1,2n+1)} \dots \Psi_{(1,1)}
\end{equation}
where ellipses contain contributions from all values of $n\in \mathbb{N}$.
Here $A_{01},A_{10} \in \Sk(T^2\times I)$ are connected curves in class $(0,2)$ with a negative (respectively, positive) crossing shown in Figure~\ref{fig:ymWCF}. See Remark \ref{rmk:VMfactors} for details about the associated factors $\Psi_{A_{10}}^{-1}\Psi_{A_{01}}^{-1}$.

\begin{remark}
Formula \eqref{eq:SWwcf} can also be derived by repeated application of the pentagon relation of Figure \ref{fig:pentagon}. The first steps of this procedure are shown in Figure \ref{fig:ymWCF}:
we first push the curves $(1,1)$ and $(-1,1)$ through each other, creating curves $A_{10}$ and $A_{01}$. 
After this step, further pushing curve $(1,1)$ down will cause it to cross $A_{10}$ generating further curves in class $(1,2n+1)$. Similarly, further pushing curve $(-1,1)$ upward will also generate curves in class $(2n+1)$ by unlinking with $A_{10}$. 
Note that $A_{01}$ does not participate in further creation of curves, as shown in Figure \ref{fig:YMwcf-deformed}.
\end{remark}

\begin{figure}
\begin{center}
\includegraphics[width=0.6\textwidth]{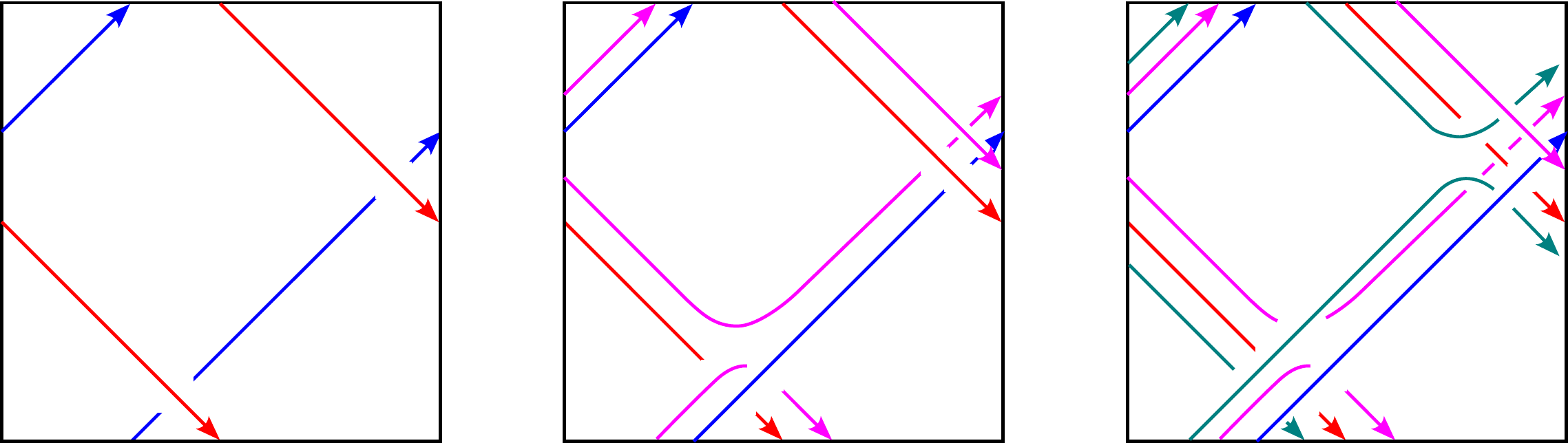}
\caption{Curves $(1,1)$ (red), $(-1,1)$ (blue), $A_{10}$ (magenta) and $A_{01}$ (green) after moving $(1,1)$ across $(-1,1)$.}
\label{fig:ymWCF}
\end{center}
\end{figure}

\begin{figure}
\begin{center}
\includegraphics[width=0.25\textwidth]{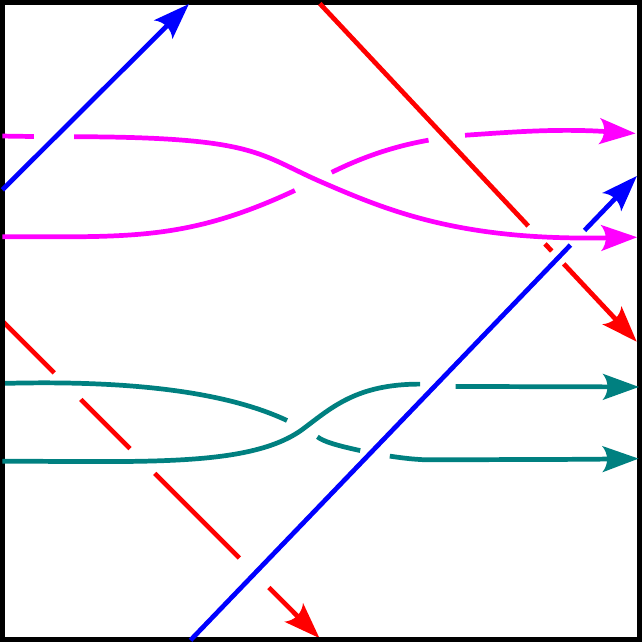}
\caption{Curves $(1,1)$, $(-1,1)$, $A_{10}$ and $A_{01}$.}
\label{fig:YMwcf-deformed}
\end{center}
\end{figure}

\begin{remark}\label{rmk:VMfactors}
$A_{10}$ and $A_{01}$ can be expressed in terms 
certain skein elements $P_{i,j}$ introduced by Morton and Samuelson \cite{Morton-Samuelson}:
\begin{equation}
\begin{split}
	&A_{10} = \frac{1}{2}(q^{1/2}+q^{-1/2}) P_{0,2} + \frac{1}{2} (q^{1/2}-q^{-1/2}) P_{0,1}^2 \\
	&A_{01} = \frac{1}{2}(q^{1/2}+q^{-1/2}) P_{0,2} - \frac{1}{2} (q^{1/2}-q^{-1/2}) P_{0,1}^2 \\
\end{split}
\end{equation}
The factors appearing in \eqref{eq:SWwcf} are defined as follows
\begin{equation}\label{eq:psi-A10-A01}
\begin{split}
    \Psi_{A_{10}}^{-1} &= \exp\left(\sum_{d\geq 1} \frac{1}{2d} \left(\frac{q^{1/2}+q^{-1/2}}{q^{1/2}-q^{-1/2}} P_{0,2d} + P_{0,d}^2\right)\right)\\
    \Psi_{A_{01}}^{-1} &= \exp\left(\sum_{d\geq 1} \frac{1}{2d} \left(\frac{q^{1/2}+q^{-1/2}}{q^{1/2}-q^{-1/2}} P_{0,2d} - P_{0,d}^2\right)\right)\,.
\end{split}
\end{equation}

Each term can be given a geometric interpretation by an application of \cite[Theorem 13]{morton2008geometrical}, which allows to compute insertions of any $Q\in \Sk(S^1\times D^2)$ along a curve in $\SkAlg(T^2\times I)$.
We have checked numerically that $\Psi_{A_{10}}^{-1}$ coincides with the insertion of a disk, i.e. a skein-valued dilogarithm \eqref{eq:explicit-skein-dilog}, along curve $A_{10}$.
Similarly, we have checked numerically that 
$\Psi_{A_{01}}$ (i.e. the second line in \eqref{eq:psi-A10-A01} with opposite overall sign in the exponent) coincides with the insertion of an anti-disk, i.e. the
inverse of the skein-valued dilogarithm \eqref{eq:skein-dilog-inverse}, along $A_{01}$.
To obtain $\Psi_{A_{01}}^{-1}$ one needs to compute the multiplicative inverse. As an operation in the skein, this is complicated by the fact that annihilation of curves in the product $\Psi_{A_{01}}^{-1} \Psi_{A_{01}}$ involves moving (multiple coverings of) the disk and anti-disk onto each other, which induces repeated wall-crossings among their boundaries.
\end{remark}

\begin{remark} 
    Formula \eqref{eq:SWwcf} is a skein-valued lift of the  wall-crossing formula for 4d $\mathcal{N}=2$ $SU(2)$ Yang-Mills theory which was previously identified with the Kontsevich-Soibelman wall-crossing formula for the CY3 category associated to the Kronecker quiver  \cite[discussion below Conjecture 1]{Kontsevich:2008fj} and \cite[Equation (2.26)]{Gaiotto-Moore-Neitzke-4dwallvia3dfield}.
The `motivic' counterpart of the Kontsevich-Soibelman wall-crossing formula is known to give, in this case, the spin refinement of the Seiberg-Witten wall crossing \cite[Equation (1.3) and Section 4.1]{Dimofte:2009tm}.
Our wall crossing formula (arrived at by entirely different considerations, having nothing to do with the Donaldson-Thomas invariants of a Calabi-Yau category) in fact specializes 
to the refined/motivic formula by specializing \eqref{eq:SWwcf} to the $\mathfrak{gl}(1)$ skein via \eqref{eq:HOMLYPT-2-gl1}.  
In particular, we note that the $\mathfrak{gl}(1)$-specialization of the middle factors $\Psi_{A_{10}}^{-1} \Psi_{A_{01}}^{-1}$ is $\Phi(q^{1/2} \hat x^2)^{-1}\Phi(q^{-1/2} \hat x^2)^{-1}$ where $\Phi(\xi) = (q^{1/2} \xi;q)_\infty^{-1}$, 
in agreement with \cite[Equation (4.8)]{Dimofte:2009tm}, \cite[Section 4.1]{Galakhov:2014xba}, and \cite[Section 8.4]{Neitzke-Yan-nonabelianization}. 
\end{remark}

\begin{question}
    What (if any) question in the Donaldson-Thomas theory of the CY3 category associated to the Kronecker quiver -- or what (if any) question about $\mathcal{N}=2$ supersymmetric Yang-Mills theory -- corresponds to the skein valued wall crossing \eqref{eq:SWwcf}?
\end{question}

\appendix

\section{Flow graphs and holomorphic curves} \label{flow graphs and holomorphic curves}

\subsection{The flow graph limit}\label{ssec: curves to flow graphs}
Consider a Lagrangian $Q\subset T^\ast M$. We will take $Q$ to lie very close to the zero section. We obtain this by fiber scaling $\sigma_\lambda\colon T^\ast M\to T^\ast M$, $\sigma_\lambda(q,p)=(q,\lambda p)$. Since $\sigma_\lambda^\ast\omega=\lambda\omega$, $\sigma_\lambda$ preserves Lagrangians and in the non-exact case it scales the flux. Let $Q_\lambda=\sigma_\lambda(Q)$. 

We first consider holomorphic curves with boundary on $Q_\lambda$ only. If we consider the standard Gromov limit of such curves we have the following straightforward result.
\begin{lemma}
Let $u_\lambda\colon (\Sigma_\lambda,\partial\Sigma_\lambda)\to (T^\ast M,Q_\lambda)$ be a family of holomorphic curves, then some subsequence $u_\lambda$ Gromov converges to a constant curve $u_0\colon (\Sigma_0,\partial\Sigma_0)$ mapping into $M$ as $\lambda\to 0$.    
\end{lemma}
\begin{proof}
This follows from the standard argument for Gromov compactness. (Note that there cannot be any derivative blow up: such blow up would result in a non-constant holomorphic curve in $T^\ast M$ with boundary in $M$).      
\end{proof}

We write $(q,p)$ for standard coordinates on $T^\ast M$ and use the almost complex structure $J$ on $T^\ast M$ (in some neighborhood of the zero section) induced by a Riemannian metric on $M$. The following result follows from the maximum principle for harmonic functions.
\begin{lemma}\label{l: estimate1}\cite[Lemma 5.4]{Ekholm-morse}
If $u=(q,p)\colon(\Sigma,\partial\Sigma)\to (T^\ast M,Q_\lambda)$ is a $J$-holomorphic map then 
$\sup_\Sigma|p|=\mathcal{O}(\lambda)$. \qed
\end{lemma}

Lemma \ref{l: estimate1} shows that we can rescale holomorphic curves $u_\lambda$ locally by $\lambda^{-1}$. Such rescaled map also satisfies a Cauchy-Riemann equation and, provided the derivative of the rescaled map is bounded, it converges to the fiber strip over a local Morse flow line. 
As explained in \cite[Section 5, Corollary 5.24]{Ekholm-morse}, it is possible to add, given that $S$ meets a finite number of general position conditions, finitely many boundary punctures on the domain of $u_\lambda$ so that for the punctured domain the rescaled derivative is uniformly bounded. 
Hence after adding a finite number of punctures the holomorphic map converges (in this rescaled sense as is standard for flow trees, see \cite[Section 5.4.4]{Ekholm-morse}) to a flow graph. We refer to this rescaled convergence as \emph{flow graph convergence}, which in particular gives $C^0$-convergence everywhere (and $C^1$-convergence outside neighborhoods of $\mathcal{O}(\lambda\log \lambda)$-neighborhoods of the graph vertices, see \cite[Lemma 5.13]{Ekholm-morse}). We state this result on flow graph convergence as a lemma.

\begin{lemma}\label{l : flow graph converegnce 1}
Any sequence of holomorphic curves $u_\lambda$ with boundary on $Q_\lambda$ has a subsequence that flow graph converges to a flow graph $\Gamma$ of $Q$. \qed      
\end{lemma}

Dimension formulas for flow graphs and holomorphic curves agree, and after small perturbation, there is therefore a discrete set of flow graphs that are finite below any fixed Euler characteristic, compare \cite[Section 3]{Ekholm-morse}.

\subsection{Flow graphs attached to curves}\label{ssec: curves to curves + graphs}

Suppose now that $M \subset X$ is a Lagrangian, $Q \subset D^*M \subset X$ is contained in a Weinstein neighborhood of $M$, and we study curves with boundary on $Q$ and possibly also on other Lagrangians $L \subset X \setminus D^*M$. 
We take $Q_\lambda\subset D^*M$, $\lambda>0$, as the corresponding family of fiber-scaled Lagrangians.  

In this case, as $\lambda \to 0$, holomorphic curves with boundary on $Q_\lambda \sqcup L$ concentrate near holomorphic curves with boundary on $M \sqcup L$ with flow graphs of $Q\subset T^\ast M$ attached along their boundaries in $M$.   Such configurations were called quantum flow trees in \cite{EENS} and were the key to the calculation of the knot contact homology differential for any link. Technical results needed to establish the results stated here are found in \cite[Section 5]{EENS}. 

We first consider the standard Gromov limit of curves in $(X,Q_\lambda)$. We have the following.
\begin{lemma}\label{l : usual Gromov}
Let $u_\lambda\colon (\Sigma_\lambda,\partial\Sigma_\lambda)\to (X,Q_\lambda)$ be a family of holomorphic curves. Then some subsequence $u_\lambda$ Gromov converges to a curve $u_0\colon (\Sigma_0,\partial\Sigma_0)\to (X,M)$ as $\lambda\to 0$.  
\end{lemma}
\begin{proof}
This follows from the standard argument for Gromov compactness. The number of non-zero area bubbles in the limit is controlled by the initial area bound of the curve. For sequences with bounded derivative, the usual Arzela-Ascoli argument gives a sequence that converges to a holomorphic curve with boundary on $L\cup M$. Since the action and flux of $S_\lambda$ is $\mathcal{O}(\lambda)$ this accounts for all area and we get Gromov convergence.
\end{proof}

We point out that unlike Gromov compactness in the case when the Lagrangian does not degenerate, in Lemma \ref{l : usual Gromov}, constant components of the limit may be located far away from the positive area components of the lifting curve. As before, to get a more precise description, we need to rescale.

We consider this more refined picture of the limit as $\lambda\to 0$. 
As in \cite[Section 5]{EENS}, we consider the domains $(\Sigma_\lambda,\partial\Sigma_\lambda)$ in a sequence of curves $u_\lambda$. Arguing as in \cite[Lemma 5.8 and 5.13]{EENS} we find subsets $(\Sigma_0,\partial\Sigma_0)$ of the domain bounded by extremal short segments where the map satisfies $\mathcal{O}(\lambda)$ derivative bounds and where $\area(u_\lambda|\Sigma_0)=\mathcal{O}(\lambda)$. As in Lemma \ref{l : flow graph converegnce 1} we then find that the restriction $u_\lambda|\Sigma_0$ can be rescaled and converges to flow graphs, and by the area argument in \cite[Lemma 5.13]{EENS} these flow graphs begin and end on the boundary of the positive area part of the limit. We say that such a configuration is a \emph{holomorphic curve with a flow graph attached}. This then leads to the following result.   

\begin{lemma}\label{l : flow graph converegnce 3}
Any sequence of holomorphic curves $u_\lambda\colon (\Sigma_\lambda,\partial\Sigma_\lambda)\to (X,Q_\lambda)$ has a subsequence that flow graph converges to a holomorphic curve $u_M\colon(\Sigma,\partial\Sigma)\to(X,M)$ with a flow graph $\Gamma$ of $Q_\lambda\subset T^\ast M$ attached along the boundary components in $M$. 
If $C$ is embedded and transversely cut out and the flow graph configuration is transversely cut out as well, then for sufficiently small $\lambda>0$ there is a unique transversely cut out holomorphic curve in $N$ with boundary (piecewise) $C^1$-close to the cotangent lift of the flow graph configuration.   
\end{lemma}
\begin{proof}
The convergence statement was discussed above.  The last statement follows from results in \cite[Section 5.4]{EENS}. We give a brief description of the argument given there. 

The existence and uniqueness of holomorphic curves near flow graph configurations is a consequence of a gluing theorem that contains the two usual ingredients in Floer gluing: establish conditions for Newton iteration and show that the resulting gluing map is surjective. In the case at hand, the domain of the curve constructed is a fat graph attached to the domain of the big curve. For generic configurations there are two local domain models in the graph: strip regions (along flow edges, switches, and ends) and strips with a slit (near trivalent vertices). In all these local models both solutions and their linearizations near the limit can be explicitly determined, see \cite[Section 6.1]{Ekholm-morse}. There is also a local model for junction points where the graph pieces are attached to the big curve, see \cite[Section 5.4.2]{EENS}. 

Along these local model domains there are Sobolev spaces of maps with positive exponential weights, pieces are glued and interpolated in bounded regions where the weight function is bounded, see \cite[Sections 6.2--3]{Ekholm-morse}. This leads to a uniform Fredholm problem and a pregluing that satisfies the necessary quadratic estimate for Floer iteration, see \cite[Section 6.4]{Ekholm-morse} and \cite[Section 5.4.3]{EENS}. The surjectivity is then established using the fact that along the standard pieces, control of the function norm in the bounded patching regions gives control in the weighted norm over the whole standard piece, see \cite[Section 6.4.4]{Ekholm-morse} and \cite[Section 5.4.3]{EENS}.         
\end{proof}

Lemma \ref{l : flow graph converegnce 3} gives a one-to-one correspondence between holomorphic curves with boundary on $Q_\lambda$ and holomorphic curves with boundary on $M$ with flow graphs attached provided the holomorphic curves in the flow graphs are transversely cut out and simple (i.e., nowhere multiple covered). We next show that if the big holomorphic curve $C$ in the limit is embedded near its boundary, then flow graph edges of limiting curves that are attached to $C$ cannot be multiply covered. 
\begin{lemma}\label{l : simple implies simple}
Let $u_\lambda\colon (\Sigma_\lambda,\partial\Sigma_\lambda)\to (X,Q_\lambda)$ be a sequence that flow graph converges to a holomorphic curve $u_M\colon(\Sigma,\partial\Sigma)\to(X,M)$ with a flow graph $\Gamma$ of $Q_\lambda\subset T^\ast M$ attached along the boundary components in $M$ with all attaching points distinct and assume that $u_M$ is an embedding in a neighborhood of $\partial \Sigma$. Then any edge of $\Gamma$ is non-multiply covered. 
\end{lemma}

\begin{proof}
We first consider an edge in $\Gamma$ attached to $u_M$.
Consider a neighborhood $(\C^3,\R^3)$, where $\R^3$ corresponds to $L$ of the point where the flow line is attached. We take coordinates so that the big curve lies in the first coordinate line. We take one of the sheets of $Q$ to be $\R^3$ and the other to be the graph of $\lambda dx_2$. (Note that the flow tree rescaling, multiplication by $\lambda^{-1}$, takes the holomorphic curve near the flow graph configuration to a standard holomorphic curve, with boundary on $\R^3$ and the graph of $dx_2$, up to an error term of size $\mathcal{O}(\lambda)$, see \cite[Lemma 5.12]{Ekholm-morse}.)

Projecting the curve with flow line attached to the complex line spanned by $\partial_{x_1}+\partial_{x_2}$ we find that the curve projects to the upper half-plane with boundary on the real line and a slit along the line with imaginary part $\lambda>0$. The assumption that $u_M$ is an embedding near the boundary implies that the degree of the half plane equals $1$ above the line with imaginary part $\lambda$. If the edge attached was a multiple cover, then the degree below the slit would be $>1$. This is not possible and hence the edge cannot be multiply covered. 

We next show that no other edge of $\Gamma$ can be multiply covered. Note first that it is clear that the multiplicity of a curve is constant along an edge in $\Gamma$. We check that it remains constant over vertices. As above, near a vertex there is a standard model for the projected boundary condition after rescaling, sheets without cusp edges correspond to lines in $\C$ parallel to the real axis, a those with cusp edge correspond to a half circles with lines attached. In the limit, the projected curve is holomorphic, possibly with interior and boundary branch points, and it is straightforward to check that edge multiplicity does not change. Since edges of $\Gamma$ attached to the big curve $u_M$ have multiplicity one, it follows that all edges of $\Gamma$ have multiplicity one.  
\end{proof}

\bibliographystyle{plain}
\bibliography{skeinrefs}

\end{document}